\documentclass{amsart}
\usepackage{amsmath,amssymb,amsthm,amsfonts,amscd,mathrsfs,mathtools}
\usepackage{comment}
\usepackage{tikz-cd}
\usepackage[hyphens]{url}
\usepackage{enumitem}
\usepackage{caption}
\usepackage{hyperref}
\usepackage{overpic}
\usepackage[utf8]{inputenc}

\theoremstyle{plain}
\newtheorem{thm}{Theorem}[section]
\newtheorem{lem}[thm]   {Lemma}
\newtheorem{cor}[thm]   {Corollary}
\newtheorem{prop}[thm]  {Proposition}

\newtheorem{athm}{Theorem}

\newtheorem{acor}{Corollary}[athm]

\theoremstyle{definition}
\newtheorem{defn}[thm]  {Definition}
\newtheorem{conj}[thm]{Conjecture}
\newtheorem{ex}[thm]{Example}
\newtheorem{rem}[thm]{Remark}
\newtheorem{nota}[thm]{Notation}

\newcommand{\fa}{\mathfrak{a}} 
\newcommand{\fA}{\mathfrak{A}} 
\newcommand{\fB}{\mathfrak{B}} 
\newcommand{\bB}{\mathbf{B}} 
\newcommand{\bS}{\mathbb{S}} 
\newcommand{\cC}{\mathcal{C}} 
\newcommand{\cD}{\mathcal{D}} 
\newcommand{\fd}{\mathfrak{d}} 
\newcommand{\fe}{\mathfrak{e}} 
\newcommand{\cE}{\mathcal{E}} 
\newcommand{\bE}{\mathbf{E}} 
\newcommand{\fF}{\mathfrak{F}} 
\newcommand{\fG}{\mathfrak{G}} 
\newcommand{\fI}{\mathfrak{I}} 
\newcommand{\fj}{\mathfrak{j}} 
\newcommand{\fJ}{\mathfrak{J}} 
\newcommand{\cK}{\mathcal{K}} 
\newcommand{\cL}{\mathcal{L}} 
\newcommand{\bk}{\mathbf{k}} 
\newcommand{\N}{\mathbb{N}} 
\newcommand{\mO}{\mathrm{O}} 
\newcommand{\cO}{\mathcal{O}} 
\newcommand{\fo}{\mathfrak{o}} 
\newcommand{\cP}{\mathcal{P}} 
\newcommand{\fr}{\mathfrak{r}} 
\newcommand{\fS}{\mathfrak{S}} 
\newcommand{\fs}{\mathfrak{s}} 
\newcommand{\cS}{\mathcal{S}} 
\newcommand{\cT}{\mathcal{T}} 
\newcommand{\ft}{\mathfrak{t}} 
\newcommand{\fU}{\mathfrak{U}} 
\newcommand{\cU}{\mathcal{U}} 
\newcommand{\cV}{\mathcal{V}} 
\newcommand{\bx}{\mathbf{x}} 
\newcommand{\by}{\mathbf{y}}
\newcommand{\bt}{\mathbf{t}}
\newcommand{\bz}{\mathbf{z}}
\newcommand{\Z}{\mathbb{Z}} 

\newcommand{\bDelta}{\mathbf{\Delta}} 
\newcommand{\inert}{\mathrm{int}}
\newcommand{\rinert}{\mathrm{r}\inert}
\newcommand{\linert}{\mathrm{l}\inert}
\newcommand{\vinert}{\mathrm{v}\inert}
\newcommand{\hinert}{\mathrm{h}\inert}
\renewcommand{\special}{\mathrm{special}}
\newcommand{\Cat}{\mathrm{Cat}} 
\newcommand{\infone}{{(\infty,1)}}
\newcommand{\inftwo}{{(\infty,2)}}
\newcommand{\infnm}{{(\infty,n-1)}}
\newcommand{\infn}{{(\infty,n)}}
\newcommand{\circhor}{\circ_{\mathrm{h}}} 
\newcommand{\Catinfone}{\Cat_\infone} 
\newcommand{\RelCatinfone}{\mathrm{Rel}\Catinfone}
\newcommand{\Catinftwo}{\Cat_\inftwo}
\newcommand{\Catinfnm}{\Cat_\infnm} 
\newcommand{\Catinfn}{\Cat_\infn} 
\newcommand{\twosimeq}{{2\simeq}} 
\newcommand{\nsimeq}{{n\simeq}} 
\newcommand{\CAlg}{\mathrm{CAlg}} 
\newcommand{\DDat}{\mathrm{DDat}} 
\newcommand{\colim}{\mathrm{colim}}
\newcommand{\ev}{\mathrm{ev}} 
\newcommand{\Fin}{\mathrm{Fin}} 
\newcommand{\Fininj}{\Fin^{\mathrm{inj}}} 
\newcommand{\Fun}{\mathrm{Fun}} 
\newcommand{\llax}{{\text{l-lax}}} 
\newcommand{\rlax}{{\text{r-lax}}} 
\newcommand{\Mod}{\mathrm{Mod}} 
\newcommand{\one}{\mathbf{1}} 
\newcommand{\op}{\mathrm{op}}
\newcommand{\PSh}{\mathscr{P}} 
\DeclareFontFamily{U}{min}{}
\DeclareFontShape{U}{min}{m}{n}{<-> udmj30}{}
\newcommand\yo{\!\text{\usefont{U}{min}{m}{n}\symbol{'207}}\!} 
\newcommand{\Day}{\mathrm{Day}} 
\newcommand{\Cospan}{\mathrm{Cospan}} 
\newcommand{\bCospan}{\mathbb{C}\mathrm{ospan}} 
\newcommand{\Tw}{\mathrm{Tw}} 
\newcommand{\Top}{\mathrm{Top}} 

\newcommand{\CFrob}{\mathrm{CFrob}} 
\newcommand{\GL}{\mathbf{GL}} 
\newcommand{\SGL}{\mathbf{SGL}} 
\newcommand{\tbeta}{\tilde\beta} 
\newcommand{\Gaf}{\mathrm{Gaf}} 
\newcommand{\bGr}{{\mathbb{G}\mathrm{r}}} 
\newcommand{\Gr}{\mathrm{Gr}} 
\newcommand{\VE}{\mathrm{VE}}
\newcommand{\CE}{\mathrm{CE}}
\newcommand{\conn}{\mathrm{conn}} 

\newcommand{\hAut}{\mathrm{hAut}} 
\newcommand{\map}{\mathrm{map}} 
\newcommand{\Out}{\mathrm{Out}} 

\newcommand{\pa}[1]{\left(#1\right)}
\newcommand{\set}[1]{\left\{#1\right\}}
\newcommand{\hto}{\hookrightarrow}
\newcommand{\ot}{\leftarrow}
\newcommand{\del}{\partial}
\newcommand{\Id}{\mathrm{Id}} 
\newcommand{\ul}[1]{\underline{#1}} 
\newcommand{\sm}{\mathrm{sm}}


\begin{document}

\title[Symmetric monoidal extensions and graph cobordisms]{Symmetric monoidal extensions and graph cobordisms between finite sets}

\author{Andrea Bianchi}
\address{Max Planck Institute for Mathematics,
Vivatsgasse 7, Bonn,
Germany
}
\email{bianchi@mpim-bonn.mpg.de}

\date{\today}
\begin{abstract}
Given a symmetric monoidal $(\infty,n)$-category $\mathcal{C}$ and a space $X$, we address the problem of explicitly describing the symmetric monoidal $(\infty,n)$-category freely obtained from $\mathcal{C}$ by adjoining $X$ new $n$-morphisms with prescribed sources and targets. We develop an apparatus of tools that allow one to detect in concrete situations such a free symmetric monoidal extension. As motivating application, we introduce a symmetric monoidal $(\infty,2)$-category ${\mathbb{G}\mathrm{r}}$ of graph cobordisms between finite sets, following classical constructions of Gersten, Culler--Vogtmann and Hatcher--Vogtmann, and we exhibit it as an extension of the symmetric monoidal $(\infty,1)$-category $\mathrm{Fin}$ of finite sets, obtained by freely adjoining a specific list of new 1-morphisms and 2-morphisms. We recover results of Barkan--Steinebrunner and of Galatius.

\end{abstract}
\maketitle
\tableofcontents
\section{Introduction}
\subsection{Statement of results}
The main goal of the article is to introduce a symmetric monoidal 2-category $\bGr$ of graph cobordisms between finite sets and to study its properties.
Roughly speaking, objects in $\bGr$ are finite sets, 1-morphisms are graph cobordisms between finite sets, and 2-morphisms are maps between graphs cobordisms obtained by collapsing trees to points. Taking disjoint unions of finite sets and graphs makes $\bGr$ into a symmetric monoidal $\inftwo$-category.
The precise construction, given in  Section \ref{sec:graphcobsets}, expands on classical constructions of Gersten \cite{Gersten}, Culler--Vogtmann \cite{CullerVogtmann}, and Hatcher--Vogtmann\cite{HatcherVogtmann}. Most notably, the classifying space of endomorphism category $\bGr(\emptyset,\emptyset)$ of the object $\emptyset\in\bGr$ contains for all $n\ge2$ a copy of $\bB\Out(F_n)$, the classifying space of the group of outer automorphisms of a free group of rank $n$.

The main result of the article is the following ``universal property'', or ``presentation'' for $\bGr$, when considered as a symmetric monoidal $\inftwo$-category. For this, we introduce in Definition \ref{defn:graphlike} the notion of \emph{graph-like structure} on a symmetric monoidal $\inftwo$-category $\cD$: roughly speaking, a graph-like structure is the datum of a symmetric monoidal $\inftwo$-functor $\Fin\to\cD$ together with the specification of a certain number of 1-morphisms and 2-morphisms in $\cD$. Here $\Fin$ denotes the symmetric monoidal category of finite sets and maps of finite sets, with the disjoint union monoidal structure. We shall in fact introduce in Definition \ref{defn:GLfunctor} a space $\GL(\cD)$ of graph-like structures on $\cD$, functorially in $\cD\in\Catinftwo$.

\begin{athm}
 \label{thm:A}
 The functor $\GL\colon\CAlg(\Catinftwo)\to\cS$, associating with a symmetric monoidal $\inftwo$-category $\cD$ its space of graph-like structures, is represented by $\bGr$, i.e. we have a natural isomorphism $\Fun^\otimes(\bGr,-)^\simeq\simeq\GL$.
\end{athm}
The universal property of the symmetric monoidal $\inftwo$-category $\bGr$ given in Theorem \ref{thm:A} allows us to deduce a universal property for the symmetric monoidal $\infone$-category $\Gr:=|\Gr|_2$ obtained from $\bGr$ by inverting all 2-morphisms. This result, stated as Corollary \ref{cor:A1}, is originally due to Barkan--Steinebrunner \cite{BarkanSteinebrunner}. Compare also with \cite[§4.4]{SteinebrunnerNotes}, in particular with Corollary 4.22 in loc.cit.
\begin{acor}[\cite{BarkanSteinebrunner}]
\label{cor:A1}
The functor $\CFrob\colon\CAlg(\Catinfone)\to\cS$, associating with a symmetric monoidal $\infone$-category its moduli space of commutative Frobenius algebras (see Definition \ref{defn:Frobeniusalgebra}), is represented by $\Gr$, i.e. we have a natural equivalence $\Fun^\otimes(\Gr,-)^\simeq\simeq\CFrob$ of functors $\CAlg(\Catinfone)\to\cS$.
\end{acor}
In turn, Corollary \ref{cor:A1} allows us to compute explicitly the classifying space $|\Gr|$ of $\Gr$ as a symmetric monoidal space. This result, stated as Corollary \ref{cor:A2}, is also due to Barkan--Steinebrunner, and it is preceeded by an extremely similar result by Galatius \cite{Galatius} (see Remark \ref{rem:Galatius} for more details); the way we derive Corollary \ref{cor:A2} from Corollary \ref{cor:A1} is also parallel to the argument of Barkan--Steinebrunner, and is only included in the present article for completeness.

\begin{acor}
[\cite{BarkanSteinebrunner}]
\label{cor:A2}
The functor $\Omega\colon\CAlg(\cS)\to\cS$, associating with a symmetric monoidal space its loop space at the identity element, is represented by $|\Gr|$, i.e. we have a natural equivalence $\Fun^\otimes (|\Gr|,-)^\simeq\simeq\Omega$ of functors $\CAlg(\cS)\to\cS$. Equivalently, there is an equivalence of symmetric monoidal spaces $|\Gr|\simeq\Omega^{\infty-1}\bS$.
\end{acor}

As an application of Corollary \ref{cor:A1}, we construct a symmetric monoidal functor $\Re\colon\Gr\to\Cospan(\cS)$, with target the $\infone$-category of cospans between spaces; building on work of Culler--Vogtmann and Hatcher--Vogtmann, we prove in Theorem \ref{thm:Reinclusion} that $\Re$ is an $\infone$-subcategory inclusion and identify its image.

The universal property of $\bGr$ stated in Theorem \ref{thm:A} exhibits it as a free symmetric monoidal extension of $\Fin$, in which spaces of new 1-morphisms and 2-morphisms are adjoined. In order to prove Theorem \ref{thm:A} we develop an apparatus of tools to study more generally free extensions of symmetric monoidal $\infn$-categories, for $n\ge1$, in which we adjoin spaces of new $n$-morphisms. 

The main tools are given by the following results, contained in Sections \ref{sec:symmonextensions} and \ref{sec:truncatedviapsh}.
The setting is the following: we freely extend a symmetric monoidal $\infn$-category $\cC$ to a symmetric monoidal $\infn$-category $\cD$ by adjoining a space $X$ of new $n$-morphisms of prescribed sources and targets; we further assume that there is an $\N$-valued grading on $n$-morphisms of $\cD$ which is additive under composition and monoidal product, such that the $\N$-grading vanishes on $\N$-morphisms coming from $\cC$ and takes values 1 on $n$-morphisms coming from $X$.

The first result tells us that, under suitable grading conditions, adjoining new $n$-morphisms to a symmetric monoidal $\infn$-category will not affect its objects nor its $i$-morphisms for any $i<n$.
\begin{athm}[Corollary \ref{cor:nsimeqinvariance}]
\label{thm:B}
The canonical symmetric monoidal $\infn$-functor $\cC\to\cD$ restricts to an equivalence $\cC^\nsimeq\to\cD^\nsimeq$ on $n$-groupoid cores.
\end{athm}
The second result compares morphism spaces in $\cD$ and in an iterated extension of $\cC$ obtained by adjoining several copies of the new $n$-morphisms.
\begin{athm}[Proposition \ref{prop:retraction}]
\label{thm:C}
Let $S\in\Fin$ and let $\cD'$ be the symmetric monoidal $\infn$-category obtained from $\cC$ by adjoining $S\times X$ new $n$-morphisms with sources and targets prescribed as for the extension $\cC\to\cD$. Observe that morphisms in $\cD'$ come with an $\N^S$-valued grading. Let $y_1$ and $y_2$ be $(n-1)$-morphisms in $\cC$ sharing all levels of sources and targets, let $Y$ denote the space of $n$-morphisms from $y_1$ to $y_2$ in $\cD$ of $\N$-grading $\#S\in\N$, and let $Y'$ denote the space of $n$-morphisms from $y_1$ to $y_2$ in $\cD'$ of $\N^S$-grading $1^S$. Then the map $Y'/\fS_S\to Y$ induced by the fold symmetric monoidal $\infn$-functor $\cD'\to\cD$ is a homotopy retraction, i.e. it admits a section up to homotopy.
\end{athm}
We actually conjecture that the map $Y'/\fS_S\to Y$ in Theorem \ref{thm:C} is an equivalence, see Conjecture \ref{conj:retractionisequivalence} for more details. Theorems \ref{thm:B} and \ref{thm:C} allow one in concrete situations to restrict the study of the extension $\cD$ to understanding its spaces of $n$-morphisms of $\N$-grading 1. We are thus led to introduce the $\infty$-category $\CAlg(\Catinfn^{01})$ of \emph{$01$-truncated symmetric monoidal $n$-categories}: roughly these are symmetric monoidal $n$-categories in which $n$-morphisms carry grading 0 or 1, and in which composition and monoidal product of $n$-morphisms is only defined if the sum of the gradings does not exceed 1. 
\begin{athm}[Corollary \ref{cor:CAlgCatinfn01pullback}]
\label{thm:D}
There is a pullback square of large $\infty$-categories
\[
\begin{tikzcd}[row sep=10pt]
\CAlg(\Catinfn^{01})\ar[r]\ar[d]\ar[dr,phantom,"\lrcorner"very near start]&\CAlg(\Catinfn)\ar[d,"{(\dot\cT^n,\dot\lambda^n)}"]\\
\CAlg(\PSh^{01})\ar[r,"\hat\iota^*"]&\CAlg(\PSh).
\end{tikzcd}
\]
Here $\PSh\to\Catinfone$ denotes the bicartesian fibration corresponding to the functor associating with a small category its category of space-valued presheaves, and $\PSh^{01}$ is similarly defined, using presheaves with values in the category of ``$01$-graded spaces'', i.e. spaces with a map to the set $\{0,1\}$. The functor $\hat\iota^*$ forgets the part of grading 1 of a $01$-graded presheaf. The symmetric monoidal functor $(\dot\cT^n,\dot\lambda^n)$
associates with $\cE\in\Catinfn$ the ``natural'' $\infone$-category $\dot\cT^n(\cE)$ on which the presheaf $\dot\lambda^n(\cE)$, packaging the information about $n$-morphisms in $\cE$, is defined.
\end{athm}
The precise definition of the functor $\dot\cT^n$ is given in Section \ref{sec:truncatedviapsh}; for $n=1$ we have $\dot\cT^1(\cE)=\cE\times\cE^\op$, over which the presheaf $\dot\lambda(\cE)=\cE(-,-)$ is defined. Theorem \ref{thm:D} allows one to compute spaces of $n$-morphisms of $\N$-grading 1 in $\cD$ via suitable Day convolutions of presheaves over $\dot\cT^n(\cC)$.

We hope that this apparatus of tools can be useful in other contexts than the one presented in this article.

\subsection{Organisation of the article}
In Section \ref{sec:preliminaries} we introduce some preliminary material, which serves as a reference. Section \ref{sec:graphcobsets} gives a careful and elementary construction of $\bGr$ as a plain 2-category (i.e. an $\inftwo$-category with discrete spaces of 2-morphisms), as well as a description of the functor of graph-like structures on symmetric monoidal $\inftwo$-categories. Sections \ref{sec:symmonextensions} and \ref{sec:truncatedviapsh} give the apparatus of tools to study free extensions of symmetric monoidal $\infn$-categories for any $n\ge1$. Section \ref{sec:proofthmA} is devoted to the proof of Theorem \ref{thm:A}. Finally, in Section \ref{sec:proofA1A2} we prove Corollaries \ref{cor:A1} and \ref{cor:A2} as well as Theorem \ref{thm:Reinclusion}.

\subsection{Motivation}
In the forthcoming work on string topology
\cite{Bianchi:stringtopology}, we are going to apply Corollaries \ref{cor:A1} and \ref{cor:A2} to construct string operations connecting the homology groups of the mapping spaces $\map(X,M)$, for varying space $X\in\cS$ and for a fixed oriented Poincar\'e duality space $M$. We originally intended to apply directly Theorem \ref{thm:A} to that purpose, and this led us to the quest for a universal property for $\bGr$ and to the study of symmetric monoidal extensions of $\infn$-categories. We later understood, through conversations with Jan Steinebrunner, that the work of Barkan--Steinebrunner would actually suffice for the string topology applications. Since we do not have any direct application of Theorem \ref{thm:A} in mind, that does not pass through Corollaries \ref{cor:A1} and \ref{cor:A2}, we declare the following twofold motivation for this article:
\begin{itemize}
\item initiate a systematic study of symmetric monoidal extensions of $\infn$-categories;
\item give an alternative proof of the results of Barkan--Steinebrunner.
\end{itemize}

\subsection{Acknowledgments}
I would like to thank 
Shaul Barkan,
Thomas Blom,
Bastiaan Cnossen,
Rune Haugseng, 
Kaif Hilman,
Geoffroy Horel,
F\'elix Loubaton, 
Lennart Meier, 
Viktoriya Ozornova, 
Maxime Ramzi, 
Jan Steinebrunner, 
Raffael Stenzel,
Vignesh Subramanian
and Nathalie Wahl
for helpful conversations.

I was supported by the Max Planck Institute for Mathematics in Bonn throughout the preparation of this article.

\section{Preliminaries}
\label{sec:preliminaries}
In this section we recollect most of the background material that we will use throughout the article. We encourage the reader to proceed directly to the next section on a first reading of the article, and to use this section as a reference.

We fix two nested Grothendieck universes $\fU_1\subsetneq\fU_2$, and we refer to objects whose size is bounded above by $\fU_1$ and $\fU_2$ as ``small'' and ``large'', respectively.
An $\infty$-category is \emph{presentable} if it is $\fU_2$-small, admits $\fU_1$-small colimits, and is $\kappa$-accessible for some regular cardinal $\kappa<\fU_1$.

\subsection{Finite sets and simplices}
We denote by $\cS$ the presentable $\infty$-category of small $\infty$-groupoids, which we refer to as ``spaces''. We say that a map of spaces $X\to Y$ is an \emph{inclusion} if it exhibits $X$ as equivalent to a union of connected components of $Y$; we also say that $X$ is a \emph{subspace} of $Y$.

We denote by $\Fin$ the category of finite sets, and by $\Fininj\subset\Fin$ the subcategory of finite sets and injections. We regard $\Fin$ as a full $\infty$-subcategory of $\cS$.
For $k\ge0$ we denote by $\ul k\in \Fin$ the finite set $\set{1,\dots,k}$. We have $\emptyset=\ul 0$.

We denote by $\bDelta$ the category of linearly ordered non-empty finite sets and weakly increasing functions; for $k\ge0$ we let $[k]\in\bDelta$ be the object $\set{0<\dots<k}$; for $k\ge1$ and $0\le i\le k$ we denote by $d^i\colon [k-1]\to[k]$ the unique injective map whose image does not contain $i$; for $k\ge0$ and $0\le i\le k$ we let $s^i\colon[k+1]\to[k]$ be the unique surjective map whose fibre over $i$ has two points. For an $\infty$-category $\cC$ we denote by $d_i,s_i\colon\Fun([k],\cC)\to\Fun([k\pm1],\cC)$ the functors induced by $d^i,s^i$.

A morphism $[k]\to[k']$ in $\bDelta$ is \emph{inert} if it is injective and its image is a subsegment $\set{i<\dots<i+k}\subseteq[k']$, and it is \emph{active} if its image contains $0$ and $k'$. Every morphism in $\bDelta$ factors uniquely, up to unique isomorphism, as the composite of an active and an inert morphism. For $n\ge1$, a morphism in $\bDelta^n$ or in $(\bDelta^\op)^n$ is inert/active if each of its coordinates is (the opposite of) an inert/active morphism.
We let $\bDelta_\inert$ denote the wide subcategory of $\bDelta$ spanned by inert morphisms.

For $n\ge0$ and a sequence $\bk=(k_1,\dots,k_n)\in\N^n$ we let $[\bk]$ denote the object $([k_1],\dots,[k_n])\in\bDelta^n$, and $\ul\bk=\ul k_1\times\dots\times \ul k_n\in\Fin$.

For an $\infty$-category $\cC$ and for $n\ge0$ we denote by $s^n\cC:=\Fun((\bDelta^\op)^n,\cC)$ the category of $n$-fold simplicial objects in $\cC$. We write $s\cC=s^1\cC$ and identify $\cC=s^0\cC$.

For $\cC=\cS$ and $\bk=(k_1,\dots,k_n)$, we let $\Delta^\bk:=\bDelta^n(-,[\bk])\in s^n\cS$ denote the functor represented by $[\bk]$. For a map $f\colon[\bk]\to[\bk']$ in $\bDelta^n$ we also write $f\colon\Delta^{\bk}\to\Delta^{\bk'}$ for the induced map of $n$-fold simplicial spaces.
\begin{nota}
\label{nota:E1}
We denote by $E^1\in s\cS$ the quotient of $\Delta^3\in s\cS$ in which we collapse the two copies of $\Delta^1\subset\Delta^3$ corresponding to the two inclusions $[1]\hto[3]$ with images $\set{0,2}$ and $\set{2,4}$, respectively, to two distinct points. In symbols
\[
E^1=\Delta^3\sqcup_{(\Delta^1\sqcup\Delta^1)}(\Delta^0\sqcup\Delta^0).
\]
\end{nota}
\begin{defn}
For an $\infty$-category $\cC$, an object $x\in\cC$ and a morphism $f\colon y\to z$ in $\cC$, we say that $x$ is \emph{local} with respect to $f$ if the map of spaces $\cC(y,x)\to\cC(z,x)$ induced by precomposition with $f$ is an equivalence.
\end{defn}

\subsection{Categories of categories}
\label{subsec:catofcat}
For $n\ge0$ we let $\Catinfn$ denote the $\infty$-category of small $\infn$-categories. 
We refer to the explicit model of $\Catinfn$ as the full and reflective sub-$\infty$-category of $s^n\cS$ spanned by complete Segal objects, i.e. objects $X\in s^n\cS$ satisfying the following conditions introduced by Rezk \cite{Rezk}.
\begin{itemize}
\item (Segal condition) The composite functor $X_\inert\colon(\bDelta_\inert^\op)^n\hto(\bDelta^\op)^n\overset{X}{\to}\cS$ is right Kan extended from the full subcategory $(\bDelta_{\inert,\le1}^\op)^n\subset(\bDelta_\inert^\op)^n$ spanned by objects $[\bk]$ with $k_i\le 1$ for all $1\le i\le n$. This can be equivalently expressed by asking $X$ to be local with respect to all of the following ``spine inclusions'', for varying $\bk\in\N^n$:
\[
\colim_{([\bk']\to[\bk])\in(\bDelta_\inert^n)_{/[\bk]}}\Delta^{\bk'}\to\Delta^{\bk}.
\]
\item(Constancy condition) For all $1\le i\le n-1$, for all $\bk\in\N^{i-1}$ and for all maps $[\bk']\to[\bk'']$ in $(\bDelta^\op)^{n-i}$, the induced map $X([\bk,0,\bk'])\to X([\bk,0,\bk''])$ is an equivalence of spaces. Equivalently, we can ask $X$ to be local with respect to the maps $\Delta^{\bk}\boxtimes\Delta^0\boxtimes\Delta^{\bk''}\to\Delta^{\bk}\boxtimes\Delta^0\boxtimes\Delta^{\bk'}$, where $\boxtimes$ denotes the exterior product of multisimplicial spaces (see Subsection \ref{subsec:presheaves}).
\item(Completeness condition)  For all $1\le i\le n$ and for all $\bk\in\N^{i-1}$, $X$ is \emph{local} with respect to the map $\Delta^\bk\boxtimes E^1\boxtimes\Delta^{0^{n-i}}\to \Delta^\bk\boxtimes \Delta^0\boxtimes\Delta^{0^{n-i}}$ induced by the unique map $E^1\to\Delta^0$ in $s\cS$.
\end{itemize}
\begin{rem}
\label{rem:Catinfnpresentable}
The description of $\Catinfn\subset s^n\cS$ as the full subcategory on objects that are local with respect to a \emph{set} of morphisms in $s^n\cS$ ensures that $\Catinfn$ is an accessible localisation of $s^n\cS$, and in particular $\Catinfn$ is presentable.
\end{rem}

\begin{nota}
\label{nota:universes}
For $n\ge0$ and $1\le i\le 2$ we denote by $\cS^{\fU_i}$ and $\Catinfn^{\fU_i}$ the $\infty$-category of $\fU_i$-small $\infty$-groupoids and $\infn$-categories, respectively; for $i=1$ we omit the exponent ``$\fU_1$'' from the notation.
\end{nota}

\begin{ex}
Both $\cS$ and $\Catinfn$ belong to $\Catinfone^{\fU_2}$ and are presentable. 
\end{ex}

Spaces are the same as $(\infty,0)$-categories; for all $n\ge1$, $(\infty,n-1)$-categories are also considered as $\infn$-categories with only invertible $n$-morphisms. From the point of view of $(n-1)$-fold and $n$-fold simplicial spaces, the fully faithful inclusion $\Cat_{(\infty,n-1)}\hto\Catinfn$ corresponds to the restriction on $\Cat_{(\infty,n-1)}$ of the functor $s^{n-1}\cS\to s^n\cS$ given by left Kan extension along the inclusion $-\times[0]\colon(\bDelta^\op)^{n-1}\to(\bDelta^\op)^n$; since $[0]$ is initial in $\bDelta^\op$, we can identify the same functor $s^{n-1}\cS\to s^n\cS$ as the one given by restriction along the projection $(\bDelta^\op)^n\to(\bDelta^\op)^{n-1}$ away from the last factor.
The above shows that for $n\ge1$, the fully faithful inclusion $\Cat_{(\infty,n-1)}\hto\Catinfn$ has a left adjoint $|-|_n\colon\Catinfn\to\Cat_{(\infty,n-1)}$, given by left Kan extension along the projection $(\bDelta^\op)^n\to(\bDelta^\op)^{n-1}$, and a right adjoint $(-)^\nsimeq\colon\Catinfn\to\Cat_{(\infty,n-1)}$, given by restriction along $-\times[0]$; we say that $\cC^\nsimeq$ is the \emph{$n$-groupoid core} of $\cC$, and we regard $\cC^\nsimeq$ as a (non-full) sub-$\infn$-category of $\cC$ via the counit of the adjunction. For each $n\ge1$ we also denote by $|-|,(-)^\simeq\colon\Catinfn\to\cS$ the left and right adjoints of the fully faithful inclusion $\cS\to\Catinfn$ given by taking constant functors out of $(\bDelta^\op)^n$.

For $\cC,\cD\in\Catinfn$ we denote by $\Fun(\cC,\cD)^\simeq\in\cS$ the morphism space from $\cC$ to $\cD$ in $\Catinfn$; we put the symbol ``$\simeq$'' in order to avoid confusion with other notions of $\infn$-categories of $\infn$-functors from $\cC$ to $\cD$ appearing in the literature.

\begin{defn}
\label{defn:ThetanmOn}
We define $\Theta_0=\set{s_0}$ and $\mO_0=\set{s_0,t_0}$ as the discrete groupoids with one, respectively two non-equivalent objects. For $n\ge1$, we recursively define $\Theta_n$ as the $n$-category with two non-equivalent objects $s_0,t_0$ having $\Theta_{n-1}$ as $(n-1)$-category of morphisms from $s_0$ to $t_0$, and otherwise having no other non-identity $i$-morphisms for any $1\le i\le n$. We similarly define $\mO_n$, replacing ``$\Theta_{n-1}$'' by ``$\mO_{n-1}$'' in the recursive definition of $\Theta_n$. Equivalently, for $n\ge1$, we may define $\mO_n$ as the $\infn$-category $\Theta_{n+1}^{(n+1)\!\simeq}$, in particular we have an inclusion $\mO_n\hto\Theta_{n+1}$.
\end{defn}
\begin{ex}
The category $\mO_2$ has two distinct objects $s,t$ and two distinct morphisms $f,g\colon s\to t$; the 2-category $\Theta_2$ is obtained from $\mO_2$ by adjoining a 2-morphism $f\Rightarrow g$.
\end{ex}
For $0\le m\le n$ and an $\infn$-category $\cC$, we identify $\Fun(\Theta_m,\cC)^\simeq$ with the space of all $m$-morphisms in $\cC$, that is the value at $[1]^m\times[0]^{n-m}\in(\bDelta^\op)^n$ of a complete Segal $n$-fold simplicial space $(\bDelta^\op)^n\to\cS$ representing $\cC$.
Similarly, $\mO_{m-1}\in\Catinfn$ corepresents the space of all choices of two $m-1$-morphisms in $\cC$ sharing all levels of sources and targets.

The word ``$n$-category'' is used for $\infn$-categories $\cC$ such that the restriction map $\Fun(\Theta_n,\cC)^\simeq\to\Fun(\mO_{n-1},\cC)^\simeq$ has discrete fibres; informally, we are requiring that all spaces of $n$-morphisms between any two $(n-1)$-morphisms sharing all levels of sources and targets is discrete. We abbreviate ``category'' for ``1-category'' and ``set'' for ``0-category''.

For $n\ge1$, an $\infn$-category $\cC$ and objects $x,y,z\in\cC$ we usually denote by $\cC(x,y)$ the $(\infty,n-1)$-category of morphisms between $x$ and $y$. We denote by $-\circhor-\colon\cC(y,z)\times\cC(x,y)\to\cC(y,z)$ the horizontal composition functor (for $n=1$ we just write $-\circ-$).

\subsection{Symmetric monoidal categories and deloopings}
\label{subsec:symmoncat}
For $n\ge0$ we denote by $\CAlg(\Catinfn)$ the presentable $\infty$-category of small symmetric monoidal $\infn$-categories, with morphisms given by symmetric monoidal $\infn$-functors. This can be equivalently defined as the category of commutative monoid objects \cite[Definition 2.4.2.1]{HA}, or of algebras over the $E_\infty$-operad, in $\Catinfn$, see \cite[Proposition 2.4.2.5]{HA}.
We usually denote by $\one\in\cC$ the monoidal unit and we denote by $-\otimes-\colon\cC\times\cC\to\cC$ the monoidal product.
When the monoidal structure on $\cC$ comes from categorical coproduct, either on $\cC$ or on a closely related $\infn$-category, we will instead use the notations $\emptyset$ and $-\sqcup-$ for the monoidal unit and the monoidal product of $\cC$.
For small symmetric monoidal $\infn$-categories $\cC$ and $\cC'$, we let $\Fun^\otimes(\cC,\cC')^\simeq\in\cS$ denote the corresponding morphism space in $\CAlg(\Catinfn)$.

We usually endow $\Fin$ and $\cS$ with the coproduct symmetric monoidal structure. When $\cS$ is endowed with the product symmetric monoidal structure, we will stress this by writing $(\cS,\times)$.

For $\cC\in\CAlg(\Catinfn)$ we let $\bB\cC\in\CAlg(\Cat_{(\infty,n+1)})$ denote the \emph{first delooping} of $\cC$, i.e. the essentially unique symmetric monoidal $(\infty,n+1)$-category with 
$\bB\cC(\one_{\bB\cC},\one_{\bB\cC})\simeq\cC$ and such that $(\bB\cC)^\simeq$ is a connected space. Taking the first delooping gives a fully faithful inclusion $\bB\colon\CAlg(\Catinfn)\hto\CAlg(\Cat_{(\infty,n+1)})$ with essential image given by the symmetric monoidal $(\infty,n+1)$-categories whose space of objects is connected.
We usually denote by $*=\one_{\bB\cC}$ the monoidal unit of a first delooping, which we regard as a ``base object''.

\begin{ex},
\label{ex:BnMgrouplike}
Let $M\in\CAlg(\cS)$ be a \emph{group-like} symmetric monoidal space, i.e. the discrete monoid $\pi_0(M)$ is an abelian group. Then for $n\ge1$ we have that $\bB^n M$ is the essentially constant $n$-fold simplicial space with value the essentially unique pointed $(n-1)$-connected space $X$ with $\Omega^nX\simeq M$. By abuse of notation, we refer to this space $X$ also as $\bB^nM$ in this case.
\end{ex}

\begin{defn}
\label{defn:ungrouplike}
A discrete abelian monoid $M$ is \emph{ungroup-like} if the neutral element $0\in M$ is the only invertible element.
\end{defn}
We remark that all finitely generated, free abelian monoids are ungroup-like.
\begin{ex}
\label{ex:BnMungrouplike}
Let $M$ be an ungroup-like abelian monoid, considered as an object in $\CAlg(\cS)$. Then for $n\ge1$ the underlying $n$-fold simplicial space of $\bB^nM$ is the levelwise discrete $n$-fold simplicial abelian monoid $[\bk]\mapsto M^{\ul\bk}$, with face maps given by addition and degeneracy maps given by inserting $0\in M$. Note that, for a generic discrete monoid $M$, the $n$-fold simplicial set $[\bk]\mapsto M^{\ul\bk}$ satisfies the Segal and constancy conditions; the ungroup-like assumption also ensures completeness.
\end{ex}
\begin{nota}
\label{nota:totalgrading}
In the setting of Example \ref{ex:BnMungrouplike}, for $\nu\in\bB^nM([\bk])\simeq M^{\ul\bk}$, we define the \emph{total grading} of $x$ as the sum of the coordinates of $\nu$, taken in $M$.
\end{nota}

\subsection{Presheaves and Day convolution}
\label{subsec:presheaves}
For $\cC\in\Catinfone$ and for any $\infty$-catego\-ry $\cD$ we denote by $\PSh(\cC;\cD):=\Fun(\cC^\op;\cD)$ the $\infty$-category of $\cD$-valued presheaves over $\cC$. If $\cD$ is endowed with a symmetric monoidal structure, we define for $\cC,\cC'\in\Catinfone$ an exterior product $-\boxtimes-\PSh(\cC;\cD)\times\PSh(\cC;\cD')\to\PSh(\cC\times\cC';\cD)$, sending $(A,B)\mapsto A\boxtimes B$ with $(A\boxtimes B)(c,c')= A(c)\otimes B(c')$.

If $\cD=\cS$ we just write $\PSh(\cC)$ for $\PSh(\cC,\cS)$, which is a presentable $\infty$-category.
We denote by $\yo\colon\cC\to\PSh(\cC)$ the Yoneda embedding. For $\cC\in\CAlg(\Catinfone)$, the category $\PSh(\cC)$ is endowed with a symmetric monoidal structure $-\otimes_\Day-$ given by \emph{Day convolution}, which can be characterised by the following properties \cite[Corollary 4.8.1.12]{HA}:
\begin{itemize}
\item the Yoneda embedding can be promoted to a symmetric monoidal functor;
\item the monoidal product $\PSh(\cC)^2\to\PSh(\cC)$ preserves colimits separately in each variable.
\end{itemize}
The two properties give rise to the following formula: if $A\simeq\colim_{i\in I}\yo(x_i)$ and $B\simeq\colim_{j\in J}\yo(x'_j)$ are objects in $\PSh(\cC)$, written as colimits of representable presheaves, then $A\otimes_{\Day}B\simeq\colim_{(i,j)\in I\times J}\yo(x_i\otimes x'_j)$.

A functor $F\colon\cC\to\cD$ induces a functor $F^*\colon\PSh(\cD)\to\PSh(\cC)$ by restriction, with left adjoint $F_!\colon\PSh(\cC)\to\PSh(\cD)$ given by left Kan extension. The functor $F_!$ is symmetric monoidal, whereas $F^*$ is in general only lax symmetric monoidal.

We denote by $\PSh\to\Catinfone$ a cartesian fibration corresponding to the functor $\PSh(-)\colon(\Catinfone)^\op\to\Catinfone^{\fU_2}$ sending $\cC\mapsto\PSh(\cC)$ and sending $F\colon\cC\to\cD$ to $F^*$. Roughly speaking, an object in $\PSh$ is a pair $(\cC,A)$ with $\cC\in\Catinfone$ and $A\in\PSh(\cC)$. Exterior product of presheaves, sending $(A,B)\in\PSh(\cC)\times\PSh(\cD)$ to the presheaf $A\boxtimes B\in\PSh(\cC\times\cD)$, makes $\PSh(-)$ into a lax symmetric monoidal functor. 
By \cite[Theorem B, Remark 0.1]{Ramzi} we have a symmetric monoidal structure on $\PSh$ as well as on the projection functor $\PSh\to\Catinfone$.

We can also regard $\PSh$ as the full subcategory of $\Fun([1],\Catinfone)$ spanned by right fibrations, recalling that the datum of $A\in\PSh(\cC)$ is equivalent to that of a right fibration over $\cC$; in this light the functor $\PSh\to\Catinfone$ is the restriction of the target functor $d_0\colon\Fun([1],\Catinfone)\to\Catinfone$, and the symmetric monoidal structure on $\PSh$ is restricted from the pointwise one on $\Fun([1],\Catinfone)$; in particular $\PSh$ has the \emph{cartesian} symmetric monoidal structure.

\subsection{Twisted arrow categories}
\label{subsec:twistedarrow}
We denote by $i_\bDelta\colon\bDelta\to\bDelta$ the non-trivial involution of $\bDelta$, by $\Delta_\bDelta\colon\bDelta\to\bDelta\times\bDelta$ the diagonal functor, by $p_{\bDelta,1},p_{\bDelta,2}\colon\bDelta\times\bDelta\to\bDelta$ the two projections, and by $c_\bDelta\colon \bDelta\times\bDelta\to\bDelta$ the concatenation functor, sending $([k],[k'])$ to the lexicographic concatenation of finite ordered sets $[k]\sqcup_{<}[k']\simeq[k+k'+1]$. For $j=1,2$ there is a natural transformation $\eta_{\bDelta,j}\colon p_{\bDelta_j}\Rightarrow c_\bDelta$, sending $([k_1],[k_2])$ to the minimal ($j=1$) or maximal ($j=2$) inclusion $[k_j]\hto[k_1+k_2+1]$.

For $X\in s\cS=\Fun(\bDelta^\op,\cS)$ we denote by $X^\op,\Tw(X)\in s\cS$ the functors $X\circ i^\op_\bDelta$ and $X\circ c_\bDelta^\op\circ(\Id_\bDelta\times i_\bDelta)^\op\circ\Delta_{\bDelta}^\op$. The natural transformations $\eta_{\bDelta,1},\eta_{\bDelta,2}$ give rise to a map of simplicial spaces $\Tw(X)\to X\times X^\op$.

If $\cC\in\Catinfone$, then also $\cC^\op,\Tw(\cC)\in\Catinfone$, and the projection $\Tw(\cC)\to\cC\times\cC^\op$ is a right fibration, corresponding to a presheaf which we denote by $\cC(-,-)\in\PSh(\cC\times\cC^\op)$, as its value on $(x,y)$ is the morphism space $\cC(x,y)$. We refer to $\Tw(\cC)$ as the \emph{twisted arrow category} of $\cC$, and we denote by $d_1\colon\Tw(\cC)\to\cC$ and $d_0\colon\Tw(\cC)\to\cC^\op$ the ``source'' and ``target'' functors, even though they are not directly related to the morphisms $d^0,d^1\colon[0]\to[1]$ in $\bDelta$.

Under the equivalence $\PSh(\cC\times\cC^\op)\simeq\Fun(\cC;\PSh(\cC))$, the presheaf $\cC(-,-)$ corresponds in turn to $\yo\colon\cC\to\PSh(\cC)$.

\subsection{Commutative Frobenius algebras}
The following material can be found in \cite[Subsection 4.6.5]{HA}.

\begin{defn}
\label{defn:duality}
Let $\cC$ be a symmetric monoidal $\infone$-category, and let $x\in\cC$ be an object. A \emph{duality datum} for $x$ is the pair $(x^\vee,c)$ of an object $x^\vee\in\cC$ and a morphism $c\colon x\otimes x^\vee\to \one$ satisfying the following property: there exists a morphism $u\colon 1\to x^\vee\otimes x$ for which there exist homotopies $\Id_x\simeq (c\otimes\Id_x)(\Id_x\otimes u)$ and $\Id_{x^\vee}\simeq(\Id_{x^\vee}\otimes c)(u\otimes\Id_{x^\vee})$. In this case we say that $x^\vee$ is a dual of $x$, and that $x$ is dualisable; the maps $c$ and $u$ are called counit and unit of the duality.
\end{defn}
\begin{nota}
Following \cite[Notation 4.6.1.8]{HA}, for $\cC\in\CAlg(\Catinfone)$ we let $\DDat(\cC)$ denote the moduli space of duality data in $\cC$: it is introduced as a certain full subcategory of the limit in $\Catinfone$ of the following diagram
\[
\begin{tikzcd}[row sep=3pt]
(\cC\times\cC)\ar[dr,"-\otimes-"']& &\Fun([1],\cC)\ar[dl,"d_1"]\ar[dr,"d_0"']& &*\ar[dl,"\one"]\\
&\cC& &\cC;
\end{tikzcd}
\]
more precisely, we take the full subcategory spanned by objects of the form $(x,x^\vee,c)$ such that $(x^\vee,c)$ is a duality datum for $x$. We have that $\DDat(\cC)$ is an $\infty$-groupoid, and the forgetful functor $\DDat(\cC)\to\cC^\simeq$ sending $(x,x^\vee,c)\mapsto x$ is fully faithful with image the subspace of dualisable objects in $\cC^\simeq$, see \cite[Lemma 4.6.1.10]{HA}.
\end{nota}
\begin{lem}[{\cite[Lemma 4.6.1.6]{HA}}]
\label{lem:dualityequivalence}
Let $\cC\in\CAlg(\Catinfone)$ and let $(x^\vee,c)$ be a duality datum for $x\in\cC$; then for all objects $y,z\in\cC$ the composite map of spaces
\[
\begin{tikzcd}[column sep=45pt]
\cC(y,z\otimes x)\ar[r,"-\otimes x^\vee"]&\cC(y\otimes x^\vee,z\otimes x\otimes x^\vee)\ar[r,"{\cC(y\otimes x^\vee,z\otimes c)}"]&\cC(y\otimes x^\vee,z)
\end{tikzcd}
\]
is an equivalence.
\end{lem}
The following is a symmetric monoidal adaptation of \cite[Definition 4.6.5.1]{HA}.
\begin{defn}
\label{defn:Frobeniusalgebra}
Let $\cC$ be a symmetric monoidal $\infone$-category.
A \emph{commutative Frobenius algebra} in $\cC$ is the datum $(\fI,\ft)$ of a symmetric monoidal functor $\fI\colon\Fin\to\cC$ and a morphism $\ft\colon\fI(\ul1)\to\fI(\ul0)$ such that $(\fI(\ul1),\ft\circ\fI(\mu))$ is a duality datum for $\fI(\ul1)$, see Notation \ref{nota:mu}.
We denote by $\CFrob\colon\CAlg(\Catinfone)\to\cS$ the functor associating with a symmetric monoidal $\infone$-category $\cC$ its moduli space of commutative Frobenius algebras. It is the pullback in $\Fun(\CAlg(\Catinfone),\cS)$ of the following diagram:
\[
\begin{tikzcd}[row sep=4pt,column sep=14pt]
\Fun^\otimes(\Fin,-)^\simeq\ar[dr,"\ev_{\mu}"']& &\Fun([2],-)^\simeq\ar[dl,"d_2"]\ar[dr,"d_1"']& &\DDat(\cC)\ar[dl,"\text{take counit}"]\\
&\Fun([1],\cC) & &\Fun([1],\cC)
\end{tikzcd}
\]
\end{defn}

\subsection{Localisation of categories}
\label{subsec:localisation}
\begin{defn}
\label{defn:RelCatinfone}
We denote by $\RelCatinfone\subset\Fun([1],\Catinfone)$ the full sub-$\infty$-category spanned by functors $W\to\cC$ such that $W([0])\overset{\simeq}{\to}\cC([0])$, and such that the map $W([1])\to \cC([1])$ exhibits $W([1])$ as a union of components of $\cC([1])$; in few words, $W$ is a wide subcategory of $\cC$. An object in $\RelCatinfone$ is usually denoted as a pair $(\cC,W)$, and is called a \emph{relative $\infone$-category}. We endow $\RelCatinfone$ with the (pointwise) product symmetric monoidal structure.

The fully faithful inclusion $\Catinfone\hto\RelCatinfone$ sending $\cC\mapsto(\cC,\cC^\simeq)$ is symmetric monoidal, and it admits a symmetric monoidal left adjoint $L\colon\RelCatinfone\to\Catinfone$, sending $(\cC,W)$ to the localisation of $\cC$ at all morphisms in $W$.
\end{defn}
\begin{defn}
\label{defn:laxarrowcategories}
We denote by $\Fun([1],\Catinfone)^\llax$ the full $\infty$-subcategory of $(\Catinfone)_{/[1]}$ spanned by cocartesian fibrations, and by $\Fun([1],\Catinfone)^\llax$ the full $\infty$-subcategory of $(\Catinfone)_{/[1]^\op}$ spanned by cartesian fibrations. In both cases an object is given by a functor of $\infone$-categories $\cC(0)\to\cC(1)$; a morphism from $(\cC(0)\to\cC(1))$ to $(\cD(0)\to\cD(1))$ is also called a left/right lax natural transformation between functors $[1]\to\Catinfone$; it looks like a square of $\infone$-category commuting up to a natural transformation, whose orientation dictates the difference between the two cases:
\[
 \begin{tikzcd}[row sep=30pt, column sep =70 pt]
\cC(0)\ar[r]\ar[d,""{name=tL,inner sep=2pt,right},""{name=sR,inner sep=2pt,right,very near end}]& \cD(0)\ar[d,""{name=sL,inner sep=2pt,left,very near start},""{name=tR,inner sep=2pt,left}] \\
\cC(1)\ar[r]&\cD(1).
\ar[Rightarrow,from=sL, to=tL,"\llax"']\ar[Rightarrow, from=sR,to=tR,"\rlax"']
 \end{tikzcd}
\]
We denote by $d_0$ and $d_1$ the functors $\Fun([1],\Catinfone)^\llax,\Fun([1],\Catinfone)^\rlax\to\Catinfone$ taking fibres of cocartesian fibrations over $[1]$ (respectively, cartesian fibrations over $[1]^\op$) at 1 (for $d_0$) and at 0 (for $d_1$).
\end{defn}

\begin{lem}
\label{lem:vertlocfibration}
Let $q\colon\cD^\op\to\RelCatinfone$ be a functor, and denote by $p_j\colon\cC_j\to\cD$ the cartesian fibration corresponding to $d_jq$ for $j=0,1$. Let $W=\cC_1\times_{\cD}\cD^\simeq$, and consider $W\subset\cC_1\subset\cC_0$ as a wide subcategory. Then the localised functor $L(\cC_0,W)\to\cD$ is a cartesian fibration, and its straightening is the composite $Lq$.
\end{lem}
\begin{proof}
Let $\cC'\to\cD$ denote a right fibration corresponding to $L\circ q\colon\cD^\op\to\Catinfone$, and let $\cE\in\Catinfone$. We then have equivalences of spaces
\[
\Fun(\cC',\cE)^\simeq\simeq\Fun_{\cD}(\cC',\cE\times\cD)^\simeq\simeq\Fun(\cD^\op,\Fun([1],\Catinfone)^\llax)^\simeq_{Lq,\mathrm{const}_\cE},
\]
where $\Fun_\cD$ denotes functors that are compatible with the projection to $\cD$, but need not preserve cartesian morphisms; and $\Fun(\cD^\op,\Fun([1],\Catinfone)^\llax)^\simeq_{Lq,\mathrm{const}_\cE}$ denotes the fibre over $(Lq,\mathrm{const}_\cE)$ of the map of spaces 
\[
(d_1,d_0)\colon\Fun(\cD^\op,\Fun([1],\Catinfone)^\llax)^\simeq\to(\Fun(\cD^\op,\Catinfone)^\simeq)^2.
\]
In other words, the last space is the space of left lax natural transformations from $Lq$ to $\mathrm{const}_\cE$.
Composition with the (strict) natural transformation $d_0q\to Lq$ allows us to identify $\Fun(\cD^\op,\Fun([1],\Catinfone)^\llax)^\simeq_{Lq,\mathrm{const}_\cE}$ with the subspace of $\Fun(\cD^\op,\Fun([1],\Catinfone)^\llax)^\simeq_{d_0q,\mathrm{const}_\cE}$ comprising those left lax natural transformations that, for each $x\in\cD^\op$, restrict to a functor $d_0q(x)\to \cE$ which is $d_1q(x)$-local. And this subspace, under the composite equivalence of spaces
\[
\Fun(\cD^\op,\Fun([1],\Catinfone)^\llax)^\simeq_{d_0q,\mathrm{const}_\cE}\simeq\Fun_\cD(\cC_0,\cE\times\cD)^\simeq\simeq\Fun(\cC_0,\cE)^\simeq,
\]
corresponds to the subspace of $\Fun(\cC_0,\cE)^\simeq$ comprising $W$-local functors.
\end{proof}
\begin{lem}
\label{lem:horlocfibration}
Let $(\cC,W)\in\RelCatinfone$, let $p\colon\cD\to\cC$ be a cartesian fibration, and assume that the functor $q\colon\cC^\op\to\Catinfone$ corresponding to $p$ is $W^\op$-local. Let $\mathrm{cart}_W\subset\cD$ denote the wide subcategory comprising all cartesian lifts of morphisms in $W$. Then the localised functor $L(\cD,\mathrm{cart}_W)\to L(\cC,W)$ is again a cartesian fibration, and it corresponds to the functor $L(\cC^\op,W^\op)\to\Catinfone$ induced by $q$.
\end{lem}
\begin{proof}
Let $\cC':=L(\cC,W)$, let
$\cD'\to \cC'$ be a cartesian fibration corresponding to $q\in \Fun((\cC')^\op,\Catinfone)^\simeq\subset\Fun(\cC^\op,\Catinfone)^\simeq$, and let $\cE\in\Catinfone$. Then we have equivalences of spaces
\[
\Fun(\cD',\cE)^\simeq\simeq\Fun_{\cC'}(\cD',\cE\times \cC')^\simeq\simeq\Fun((\cC')^\op,\Fun([1],\Catinfone)^\llax)^\simeq_{q,\mathrm{const}_\cE}.
\]
The localisation functor $\cC^\op\to(\cC')^\op$ 
allows us to identify the last space with the subspace of $\Fun(\cC^\op,\Fun([1],\Catinfone)^\llax)^\simeq_{q,\mathrm{const}_\cE}$ comprising $W$-local functors, and under the composite equivalence
\[
\Fun(\cC^\op,\Fun([1],\Catinfone)^\llax)^\simeq_{q,\mathrm{const}_\cE}\simeq\Fun_{\cC}(\cD,\cC\times\cE)\simeq\Fun(\cD,\cE),
\]
this corresponds to the subspace of $\Fun(\cD,\cE)$ comprising $\mathrm{cart}_W$-local functors.
\end{proof}

The following is an adaptation of \cite[Chapter 7.2]{Cisinski} for \emph{left} calculi of fractions.
\begin{defn}
Let $(\cC,W)\in\RelCatinfone$ and let $x\in\cC$. A \emph{left calculus of fractions at $x$} is a functor $\pi(x)\colon W(x)\to\cC_{x/}$ satisfying the following properties:
\begin{itemize}
\item the $\infone$-category $W(x)$ admits an initial object $x_0$, and $\pi(x)\colon x_0\mapsto\Id_x$;
\item $\pi(x)$ takes value in the full subcategory of $\cC_{x/}$ spanned by arrows that are contained in $W$;
\item the left Kan extension functor $(d_0\pi(x))_!\colon\PSh(W(x))\to\PSh(\cC)$ sends the terminal presheaf $*\in\PSh(W(x))$ to a presheaf $(d_0\pi(x))_!(*)\in\PSh(\cC)$ that factors through $L(\cC,W)$, i.e. morphisms in $W^\op$ are sent to invertible morphisms in $\cS$ along $(d_0\pi(x))_!(*)$.
\end{itemize}
\end{defn}
We remark that, for any functor $\pi(x)\colon W(x)\to\cC_{x/}$ and for $y\in\cC$, we have $(d_0\pi(x))_!(*)(y)\simeq \colim_{z\in W(x)}\cC(y,d_0\pi(z))$; we can also consider the category $\cC_{y/}\times_{\cC}W(x)$, where the fibre product is taken along the functors $d_0\colon\cC_{y/}\to\cC$ and $d_0\pi\colon W(x)\to\cC$; then we have $(d_0\pi(x))_!(*)(y)\simeq|\cC_{y/}\times_{\cC}W(x)|$.
\begin{thm}[{\cite[Theorem 7.2.8]{Cisinski}}]
\label{thm:Cisinski}
Let $(\cC,W)$ be a relative $\infty$-category, let $x\in\cC$ and let $\pi(x)\colon W(x)\to\cC_{x/}$ be a left calculus of fractions at $x$. Then we have an equivalence in $\PSh(\cC)$
\[
L(\cC,W)(-,x)\simeq (d_0\pi(x))_!(*).
\]
\end{thm}

\section{The \texorpdfstring{$\inftwo$}{inftwo}-category \texorpdfstring{$\bGr$}{Gr} of graph cobordisms between finite sets}
\label{sec:graphcobsets}
In this section we introduce an (ordinary) symmetric monoidal 2-category of graph cobordisms between finite sets, denoted $\bGr$; the main result of the next section will be a universal property of $\bGr$ as a symmetric monoidal $\inftwo$-category.

Roughly speaking, objects of $\bGr$ are finite sets; a 1-morphism between the finite sets $B$ and $A$ is the datum of a graph $G$ attached to $A$ and with a marking from $B$; and a 2-morphism between the 1-morphisms $B\overset{G}{\to}A$ and $B\overset{G'}{\to}A$ is a map $G\to G'$ collapsing trees and compatible with the inclusions of $A$ and the maps from $B$. Disjoint union of finite sets and graphs gives the symmetric monoidal structure.

\begin{rem}
\label{rem:classicalgraphs}
 Categories of graphs and graph collapses are classical objects. See for example \cite[§1.1]{Gersten} and \cite[§2.1]{CullerVogtmann}, where all maps between graphs are allowed (and not only those collapsing trees), and \cite[§1.1]{CullerVogtmann}, where graphs with a marking by the standard rose, together with tree collapse maps, are considered.
\end{rem}

\subsection{Marked graphs attached to a finite set}
\label{subsec:gafs}
The following definition is a variation of the notion of combinatorial graph introduced by Gersten \cite{Gersten}.
\begin{defn}
\label{defn:gaf}
A \emph{marked graph attached to a finite set} (shortly ``gaf''), is a sequence $G=(A,B,V,H,\rho,\sigma,\upsilon)$ consisting of:
 \begin{itemize}
  \item finite sets $A,B,V,H$ of ``attaching vertices'', ``markings'', ``inner vertices'' and ``half-edges'';
  \item maps of finite sets $\rho,\sigma,\upsilon\colon  A\sqcup B\sqcup V\sqcup H\to A\sqcup B\sqcup V\sqcup H$,
\end{itemize}
 satisfying the following requirements:
 \begin{itemize}
  \item $\rho$ restricts to the identity on $A\sqcup V\sqcup H$ and to a map $B\to A\sqcup V$;
  \item $\sigma$ and $\upsilon$ restrict to the identity on $A\sqcup B\sqcup V$;
  \item $\sigma$ restricts to a map $H\to  A\sqcup V$.
  \item $\upsilon$ is an involution acting freely on $H$.
 \end{itemize}
We say that $G$ is a graph \emph{attached to $A$} and \emph{marked by $B$}.
A sub-gaf of a gaf $G$ as above is obtained by selecting subsets $A'\subseteq A$, $B'\subseteq B$, $V'\subseteq  V$ and $H'\subseteq H$ such that the maps $\rho,\sigma,\upsilon$ restrict to self-maps of $A'\sqcup B'\sqcup V'\sqcup H'$.
\end{defn}
\begin{nota}
\label{nota:gaf}
We usually expand a gaf $G$ as $(A,B,V,H,\rho,\sigma,\upsilon)$, we expand a gaf $G'$ as $(A',B',V',H',\rho',\sigma',\upsilon')$, and so on.
We usually define sub-gaf's of a gaf $G$ by declaring a (suitable) subset of $A\sqcup B\sqcup V\sqcup H$.

For a gaf $G$, we usually denote by $E=H/\upsilon$ the set of $\upsilon$-orbits contained in $H$, and refer to elements in $E$ as ``edges'' of $G$. Similarly, we write $E'=H'/\upsilon'$ for the set of edges of $G'$, and so on.
\end{nota}

\begin{defn}
\label{defn:ReG}
Given a gaf $G$, we let $\Re(G)$ denote the 1-dimensional cell complex having a 0-cell for each element in $A\sqcup V$ and a 1-cell for each element in $E$. The two attaching points of the 1-cell corresponding to $e=\set{h_1,h_2}\subseteq H$ are the (possibly equal) 0-cells $\sigma(h_1)$ and $\sigma(h_2)$. 
\end{defn}
Definition \ref{defn:ReG} explains the use of the word ``graph'' in Definition \ref{defn:gaf}.
A morphism of gaf's shall be required to be compatible with attaching vertices and markings, and to induce at the level of cell complexes a ``nice'' homotopy equivalence. This is achieved by the following combinatorial implementation.
\begin{defn}
 \label{defn:tree}
A gaf $G$ is a \emph{based tree} (respectively, 
a \emph{non-based tree}) if $\Re(G)$ is contractible, and moreover $|A|=1$ (respectively, $A=\emptyset$).
\end{defn}

\begin{defn}
\label{defn:morgaf}
A \emph{morphism of gaf's} $f\colon G\to G'$ is a map of sets $f\colon A\sqcup B\sqcup V\sqcup H\to A'\sqcup B'\sqcup V'\sqcup H'$ satisfying the following conditions:
\begin{enumerate}
 \item $f$ restricts to maps $A\to A'$ and $B\to B'$;
 \item $f\circ\rho=\rho'\circ f$, $f\circ\sigma=\sigma'\circ f$ and $f\circ\upsilon=\upsilon'\circ f$; in particular $f$ restricts to maps $V\to A\sqcup V'$ and $H\to A'\sqcup V'\sqcup H'$;
 \item for all $v'\in  V'$, $f^{-1}(v')$ is a non-based subtree of $G$;
 \item for all $a'\in  A'$, $f^{-1}(a')$ is a disjoint union of based subtrees of $G$, one for each element in $f^{-1}(a')\cap A$.
 \item for all $h'\in H'$, $f^{-1}(h')$ is a singleton.
\end{enumerate}
We let $\Gaf$ denote the 1-category of gaf's and morphisms of gaf's. We consider $\Gaf$ as a symmetric monoidal category by disjoint union of finite sets and graphs.

We denote by $\fA,\fB\colon\Gaf\to\Fin$ the symmetric monoidal functors given by $\fA(G)=A$ and $\fB(G)=B$, respectively.
\end{defn}

\subsection{Graph cobordisms between finite sets}
\begin{defn}
\label{defn:GrAB}
Let $A$ and $B$ be finite sets. We denote by $\bGr(B,A)$ the category obtained as fibre over the object $(B,A)$ of the functor $\fB\times\fA\colon\Gaf\to\Fin\times\Fin$. This is the subcategory of $\Gaf$ whose objects are those gaf's attached to $A$ and marked by $B$, and whose morphisms restrict to the identities of $A$ and $B$.
\end{defn}

\begin{ex}
\label{ex:GrconnBoutFn}
 For $A=B=\emptyset$, the category $\bGr(\emptyset,\emptyset)$ admits a symmetric monoidal structure by disjoint union, and it is symmetric monoidally freely generated by its full subcategory $\bGr(\emptyset,\emptyset)^{\conn}$ spanned by gaf's $G$ with $\Re(G)$ connected and non-empty. In turn, $\bGr(\emptyset,\emptyset)^{\conn}$ is the disjoint union of its full subcategories $\bGr(\emptyset,\emptyset)^{\conn}_n$, for $n\ge0$, spanned by non-empty connected graphs of rank $n$. 
 
 As we will see in Theorem \ref{thm:Reinclusion}, the classifying space of the category $\bGr(\emptyset,\emptyset)^{\conn}_n$

is an Eilenberg-MacLane space of type $K( \Out(F_n),1)$, where
$\Out(F_n)$ denote the group of outer automorphisms of a free group of rank $n$.
\end{ex}

Given finite sets $A,A',A''$, we construct a functor $\circhor\colon\bGr(A',A)\times\bGr(A'',A')\to\bGr(A'',A)$ by \emph{glueing graphs along $A'$} as follows. Let $G'\in\bGr(A'',A')$ and $G\in\bGr(A',A)$; by Notation \ref{nota:gaf}, we have in particular $B'=A''$ and $B=A'$. We let $G\circhor G'$ be the gaf $\hat G$ with $\hat A=A$, $\hat B=B'=A''$, $\hat V=V\sqcup V'$, $\hat H=H\sqcup H'$ and whose structure maps $\hat\rho$, $\hat\sigma$ and $\hat\upsilon$ characterised by the following:
\begin{itemize}
 \item $\hat\rho|_{A''}$ is the composite of $\rho'\colon B'\to A'\sqcup V'=B\sqcup V'$ and $(\rho\sqcup \Id_V)\colon B\sqcup V'\to A\sqcup V\sqcup V'= A\sqcup \hat V$;
 \item $\hat\sigma|_{H}=\sigma$, and $\hat\sigma|_{H'}=(\rho\sqcup\Id_{V'})\circ\sigma'$.
 \item $\hat\upsilon|_{\hat H}=\upsilon\sqcup\upsilon'$;
\end{itemize}

The behaviour of $\circhor$ on morphisms is similar: given morphisms $f'\colon G'\to G'_1$ in $\bGr(A'',A')$ and $f\colon G\to G_1$ in $\bGr(A',A)$, we let $f\circhor f'$ be the map of gaf's $\hat f\colon \hat G\to \hat G_1$ agreeing with $f$ and $f'$ on the different finite sets defining $\hat G$, where we postcompose with $\rho_1$ whenever $f'$ lands in $A_1'$. Among the conditions of Definition \ref{defn:morgaf}, the less obvious to check for $\hat f$ are (3) and (4). For this, we observe the following:
\begin{itemize}
 \item for $v'_1\in V'_1\subset\hat V_1$, $\hat f^{-1}(v'_1)$ is isomorphic to the non-based tree $(f')^{-1}(v'_1)$;
 \item for $v_1\in V_1\subset \hat V_1$, $\hat f^{-1}(v_1)$ is isomorphic to the gaf obtained by glueing the non-based tree $f^{-1}(v_1)$ with the based trees $(f')^{-1}(a'_1)$, for $a'_1$ ranging in $\rho_1^{-1}(v_1)\subset A'_1$: we glue the unique vertex $a'\in (f')^{-1}(a'_1)\cap A'$ to the vertex $\rho(a')\in f^{-1}(v)$; the result is a non-based tree;
 \item for $a\in A$, the situation is analogous to the previous one, except that the result is a based tree, as $f^{-1}(a)$ is now a based tree.
\end{itemize}

The functors $\circhor$ satisfy associativity up to natural equivalence, coherently for iterated compositions. We obtain therefore a 2-category $\bGr$, whose objects are finite sets, and whose morphism 1-category from $B$ to $A$ is $\bGr(B,A)$. Horizontal composition of morphisms is given by the functors $\circhor$. For each finite set $A$, the gaf $(A,A,\emptyset,\emptyset,\Id_A,\emptyset,\emptyset)\in\bGr(A,A)$ is an identity object in $\bGr(A,A)$, where ``$\emptyset$'' denotes the empty set or the empty map.

We also want to enhance $\bGr$ to a \emph{symmetric monoidal} 2-category, with monoidal product given by taking disjoint unions of finite sets, gaf's, and maps of gaf's. To do this formally, we first define for a finite pointed set $S_*=S\sqcup\set{*}$ a variation $\Gaf_S$ of $\Gaf$: an object in $\Gaf_S$ is a tuple $(A,B,V,H,\rho,\sigma,\upsilon)$ of objects and morphisms in $\Fin_{/S}$, whose image along $d_1\colon\Fin_{/S}\to\Fin$ satisfies the requirements of Definition \ref{defn:gaf}; a morphism in $\Gaf_S$ is similarly a morphism in $\Fin_{/S}$ whose image along $d_1$ satisfies the requirements from Definition \ref{defn:morgaf}. The assignment $S_*\mapsto\bGr_S$ extends to a functor $\Fin_*\to\mathrm{Cat}_2$: given a map of finite pointed sets $g\colon S_*\to S_*'$, we consider the functor $\bGr_S\to\bGr_{S'}$ that forgets all data lying over $g^{-1}(*)$, and composes the remaining data with the restricted map $g\colon g^{-1}(S')\to S'$.

The constructed functor $\bGr_{(-)}\Fin_*\to\mathrm{Cat}_2$ satisfies the Segal condition: for $S_*\in\Fin_*$ we have an equivalence $\bGr_S\xrightarrow{\simeq}\bGr^S$ induced by the $S$ maps $S_*\to\ul1_*$ in $\Fin_*$ that send a single element of $S$ to $\ul1$, and the rest to $*$. A functor $\Fin_*\to\mathrm{Cat}_2$ satisfying the Segal condition shall be our notion of symmetric monoidal 2-category.

\begin{rem}
The lack of strictness for $\bGr$ is due to the fact that, in defining $\circhor$, we have chosen a disjoint union functor $-\sqcup-\colon\Fin\times\Fin\to \Fin$. One can in fact give a \emph{strict} 2-category equivalent to $\bGr$ by requiring all finite sets $V,H,E$ appearing in the description of an object in $\bGr(B,A)$ to be initial segments $\ul k$ of $\Z_{\ge1}$, and by using concatenation of initial segments of $\Z_{\ge1}$ to implement disjoint union. However, even with this model, $\bGr$ will not be strict as a (symmetric) \emph{monoidal} 2-category: given pairs of composable 1-morphisms $(G_1,G_2)$ and $(G'_1,G'_2)$, the 1-morphisms $(G_1\sqcup G'_1)\circhor(G_2\sqcup G'_2)$ and $(G_1\circhor G_2)\sqcup(G'_1\circhor G'_2)$ will be connected by a canonical 2-isomorphism which is in general not an identity. This problem is unavoidable: the braided monoidal category $\bGr(\emptyset,\emptyset)$, of endomorphisms of the monoidal unit, has a non-trivial braiding $G\sqcup G\simeq G\sqcup G$ for all non-empty gafs $G\in\bGr(\emptyset,\emptyset)$, contradicting the Eckmann--Hilton argument that would hold in a strict monoidal 2-category.
\end{rem}
We will mainly consider $\bGr$ as a symmetric monoidal $\inftwo$-category, and we will call it the $\inftwo$-category of \emph{graph cobordisms between finite sets}.
\begin{nota}
We let $\Gr\CAlg(\Catinfone)$ denote the symmetric monoidal $\infone$-category $|\bGr|_2$ obtained from $\Gr$ by inverting all 2-morphisms.
\end{nota}

\begin{rem}
\label{rem:Galatius}
Galatius \cite[Section 4]{Galatius} introduces a category of graph cobordisms between finite sets, denoted $\cC_\infty$: it is a category internal to $\Top$. Jan Steinebrunner pointed out to us that the underlying $\infone$-category of $\cC_\infty$ is \emph{not} equivalent to $\Gr$ (though it is likely equivalent to a wide $\infone$-subcategory of $\Gr$): indeed, Galatius only considers graphs whose vertices have valence at least 3, hence $\cC_\infty$ cannot account for graph cobordisms having components homotopy equivalent to a point or a circle that are not attached to the target finite set. 
Jan Steinebrunner has also pointed out to us that the combination of Propositions 4.8 and 4.9, and Theorem 5.1 in \cite{Galatius}, together with Corollary \ref{cor:A2}, prove at least the equivalence $|\cC_\infty|\simeq\Omega^{\infty-1}\bS\simeq|\bGr|$; in this sense, the first statement akin to Corollary \ref{cor:A2} appearing in the literature is due to Galatius. Compare also with \cite[Remark 4.18]{SteinebrunnerNotes}.
\end{rem}

\subsection{Some properties of \texorpdfstring{$\bGr$}{bGr}}
\label{subsec:propertiesGr}
We next highlight some features of $\bGr$.

\begin{nota}
\label{nota:mu}
We denote by $\mu\colon \ul 2\to \ul 1$ be the unique map of finite sets.
\end{nota}
\begin{description}[style=unboxed,leftmargin=0cm]
 \item[The functor $\fI_\bGr$] Consider $\Fin$ as a symmetric monoidal 2-category having only identity 2-morphisms; then we have a symmetric monoidal 2-functor $\fI_\bGr\colon\Fin\to\bGr$, sending $A\mapsto A$ and sending a map of finite sets $f\colon B\to A$ to the unique gaf $G\in\bGr(B,A)$ with $V=H=\emptyset$ and $\rho=f$.
  \item[The 1-morphism $\ft_\bGr$] We have a special 1-morphism $\ft_\bGr\colon\fI_\bGr(\ul 1)\to\fI_\bGr(\ul 0)$, given by the unique gaf $G_\ft\in\bGr(\ul 1,\ul 0)$ with $V_\ft=\ul 1$. We represent $\ft_\bGr$ pictorially as follows:
  \begin{center}
  \begin{tikzpicture}
   \node[anchor=east] at(0,0){\tiny 1 $\bullet$};
   \node at(1,0){\tiny $\bullet$};\node[anchor=south] at(1,0){$\ft_\bGr$};
   \node at(2,0){$\emptyset$};
   \draw (0,0) to ++(1,0);
  \end{tikzpicture}
 \end{center}
 \item[The $C_2$-equivariant 1-morphism $\fe_\bGr$] We have in $\bGr$ a special 1-morphism $\fe_\bGr\colon \fI_\bGr(\ul 0)\to\fI_\bGr(\ul 2)$, given by the unique gaf $G_\fe\in\bGr(\ul 0,\ul 2)$ with $V_\fe=\emptyset$, $H_\fe=\ul 2$ and $\sigma_\fe=\Id_{\ul 2}$. We represent $\fe_\bGr$ pictorially as follows:
 \begin{center}
  \begin{tikzpicture}
   \node at(0,.5){$\emptyset$};
   \node[anchor=west] at(2,1){\tiny$\bullet$ 1};
   \node[anchor=west] at(2,0){\tiny$\bullet$ 2};
   \draw[thick] (2,0) [out=180,in=-90] to ++(-.5,.5) node[right]{$\fe_\bGr$} [out=90, in=180] to ++(.5,.5);
  \end{tikzpicture}
 \end{center}
We further consider the trivial and free action of the group $C_2$ on the objects $\ul 0$ and $\ul 2$ of $\Fin$, respectively, and obtain an action of $C_2$ on the objects $\fI_\bGr(\ul 0)$ and $\fI_\bGr(\ul 2)$ of $\bGr$ and, by conjugation, on the category $\bGr(\fI_\bGr(\ul 0),\fI_\bGr(\ul 2))$. Then $G_\fe\in\bGr(\fI_\bGr(\ul 0),\fI_\bGr(\ul 2))$ admits a natural $C_2$-equivariant structure, induced by the unique non-trivial involution $f\colon G_\fe\xrightarrow{\simeq} G_\fe$ in $\Gaf$ acting freely on both $A=H_\fe=\ul 2$.

 \item[The 2-morphism $\beta_\bGr$] Using the above, we can construct the following 1-mor\-phisms $G_{\beta,1}, G_{\beta,2}\colon \fI_\bGr(\ul 1)\to\fI_\bGr(\ul 1)$ in $\bGr$. We let
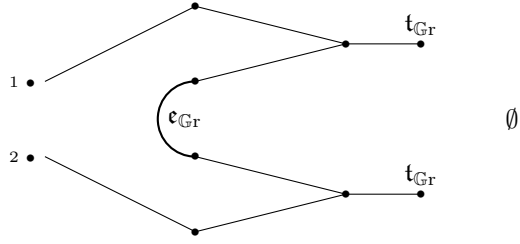
 \[
 G_{\beta,1} =((\ft_\bGr\circhor\fI_\bGr(\mu))\sqcup\Id_{\fI_\bGr(\ul 1)})\circhor(\Id_{\fI_\bGr(\ul 1)}\sqcup\fe_\bGr),
 \]
 i.e. the following horizontal composite, where we remove all instances of ``$\fI_\bGr$'':
 \[
 \begin{tikzcd}[column sep=55pt]
 \ul 1\cong\ul 1\sqcup\ul 0\ar[r,"{\Id_{\ul 1}\sqcup\fe_\bGr}"] &\ul 1\sqcup\ul 2\cong\ul 3 \cong\ul 2\sqcup\ul 1\ar[r,"(\ft_\bGr\circhor\mu)\sqcup\Id_{\ul 1}"] & \ul 0\sqcup \ul 1\cong\ul 1.
 \end{tikzcd}
 \]
 We represent $G_{\beta,1}$ pictorially as follows:
  \begin{center}
  \begin{tikzpicture}
   \node[anchor=east] at(0,1){\tiny 1 $\bullet$};
   \draw(0,1) to ++(2,1);
   \node at(2,2){\tiny$\bullet$};
   \node at(2,1){\tiny$\bullet$};
   \node at(2,0){\tiny$\bullet$};
   \draw[thick] (2,0) [out=180,in=-90] to ++(-.5,.5) node[right]{$\fe_\bGr$} [out=90, in=180] to ++(.5,.5);
   \node at(4,1.5){\tiny$\bullet$};
   \draw (2,2) to ++(2,-.5);
   \draw (2,1) to ++(2,.5);
   \node at(5,1.5){\tiny $\bullet$};\node[anchor=south] at(5,1.5){$\ft_\bGr$};
   \draw (4,1.5) to ++(1,0);
   \node[anchor=west] at(6,1){\tiny $\bullet$ 1};
   \draw (2,0) to ++(4,1);
  \end{tikzpicture}
 \end{center}
We let $G_{\beta,2}$ be the identity 1-morphism $\Id_{\fI_\bGr(\ul 1)}$. Then we have a special 2-morphism $\beta_\bGr\colon G_{\beta,1} \Rightarrow G_{\beta,2}$, which we describe in the following. As gaf's, $G_{\beta,1}$ is characterised up to unique isomorphism by the following: $A_{\beta,1}=B_{\beta,1}=\ul 1$ and $V_{\beta,1}=\set{v}\simeq\ul1$ are singletons; $H_{\beta,1}=\set{h_1,h_v}\simeq\ul2$ has two elements; $\sigma_{\beta,1}$ sends $h_1\mapsto 1\in A_{\beta,1}$ and $h_v\mapsto v\in V_{\beta,1}$; and $\rho_{\beta,1}$ sends $1\mapsto v$. We similarly characterise $G_{\beta,2}$ by $A_{\beta,2}=B_{\beta,2}=\ul 1$ and $V_{\beta,2}=H_{\beta,2}=\emptyset$. We then define $\beta_\bGr\colon G_{\beta,1}\to G_{\beta,2}$ as the unique morphism of gaf's, sending 
$V_{\beta,1}\sqcup H_{\beta,1}$ constantly to $1\in A_{\beta,2}$.  At the level of graphs, $G_{\beta,1}$ is a segment with one attaching vertex and marked once at the other vertex, whereas $G_{\beta,2}$ consists of a single once-marked attaching vertex; then $\beta_\bGr$ is the contraction of $G_{\beta,1}$ onto $G_{\beta,2}$.
Pictorially, $\beta_\bGr$ straightens the above S-shaped composition giving $G_{\beta,1}$ to a straight segment.
\item[The $C_2$-equivariant 2-morphism $\tbeta_\bGr$] Using the above, there is a canonical identification of the following two 1-morphisms in $\bGr(\fI_\bGr(\ul 2),\fI_\bGr(\ul 0))$:
\[
\begin{array}{c}
\ft_\bGr\circhor\fI_\bGr(\mu)\circhor(G_{\beta,1}\sqcup\Id_{\fI_\bGr(\ul 1)});\\[.1cm]
(\ft_\bGr\sqcup\ft_\bGr)\circhor(\fI_\bGr(\mu)\sqcup\fI_\bGr(\mu))\circhor(\Id_{\fI_\bGr(\ul 1)}\sqcup\fe_\bGr\sqcup\Id_{\fI_\bGr(\ul 1)}),
\end{array}
\]
given by replacing ``$G_{\beta,1}$'' by its formula and using the axioms of symmetric monoidal 2-category. Both 1-morphisms are represented by the unique gaf $G_{\tbeta,1}\in\bGr(\ul 2,\ul 0)$ with $V_{\tbeta,1}=H_{\tbeta,1}=\ul 2$ and $\rho_{\tbeta,1}=\sigma_{\tbeta,1}=\Id_{\ul 2}$. As a graph, $G_{\tbeta,1}$ is a segment attached to the empty set and marked once at each endpoint; the first formula highlights how $G_{\tbeta,1}$ can be obtained from $G_{\beta,1}$ by converting the attaching vertex $A_{\beta,1}$ into a once-marked inner vertex, namely the element $2\in V_{\tbeta,1}\simeq\ul2$; instead $V_{\beta,1}$ gives rise to $1\in V_{\tbeta,1}$. Both of the following represent pictorially $G_{\tbeta,1}$, according to the first and to the second formula, respectively:
  \begin{center}
  \begin{tikzpicture}
   \node[anchor=east] at(0,1){\tiny 1 $\bullet$};
   \node[anchor=east] at(0,0){\tiny 2 $\bullet$};
   \draw(0,1) to ++(2,1);
   \draw(0,0) to [out=-20,in=200] (6,0);
   \node at(2,2){\tiny$\bullet$};
   \node at(2,1){\tiny$\bullet$};
   \node at(2,0){\tiny$\bullet$};
   \draw[thick] (2,0) [out=180,in=-90] to ++(-.5,.5) node[right]{$\fe_\bGr$} [out=90, in=180] to ++(.5,.5);
   \node at(4,1.5){\tiny$\bullet$};
   \draw (2,2) to ++(2,-.5);
   \draw (2,1) to ++(2,.5);
   \draw (6,1) to ++(2,-.5);
   \draw (6,0) to ++(2,.5);
   \node at(6,1){\tiny $\bullet$};
   \node at(6,0){\tiny $\bullet$};
   \draw (2,0) to ++(4,1);
   \node at(5,1.5){\tiny $\bullet$};\node[anchor=south] at(5,1.5){$\ft_\bGr$};\
   \node at(8,.5){\tiny$\bullet$};
   \node at(9,.5){\tiny $\bullet$};\node[anchor=south] at(9,.5){$\ft_\bGr$};
   \draw (4,1.5) to ++(1,0);
   \draw (8,.5) to ++(1,0);
   \node[anchor=west] at(10,.5){$\emptyset$};
  \end{tikzpicture}
 \end{center}
   \begin{center}
  \begin{tikzpicture}
   \node[anchor=east] at(0,1){\tiny 1 $\bullet$};
   \node[anchor=east] at(0,0){\tiny 2 $\bullet$};
   \draw(0,1) to ++(2,1);
   \draw(0,0) to ++(2,-1);
   \node at(2,2){\tiny$\bullet$};
   \node at(2,1){\tiny$\bullet$};
   \node at(2,0){\tiny$\bullet$};
   \node at(2,-1){\tiny$\bullet$};
   \draw[thick] (2,0) [out=180,in=-90] to ++(-.5,.5) node[right]{$\fe_\bGr$} [out=90, in=180] to ++(.5,.5);
   \node at(4,1.5){\tiny$\bullet$};
   \node at(4,-.5){\tiny$\bullet$};
   \draw (2,2) to ++(2,-.5);
   \draw (2,1) to ++(2,.5);
   \draw (2,0) to ++(2,-.5);
   \draw (2,-1) to ++(2,.5);
   \node at(5,1.5){\tiny $\bullet$};\node[anchor=south] at(5,1.5){$\ft_\bGr$};
   \node at(5,-.5){\tiny $\bullet$};\node[anchor=south] at(5,-.5){$\ft_\bGr$};
   \draw (4,1.5) to ++(1,0);
   \draw (4,-.5) to ++(1,0);
   \node[anchor=west] at(6,.5){$\emptyset$};
  \end{tikzpicture}
 \end{center}
 We can also consider the 1-morphism $\ft_\bGr\circhor\fI_\bGr(\mu)$, given by the unique gaf $G_{\tbeta,2}\in\bGr(\ul 2,\ul 0)$ with $V_{\tbeta,2}=\ul1$ and $H_{\tbeta,2}=\emptyset$: as a graph, this is a twice-marked vertex attached to the empty set. We represent $G_{\tbeta,2}$ pictorially as follows:
   
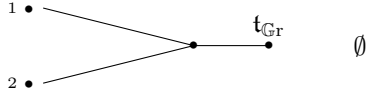
\begin{center}
  \begin{tikzpicture}
   \node[anchor=east] at(0,1){\tiny 1 $\bullet$};
   \node[anchor=east] at(0,0){\tiny 2 $\bullet$};
   \draw (0,1) to ++(2,-.5);
   \draw (0,0) to ++(2,.5);
   \node at(2,.5){\tiny$\bullet$};
   \node at(3,.5){\tiny $\bullet$};\node[anchor=south] at(3,.5){$\ft_\bGr$};
   \draw (2,.5) to ++(1,0);
   \node[anchor=west] at(4,.5){$\emptyset$};
  \end{tikzpicture}
 \end{center}
 Using the first formula for $G_{\tbeta,1}$, we obtain a 2-morphism
 \[
 \tbeta_\bGr:=\ft_\bGr\circhor\fI_\bGr(\mu)\circhor(\beta_\bGr\sqcup\Id_{\fI_\bGr(\ul 1)})\colon G_{\tbeta,1}\Rightarrow G_{\tbeta,2},
 \]
 where we write identities of 1-morphisms as the 1-morphisms themselves.
 
 We have actions of $C_2$ on $\ul 2$ (by swap), on $\ul 1$ (trivial action) and on $\ul 4$ (by the map exchanging 1-4 and 2-3). The horizontal factors of the second formula for $G_{\tbeta,1}$, namely $(\ft_\bGr\sqcup\ft_\bGr)\in\bGr(\fI_\bGr(\ul 2),\fI_\bGr(\ul 0))$, $(\fI_\bGr(\mu)\sqcup\fI_\bGr(\mu))\in\bGr(\fI_\bGr(\ul 4),\fI_\bGr(\ul 2))$ and $(\Id_{\fI_\bGr(\ul 1)}\sqcup\fe_\bGr\sqcup\Id_{\fI_\bGr(\ul 1)})\in\bGr(\fI_\bGr(\ul 2),\fI_\bGr(\ul 4))$ have a natural structure of $C_2$-equivariant objects in the corresponding $C_2$-equivariant categories: the last one, for instance, is obtained combining the $C_2$-equivariant structure on $\fe_\bGr$ and symmetric monoidality of $\sqcup$. Similarly, $G_{\tbeta,2}$ has a $C_2$-equivariant structure in $\bGr(\fI_\bGr(\ul 2),\fI_\bGr(\ul 0))$. It makes thus sense to ask whether or not $\tbeta_\bGr$ is a $C_2$-equivariant morphism \footnote{Note that $\bGr(\ul 2,\ul 0)$ is a \emph{plain 1-category} with $C_2$-action, so even if ``being $C_2$-equivariant'' is a structure on an object, it is a property for a morphism between $C_2$-equivariant objects.}. At the level of graphs, $\tbeta_\bGr$ is the contraction of the segment $G_{\tbeta,1}$ onto the point $G_{\tbeta,2}$, and this is clearly invariant under swapping the two endpoints of $G_{\tbeta,1}$, together with their markings, and also swapping the two half-edges of $G_{\tbeta,1}$. So $\tbeta_\bGr$ is indeed $C_2$-equivariant.
\end{description}

Passing to $\inftwo$-categories, we have the following structure on $\bGr$:
\begin{itemize}
 \item $\bGr$ is a symmetric monoidal $\inftwo$-category;
 \item we are given a symmetric monoidal $\inftwo$-functor $\fI_\bGr\colon\Fin\to\bGr$;
 \item we are given a 1-morphism $\ft_\bGr\colon\fI_\bGr(\ul 1)\to \fI_\bGr(\ul 0)$;
 \item we are given a 1-morphism $\fe_\bGr\colon\fI_\bGr(\ul 0)\to\fI_\bGr(\ul 2)$, and a $C_2$-equivariant structure on $\fe_\bGr$ as an object of the $C_2$-equivariant $\infone$-category
 \[
 \bGr(\fI_\bGr(\ul 0),\fI_\bGr(\ul 2));
 \]
 \item we are given a 2-morphism
 \[
 \beta_\bGr\colon ((\ft_\bGr\circhor\fI_\bGr(\mu))\sqcup\Id_{\fI_\bGr(\ul 1)})\circhor(\Id_{\fI_\bGr(\ul 1)}\sqcup\fe_\bGr)\Rightarrow\Id_{\fI_\bGr(\ul 1)};
 \]
 \item we are given a $C_2$-equivariant structure on the resulting 2-morphism
 \[
 \tbeta_\bGr:=\ft_\bGr\circhor\fI_\bGr(\mu)\circhor(\beta_\bGr\sqcup\Id_{\fI_\bGr(\ul 1)}),
 \]
 considered as a morphism between $C_2$-equivariant objects in the $C_2$-equiva\-riant $\infone$-category $\bGr(\fI_\bGr(\ul 2),\fI_\bGr(\ul 0))$.
\end{itemize}

\subsection{Graph-like structures}
We axiomatise the previous discussion about $\bGr$.
\begin{defn}
 \label{defn:graphlike}
 A \emph{semi-graph-like structure} on a symmetric monoidal $\inftwo$-category $\cD$ is the datum of:
 \begin{itemize}
  \item a symmetric monoidal $\inftwo$-functor $\fI\colon\Fin\to\cD$;
  \item an object $\ft$ in the $\infone$-category $\cD(\fI(\ul 1),\fI(\ul 0))$;
  \item a $C_2$-equivariant object $\fe$ in $\cD(\fI(\ul 0), \fI(\ul 2))$, which is a $C_2$-equivariant $\infone$-category.
\end{itemize}
A \emph{graph-like structure} on $\cD$ is the datum of a semi-graph-like structure as above, together with:
\begin{itemize}
\item a 2-morphism $\beta\colon ((\ft\circhor \fI(\mu))\otimes\Id_{\fI(\ul 1)})\circhor(\Id_{\fI(\ul 1)}\otimes \fe)\Rightarrow\Id_{\fI(\ul 1)}$, which can be regarded as a morphism in the $\infone$-category $\cD(\fI(\ul 1),\fI(\ul 1))$;
  \item a $C_2$-equivariant structure on the 2-morphism $\tbeta:=\ft\circhor \fI(\mu)\circhor(\beta\otimes\Id_{\fI(\ul 1)})$, which can be regarded as a morphism in the $C_2$-equivariant $\infone$-category $\cD(\fI(\ul 2),\fI(\ul 0))$ between $C_2$-equivariant objects.
\end{itemize}
\end{defn}
\begin{nota}
We usually denote a graph-like structure as in Definition \ref{defn:graphlike} by $(\fI,\ft,\fe,\beta)$, leaving the other data, especially the various $C_2$-equivariant structures, understood.
For instance, the graph-like structure on $\bGr$ from Subsection \ref{subsec:propertiesGr} will be denoted $(\fI_\bGr,\ft_\bGr,\fe_\bGr,\beta_\bGr)$.
\end{nota}
In the following we construct functors $\SGL,\GL\colon\CAlg(\Catinftwo)\to\cS$ associa\-ting with a symmetric monoidal $\inftwo$-category $\cD$ its \emph{spaces} of semi-graph-like and graph-like structures, respectively.

\begin{defn}
\label{defn:SGL}
We denote by $\SGL$ the limit of the following W-shaped diagram of limit-preserving and accessible functors $\CAlg(\Catinftwo)\to\cS$:
 \[
 \begin{tikzcd}
  \Fun([1],-)^\simeq\ar[d,"{(d_1,d_0)}"']
  & \Fun^\otimes(\Fin,-)^\simeq \ar[dr,"\ev_{\ul 0}\times\ev_{\ul 2}"']\ar[dl,"\ev_{\ul 1}\times\ev_{\ul 0}"]&
  \Fun(\bB C_2\times[1],-)^\simeq\ar[d,"{(d_1,d_0)}"]
\\
  \Fun(\set{0,1},-)^\simeq& &\Fun(\bB C_2\times\set{0,1},-)^\simeq.
 \end{tikzcd}
 \]
\end{defn}
Informally, for $\cD\in\CAlg(\Catinftwo)$, $\SGL(\cD)$ is the space of choices $(\fI,\ft,\fe)$ of a symmetric monoidal functor $\fI\colon \Fin\to\cD$, an object $\ft\in\cD(\fI(\ul 1),\fI(\ul 0))$, and a $C_2$-equivariant object $\fe\in\cD(\fI(\ul 0),\fI(\ul 2))$.
We have a natural transformation $\SGL\Rightarrow\Fun(\mO_1,-)^\simeq$ of functors $\CAlg(\Catinftwo)\to\cS$, sending $(\fI,\ft,\fe)\in \SGL(\cD)$ to the following $\mO_1$-shaped diagram in $\cD$

\[
 \begin{tikzcd}[column sep=70pt]
  \fI(\ul 1)\ar[r,bend left=10,"((\ft\circhor \fI(\mu))\otimes\Id_{\fI(\ul 1)})\circhor(\Id_{\fI(\ul 1)}\otimes \fe)"]\ar[r,bend right=10,"\Id_{\fI(\ul 1)}"'] &\fI(\ul 1).
 \end{tikzcd}
\]
We then let $\SGL'\colon\CAlg(\Catinftwo)\to\cS$ be pullback of the cospan diagram of functors $\SGL\to \Fun(\mO_1,-)^\simeq\ot\Fun(\Theta_2,-)^\simeq$: again $\SGL'$ is limit-preserving and accessible; the space $\SGL'(\cD)$ parametrises choices of a sequence $(\fI,\ft,\fe,\beta)$ as in Definition \ref{defn:graphlike}, with a $C_2$-equivariant structure on $\fe$. Finally, we construct a natural transformation $\SGL'\Rightarrow \Fun((\bB C_2\times \mO_1)\sqcup_{(*\times \mO_1)}(*\times\Theta_2),-)^\simeq$, sending $(\fI,\ft,\fe,\beta)\in \SGL'(\cD)$ to the following $\Theta_2$-diagram in $\cD$, whose sub-$\mO_1$-diagram is endowed with a $C_2$-equivariant structure

\[
 \begin{tikzcd}[column sep=230pt]
  \fI(\ul 2)\ar[r,bend left=10,"\ft\circhor \fI(\mu)\circhor((((\ft\circhor \fI(\mu))\otimes\Id_{\fI(\ul 1)})\circhor(\Id_{\fI(\ul 1)}\otimes \fe))\otimes\Id_{\fI(\ul 1)})",""{name=U,inner sep=1pt,below}]
  \ar[r,bend right=10,"\ft\circhor \fI(\mu)"{below},""{name=D,inner sep=1pt}]
  \ar[Rightarrow,from=U, to=D, "\tbeta:=\ft\circhor \fI(\mu)\circhor(\beta\otimes\Id_{\fI(\ul 1)})"]
  &\fI(\ul 0)
 \end{tikzcd}
\]
\begin{defn}
\label{defn:GLfunctor}
 We define $\GL\colon\CAlg(\Catinftwo)\to\cS$ as the pullback of the cospan diagram of functors
 \[
\SGL'\to \Fun\big((\bB C_2\times \mO_1)\sqcup_{(*\times \mO_1)}(*\times\Theta_2),-\big)^\simeq \ot \Fun(\bB C_2\times\Theta_2,-)^\simeq.
\]
\end{defn}
Again, $\GL$ is limit-preserving and accessible;
the space $\GL(\cD)$ is by definition the space of graph-like structure on $\cD$.

Since the functors $\SGL$ and $\GL$ are limit-preserving and accessible functors between the presentable $\infone$-categories $\CAlg(\Catinftwo)$ and $\cS$, we can invoke \cite[Proposition 5.5.2.7]{LurieHTT} and conclude that $\SGL$ and $\GL$ are representable.

\begin{nota}
\label{nota:cUcV}
We denote by $\cU$ and $\cV$ the symmetric monoidal $\inftwo$-categories representing $\SGL$ and $\GL$, respectively, i.e. we have equivalences $\Fun^\otimes(\cU,-)^\simeq\simeq\SGL$ and $\Fun^\otimes(\cV,-)^\simeq\simeq\GL$.
We denote by $(\fI_\cU,\ft_\cU,\fe_\cU)$ the universal semi-graph-like structure on $\cU$, corresponding to $\Id_\cU\in\Fun^\otimes(\cU,\cU)$; similarly, $(\fI_\cV,\ft_\cV,\fe_\cV,\beta_\cV)$ denotes the universal graph-like structure on $\cV$. 
Finally, we denote by $\fJ\colon\cU\to\cV$ the symmetric monoidal $\inftwo$-functor corresponding to $(\Id_\cV,\ft_\cV,\fe_\cV)$, and by $\fG\colon \cV\to\bGr$ the symmetric monoidal $\inftwo$-functor corresponding to $(\fI_\bGr,\ft_\bGr,\fe_\bGr,\beta_\bGr)$.
\end{nota}
\begin{rem}
\label{rem:cUinfone}
We observe that $\SGL$ factors through $(-)^\twosimeq\colon\Catinftwo\to\Catinfone$, or in other words, that for any $\cD\in\CAlg(\Catinftwo)$ the inclusion $\cD^\twosimeq\to\cD$ induces an equivalence $\SGL(\cD^\twosimeq)\simeq\SGL(\cD)$: indeed the sources $\star$ of the functors $\Fun(\star,-)^\simeq$ and $\Fun^\otimes(\star,-)$ in the diagram of Definition \ref{defn:SGL} are all $\infone$-categories. As a consequence we have that $\cU$ is an $\infone$-category, i.e. all of its 2-morphisms are invertible, and $\fJ$ restricts to a symmetric monoidal $\infone$-functor $\cU\to\cV^\twosimeq$. We will see after Corollary \ref{cor:nsimeqinvariance} that the latter is in fact an equivalence.
\end{rem}

We can now rephrase Theorem \ref{thm:A} as the statement that $\fG\colon\cV\to\bGr$ is an equivalence of $\inftwo$-categories; we will prove this statement in
Section \ref{sec:proofthmA}.

A key step in proving Theorem \ref{thm:A} will be the following theorem, implying that $\bGr^\twosimeq$ represents $\SGL$.
\begin{thm}
\label{thm:Abis}
The symmetric monoidal $\infone$-functor $\fF\colon\cU\to\bGr^\twosimeq$ corresponding to $(\fI_\bGr,\ft_\bGr,\fe_\bGr)\in\SGL(\bGr^\twosimeq)$ is an equivalence.
\end{thm}

\section{Free extensions of symmetric monoidal \texorpdfstring{$\infn$}{infn}-categories}
\label{sec:symmonextensions}
Recall from Notation \ref{nota:cUcV} that $\cV$ denotes the symmetric monoidal $\inftwo$-catego\-ry freely obtained from $\Fin$ by adjoining the 1-morphism $\ft_\cV$, the $C_2$-equivariant 1-morphisms $\fe_\cV$, the 2-morphism $\beta_\cV$, and the $C_2$-equivariant structure on the resulting 2-morphism $\tbeta_\cV$.
In order to identify $\cV$ with $\bGr$, we need to understand how a free symmetric monoidal extension of an $\inftwo$-category looks like. In this section we fix $n\ge1$ and consider the problem of describing extensions of symmetric monoidal $\infn$-categories obtained by freely adjoining spaces of new $n$-morphisms.
\subsection{Definition of free extensions}
\begin{nota}
Recall Definition \ref{defn:ThetanmOn}. For a pair of spaces $(X,\del X)$, i.e. a map of spaces $i\colon\del X\to X$ we denote by $\del(X\times \Theta_n)$, leaving $\del X$ and $i$ implicit in the notation, the $\infn$-category $X\times\mO_{n-1}\sqcup_{\del X\times\mO_{n-1}}\del X\times\Theta_n$; we have a canonical $\infn$-functor $\del(X\times\Theta_n)\to X\times\Theta_n$.
\end{nota}

\begin{defn}
\label{defn:extension}
Let $n\ge1$, let $(X,\del X)$ be a pair of spaces, let $\cC\in\CAlg(\Catinfn)$, and let $F\colon\del(X\times\Theta_n)\to\cC$ be an $\infn$-functor, whose datum is equivalent to a commutative square of spaces as follows:
\[
\begin{tikzcd}[row sep=10pt, column sep=70pt]
\del X\ar[d]\ar[r,"F|_{\del X\times\Theta_n}"]&\Fun(\Theta_n,\cC)^\simeq\ar[d,"\text{restriction}"]\\
X\ar[r,"F_{X\times\mO_{n-1}}"]&\Fun(\mO_{n-1},\cC^\nsimeq)^\simeq.
\end{tikzcd}
\]
We denote by $\cC[X,\del X]^\sm$, leaving $F$ implicit in the notation, the symmetric monoidal $\infn$-category representing the limit-preserving and accessible functor
\[
\Fun^\otimes(\cC,-)^\simeq\times_{\Fun(\del(X\times \Theta_n),-)^\simeq}\Fun(X\times\Theta_n,-)^\simeq\colon\CAlg(\Catinfn)\to\cS.
\]
If $\del X=\emptyset$, we just write $\cC[X]^\sm$.
\end{defn}
Informally, $\cC[X,\del X]^\sm$ is the free symmetric monoidal $\infn$-category obtained from $\cC$ by adjoining ``$X$ many'' new $n$-morphisms, with sources and targets of all depths specified by the map $X\to\Fun(\mO_{n-1},\cC^\nsimeq)^\simeq$, and by identifying ``$\del X$'' many of them with the $n$-morphisms of $\cC$ specified by the map $\del X\to\Fun(\Theta_n,\cC)^\simeq$.
\begin{nota}
\label{nota:multiextension}
Let $k\ge0$, let $\cC=\cC_0\to\dots\to\cC_k=\cD$ be a sequence of morphisms in $\CAlg(\Catinfn)$, and for $1\le i\le k$ let $(X_i,\del X_i)$ be a pair of spaces and $F_i\colon\del(X_i\times\Theta_n)\to\cC_{i-1}$ be an $\infn$-functor.
If the canonical map $\cC_{i-1}[X_i,\del X_i]^\sm\to\cC_i$ is an equivalence for all $1\le i\le k$, we say that $\cD$ is an \emph{iterated extension} and write $\cD\simeq\cC[X_1,\del X_1;\dots;X_k,\del X_k]^\sm$, or shortly $\cD\simeq\cC[\ul X,\ul{\del X}]^\sm$.
\end{nota}
\begin{ex}
\label{ex:cUcVasextensions}
For $n=1$ we have $\cU\simeq\Fin[*\sqcup\bB C_2]^\sm$ for the functor
\[
F\colon\del((*\sqcup\bB C_2)\times\Theta_1)\simeq (*\sqcup\bB C_2)\times \mO_0\simeq (*\sqcup\bB C_2)\times\set{s_0,t_0}\to\Fin
\]
restricting to the constant functor at $\ul0$ on $\bB C_2\times s_0\sqcup *\times t_0$, the constant functor at $\ul1$ on $*\times s_0$, and the inclusion $\bB C_2\hto\Fin^\simeq\hto\Fin$ on $\bB C_2\times t_0\simeq \bB C_2$.

For $n=2$ we have $\cV\simeq\cU[*\ ;\ \bB C_2,*]^\sm$: the first extension is given by
\[
F_1\colon\del(*\times\Theta_2)\simeq\mO_1\to\cU
\]
sending $s_0,t_0\mapsto \ul1$, and sending $s_1\mapsto((\ft_\cU\mu)\sqcup \Id_{\ul1})(\Id_{\ul1}\sqcup\fe_\cU)$ and $t_1\mapsto \Id_{\ul1}$; the result of this extension is adjoining the 2-morphism $\beta_\cV$. The second extension, using the pair of spaces$(\bB C_2,*)$, is given by
\[
F_2\colon\del(\bB C_2\times \Theta_2)\simeq\bB C_2\times\mO_1\sqcup_{*\times\mO_1}*\times\Theta_2\to\cU[*],
\]
whose description is precisely given by the diagram before Definition \ref{defn:GLfunctor}.
\end{ex}

\subsection{Multigraded categories}
In the study of extensions of symmetric monoidal $\infn$-catego\-ries as in Definition \ref{defn:extension} and Notation \ref{nota:multiextension}, it is convenient to consider the ``old $n$-morphisms'' as living in grading 0, and the ``new $n$-morphisms'' as living in positive grading. We introduce some definitions to make this precise.
In the rest of the section we denote by $M$ an \emph{ungroup-like} abelian monoid as in Definition \ref{defn:ungrouplike}.
\begin{defn}
An \emph{$M$-graded $\infn$-category} is an $\infn$-category $\cC$ together with an $\infn$-functor $\fd\colon\cC\to\bB^nM$. We denote by $\Catinfn^M:=(\Catinfn)_{/\bB^nM}$ the presentable $\infone$-category of small $M$-graded $\infn$-categories; we denote by $\Fun_M(\cC,\cD)^\simeq$ the morphism space between $\cC,\cD\in\Catinfn^M$ . 
For $\cC\in\Catinfn^{M}$ we usually denote by $\delta$ the composite map of spaces
\[
\Fun(\Theta_n,\cC)^\simeq\overset{\fd}{\to}\Fun(\Theta_n,\bB^nM)^\simeq\simeq M,
\]
and refer to it as an \emph{$M$-grading} on $\cC$. We present an object in $\Catinfn^{M}$ as a pair $(\cC,\fd)$ or as a pair $(\cC,\delta)$. For $a\in M$ we denote by $\Fun(\Theta_n,\cC)^\simeq_a$ the fibre at $a$ of $\delta$, which we think of as the space of all $n$-morphisms in $\cC$ of grading $a$.
\end{defn}
We note that $\delta$ sends invertible $n$-morphisms in $\cC$ to invertible $n$-morphisms in $\bB^nM$; since $M$ is ungroup-like abelian, the subspace of invertible $n$-morphisms of $\cC$ is contained in $\Fun(\Theta_n,\cC)^\simeq_0$.
We also note that, since $\bB^nM$ is a symmetric monoidal $\infn$-category, the overcategory $\Catinfn^M=(\Catinfn)_{/\bB^nM}$ inherits a symmetric monoidal structure; moreover we have an equivalence of presentable $\infone$-categories $\CAlg(\Catinfn^M)\simeq\CAlg(\Catinfn)_{/\bB^nM}$.

We further observe that a map of spaces $\delta\colon\Fun(\Theta_n,\cC)^\simeq\to M$ is a valid $M$-grading, making $(\cC,\delta)$ into an object in $\Catinfn^M$, precisely if $\delta$ is additive along compositions in all of the $n$ directions; and for $\cC\in\CAlg(\Catinfn)$, we have can upgrade $(\cC,\delta)$ to an object in $\CAlg(\Catinfn)$ (in an essentially unique way) if moreover $\delta$ is additive along symmetric monoidal products of $n$-morphisms.

\begin{defn}
\label{defn:DeltaopnM}
We denote by $\Delta^{\op,n}_M$ the category of multisimplices of the $n$-fold simplicial set $\bB^nM$ from Example \ref{ex:BnMungrouplike}, i.e. $\bDelta^{\op,n}_M\to(\bDelta^\op)^n$ is a left fibration corresponding to the functor $\bB^nM\colon(\bDelta^\op)^n\to\cS$. We denote by $\bDelta^n_M\to\bDelta^n$ the opposite, right fibration. An object in $\bDelta^{\op,n}_M$ or $\bDelta^n_M$ is usually denoted $([\bk],\nu)$ with $\bk\in\N^n$ and $\nu\in M^{\ul\bk}\simeq \bB^nM([\bk])$.

A morphism in $\bDelta^{\op,n}_M$ or $\bDelta^n_M$ is inert/active if the underlying morphism in $(\bDelta^\op)^n$ or $\bDelta^n$ is inert/active. We denote by $\bDelta^{\op,n}_{M,\inert}\subset \bDelta^{\op,n}_M$ and $\bDelta^n_{M,\inert}\subset \bDelta^n_M$ the subcategories spanned by all objects and inert morphisms.
\end{defn}
We remark that, for $M\neq0$, $\bDelta^n_M$ is \emph{not} equivalent to $(\bDelta_M)^n$.

\begin{defn}
We let $s^n_M\cS:=\Fun(\bDelta^{\op,n}_M,\cS)$ and refer to it as the $\infty$-category of $M$-graded $n$-fold simplicial spaces. For $X\in s^n_M\cS$ and $a\in X([\bk],\nu)$, the total grading of $a$ is defined as the total grading of $\nu$ (see Notation \ref{nota:totalgrading}).
\end{defn}
\begin{rem}
\label{rem:snMasovercategory}
We can equivalently define $s^n_M\cS$ as the overcategory $s^n\cS_{/\bB^nM}$;
with this second description, $s^n_M\cS$ is automatically endowed with a symmetric monoidal structure.
Explicitly, the tensor product of $X,Y\in s^n_M\cS$ is the functor $\bDelta^{\op,n}_M\to\cS$ sending the object $([\bk],\nu)$ to $\coprod_{y+z=x} X([\bk],y)\times Y([\bk],z)$.
\end{rem}

Just as $\Catinfn\subseteq s^n\cS$ is a full and reflective subcategory of $s^n\cS$ (see Subsection \ref{subsec:catofcat}), we can regard $\Catinfn^M$ as the full and reflective subcategory of $s^n_M\cS$ spanned by objects $X$ satisfying the following conditions:
\begin{itemize}
\item ($M$-Segal condition) The composite functor $X_\inert\colon\bDelta_{M,\inert}^{\op,n}\hto\bDelta^{\op,n}_M\overset{X}{\to}\cS$ is right Kan extended from the full subcategory $\bDelta_{M,\inert,\le1}^{\op,n}\subset\bDelta_{M,\inert}^{\op,n}$ spanned by objects $([\bk],\nu)$ with $\bk\in\set{0,1}^n$.
\item (Constancy and completeness condition) The restriction of $X$ to the full subcategory $(\bDelta^\op)^n\simeq\bDelta^{\op,n}_0\subseteq\bDelta^{\op,n}_M$ satisfies the constancy and completeness conditions from Subsection \ref{subsec:catofcat}. Here $\bDelta^{\op,n}_0$ denotes the full subcategory of $\bDelta^{\op,n}_M$ corresponding to the submonoid $0\subset M$.
\end{itemize}
Moreover $\Catinfn^M$ is a symmetric monoidal $\infone$-subcategory of $s^n_M\cS$; in fact, an object $X\in s^n_M\cS$ lies in $\Catinfn^M$ if and only if its image under the symmetric monoidal forgetful functor $s^n_M\cS\to s^n\cS$ lies in $\Catinfn$.
\begin{rem}
\label{rem:CatinfnMcharacterisation}
One can characterise $M$-graded $\infn$-categories as those $M$-graded $n$-fold simplicial spaces that are local with respect to a suitable set of maps of $M$-graded simplicial spaces, in a way which is analogous to Remark \ref{rem:Catinfnpresentable}. In particular $\Catinfn^M$ is an accessible localisation of $s^n_M\cS$.
\end{rem}

\begin{rem}
The abelian monoid structure on $M$ yields a symmetric monoidal structure on $\cS_{/M}$, and $\Catinfone^M$ can be identified with $\Cat_\infty^{\cS_{/M}}$ in the sense of \cite{GepnerHaugseng}.
\end{rem}

\begin{nota}
\label{nota:fshrik}
For a monoid homomorphism between ungroup-like abelian monoids $f\colon M\to M'$ we denote by abuse of notation also by $f$ the induced functor $\bDelta^{\op,n}_M\to\bDelta^{\op,n}_{M'}$, inducing by restriction a functor
$f^*\colon s^n_{M'}\cS\to s^n_M\cS$; the left adjoint to $f^*$, given by left Kan extension along $f$, is denoted $f_!$.
\end{nota}
In the setting of Notation \ref{nota:fshrik}, and in the light of Remark \ref{rem:snMasovercategory}, the functor $f_!\colon s^n\cS_{/\bB^n M}\to s^n\cS_{/\bB^n M'}$ sends $(X\overset{\fd}{\to}\bB^nM)\in\Catinfn^M$ to the composite
\[
(X\overset{\fd}{\to}\bB^nM\overset{\bB^nf}{\to}\bB^nM')\in\Catinfn^{M'};
\]
from the point of view of the grading, we replace $\delta\colon\Fun(\Theta_n,\cC)^\simeq\to M$ with $f\circ\delta$.
Similarly, $f^*\colon Y\mapsto Y\times_{\bB^nM'}\bB^nM$.
From this we see that $f_!$ is symmetric monoidal and $f^*$ is lax symmetric monoidal, and $f_!$ and $f^*$ restrict to an adjunction $\Catinfn^M\rightleftarrows\Catinfn^{M'}$. In particular we obtain adjunctions
\[
f_!\colon\CAlg(s^n_M\cS)\rightleftarrows\CAlg(s^n_{M'}\cS)\colon f^*;\quad f_!\colon\CAlg(\Catinfn^M)\rightleftarrows\CAlg(\Catinfn^{M'})\colon f^*.
\]
\begin{ex}
\label{ex:iotatau}
The initial map of monoids $\iota\colon 0\to M$ gives rise to a fully faithful functor $\iota_!\colon\CAlg(\Catinfn)\to\CAlg(\Catinfn^M)$, with essential image given by the $M$-graded symmetric monoidal $\infn$-categories all of whose $n$-morphisms have vanishing $M$-grading. The right adjoint $\iota^*$ sends $(\cC,\delta)$ to the sub-$\infn$-category of $\cC$ comprising all $i$-morphisms for $i\le n-1$, and all $n$-morphisms of vanishing $M$-grading.
Conversely, the terminal map $\tau\colon M\to 0$ gives rise to the functor $\tau_!\colon\CAlg(\Catinfn^M)\to\CAlg(\Catinfn)$ forgetting the $M$-grading, with right adjoint given by $\tau^*\colon\cC\mapsto\cC\times\bB^nM$.
\end{ex}
\begin{rem}
Let $\cD=\cC[X,\del X]^\sm$ be an extension as in Definition \ref{defn:extension}. Then an $M$-grading $\delta\colon\Fun(\Theta_n,\cD)^\simeq\to M$ is the same as a functor $\fd\colon\cD\to\bB^nM$, so it is uniquely determined by an $M$-grading $\delta|_{\cC}\colon\Fun(\Theta_n,\cC)^\simeq\to M$, a map $\delta_X\colon X\to M$, and homotopy filling the diagram of spaces
\[
\begin{tikzcd}[column sep=70pt, row sep=10pt]
\del X\ar[d]\ar[r,"F|_{\del X\times\Theta_n}"]&\Fun(\Theta_n,\cC)^\simeq\ar[d,"\delta|_\cC"]\\
X\ar[r,"\delta_X"]&M.
\end{tikzcd}
\]
Note that, since $M$ is discrete, such a homotopy, if it exists, is essentially unique.
\end{rem}

\begin{defn}
\label{defn:Nadmissible}
Recall Notation \ref{nota:multiextension}. We say that an iterated extension of symmetric monoidal $\infn$-categories $\cD=\cC[\ul X,\ul{\del X}]^\sm$ is \emph{$\N$-admissible} if there exists a $\N$-grading $\delta\colon\Fun(\Theta_n,\cD)^\simeq\to\N$ such that the composite
\[
\Fun(\Theta_n,\cC)^\simeq\to\Fun(\Theta_n,\cD)^\simeq\overset{\delta}{\to}\N
\]
is constantly 0, and each composite $X_i\to\Fun(\Theta_n,\cD)^\simeq\overset{\delta}{\to}\N$ is constantly 1. By the previous remarks, such an $\N$-grading, if it exists, is essentially unique.
\end{defn}
Both extensions $\Fin\to\cU$ and $\cU\to\cV$ from Example \ref{ex:cUcVasextensions} are $\N$-admissible. 

\subsection{Truncation and \texorpdfstring{$n$}{n}-groupoid cores of extensions.}
\begin{defn}
\label{defn:CatinfnI}
We say that a subset $I\subseteq M$ is a \emph{coideal} if $0\in I$ and if for all $a,b\in M$ with $a+b\in I$ we have $a,b\in I$.
We denote by $\bDelta^{\op,n}_I$ the full subcategory of $\bDelta^{\op,n}_M$ spanned by objects $([\bk],\nu)$ with $\nu$ having total grading in $I$ (see Notation \ref{nota:totalgrading}): note that this is stronger than just requiring all coordinates of $\nu$ to be in $I$. We define similarly $\bDelta^n_I\subset\bDelta^n_M$. We denote by $\iota^{I,M}\colon\bDelta^{\op,n}_I\hto\bDelta^{\op,n}_M$ the inclusion functor.
We denote by $\bDelta^{\op,n}_{I,\inert}$ the subcategory of $\bDelta^{\op,n}_I$ spanned by inert morphisms, and by $\iota^{I,M}_{\inert}\colon\bDelta^{\op,n}_{I,\inert}\hto\bDelta^{\op,n}_{M,\inert}$ the inclusion functor.

We denote by $s^n_I\cS:=\Fun(\bDelta^{\op,n}_I,\cS)$, and by $\Catinfn^I\subseteq s^n_I\cS$ the full and reflective subcategory spanned by objects $X$ satisfying the following conditions:
\begin{itemize}
\item($I$-Segal condition) the restriction of $X$ along the inclusion $\bDelta^{\op,n}_{I,\inert}\hto\bDelta^{\op,n}_I$ is right Kan extended from the full subcategory $\bDelta^{\op,n}_{I,\inert,\le1}\subset\bDelta^{\op,n}_{I,\inert}$ spanned by objects $([\bk],\nu)$ with $\bk\in\set{0,1}^n$;
\item(Constancy and completeness condition) the restriction of $X$ on the full subcategory $(\bDelta^\op)^n\simeq\bDelta^{\op,n}_0\subseteq\bDelta^{\op,n}_I$ satisfies the constancy and completeness conditions from Subsection \ref{subsec:catofcat}.
\end{itemize}
\end{defn}
An object in $\Catinfn^I$ can be thought of as an $I$-graded partial $\infn$-category, in the sense that $n$-morphisms are endowed with an $I$-grading, and compositions of $n$-morphisms (in any direction) is only defined if the sum of the gradings still lies in $I$.
It is possible to exhibit $\Catinfn^I$ as an accessible localisation of $s^n_I\cS$, similarly to Remark \ref{rem:CatinfnMcharacterisation}; in particular $\Catinfn^I$ is presentable.
\begin{defn}
\label{defn:Itruncated}
Let $I\subseteq M$ be a coideal.
An $M$-graded $\infn$-category $\cC$ is \emph{$I$-truncated} if the canonical map $\Fun(\Theta_n,\cC)^\simeq\times_M(M\setminus I)\to\Fun(\mO_{n-1},\cC)^\simeq\times (M\setminus I)$ is an equivalence.
\end{defn}
We note that $\cC\in\Catinfn^M$ is $I$-truncated if and only if for all $a\in M\setminus I$ the map $\Fun(\Theta_n,\cC)^\simeq_a\to\Fun(\mO_{n-1},\cC)^\simeq$ given by restriction is an equivalence. Roughly speaking, the last condition means that for all pairs of $(n-1)$-morphisms $s,t$ sharing all levels of sources and targets, and for all $a\in M\setminus I$, the space of $n$-morphisms from $s$ to $t$ of grading $a$ is contractible.

\begin{lem}
\label{lem:inertinitial}
Let $I\subseteq M$ be a coideal.
Let $([\bk],\nu)\in\bDelta^{\op,n}_M$, let $(\bDelta^{\op,n}_I)_{([\bk],\nu)/}$ denote the full subcategory of $(\bDelta^{\op,n}_M)_{([\bk],\nu)/}$ spanned by arrows with target in $\bDelta^{\op,n}_I$, and let 
$(\bDelta^{\op,n}_{I,\inert})_{([\bk],\nu)/}$ denote the full subcategory of $(\bDelta^{\op,n}_{M,\inert})_{([\bk],\nu)/}$ spanned by arrows with target in $\bDelta^{\op,n}_{I,\inert}$. Then the inclusion $(\bDelta^{\op,n}_{I,\inert})_{([\bk],\nu)/}\hto(\bDelta^{\op,n}_I)_{([\bk],\nu)/}$ is an initial functor, i.e. limits of diagrams parametrised by $(\bDelta^{\op,n}_{I})_{([\bk],\nu)/}$ can be computed upon restriction to $(\bDelta^{\op,n}_{I,\inert})_{([\bk],\nu)/}$. 
\end{lem}
\begin{proof}
We compute, for any $(\alpha\colon([\bk],\nu)\to([\bk'],\nu'))\in(\bDelta^{\op,n}_{I})_{([\bk],\nu)/}$, the overcategory $((\bDelta^{\op,n}_{I,\inert})_{([\bk],\nu)/})_{/\alpha}$: this is the same as the category whose objects and morphisms are described as follows:
\begin{itemize}
\item an object is a factorisation $([\bk],\nu)\overset{\beta}{\to}([\bk''],\nu'')\overset{\gamma}{\to}([\bk'],\nu')$ of $\alpha$ in $\bDelta^{\op,n}_M$, with $\beta$ inert and with $([\bk''],\nu'')\in\bDelta^{\op,n}_I$;
\item a morphism from $([\bk],\nu)\overset{\beta}{\to}([\bk''],\nu'')\overset{\gamma}{\to}([\bk'],\nu')$ to $([\bk],\nu)\overset{\beta'}{\to}([\bk'''],\nu''')\overset{\gamma'}{\to}([\bk'],\nu')$ is an inert morphism $\epsilon\colon([\bk''],\nu'')\to([\bk'''],\nu''')$ with $\epsilon\beta=\beta'$ and $\gamma'\epsilon=\gamma$.
\end{itemize}
The category $((\bDelta^{\op,n}_{I,\inert})_{([\bk],\nu)/})_{/\alpha}$ has a terminal object, given by the unique factorisation $([\bk],\nu)\overset{\beta}{\to}([\bk''],\nu'')\overset{\gamma}{\to}([\bk'],\nu')$ of $\alpha$ with $\beta$ inert and with $\gamma$ active: here we crucially use that $I$ is a coideal to conclude that, since $\beta$ is an active morphism in $\bDelta^{\op,n}_M$ with target in $\bDelta^{\op,n}_I$, also the source $([\bk''],\nu'')$ of $\beta$ lies in $\bDelta^{\op,n}_I$.

This shows that $((\bDelta^{\op,n}_{I,\inert})_{([\bk],\nu)/})_{/\alpha}$ is weakly contractible, and the claim follows from Quillen's Theorem A.
\end{proof}
\begin{cor}
\label{cor:inertinitial}
Let $I\subseteq M$ be a coideal.
Then the Beck--Chevalley transformation in the following diagram is an equivalence:
\[
\begin{tikzcd}[column sep=90pt]
s^n_I\cS\ar[r,"\text{restrict}"]\ar[d,hook,"(\iota^{I,M})_*"'] & \Fun(\bDelta^{n,\op}_{I,\inert},\cS)\ar[d,hook,"(\iota^{I,M}_{\inert})_*"]\\
s^n_M\cS\ar[r,"\text{restrict}"]\ar[ur,Rightarrow,"\simeq"] &\Fun(\bDelta^{n,\op}_{M,\inert},\cS).
\end{tikzcd}
\]
\end{cor}
\begin{proof}
The value at $([\bk],\nu)\in \bDelta^{\op,n}_{M,\inert}$ of the given natural transformation applied to $X\in s^n_I\cS$ is the comparison map of spaces $\lim_{(\bDelta^{\op,n}_I)_{([\bk],\nu)/}}X\to \lim_{(\bDelta^{\op,n}_{I,\inert})_{([\bk],\nu)/}}X$, which is an equivalence by Lemma \ref{lem:inertinitial}.
\end{proof}
We stress that the vertical functors in the diagram of Corollary \ref{cor:inertinitial} are fully faithful inclusions, since they are given by right Kan extension from a full subcategory; moreover the horizontal functors detect equivalences, since they are given by restriction to a wide subcategory.
\begin{lem}
\label{lem:Itruncateddescription}
Let $I\subseteq M$ be a coideal
and let $X\in s^n_M\cS$. Then $X$ is an $I$-truncated $M$-graded $\infn$-category if and only if both of the following hold:
\begin{enumerate}
\item $X$ is right Kan extended from its restriction on $\bDelta^{\op,n}_I$;
\item $X|_{\bDelta^{\op,n}_I}\in\Catinfn^I$.
\end{enumerate}
In other words, the right Kan extension functor $(\iota^{I,M})_*\colon s^n_I\cS\to s^n_M\cS$ restricts to a fully faithful inclusion $\Catinfn^I\hto\Catinfn^M$ with essential image given by the full subcategory of $I$-truncated $M$-graded $\infn$-categories.
\end{lem}
\begin{proof}
Let $X\in s^n_M\cS$, and let $X_\inert\colon\bDelta^{\op,n}_{M,\inert}$ denote the restriction of $X$ on $\bDelta^{\op,n}_{M,\inert}$. Then both the hypothesis (2) or the hypothesis ``$X\in\Catinfn^M$'' imply the (same) constancy and completeness conditions for $X$, as these are conditions on $X|_{\bDelta^{\op,n}_0}$.

Assume henceforth that $X$ does satisfy these constancy and completeness conditions. Then (2) is equivalent to requiring that $X_\inert|_{\bDelta^{\op,n}_{I,\inert}}$ be right Kan extended from the full subcategory $\bDelta^{\op,n}_{I,\inert,\le1}\subset\bDelta^{\op,n}_{I,\inert}$. Instead ``$X\in\Catinfn^M$'' is equivalent to requiring that $X_\inert$ be right Kan extended from the full subcategory $\bDelta^{\op,n}_{M,\inert,\le1}\subset\bDelta^{\op,n}_{M,\inert}$; the further hypothesis ``$X$ is $I$-truncated'' is then equivalent to requiring that $X_\inert|_{\bDelta^{\op,n}_{M,\inert,\le1}}$ is in turn right Kan extended from the full subcategory $\bDelta^{\op,n}_{I,\inert,\le1}\subset\bDelta^{\op,n}_{M,\inert,\le1}$. Moreover, by Corollary \ref{cor:inertinitial}, (1) is equivalent to requiring that $X_\inert$ be right Kan extended from the full subcategory $\bDelta^{\op,n}_{I,\inert}\subset\bDelta^{\op,n}_{M,\inert}$. Hence both (1)-(2) and ``$X$ is an $I$-truncated $M$-graded $\infn$-category'' are equivalent to the requirement that $X_\inert$ be right Kan extended from $\bDelta^{\op,n}_{I,\inert,\le1}$.
\end{proof}

\begin{nota}
\label{nota:TI}
We denote by $T_I\colon s^n_M\cS\to s^n_M\cS$ the \emph{$I$-truncation functor}, i.e. the composite $(\iota^{I,M})_*(\iota^{I,M})^*$. It restricts to an endofunctor of $\Catinfn^M$.
\end{nota}
\begin{cor}
\label{cor:Itruncateddescription}
Let $I\subseteq M$ be a coideal
and let $\cC\in\Catinfn^M$. Then the unit of the adjunction $\cC\to T_I(\cC)$ induces equivalences $\cC^\nsimeq\simeq T_I(\cC)^\nsimeq$ and $\Fun(\Theta_n,\cC)^\simeq_a\simeq\Fun(\Theta_n,T_I(\cC))^\simeq_a$ for all $a\in I$. Equivalently, a functor $\cC\to\cD$ in $\Catinfn^M$ is a $T_I$-equivalence if and only if it induces equivalences $\cC^\nsimeq\simeq\cD^\nsimeq$, and equivalences $\Fun(\Theta_n,\cC)^\simeq_a\simeq\Fun(\Theta_n,\cD)^\simeq_a$ for all $a\in I$.
\end{cor}
\begin{proof}
Since $(\iota^{I,M})_*\colon\Catinfn^I\hto\Catinfn^M$ is fully faithful, a $T_I$-equivalence is the same as a morphism $\cC\to\cD$ in $\Catinfn^M$ which is sent to an equivalence along restriction to $\bDelta^{\op,n}_I$. We can further restrict $\cC$ and $\cD$ to the wide subcategory $\bDelta^{\op,n}_{I,\inert}$, and using that both $\cC|_{\bDelta^{\op,n}_{I,\inert}}$ and $\cD|_{\bDelta^{\op,n}_{I,\inert}}$ are right Kan extended from $\bDelta^{\op,n}_{I,\inert,\le1}$, we obtain that a $T_I$-equivalence is the same as a morphism $\cC\to\cD$ in $\Catinfn^M$ evaluating to an equivalence (of spaces) on objects in $\bDelta^{\op,n}_{I,\le1}$: objects in this subcategory have the following forms:
\begin{itemize}
\item $([\bk],\nu)$ for $\bk=(k_1,\dots,k_n)$ with at least one $k_i=0$ (and necessarily $\nu=0\in M^0$); the equivalences $\cC([\bk],\nu)\overset{\simeq}{\to}\cD([\bk],\nu)$ for such objects combine to the equivalence $\cC^\nsimeq\simeq\cD^\nsimeq$;
\item $([\bk],\nu)$ with $\bk=1^n$ and $\nu=a\in I$; the equivalence $\cC([\bk],\nu)\overset{\simeq}{\to}\cD([\bk],\nu)$ is then the same requirement as the equivalence $\Fun(\Theta_n,\cC)^\simeq_a\simeq\Fun(\Theta_n,\cD)^\simeq_a$.
\end{itemize}
\end{proof}

\begin{lem}
Let $I\subseteq M$ be a coideal. 
Then the symmetric monoidal structure on $\Catinfn^M$ induces under localisation a symmetric monoidal structure on $\Catinfn^I$, such that $(\iota^{I,M})^*\colon\Catinfn^M\to\Catinfn^I$ is symmetric monoidal, and its right adjoint $(\iota^{I,M})_*$ is lax symmetric monoidal. In particular we have an induced adjunction
\[
(\iota^{I,M})^*\colon\CAlg(\Catinfn^M)\rightleftarrows\CAlg(\Catinfn^I)\colon(\iota^{I,M})_*,
\]
whose right adjoint is still fully faithful.
\end{lem}
\begin{proof}
By \cite[Proposition 2.2.1.9, Example 2.2.1.7]{HA}, it suffices to check that for every $T_I$-equivalence $\cC\to\cC'$ in $\Catinfn^M$ and any $\cD\in\Catinfn^M$, the $\infn$-functor $\cC\times\cD\to\cC'\times\cD$ is again a $T_I$-equivalence. By Lemma \ref{lem:Itruncateddescription}, it suffices to check equivalence on $n$-groupoid cores and on spaces of $n$-morphisms of gradings in $a\in I$; using that $I$ is a coideal, both statements are evident.
\end{proof}

\begin{cor}
\label{cor:nsimeqinvariance}
Let $I\subseteq M$ be a coideal 
and let $\cD=\cC[\ul{X},\ul{\del X}]^\sm\in\CAlg(\Catinfn)$ be an extension of symmetric monoidal $\infn$-categories. Assume that there is an $M$-grading $\delta\colon\Fun(\Theta_n,\cD)^\simeq\to M$ such that the composites $X_i\to\Fun(\Theta_n,\cD)^\simeq\to M$ have values in $M\setminus I$. Then the canonical symmetric monoidal $\infn$-functor $\cC\to\cD$ induces an equivalence $T_I(\cC)\simeq T_I(\cD)$ and an equivalence $\cC^\nsimeq\simeq\cD^\nsimeq$.
\end{cor}
\begin{proof}
We may consider $\cC\to\cD$ as a morphism in $\CAlg(\Catinfn^M)$.
By Lemma \ref{lem:Itruncateddescription}, once $T_I(\cC)\to T_I(\cD)$ is an equivalence, also $\cC^\nsimeq\to\cD^\nsimeq$ is an equivalence. For the first, we let $\cE\in\CAlg(\Catinfn^M)$ be $I$-truncated and using Notation \ref{nota:multiextension} we compute the following chain of equivalences of spaces, for $1\le i\le k$:
\[
\begin{split}
\Fun^\otimes_M(T_I(\cC_i),\cE)^\simeq&\simeq\Fun^\otimes_M(\cC_i,\cE)^\simeq\\
&\simeq\Fun^\otimes_M(\cC_{i-1},\cE)^\simeq\times_{\Fun_M(\del(X_i\times\Theta_n),\cE)^\simeq}\Fun_M(X_i\times\Theta_n,\cE)^\simeq\\
&\simeq\Fun^\otimes_M(\cC_{i-1},\cE)^\simeq\times_{\Fun(X_i\times\mO_{n-1},\cE)^\simeq}\Fun(X_i\times\mO_{n-1},\cE)^\simeq\\
&\simeq \Fun^\otimes_M(\cC_{i-1},\cE)^\simeq\simeq\Fun^\otimes_M(T_I(\cC_{i-1}),\cE)^\simeq.
\end{split}
\]
This implies the equivalence $T_I(\cC_{i-1})\simeq T_I(\cC_i)$; the result follows by combining all these equivalences and using $\cC=\cC_0$ and $\cD=\cC_k$.
\end{proof}
\begin{ex}
\label{ex:cUcVcores}
As promised in Remark \ref{rem:cUinfone}, since the extension of symmetric monoidal $\inftwo$-categories $\cV\simeq \cU[*;\bB C_2,*]^\sm$ is $\N$-admissible, Corollary \ref{cor:nsimeqinvariance} implies that the $\infone$-functor $\fJ\colon\cU\to\cV^\twosimeq$ is an equivalence. Similarly, the extension of symmetric monoidal $\infone$-categories $\cU\simeq\Fin[\bB C_2\sqcup*]^\sm$ is $\N$-admissible, so that $\Fin^\simeq\simeq\cU^\simeq\simeq\cV^\simeq$. We will therefore pretend that $\Fin$, $\cU$ and $\cV$ have the same space of objects, namely $\Fin^\simeq$; and that $\cU,\cV$ have the same space of 1-morphisms.
\end{ex}
We conclude the subsection with a result on $M$-graded localisations.
\begin{defn}
\label{defn:RelCatinfoneM}
We consider the object $(*\hto\bB M)\in\RelCatinfone$ and denote by $\RelCatinfone^M$ the overcategory $(\RelCatinfone)_{/(*\hto\bB M)}$: its objects are relative categories $(\cC,W)$ with an $M$-grading on $\cC$ that vanishes on $W$.

For a coideal $I\subseteq M$ we further denote by $\RelCatinfone^I$ the full and reflective subcategory of $\RelCatinfone^M$ spanned by objects $(\cC,W)$ with $\cC$ being $I$-truncated.
\end{defn}
We observe that for $(\cC,W)\in\RelCatinfone^M$, the $M$-grading $\fd\colon\cC\to\bB M$ induces an $M$-grading on the localisation $L(\cC,W)\to\bB M$.
\begin{lem}
\label{lem:Mgradedlocalisation}
Let $I\subseteq M$ be a coideal.
Let $(\cC,W)\to(\cC',W')$ be a morphism in $\RelCatinfone^M$ restricting to an equivalence $W\xrightarrow{\simeq} W'$ and to a $T_I$-equivalence $\cC\to\cC'$. Then the induced functor $L(\cC,W)\to L(\cC',W')$ is a $T_I$-equivalence.
\end{lem}
\begin{proof}
For simplicity, we consider $\Catinfone^I$ as a full subcategory of $\Catinfone^M$ via $(\iota^{I,M})_*$.
The following commutative square of right adjoint functors, on left, admits a corresponding commutative square of left adjoint functors, on right:
\[
\begin{tikzcd}[row sep=15pt,column sep=35pt]
\RelCatinfone^M&\Catinfone^M\ar[l,"{((-),(-)^\simeq)}"']\\
\RelCatinfone^I\ar[u,hook]&\Catinfone^I;\ar[u,hook]\ar[l,hook,"{((-),(-)^\simeq)}"']
\end{tikzcd}
\hspace{1cm}
\begin{tikzcd}[row sep=15pt,column sep=35pt]
\RelCatinfone^M\ar[d,"T_I"]\ar[r,"L"]&\Catinfone^M\ar[d,"T_I"]\\
\RelCatinfone^I\ar[r,"L"]&\Catinfone^I.
\end{tikzcd}
\]
The left-bottom composition sends the morphism $(\cC,W)\to(\cC',W')$ to an equivalence because already the left vertical functor $T_I$ does; hence the top-right composition sends the same morphism to an equivalence; in other words, the image $L(\cC,W)\to L(\cC',W')$ along the top horizontal functor is a $T_I$-equivalence.
\end{proof}

\subsection{Extensions iterated over a finite set}
\begin{defn}
\label{defn:Sextension}
For an extension $\cD=\cC[\ul X,\ul{\del X}]^\sm$ as in Notation \ref{nota:multiextension}, we regard $\cD$ as an object in $(\CAlg(\Catinfn))_{\cC/}$;  for a finite set $S$ we let $\cC[S|(\ul X,\ul{\del X})]^\sm$ denote the image along the target functor $(\CAlg(\Catinfn))_{\cC/}\to \CAlg(\Catinfn)$ of the $S$-fold coproduct of $\cD$ with itself inside $(\CAlg(\Catinfn))_{\cC/}$.
\end{defn}
The notation in Definition \ref{defn:Sextension} is due to the fact that $\cC[S|\ul X,\ul{\del X}]^\sm$ agrees with an extension of the form
$\cC[S\times X_1,S\times\del X_1;\dots;S\times X_k,S\times\del X_k]^\sm$; roughly speaking, the difference between $\cC[\ul X,\ul{\del X}]^\sm$ and $\cC[S|\ul X,\ul{\del X}]^\sm$ is that in the second case we extend $\cC$ by adjoining $S$ copies of the new generating $n$-morphisms.

\begin{defn}
\label{defn:fdS}
Let $\cD=\cC[\ul X,\ul{\del X}]^\sm$ be an extension endowed with an $M$-grading $\fd\colon\cD\to\bB^nM$ that restricts to the zero grading on $\cC$. We $\fd^S\colon \cC[S|\ul X,\ul{\del X}]^\sm\to\bB^n M^S$ as the $M^S$-grading characterised by the requirement that for all $s\in S$ we have a commutative diagram in $\CAlg(\Catinfn)_{\cC/}$
\[
\begin{tikzcd}[row sep=10pt]
\cC[\ul X,\ul{\del X}]\ar[d,"s\text{\textsuperscript{th} inclusion}"']\ar[r,"\fd"]&\bB^nM\ar[d,"s\text{\textsuperscript{th} inclusion}"]\\
\cC[S|\ul X,\ul{\del X}]\ar[r,"\fd^S"]&\bB^nM^S.
\end{tikzcd}
\]
\end{defn}
Roughly speaking, to distinguish the $S$ copies of new $n$-morphisms we put them in ``linearly independent'' gradings, and to achieve this we use $M^S$ as the monoid for gradings instead of $M$.
In fact we may regard $\bB^nM^S$ as the coproduct in $\CAlg(\Catinfn)$ of $S$ copies of $\bB^nM$; the $M$-grading $\fd^S$ is then the following composite, whose second arrow is the assembly map along the functor $\CAlg(\Catinfn)\to\CAlg(\Catinfn)_{\cC/}$ endowing a symmetric monoidal $\infn$-category with the constant functor from $\cC$ at the monoidal unit:
\[
\coprod_{S}^{\CAlg(\Catinfn)_{\cC/}}\cC[\ul X,\ul{\del X}]^\sm\xrightarrow{\coprod_S\fd} \coprod_{S}^{\CAlg(\Catinfn)_{\cC/}}\bB^nM\to \coprod_{S}^{\CAlg(\Catinfn)}\bB^nM.
\]
\begin{nota}
\label{nota:1S}
We denote by $\fS_S$ the group of automorphisms of a finite set $S$; and we let $1^S$ denote the $\fS_S$-invariant element $(1,\dots,1)\in\N^S$.
\end{nota}
We observe that the grading $\fd^S$ admits (together with its source and target) a natural $\fS_S$-equivariant structure. In particular, if $M=\N$, we have an action of $\fS_S$ on the space $\Fun(\Theta_n,\cC[S|\ul X,\ul{\del X}]^\sm)^\simeq_{1^S}$.
The main goal of this subsection is to prove the following proposition.
\begin{prop}
\label{prop:retraction}
Let $\cC[\ul X,\ul{\del X}]^\sm$ be an $\N$-admissible symmetric monoidal extension, let $S$ be a finite set, and consider the following commutative diagram in $(\CAlg(\Catinfn))^{\bB\fS_S}$, whose bottom row has trivial action
\[
\begin{tikzcd}
\cC[S|\ul X,\ul{\del X}]^\sm\ar[d,"\text{fold}"]\ar[r,"\fd^S"]&\bB^n\N^S\ar[d,"\text{fold (sum)}"]\\
\cC[\ul X,\ul{\del X}]^\sm\ar[r,"\fd"]&\bB^n\N.
\end{tikzcd}
\]

Then the following canonical map of spaces admits a (homotopy) section
\[
\Fun(\Theta_n,\cC[S|\ul X,\ul{\del X}]^\sm)^\simeq_{1^S}\ /\fS_S\to\Fun(\Theta_n,\cC[\ul X,\ul{\del X}]_\N)^\simeq_{|S|}.
\]
\end{prop}
The advantage of the description of $\cC[S|\ul X,\ul{\del X}]^\sm$ in Definition \ref{defn:Sextension} in terms of an $S$-fold coproduct inside $(\CAlg(\Catinfn))_{\cC/}$ comes from the fact that the assignment $S\mapsto \cC[S|\ul X,\ul{\del X}]^\sm$ gives rise to a symmetric monoidal functor
$\Fin\to((\CAlg(\Catinfn))_{\cC/},\sqcup)$. Precomposing with the symmetric monoidal inclusion $\Fininj\hto\Fin$, and postcomposing with the lax symmetric monoidal target functor $((\CAlg(\Catinfn)_{\cC/},\sqcup)\to (\CAlg(\Catinfn),\sqcup)$, we obtain a lax symmetric monoidal functor $\Fininj\to(\CAlg(\Catinfn),\sqcup)$.

The forgetful functor $(\CAlg(\Catinfn),\sqcup)\to(\Catinfn,\times)$ is symmetric monoidal by \cite[Proposition 3.2.4.7]{HA}. Combining this with the previous, we obtain a lax symmetric monoidal functor $\cC[-|\ul X,\ul{\del X}]^\sm\colon\Fininj\to\Catinfn$;  the latter functor still sends a finite set $S$ to (the underlying $\infn$-category of) $\cC[S|\ul X, \ul{\del X}]^\sm$. We will also consider $\cC[-|\ul X,\ul{\del X}]^\sm$ as a lax symmetric monoidal functor to $s^n\cS$.

A similar construction makes $S\mapsto \bB^n\N^S$ into a symmetric monoidal functor $\Fininj\to \Catinfn\subset s^n\cS$, by observing that $\bB^n\N^S$ is the $S$-fold coproduct of $\bB^n\N$ with itself in $\CAlg(\Catinfn)$. In fact $\bB^n\N^{(-)}$ takes values in $n$-fold simplicial sets.
\begin{defn}
\label{defn:01S}
For a finite set $S$ we consider the coideal $\set{0,1}^S\subset\N^S$, and we denote by $\bB^n01^S$ the subsimplicial set of $\bB^n\N^S$ comprising all multisimplices of total grading in $\set{0,1}^S$ (see Notation \ref{nota:totalgrading}).
\end{defn}
The assignment $S\mapsto \set{0,1}^S$ gives a symmetric monoidal functor $\Fininj\to(\Fin,\times)$, and the assignment $S\mapsto \bB^n01^S$ gives a symmetric monoidal subfunctor $\bB^n01^{(-)}\colon\Fininj\to s^n\cS$ of $\bB^n\N^{(-)}$.

\begin{defn}
For an $\N$-admissible symmetric monoidal extension $\cC[\ul X,\ul{\del X}]^\sm$,
we denote by $\cC[-|\ul X,\ul{\del X}]_{01}^\sm\colon\Fininj\to s^n\cS$ the lax symmetric monoidal functor obtained as pullback of the cospan
$\cC[-|\ul X,\ul{\del X}]^\sm\to\bB\N^{(-)}\ot\bB 01^{(-)}$ of lax symmetric monoidal functors $\Fininj\to s^n\cS$.
We further define
\[
\cC[\ul X,\ul{\del X}]^\sm_{\Fininj}:=\underset{\Fininj}{\colim}\ \cC[-|\ul X,\ul{\del X}]^\sm_{01}\in s^n\cS.
\]
\end{defn}
The fact that the cartesian symmetric monoidal structure on $s^n\cS$ preserves colimits in each variable implies that $\Fun(\Fininj,s^n\cS)$ has a Day convolution symmetric monoidal structure \cite[Subsection 2.2.6]{HA}, and the lax symmetric monoidal functor $\cC[-|\ul X,\ul{\del X}]^\sm_{01}$ is precisely a commutative algebra object in $\Fun(\Fininj,s^n\cS)$; since $\colim_{\Fininj}\colon \Fun(\Fininj,s^n\cS)\to s^n\cS$ is symmetric monoidal, we obtain a commutative algebra structure on $\cC[\ul X,\ul{\del X}]^\sm_{\Fininj}$.

Moreover we may consider $\cC[-|\ul X,\ul{\del X}]^\sm$, and hence also $\cC[-|\ul X,\ul{\del X}]^\sm_{01}$, as a lax symmetric monoidal functor $\Fininj\to s^n\cS_{/\cC[\ul X,\ul{\del X}]^\sm}$, by first mapping $\Fininj\to\Fin_{/*}\subset\Fun([1],\Fin)$, which is a lax symmetric monoidal functor, and the repeating the rest of the construction at the level of arrow categories.
\begin{nota}
\label{nota:fr}
We denote by $\fr\colon\cC[\ul X,\ul{\del X}]^\sm_{\Fininj}\to\cC[\ul X,\ul{\del X}]^\sm$ 
the resulting morphism in 
$\CAlg(s^n\cS)$. Composing with $\fd\colon\cC[\ul X,\ul{\del X}]^\sm\to\bB^n\N$ makes $\cC[\ul X,\ul{\del X}]^\sm_{\Fininj}$ into a commutative algebra in $\N$-graded $n$-fold simplicial spaces.
\end{nota}

\begin{lem}
\label{lem:ccuXsmFininjcatinfn}
$\cC[\ul X,\ul{\del X}]^\sm_{\Fininj}$ is an $\N$-graded $\infn$-category.
\end{lem}
\begin{proof}
Since colimits in $s^n\cS$ are computed pointwise, for all $[\bk]\in(\bDelta^\op)^n$ we have
\[
\cC[\ul X,\ul{\del X}]^\sm_{\Fininj}([\bk])\simeq\colim_{\Fininj}\cC[-|\ul X,\ul{\del X}]^\sm_{01}([\bk]).
\]
For an inclusion of finite sets $S\hto T$, the induced functor $\cC[S|\ul X,\ul{\del X}]^\sm\to\cC[T|\ul X,\ul{\del X}]^\sm$ is compatible with the $\N^S$-grading and $\N^T$-grading on the source and target, along the induced inclusion of monoids $\N^S\hto\N^T$; moreover Corollaries \ref{cor:Itruncateddescription} and \ref{cor:nsimeqinvariance} imply that the map $\cC[S|\ul X,\ul{\del X}]^\sm_{01}([\bk])\hto\cC[T|\ul X,\ul{\del X}]^\sm_{01}([\bk])$ is an inclusion of components, namely all components of the space $\cC[T|\ul X,\ul{\del X}]^\sm_{01}([\bk])$ whose total grading, which is an element in $\set{0,1}^T$, lies inside the image of $\set{0,1}^S\hto\set{0,1}^T$. This shows that the functor
\[
\cC[-|\ul X,\ul{\del X}]^\sm_{01}([\bk])\colon\Fininj\to\cS
\]
is the left Kan extension along the inclusion $\Fin^\simeq\hto\Fininj$ of the functor 
\[
\Fin^\simeq\to\cS,\quad\quad
S\mapsto\cC[S|\ul X,\ul{\del X}]^\sm([\bk])_{1^S},
\]
where the index ``$1^S$'' selects $\bk$-multisimplices of total grading $1^S\in\N^S$. It follows that $\cC[\ul X,\ul{\del X}]^\sm_{\Fininj}([\bk])$ is equivalent to
\[
\colim_{\Fin^\simeq}\ \cC[-|\ul X,\ul{\del X}]^\sm([\bk])_{1^{(-)}}\simeq\coprod_{r\ge0}\cC[\ul r|\ul X,\ul{\del X}]^\sm([\bk])_{1^{\ul r}}\ /\fS_{\ul r},
\]
where the last formula gives a decomposition of $\bk$-multisimplices according to their $\N$-valued total grading.
In particular, if $\bk=(k_1,\dots,k_n)$ has at least one entry $k_i=0$, then only the total grading $r=0$ is allowed, leading to the equivalence $\cC[\ul X,\ul{\del X}]^\sm_{\Fininj}([\bk])\simeq\cC([\bk])$; this implies in particular that $\cC[\ul X,\ul{\del X}]^\sm_{\Fininj}$ satisfies the constancy and completeness conditions.

For the $\N$-Segal condition we first 
write $\cC[\ul{r}|\ul X,\ul{\del X}]^\sm([\bk])_{1^{\ul r}}$ as a disjoint union of spaces $\cC[\ul r|\ul X,\ul{\del X}]^\sm([\bk])_\nu$ for varying $\nu\in(\N^{\ul r})^{\ul\bk}$ of total grading $1^{\ul r}$; this decomposition is given by the map $\cC[\ul r|\ul X,\ul{\del X}]^\sm([\bk])_{1^{\ul r}}\to\bB^n\N^{\ul r}([\bk])_{1^{\ul r}}$. Note that the datum of $\nu\in\bB^n\N^{\ul r}([\bk])_{1^{\ul r}}$ is the same as a partition of $\ul r$ into $\ul\bk$ subsets. The action of $\fS_{\ul r}$ permutes among each other the subspaces of $\cC[\ul{r}|\ul X,\ul{\del X}]^\sm([\bk])_{1^{\ul r}}$ corresponding to partitions $\nu$ that yield the same splitting of $r$ as a sum of $\ul\bk$ natural numbers, so that $\cC[\ul X,\ul{\del X}]^\sm_{\Fininj}([\bk])_r \simeq\cC[\ul{r}|\ul X,\ul{\del X}]^\sm([\bk])_{1^{\ul r}}\ /\fS_{\ul r}$ decomposes as a disjoint union of subspaces corresponding to elements $\check\nu\in\bB^n\N([\bk])_{r}$. For each $\nu\in\bB^n\N^{\ul r}([\ul\bk])_{1^{\ul r}}$, the stabiliser in $\fS_{\ul r}$ of the subspace $\cC[\ul r|\ul X,\ul{\del X}]^\sm([\bk])_\nu$ is the subgroup $\fS_\nu\subset\fS_{\ul r}$ stabilising the partition $\nu$; regarding $\nu$ as a map of finite sets $\nu\colon\ul r\to\ul\bk$, this is the product $\prod_{i\in\ul\bk}\fS_{\nu^{-1}(i)}$. The $\N^{\ul r}$-Segal condition for $\cC[\ul r|\ul X,\ul{\del X}]^\sm$ allows us to identify $\cC[\ul r|\ul X,\ul{\del X}]^\sm([\bk])_\nu$ as the limit of the spaces $\cC[\ul r|\ul X,\ul{\del X}]^\sm([\bk'])_{\nu'}$ indexed by the category $(\bDelta^{\op,n}_{\N^{\ul r},\inert,\le1})_{([\bk],\nu)/}\simeq(\bDelta^{\op,n}_{\inert,\le1})$, which is a weakly contractible poset. Quotienting by the action of $\fS_\nu$, and denoting by $\check\nu\in\N^{\ul\bk}$ the image of $\nu$ along the fold map $(\N^{\ul r}\to\N)^{\ul\bk}$, we obtain the analogous identification of $\cC[\ul X,\ul{\del X}]^\sm_{\Fininj}([\bk])_{\check\nu}$ as the limit over the category $(\bDelta^{\op,n}_{\N,\inert,\le1})_{([\bk],\check\nu)/}\simeq(\bDelta^{\op,n}_{\inert,\le1})$ of the spaces $\cC[\ul X,\ul{\del X}]^\sm_{\Fininj}([\bk'])_{\check\nu'}$. In this last step we use that, for diagrams in $\cS$, colimits over $\bB\fS_\nu$ commute with limits over a weakly contractible category.
\end{proof}

\begin{proof}[Proof of Proposition \ref{prop:retraction}]
By Lemma \ref{lem:ccuXsmFininjcatinfn} we have that $\cC[\ul X,\ul{\del X}]^\sm_{\Fininj}$ is a symmetric monoidal $\infn$-category; moreover the functor
\[
\cC\simeq\cC[\ul0|\ul X,\ul{\del X}]^\sm_{01}\to\cC[\ul X,\ul{\del X}]^\sm_{\Fininj},
\]
induced by the symmetric monoidal inclusion $\set{0}\hto\Fininj$ upon taking colimits of restrictions of $\cC[-|\ul X,\ul{\del X}]^\sm_{01}$, is symmetric monoidal. The same functor is also homotopic, via the map $\ul0\hto\ul1$ in $\Fininj$, to the following composite map of $n$-fold simplicial spaces, whose source and targets are $\infn$-categories, and whose middle term does \emph{not} have a symmetric monoidal structure:
\[
\cC\to\cC[\ul1|\ul X,\ul{\del X}]^\sm_{01}\to\cC[\ul X,\ul{\del X}]^\sm_{\Fininj}.
\]
The second arrow of the last composition can be restricted to $\ul X$ at the level of $n$-morphism spaces, compatibly over $\ul{\del X}$.
The universal property of $\cC[\ul X,\ul{\del X}]^\sm$ then gives rise to a symmetric monoidal functor
\[
\fs\colon\cC[\ul X,\ul{\del X}]^\sm\to \cC[\ul X,\ul{\del X}]^\sm_{\Fininj},
\]
and the composite of $\fs$ with the symmetric monoidal functor $\fr$ from Notation \ref{nota:fr} is homotopic to the identity of $\cC[\ul X,\ul{\del X}]^\sm$, as can be checked by restricting to $\cC$ and to the spaces $\ul X$.

Proposition \ref{prop:retraction} now follows from considering spaces of $n$-morphisms, using also the computation from the proof of Lemma \ref{lem:ccuXsmFininjcatinfn} specialised to $\bk=1^n$.
\end{proof}

Proving that the map from Proposition \ref{prop:retraction} is a retraction is sufficient for our purposes, but we conjecture that the given map is in fact an equivalence. Equivalently, we conjecture that the functors $\fr$ and $\fs$ from the proof of Proposition \ref{prop:retraction} are inverse equivalences between $\cC[\ul X,\del{\ul X}]^\sm$ and $\cC[\ul X,\del{\ul X}]^\sm_{\Fininj}$. These statements should follow from the following stronger conjecture, by considering spaces of $n$-morphisms of grading $1^S$ in $\cC[S|\ul X,\ul{\del X}]^\sm$ and grading $|S|$ in $\cC[\ul X,\ul{\del X}]^\sm$, and by observing that the left vertical functor in the diagram gives an $\fS_S$-fold covering between these two spaces since the right vertical functor obviously gives gives an $\fS_S$-fold covering between the corresponding spaces.
\begin{conj}
\label{conj:retractionisequivalence}
The following commutative square in $\CAlg(\Catinfn)$, coming from the symmetric monoidal $\infn$-functor $\cC\to*$ and the maps of spaces $\ul X\to*$, is a pullback square:
\[
\begin{tikzcd}
\cC[S|\ul X,\ul{\del X}]^\sm\ar[d,"\text{fold}"]\ar[r]\ar[dr,phantom,"\lrcorner"very near start]&*[S|*]^\sm\ar[d,"\text{fold}"]\ar[r,phantom,"\simeq"]&\bB^n(\Fin^\simeq)^S\\
\cC[\ul X,\ul{\del X}]^\sm\ar[r]& *[*]^\sm\ar[r,phantom,"\simeq"]&\bB^n\Fin^\simeq.
\end{tikzcd}
\]

\end{conj}

\section{Truncated categories via presheaves}
\label{sec:truncatedviapsh}
In this section, for $n\ge1$, we study more closely the $\infty$-category $\Catinfn^{01}\subset s^n_{01}\cS$, where we abbreviate by ``$01$'' the coideal $\set{0,1}\subset\N$. We start with a quick analysis of the case $n=1$. An object $\cC\in\Catinfone^{01}$ has an underlying $\infone$-category $\cC|_{\bDelta^\op_0}$, and the space $\Fun(\Theta_1,\cC)^\simeq_1$ of all morphisms of grading 1 is moreover endowed with composition laws on left and on right with morphisms in $\cC$ of grading 0, so that assigning to each pair of objects $(x,y)\in\Fun(\mO_0,\cC)^\simeq$ the fibre at $(x,y)$ of the map $\Fun(\Theta_1,\cC)^\simeq_1\to\Fun(\mO_0,\cC)^\simeq$ gives rise to a presheaf over $(\cC|_{\bDelta^\op_0})\times(\cC|_{\bDelta^\op_0})^\op$. We will indeed construct a pullback square in $\Catinfone^{\fU_2}$ of the form
\[
\begin{tikzcd}[row sep=10pt, column sep=60pt]
\Catinfone^{01}\ar[r,"(-)|_{\bDelta^\op_0}"]\ar[d]\ar[dr,phantom,"\lrcorner" very near start] &\Catinfone\ar[d,"(-)\times(-)^\op"]\\
\PSh\ar[r,"d_0"]&\Catinfone,
\end{tikzcd}
\]
where $d_0\colon\PSh\to\Catinfone$ is the bicartesian fibration from Subsection \ref{subsec:presheaves}. For $n\ge2$ we will have a similar formula, where $(-)\times(-)^\op$ will be suitably replaced by another symmetric monoidal functor $\dot\cT^n\colon\Catinfn\to\Catinfone$.

Throughout the section we denote by $\iota\colon0\to\N$ and $\tau\colon\N\to0$ the initial and terminal maps of monoids, see Example \ref{ex:iotatau}; and we denote by $T_0=\iota_*\iota^*\colon s^n_\N\cS\to s^n_\N\cS$ the truncation functor. We consider $s^n_{01}\cS$ as a full subcategory of $s^n_\N\cS$ via the fully faithful functor $(\iota^{01,\N})_*$ associated with the inclusion $\iota^{01,\N}\colon\bDelta^{\op,n}_{01}\hto\bDelta^{\op,n}_{\N}$.

\subsection{Expressing \texorpdfstring{$\Catinfn^{01}$}{Catinfn01} as a pullback}
\begin{defn}
We denote by $\bDelta^{\op,n}_1$ the full subcategory of $\bDelta^{\op,n}_{01}\subset\bDelta^{\op,n}_{\N}$ spanned by objects of total grading $1\in\N$. Concretely, $([\bk],\nu)\in\bDelta^{\op,n}_1$ if and only if $\nu\in\N^{\ul\bk}$ has exactly one coordinate equal to 1 and all other coordinates equal to 0. We similarly define $\bDelta^n_1:=(\bDelta^{\op,n}_1)^\op\subset\bDelta^n_{01}$. We let $\bDelta^{\op,n}_{1,\inert}\subset\bDelta^{\op,n}_1$ denote the wide subcategory spanned by inert morphisms.
\end{defn}
\begin{defn}
\label{defn:cKcTcL}
Let $X\in s^n_\N\cS$ and $Y\in s^n\cS\simeq s^n_0\cS$.
We let $\cK^n(X)\to \bDelta^n_{01}$ be a right fibration corresponding to $X|_{\bDelta^{\op,n}_{01}}\colon\bDelta^{\op,n}_{01}\to\cS$.
We let $\cT^n(Y):=\cK^n(\iota_*(Y))$, and by abuse of notation, we define $\cT^n(X)$ as $\cT^n(\iota^*(X))\simeq \cK^n(T_0(X))$.
The unit of the adjunction $X\to \iota_*\iota^*(X)\simeq T_0(X)$ gives rise to a right fibration $\cK^n(X)\to\cT^n(X)$, and we denote by $\kappa^n(X)\in\PSh(\cT^n(X))$ the corresponding presheaf.
We further set $\cL^n(Y):=\cK^n(\tau^*(Y))$, and $\lambda^n(Y):=\kappa^n(\tau^*(Y))\in\PSh(\cT^n(\tau^*(Y)))\simeq\PSh(\cT^n(Y))$ for the corresponding presheaf.

For $X\in s^n_\N\cS$, the total grading of an object in  $\cK^n(X)$ is defined as that of its projection in $\bDelta^n_{01}$, and it is therefore equal to $0$ or $1$. For $j=0,1$ we denote by $\cK^n_j(X)$ the full subcategory of $\cK^n(X)$ spanned by objects of total grading $j$. We introduce similar notation $\cT^n_j(X)$, $\cT^n_j(Y)$ and $\cL^n_j(Y)$. We denote by $\kappa^n_j(X)\in\PSh(\cT^n_j(X))$ the restriction of $\kappa^n(X)$ to $\cT^n_j(X)$; and we set $\lambda^n_j(Y):=\kappa^n_j(\tau^*(Y))\in\PSh(\cT^n_j(\tau^*(Y)))\simeq\PSh(\cT^n_j(Y))$.

We say that a morphism in $\cK^n(X)$ is \emph{inert} if its projection to $\bDelta^n_{01}$ is inert. We use the index ``$(-)_{\inert}$'' to denote wide subcategories spanned by inert morphisms.
\end{defn}

For $Y\in s^n\cS$ we observe that $\cT^n_0(Y)$ is the opposite of the category of simplices of $Y$, i.e. $\cT^n_0(Y)\to\bDelta^n$ is a right fibration corresponding to $Y$.
We also observe that $\cT^n_0(Y)$ is a \emph{left closed} full subcategory of $\cT^n(Y)$, i.e. a morphism in $\cT^n(Y)$ with target in $\cT^n_0(Y)$ also has source in $\cT^n_0(Y)$; this is essentially because the total grading is weakly increasing along morphisms in $\bDelta^n_\N$. In particular, right Kan extension from $\cT^n_1(Y)^\op$ exhibits $\PSh(\cT^n_1(Y))$ as the full subcategory of $\PSh(\cT^n(Y))$ spanned by presheaves whose restriction on $\cT^n_0(Y)$ is equivalent to the terminal presheaf, i.e. the constant presheaf with value $*\in\cS$.

\begin{lem}
\label{lem:kappaXproperties}
Let $X\in s^n_{01}\cS\subset s^n_\N\cS$; then the following holds:
\begin{enumerate}
\item $\kappa^n_0(X)\in\PSh(\cT^n_0(X))$ is the terminal presheaf.
\end{enumerate}
If moreover $\iota^*(X)\in s^n\cS$ satisfies the Segal condition, then the following hold:
\begin{enumerate}
\setcounter{enumi}{1}
\item $X$ satisfies the $01$-Segal condition from Definition \ref{defn:CatinfnI} if an only if the presheaf $\kappa^n_1(X)\in\PSh(\cT^n_1(X))$ is local with respect to $\cT^n_1(X)_{\inert}$.
\end{enumerate}
\end{lem}
\begin{proof}
Let $([\bk],\nu,x)$ be an object in $\cT^n(X)$, with $\nu\in\N^{\ul\bk}$ and $x\in T_0(X)([\bk],\nu)$. Then the value at $([\bk],\nu,x)$ of $\kappa^n(X)$ is the fibre at $x$ of the map of spaces $X([\bk],\nu)\to T_0(X)([\bk],\nu)$. By Corollary \ref{cor:Itruncateddescription} the previous map of spaces is an equivalence if $([\bk],\nu)\in\bDelta^n_0$, i.e. if $([\bk],\nu,x)\in\cT^n_0(X)$, thus proving (1).

For (2), the condition that $\kappa^n(X)$ sends inert morphisms in $\cT^n_1(X)^\op$ to equivalences can be rephrased as follows: for any inert morphism $f\colon([\bk'],\nu')\to([\bk],\nu)$ in $\bDelta^n_1$ the following square of spaces is cartesian (all maps between vertical fibres are equivalences):
\[
\begin{tikzcd}[row sep=10pt]
X([\bk],\nu)\ar[d]\ar[r]& X([\bk'],\nu')\ar[d]\\
T_0(X)([\bk],\nu)\ar[r]&T_0(X)([\bk'],\nu').
\end{tikzcd}
\]
We may equivalently only require the squares corresponding to inert morphisms in $\bDelta^n_1$ of the form $([1^n],1)\to([\bk],\nu)$ to be cartesian: indeed each object $([\bk],\nu)\in\bDelta^n_1$ admits an essentially unique morphism of this special form of which it is the target.

The hypothesis that $\iota^*(X)$ satisfies the Segal condition can be rephrased as the requirement that $T_0(X)\in s^n_{01}\cS\subset s^n_\N\cS$ satisfies the $01$-Segal condition, whence we may in particular expand
\[
T_0(X)([\bk],\nu)\simeq\lim_{(\bDelta^{\op,n}_{0,\inert,\le1})_{([\bk],\nu)/}}X, \qquad T_0(X)([1^n],1)=\lim_{(\bDelta^{\op,n}_{0,\inert,\le1})_{([1^n],1)/}}X,
\]
and the above square, for $([\bk'],\nu')=([1^n],1)$, is cartesian if and only if $X([\bk],\nu)\simeq\lim_{(\bDelta^{\op,n}_{01,\inert,\le1})_{([\bk],\nu)/}}X$, i.e. if the $01$-Segal condition for $X$ holds at $([\bk],\nu)$. 
\end{proof}
Motivated by Lemma \ref{lem:kappaXproperties}, we give the following definition.
\begin{defn}
\label{defn:dotKappa}
Let $X\in s^n_\N\cS$ and $Y\in s^n\cS$.
We define $\dot\cK^n(X)$ as the localisation $L(\cK^n_1(X),\cK^n_1(X)_\inert)\in\Catinfone$ of $\cK^n(X)$ at inert morphisms (see Subsection \ref{subsec:localisation}). Similarly, we define $\dot\cT^n(Y):=\dot\cK^n(\iota_*(Y))$, and by abuse of notation we define $\dot\cT^n(X)$ as $\dot\cT^n(\iota^*(X))$. Finally, we set $\dot\cL^n(Y)=\dot\cK(\tau^*(Y))$.

For $\cC\in\Catinfn^\N$ we denote by $\dot\kappa^n(\cC)\in\PSh(\dot\cT(\cC))$ the presheaf corresponding to $\kappa^n_1(\cC)$, which is local with respect to inert morphisms by Lemma \ref{lem:kappaXproperties}.
Equivalently, $\dot\kappa^n(\cC)$ corresponds to the functor $\dot\cK^n(\cC)\to\dot\cT^n(\cC)$ obtained as localisation of $\cK^n_1(\cC)\to\cT^n_1(\cC)$, which is again a right fibration by Lemma \ref{lem:horlocfibration}.

Similarly, for $\cC\in\Catinfn$, we denote by $\dot\lambda^n(\cC)\in\PSh(\dot\cT^n(\cC))$ the presheaf corresponding to $\lambda^n_1(\cC)$, and by abuse of notation, for $\cC\in\Catinfn^\N$, we denote by $\dot\lambda^n(\cC)$ the presheaf $\dot\lambda^n(\iota^*(\cC))\in\PSh(\dot\cT^n(\cC))$.
\end{defn}
\begin{prop}
\label{prop:Catinfn01pullback}
We have a pullback square in $\Catinfone^{\fU_2}$
\[
\begin{tikzcd}[row sep=10pt]
\Catinfn^{01}\ar[r,"\iota^*"]\ar[d,"{(\dot\cK^n\to\dot\cT^n)\simeq(\dot\cT^n,\dot\kappa^n)}"']\ar[dr,phantom,"\lrcorner"very near start] &\Catinfn\ar[d,"\dot\cT^n"]\\
\PSh\ar[r,"d_0"]&\Catinfone.
\end{tikzcd}
\]
\end{prop}
\begin{proof}
The functor $\cT^n_1\colon s^n_\N\cS\to\Catinfone$ can be promoted to a functor that we still denote $\cT^n_1\colon s^n_\N\cS\to\RelCatinfone$, sending $X\mapsto(\cT^n_1(X),\cT^n_1(X)_{\inert})$, and point (1) of Lemma \ref{lem:kappaXproperties} gives rise to a pullback square in $\Catinfone^{\fU_2}$
\[
\begin{tikzcd}[row sep=10pt]
s^n_{01}\cS\ar[r,"\iota^*"]\ar[d,"{(\cK^n_1\to\cT^n_1)\simeq (\cT^n_1,\kappa^n_1)}"']\ar[dr,phantom,"\lrcorner"very near start]&s^n\cS\ar[d,"\cT^n_1"]\\
\PSh_{\RelCatinfone}\ar[r]&\RelCatinfone,
\end{tikzcd}
\]
where the bottom arrow is the pullback of the bicartesian fibration $d_0\colon\PSh\to\Catinfone$ along the functor 
$d_0\colon\RelCatinfone\subset\Fun([1],\Catinfone)\to\Catinfone$.

Let moreover $\PSh^{\mathrm{loc}}_{\RelCatinfone}$ denote the full subcategory of $\PSh_{\RelCatinfone}$ spanned by objects $((\cC,W),A)$ such that $A\in\PSh(\cC)$ is $W$-local. Then on the one hand $\PSh^{\mathrm{loc}}_{\RelCatinfone}$ can be identified with the following pullback on left in $\Catinfone^{\fU_2}$:
\[
\begin{tikzcd}[row sep=10pt]
\PSh^{\mathrm{loc}}_{\RelCatinfone}\ar[r]\ar[d,"L"']\ar[dr,phantom,"\lrcorner"very near start]&\RelCatinfone\ar[d,"L"]\\
\PSh\ar[r,"d_0"]&\Catinfone;
\end{tikzcd}
\hspace{.5cm}
\begin{tikzcd}[row sep=10pt]
\Catinfn^{01}\ar[r,"\iota^*"]\ar[d,"{(\cK^n_1\to\cT^n_1)\simeq (\cT^n_1,\kappa^n_1)}"']\ar[dr,phantom,"\lrcorner"very near start]&\Catinfn\ar[d,"\cT^n_1"]\\
\PSh^{\mathrm{loc}}_{\RelCatinfone}\ar[r]&\RelCatinfone.
\end{tikzcd}
\]
On the other hand point (2) of Lemma \ref{lem:kappaXproperties}, together with the observation that the completeness and constancy condition on $X\in s^n_{01}\cS$ only depend on $\iota^*(X)$, allows us to identify $\Catinfn^{01}$ as the preimage of $\PSh^{\mathrm{loc}}_{\RelCatinfone}$ along $(\cT^n_1,\kappa^n_1)$, obtaining the pullback on right. We conclude by gluing the last two pullback squares.
\end{proof}
We conclude this subsection with two basic lemmas. The first regards isomorphism classes of objects in $\dot\cT^n(\cC)$ for $\cC\in\Catinfn$, and will allow us to limit our study of mapping spaces in $\dot\cT^n(\cC)$ in the next two subsections. Note that for $\cC\in\Catinfn$ we have $\Fun(\mO_{n-1},\cC)^\simeq\simeq \iota_*(\cC)([1^n],1)$ and $\Fun(\Theta_n,\cC)^\simeq\simeq\tau^*(\cC)([1^n],1)$.
\begin{lem}
\label{lem:essentiallysurjective}
For $X\in s^n_\N\cS$, the natural map of spaces $X([1^n],1)\to\dot\cK^n(X)^\simeq$, sending $x\mapsto ([1^n],1,x)$, is surjective on $\pi_0$.
\end{lem}
\begin{proof}
Let $([\bk],\nu,x)\in\cK^n_1(X)$; then the inert morphism $([1^n],1)\to([\bk],\nu)$ in $\bDelta_1$ admits a cartesian lift $([1^n],1,x')\to([\bk],\nu,x)$ in $\cT^n_1(\cC)$; this morphism is inert and hence it is inverted in $\dot\cT^n(\cC)$.
\end{proof}
\begin{conj}
For $\cC\in\Catinfn$ the maps $\Fun(\mO_{n-1},\cC)^\simeq\to\dot\cT^n(\cC)^\simeq$ and $\Fun(\Theta_n,\cC)^\simeq\to\dot\cL^n(\cC)^\simeq$ from Lemma \ref{lem:essentiallysurjective} are in fact equivalences of spaces.
\end{conj}
The second lemma is a computation of $\dot\cT^n(Y)$ and $\dot\cL^n(Y)$ for $Y\in\cS\subset\Catinfn$.
\begin{lem}
\label{lem:dotcTcLspace}
Let $Y\in\cS$; then $\dot\cT^n(Y),\dot\cL^n(Y)\in\cS$, and the map of spaces $\dot\cL^n(Y)\to\dot\cT^n(Y)$ can be identified with the restriction map $\Fun(\Theta_n,Y)^\simeq\to\Fun(\mO_{n-1},Y)^\simeq$, or equivalently, with the diagonal map $Y\to Y^{S^{n-1}}$.
\end{lem}
\begin{proof}
The functor $(\tau^*(Y)\to\iota_*\iota^*\tau^*(Y)\simeq\iota_*(Y))\colon\bDelta^{\op,n}_\N\to\Fun([1],\cS)$ restricts to the constant functor at $(Y\to Y^{S^{n-1}})$ on $\bDelta^{\op,n}_1$. It follows that the functor $\dot\cL^n(Y)\to\dot\cT^n(Y)$ is equivalent to the product functor $Y\times\dot\cL^n(*)\to Y^{S^{n-1}}\times\dot\cT^n(*)$. Since $\cL^n_1(*)\simeq\cT^n_1(*)\simeq\bDelta^n_1$ has a terminal object $([1^n],1)$, which admits an inert map to any other object in $\bDelta^n_1$, we have that every morphism of $\bDelta^n_1$ is inverted in $L(\bDelta^n_1,\bDelta^n_{1,\inert})$, and in particular $\dot\cL^n(*)\simeq\dot\cT^n(*)\simeq|\bDelta^n_1|\simeq *$.
\end{proof}

\subsection{Mapping spaces in \texorpdfstring{$\dot{\mathcal{T}}^1(\cC)$}{dotcT1cC}}
In this subsection we fix $\cC\in\Catinfone$ and aim at computing mapping spaces in $\dot\cT^1(\cC)$.
\begin{defn}
An inert morphism $[k]\to[k']$ in $\bDelta$ is \emph{left inert}, respectively \emph{right inert}, if its image is $\set{0,\dots,k}$, respectively $\set{k'-k+1,\dots,k'}$. For $X\in s\cS$ or $X\in s_\N\cS$, a morphism in $\cT^1_0(X)$ is left/right inert if its projection to $\bDelta$ is left/right inert. We use the indices ``$(-)_{\linert}$'' and ``$(-)_{\rinert}$'' to denote wide subcategories spanned by left/right inert morphisms.
\end{defn}
Recall Subsection \ref{subsec:twistedarrow}.
For $Y\in s\cS$ we have an isomorphism $\cT_0(Y)\simeq\cT_0(Y^\op)$ covering $i_\bDelta$; in particular we have an isomorphism $\cT_0(\cC)_{\linert}\simeq\cT_0(\cC^\op)_{\rinert}$.

Let $([k],\nu)\in\bDelta_1$, and expand $\nu=(0,\dots,1,\dots,0)$, where the ``1'' is in position $1\le i\le k$. Consider the overcategory $(\bDelta_0)_{/([k],\nu)}$, i.e. the full subcategory of $(\bDelta_\N)_{/([k],\nu)}$ spanned by arrows with source in $\bDelta_0$; then $(\bDelta_0)_{/([k],\nu)}$ naturally splits as the disjoint union of two subcategories, accounting for maps $([k'],0)\to([k],\nu)$ whose underlying map $[k']\to[k]$ has image contained in $\set{0,\dots,i-1}$ or in $\set{i,\dots,k}$. Each of these subcategories has a terminal object, given by the left inert morphism $[i-1]\to[k]$ and the right inert morphism $[k-i+1]\to[k]$, respectively.
\begin{defn}
\label{defn:fafo}
We denote by $\fa,\fo\colon\bDelta_1\to\bDelta\simeq\bDelta_0$ the two functors selecting the sources of the terminal objects of the two subcategories in the splitting of $(\bDelta_0)_{/(-)}$. Using the notation above, we have $\fa([k],\nu)=[i-1]$ and $\fo([k],\nu)=[k-i+1]$.

We denote by $\epsilon_\fa\colon\fa\Rightarrow\Id_{\bDelta_1}$ and $\epsilon_\fo\colon\fo\Rightarrow\Id_{\bDelta_1}$ the canonical natural transformations, attaining left, respectively right inert morphisms as values: here we consider all functors $\fa,\fo,\Id_{\bDelta_1}$ as functors $\bDelta_1\to\bDelta_{01}$. By abuse of notation, we also denote by $\epsilon_\fa$ the natural transformation $\tau\circhor\epsilon_\fa\colon\tau\fa\simeq\fa\Rightarrow\tau$ of functors $\bDelta_1\to\bDelta_0$, and similarly we denote $\tau\circhor\epsilon_\fo$ just by $\epsilon_\fo$.
\end{defn}
We observe that $\fa\times\fo\colon\bDelta_1\to\bDelta\times\bDelta$ is an equivalence of categories. Moreover, for $Y\in s\cS$, the restriction $\iota_*(Y)|_{\bDelta_1^\op}\in\PSh(\bDelta_1)$ agrees with the pullback of $Y\boxtimes Y\in\PSh(\bDelta\times\bDelta)$ along the functor $\fa\times\fo$. Equivalently, the right fibration $\cT^1_1(Y)\to\bDelta_1$ is the pullback of the right fibration $\cT^1_0(Y)\times\cT^1_0(Y)\to\bDelta\times\bDelta$, so that in particular we have an equivalence of $\infone$-categories $\cT^1_1(Y)\simeq\cT^1_0(Y)\times\cT^1_0(Y)$. Along this equivalence, the wide subcategory $\cT^1_1(Y)_{\inert}$ corresponds to the wide subcategory $\cT^1_0(Y)_{\rinert}\times\cT^1_0(Y)_{\linert}$. It follows that $\dot\cT^1(Y)$ is equivalent to the product of localisations $L(\cT^1_0(Y),\cT^1_0(Y)_{\rinert})\times L(\cT^1_0(Y),\cT^1_0(Y)_{\linert})$; and 
we may rewrite the second factor as $L(\cT^1_0(Y^\op),\cT^1_0(Y^\op)_{\rinert})$.

Moreover, the functor $c_\bDelta\colon\bDelta\times\bDelta\to\bDelta$ from Subsection \ref{subsec:twistedarrow} can be identified with the composite $\bDelta\times\bDelta\xrightarrow{(\fa\times\fo)^{-1}}\bDelta_1\xrightarrow{\tau}\bDelta$. Hence for $Y\in s\cS$ the right fibration $\cL^1_1(Y)\to\bDelta_1$, which is the pullback of the right fibration $\cT^1_0(Y)\to\bDelta$ along $\tau\colon\bDelta_1\to\bDelta$, corresponds to the right fibration $(c_\bDelta)^*(\cT^1_0(Y))\to\bDelta\times\bDelta$. Along the equivalence $\cL^1_1(Y)\simeq(c_\bDelta)^*(\cT^1_0(Y))$, the wide subcategory $\cL^1_1(Y)_{\inert}$ corresponds to the wide subcategory of $(c_\bDelta)^*(\cT^1_0(Y))$ spanned by morphisms whose projection to $\bDelta\times\bDelta$ is a pair of a right inert and a left inert morphism.

\begin{prop}
\label{prop:rinertloc}
Let $\cC\in\Catinfone$; then $L(\cT^1_0(\cC),\cT^1_0(\cC)_{\rinert})\simeq\cC$.
\end{prop}
\begin{proof}
We start by constructing a functor $\mathbf{F}\colon\cT^1_0(\cC)\to\cC$; equivalently, we want a functor $\cT^1_0(\cC)\to\cC\times\bDelta$ over $\bDelta$. After straightening the two cartesian fibrations $\cT^1_0(\cC)\to\bDelta$ and $\cC\times\bDelta\to\bDelta$, we have two functors $\cC,\mathrm{const}_\cC\colon\bDelta^\op\to\Catinfone$, the first taking values in $\cS\subset\Catinfone$, and we have to provide a left lax natural transformation between the two, by which we mean a functor $\bDelta^\op\to\Fun([1],\Catinfone)^\llax$ whose images along $d_1$ and $d_0$ are $\cC$ and $\mathrm{const}_\cC$, respectively (see Definition \ref{defn:laxarrowcategories}).

We start by considering the functor $\mathbf{G}\colon \bDelta\to\Fun([1],\Catinfone)^\rlax$ sending the object $[k]$ to the functor $*\to[k]$, $*\mapsto k$ exhibiting $*$ as the maximum of $[k]$. The image along $\mathbf{G}$ of a morphism $f\colon[k]\to[k']$ in $\bDelta$ is forced to be the following square:
\[
\begin{tikzcd}[column sep=70pt, row sep=10pt]
*\ar[r,equal]\ar[d,"k"',""{name=L,inner sep=2pt,right, very near end}]&*\ar[d,"k'",""{name=R,inner sep=2pt,left,very near start}]\\
{[k]}\ar[r,"f"]&{[k']}.
\ar[Rightarrow,from=L,to=R, "f(k)\to k'"]
\end{tikzcd}
\]
We may compose $\mathbf{G}^\op$ with the functor $\Fun(-,\cC)\colon ((\Fun([1],\Catinfone)^\rlax)^\op\to\Fun([1],\Catinfone)^\llax$, and then further compose with the endofunctor of the $\infone$-category $\Fun([1],\Catinfone)^\llax$ sending $(\cD\to\cD')$ to $(\cD^\simeq\to\cD')$. This gives the left lax natural transformation $\cC\Rightarrow\mathrm{const}_\cC$, giving in turn the functor $\mathbf{F}\colon\cT^1_0(\cC)\to\cC$.

By construction $\mathbf{F}$ sends $([k],x)\in\cT^1_0(\cC)$ to $x'\in\cC$, where $([0],x')\cT^1_0(\cC)$ is a cartesian lift of the right inert morphism $[0]\xrightarrow{k}[k]$ in $\bDelta$. In particular $\mathbf{F}$ is essentially surjective, since any $x'\in\cC^\simeq$ is equivalent to $\mathbf{F}([0],x')$; and $\mathbf{F}$ sends right inert morphisms in $\cT_0(\cC)$ to equivalences, inducing a functor $\hat{\mathbf{F}}\colon L(\cT^1_0(\cC),\cT^1_0(\cC)_{\rinert})\to\cC$.

It is left to prove that $\hat{\mathbf{F}}$ is fully faithful. We fix an object $\bx=([0],x)\in\cT^1_0(\cC)$ and we denote by $W(\bx)$ the full subcategory of $\cT^1_0(\cC)_{\bx/}$ spanned by right inert morphisms, and by $\pi(\bx)\colon W(\bx)\to\cT^1_0(\cC)_{\bx/}$ the natural inclusion. 
Let $\by=([k],y)$ be some object in $\cT^1_0(\cC)$, and let $\by'=([0],y')$ be the source of a cartesian lift of the right inert morphism $[0]\to[k]$ in $\bDelta$. Recall from Subsection \ref{subsec:localisation} that $d_0\pi(\bx)_!(*)(\by)$ can be computed as $|\cT^1_0(\cC)_{\by/,\bx/\rinert}|$, where $\cT^1_0(\cC)_{\by/,\bx/\rinert}$ denotes the category of cospans $(\by\to \bz\ot \bx)$ in $\cT^1_0(\cC)$ with $\bx\to \bz$ right inert.

Let us say that an object $(\by\to \bz\ot \bx)$ in $\cT^1_0(\cC)_{\by/,\bx/\rinert}$ is \emph{special} if $\bz$ has the form $([k+1],z)$ and if the cospan $[k]\to[k+1]\ot[0]$ obtained by projecting to $\bDelta$ is the unique cospan exhibiting $[k+1]$ as the disjoint union of $[k]\simeq\set{0,\dots,k}\subset[k+1]$ and $[0]\simeq\set{k+1}\subset[k+1]$.
Let $\cT^1_0(\cC)_{\by/,\bx/\rinert}^{\special}$ denote the full subcategory of $\cT^1_0(\cC)_{\by/,\bx/\rinert}$ spanned by special objects. 
Then any morphism
in $\cT^1_0(\cC)_{\by/,\bx/\rinert}^{\special}$ is an equivalence, since its projection to $\bDelta$ is an equivalence: indeed the only map $[k+1]\to[k+1]$ in $\bDelta$ which is compatible with the two given inclusions of $[k]\to[k+1]\ot[0]$ is $\Id_{[k+1]}$.
In other words, $\cT^1_0(\cC)_{\by/,\bx/\rinert}^{\special}$ is an $\infty$-groupoid, i.e. $\cT_0(\cC)_{\by/,\bx/\rinert}^{\special}\simeq(\cT_0(\cC)_{\by/,\bx/\rinert}^{\special})^\simeq$.
The $\infty$-groupoid core $(\cT_0(\cC)_{\by/,\bx/\rinert}^{\special})^\simeq$ is a priori equivalent to the fibre at $(y,x)$ of the map $\cC([k+1])\to\cC([k])\times\cC([0])$ induced by the two maps $[k]\to[k+1]\ot[0]$, and the Segal condition for $\cC$ both implies that $(\cT_0(\cC)_{\by/,\bx/\rinert}^{\special})^\simeq\simeq\cC(y',x)$, and that any right inert morphism $\bt\to\by$ in $\cT_0(\cC)$ induces an equivalence $(\cT_0(\cC)_{\by/,\bx/\rinert}^{\special})^\simeq\overset{\simeq }{\to}(\cT_0(\cC)_{\bt/,\bx/\rinert}^{\special})^\simeq$. 

We next argue that $\cT^1_0(\cC)_{\by/,\bx/\rinert}^{\special}$ is initial in $\cT^1_0(\cC)_{\by/,\bx/\rinert}$. Given an object $(\by\to\bz'\ot\bx)$ in $\cT^1_0(\cC)_{\by/,\bx/\rinert}$ lying over some cospan $[k]\to[k']\ot[0]$, we argue that
the overcategory $(\cT^1_0(\cC)_{\by/,\bx/\rinert}^{\special})_{/(\by\to\bz'\ot\bx)}$ is weakly contractible because it has a terminal object. To see this, we observe that the two morphisms $[k]\to[k']\ot[0]$ combine to a unique
morphism $[k+1]\simeq[k]\sqcup[0]\to[k']$ in $\bDelta$, and taking a cartesian lift of $[k+1]\to[k']$ with target $\bz'$ gives rise to a morphism of cospans $(\by\to\bz\ot\bx)\to(\by\to\bz'\ot\bx)$ restricting to the identity on $\by$ and $\bx$: this is a terminal object in $(\cT^1_0(\cC)_{\by/,\bx/\rinert}^{\special})_{/(\by\to\bz'\ot\bx)}$.

The previous considerations allow us to apply Theorem \ref{thm:Cisinski}, yielding an equivalence between $\cC(y',x)$ and the mapping space from $\by$ to $\bx$ in $L(\cT^1_0(\cC),\cT^1_0(\cC)_{\rinert})$; and the same equivalence of mapping spaces is also induced by the functor $\hat{\mathbf{F}}$.
\end{proof}

\begin{cor}
\label{cor:dotcTCCop}
For $\cC\in\Catinfone$ we have $\dot\cT^1(\cC)\simeq\cC\times\cC^\op$, and the presheaf $\dot\lambda^1(\cC)\in\PSh(\dot\cT^1(\cC))$ from Definition \ref{defn:dotKappa} corresponds to $\cC(-,-)\in\PSh(\cC\times\cC^\op)$. In other words, the right fibration $\dot\cL^1(\cC)\to\dot\cT^1(\cC)$ is equivalent to the right fibration $(d_1,d_0)\colon\Tw(\cC)\to\cC\times\cC^\op$ from Subsection \ref{subsec:twistedarrow}
\end{cor}
\begin{proof}
The construction of the functor $\mathbf{F}$ from the proof of Proposition \ref{prop:rinertloc} can be adapted to obtain the following
commutative square in $\RelCatinfone$, whose horizontal arrows are sent to right fibrations along the functor $d_0\colon\RelCatinfone\to\Catinfone$:
\[
\begin{tikzcd}[row sep=10pt]
(\cL^1_1(\cC),\cL^1_1(\cC)_{\inert})\ar[r]\ar[d,"\mathbf{F}'"]\ar[dr,phantom,"\lrcorner"very near start] &(\cT^1_1(\cC),\cT^1_1(\cC)_{\inert})\ar[d,"\mathbf{F}''"]\\
(\Tw(\cC),\Tw(\cC)^\simeq)\ar[r,"{(d_1,d_0)}"]&(\cC\times\cC^\op,\cC^\simeq\times\cC^\simeq).
\end{tikzcd}
\]
We argue that the previous square is a pullback square in $\RelCatinfone$: to see that the image along $d_0\colon\RelCatinfone\to\Catinfone$ is a pullback square, we fix $(x,x')\in\cC\times\cC^\op$ and consider the object $\bx=([1],1,(x,x'))\in\cT^1_1(\cC)$; we readily compute $\lambda^1_1(\bx)$, i.e. the fibre at $\bx$ of $\cL^1_1(\cC)]\to\cT^1_1(\cC)$, as the fibre at $(x,x')$ of the map $\tau^*(\cC)([1],1)\to\iota_*(\cC)([1],1)$: this is the map $(d_1,d_0)\colon\cC([1])\to\cC([0])^2$, so the required fibre is precisely $\cC(x,x')$, i.e. it agrees with the fibre at $(x,x')$ of $\Tw(\cC)\to\cC\times\cC^\op$. For generic $\bx'\in\cT^1_1(\cC)$ we may find an inert morphism $\bx\to\bx'$ with $\bx$ as above; by Lemma \ref{lem:kappaXproperties} we have $\lambda^1_1(\bx')\xrightarrow{\simeq}\lambda^1_1(\bx)$, and the functor $\mathbf{F}''$ sends $\bx\to\bx'$ to an equivalence: this allows us to reduce the generic case of comparing the fibres of the horizontal maps to the specific one discussed above.  We finally observe that the image of the above square along $d_1$ is also a pullback square, obtained by selecting compatible wide subcategories from the image along $d_0$.

The horizontal arrows of the above pullback square in $\RelCatinfone$ satisfy the assumptions of Lemma \ref{lem:horlocfibration}, so that the localised square is again a pullback square in $\Catinfone$. By Proposition \ref{prop:rinertloc}, the localisation of the right vertical functor $\mathbf{F}''$ is an equivalence; it follows that the localisation of $\mathbf{F}'$ is an equivalence as well.
\end{proof}

\subsection{The partial localisation \texorpdfstring{$\ddot{\mathcal{T}}^n(\cC)$}{ddotcTn(cC)}}
In this subsection and the next we give a recursive formula for the mapping spaces of $\dot\cT^n(\cC)$ for $\cC\in\Catinfn$ and $n\ge2$. A key result that we will prove, inductively on $n$, is the following, which for $n=1$ is a direct consequence of Corollary \ref{cor:dotcTCCop}.
\begin{prop}
\label{prop:dotcTcLsm}
The functor
$(\dot\cL^n\to\dot\cT^n)\simeq(\dot\cT^n,\dot\lambda^n)\colon\Catinfn\to\PSh$
preserves finite products, i.e. it is a symmetric monoidal functor. More generally, $(\dot\cT^n,\dot\lambda^n)$ preserves pullbacks of cospan diagrams $\cC\to Y\ot \cC'$ in $\Catinfn$ with $Y\in\cS$.
\end{prop}

We fix $n\ge2$ throughout this and the next subsection, and keep considering $s^n_{01}\cS$ as a full subcategory of $s^n_\N\cS$ via the inclusion $(\iota^{01,\N})_*$.

For every $k\in\N$ and $\bk\in\N^{n-1}$ we have a map $-\otimes-\colon\N^{\ul k}\times\N^{\ul\bk}\to\N^{\ul k\times\ul \bk}$ induced by multiplication of natural numbers: explicitly, we send a pair $(\nu,\nu')$ of functions $\nu\colon\ul k\to\N$ and $\nu'\colon\ul\bk\to\N$ to the function $(i,j)\in\ul k\times\ul\bk\mapsto \nu(i)\cdot\nu'(j)\in\N$.
The maps $-\otimes-$ assemble into a map of $n$-fold simplicial spaces
$-\otimes-\colon\bB\N\boxtimes\bB^{n-1}\N\to\bB^n\N$, and taking categories of $n$-simplices we obtain a functor $-\otimes-\colon\bDelta_\N\times\bDelta^{n-1}_\N\to\bDelta^n_\N$. The latter restricts to a functor $-\otimes-\colon\bDelta_{01}\times\bDelta^{n-1}_{01}\to\bDelta^n_{01}$, and further to an equivalence $-\otimes-\colon\bDelta_1\times\bDelta^{n-1}_1\overset{\simeq}{\to}\bDelta^n_1$.

\begin{nota}
For $[k]\in\bDelta$ and $Y\in s^n\cS$ we denote by $Y_{[k]}=Y([k]\times-)\in s^{n-1}\cS$ the pullback of the presheaf $Y$ along the functor $[k]\times-\colon\bDelta^{n-1}\to\bDelta^n$.
Similarly, for $([k],\nu)\in\bDelta_{01}$ and $X\in s^n_{01}\cS$ we denote by $X_{([k],\nu)}:=X(([k],\nu)\otimes-)\in s^{n-1}_{01}\cS$ the pullback of $X$ along $([k],\nu)\otimes-\colon\bDelta^{n-1}_{01}\to\bDelta^n_{01}$.
\end{nota}
We obtain a functor $\bDelta_{01}^\op\times s^n_{01}\cS\to s^{n-1}_{01}\cS$ sending $(([k],\nu),X)\mapsto X_{([k],\nu)}$, which restricts to $\bDelta_{01}^\op\times\Catinfn^{01}\to\Catinfnm^{01}$. For $X\in s^n_{01}\cS$ the following hold:
\begin{itemize}
\item the restricted functor $X_{(-)}|_{\bDelta_0^\op}\colon\bDelta_0^\op\to s^{n-1}_{01}\cS$ agrees with $\tau^*(\iota^*(X)_{(-)})$, i.e. for $[k]\in\bDelta$ we have $X_{([k],(0,\dots,0))}\simeq \tau^*(\iota^*(X)_{[k]})$;
\item the cartesian fibration $\cK^n_1(X)\to\bDelta_1$ unstraightening the composite functor 
\[
\begin{tikzcd}
\bDelta_1^\op\subset\bDelta_{01}^\op\ar[r,"X_{(-)}"] &s^{n-1}_{01}\cS\ar[r, "\cK^{n-1}_1"]& \Catinfone
\end{tikzcd}
\]
is the composite of the right fibration $\cK^n_1(X)\to\bDelta^n_1$ and the projection on the first factor $\bDelta^n_1\simeq(\bDelta_1)^n\xrightarrow{\pi_1}\bDelta_1$.
\end{itemize}
\begin{lem}
\label{lem:Kpullback}
Recall Definitions \ref{defn:cKcTcL} and \ref{defn:fafo}. Let $X\in s^n_{01}\cS$ be $0$-truncated, i.e. $X\simeq T_0(X)$. Then the following square in $\Fun(\bDelta^\op_1,\Catinfone)$,
whose arrows are induced by $\epsilon_\fa,\epsilon_\fo$ and by the unit of the adjunction $\iota^*\dashv\iota_*$, is cartesian:
\[
\begin{tikzcd}[column sep=10pt,row sep=10pt]
\cK^{n-1}_1(X_{(-)})\ar[r]\ar[d]&\cK^{n-1}_1(X_{\fa(-)})\times \cK^{n-1}_1(X_{\fo(-)})\ar[d]\\
\cK^{n-1}_1(\iota_*\iota^*(X_{(-)}))\ar[r]&\cK^{n-1}_1(\iota_*\iota^*(X_{\fa(-)}))\times \cK^{n-1}_1(\iota_*\iota^*(X_{\fo(-)})).
\end{tikzcd}
\]
\end{lem}
\begin{proof}
It suffices to fix $\bx=([k],\nu)\in\bDelta^\op_1$ and prove that the square in $\Catinfone$ obtained by evaluation at $\bx$ is cartesian.
This square admits a canonical map to the following other square in $\Catinfone$, which is obviously cartesian; the mentioned canonical map is pointwise a right fibration:
\[
\begin{tikzcd}[row sep=5pt,column sep=30pt]
\bDelta^{n-1}_1\ar[d,equal]\ar[r,"\text{diagonal}"] & \bDelta^{n-1}_1\times \bDelta^{n-1}_1\ar[d,equal]\\
\bDelta^{n-1}_1\ar[r,"\text{diagonal}"] & \bDelta^{n-1}_1\times \bDelta^{n-1}_1.
\end{tikzcd}
\]
It is therefore sufficient to fix $\by=([\bk],\nu')\in\bDelta^{\op,n-1}_1$ and prove that the square of spaces formed by the fibres at $\by$, respectively at $(\by,\by)$, of the above four right fibrations is cartesian. The resulting square of spaces is the following:
\[
\begin{tikzcd}[row sep=13pt]
X(\bx\otimes\by)\ar[r]\ar[d] &X(\fa(\bx)\otimes\by)\times X(\fo(\bx)\otimes\by)\ar[d]\\
\lim_{(\by\to\by')\in(\bDelta^{\op,n-1}_0)_{\by/}}X(\bx\otimes\by')\ar[r] &
\lim_{(\by\to\by')\in(\bDelta^{\op,n-1}_0)_{\by/}}\begin{array}{c}X(\fa(\bx)\otimes\by')\times\\ X(\fo(\bx)\otimes\by').\end{array}
\end{tikzcd}
\]
The pullback of the right-bottom cospan in the previous square admits the following interpretation. Let $(\bDelta^{\op,n}_0)_{\bx\otimes\by/}^{\special}$ denote the full subcategory of $(\bDelta^{\op,n}_0)_{\bx\otimes\by/}$ spanned by those arrows $\bx\otimes\by\to\bz$ in $\bDelta^{\op,n}_{01}$ such that, along the composite $\bDelta^{\op,n}_{01}\xrightarrow{\tau}\bDelta^{\op,n}_0\simeq(\bDelta^\op)^n\xrightarrow{\pi_1}\bDelta^\op$, the image of $\bx\otimes\by\to\bz$ is either $\Id_{[k]}$ or the left inert morphism $[k]\to\fa(\bx)$ or the right inert morphism $[k]\to\fo(\bx)$. Then the pullback of the right-bottom cospan in the above square agrees with the limit of $X$ over $(\bDelta^{\op,n}_0)_{\bx\otimes\by/}^{\special}$. More precisely, we have an immediate identification after rewriting the top right corner as the limit over the category $(\bDelta_{01}^{\op,n-1})_{\by/}$ (which has $\Id_{\by}$ as initial object) of the functor $X(\fa(\bx)\otimes-)\times X(\fo(\bx)\otimes-)$.

We next argue that $(\bDelta^{\op,n}_0)_{\bx\otimes\by/}^{\special}$ is initial in $(\bDelta^{\op,n}_0)_{\bx\otimes\by/}$. We fix  $(\bx\otimes\by\to\bz)\in (\bDelta^{\op,n}_0)_{\bx\otimes\by/}$ and claim that the overcategory $((\bDelta^{\op,n}_0)_{\bx\otimes\by/}^{\special})_{/(\bx\otimes\by\to\bz)}$ is weakly contractible; we expand $\bz=[k_\bz]\otimes[\bk_\bz]$ with $[k_\bz]\in\bDelta^\op\simeq\bDelta^\op_0$ and $[\bk_\bz]\in(\bDelta^{\op})^n\simeq\bDelta^{\op,n}_0$. We say that a factorisation $(\bx\otimes\by\overset{\text{special}}{\longrightarrow}\bz'\to\bz)\in((\bDelta^{\op,n}_0)_{\bx\otimes\by/}^{\special})_{/(\bx\otimes\by\to\bz)}$ is \emph{very special} if the morphism $\bz'\to\bz$ is sent to the identity of $[\bk_\bz]$ along the composite $\bDelta^{\op,n}_{01}\xrightarrow{\tau}\bDelta^{\op,n}_0\simeq(\bDelta^{\op})^n\to(\bDelta^\op)^{n-1}$ ending with the projection on the last $n-1$ factors.
Then on the one hand the full subcategory $((\bDelta^{\op,n}_0)_{\bx\otimes\by/}^{\special})_{/(\bx\otimes\by\to\bz)}^{\mathrm{very special}}$ spanned by very special factorisations is final in $((\bDelta^{\op,n}_0)_{\bx\otimes\by/}^{\special})_{/(\bx\otimes\by\to\bz)}$; and on the other hand $((\bDelta^{\op,n}_0)_{\bx\otimes\by/}^{\special})_{/(\bx\otimes\by\to\bz)}^{\mathrm{very special}}$ is isomorphic to either of the following weakly contractible categories:
\begin{itemize}
\item it is isomorphic to $[1]$ if the morphism $[k]\to[k_\bz]$ in $\bDelta^\op$ lifts to a morphism $\bx\to([k_\bz],0)$ in $\bDelta^\op_{01}$, and the morphism $[\bk]\to[\bk_\bz]$ in $(\bDelta^\op)^{n-1}$ lifts to a morphism $\by\to[\bk_\bz]$ in $\bDelta^{\op,n-1}_{01}$;
\item it is isomorphic to $[0]$ otherwise.
\end{itemize}
It follows that the last square of spaces is indeed cartesian, since the assumption that $X$ be $0$-truncated implies that $X(\bx\otimes\by)$ is equivalent to $\lim_{(\bDelta^{\op,n}_0)_{\bx\otimes\by/}}X$.
\end{proof}
\begin{cor}
\label{cor:straighteningLambda1}
Let $Y\in s^n\cS$; then the straightening of the cartesian fibration $\cT^n_1(Y)\to\bDelta_1$ is the functor
$\bDelta^\op_1\to\Catinfone$ sending $\bx\in\bDelta^\op_1$ to the pullback of the following cospan
\[
\begin{tikzcd}[row sep=10pt, column sep=140pt]
&\cL^{n-1}_1(Y_{\fa(\bx)})\times\cL^{n-1}_1(Y_{\fo(\bx)})\ar[d]\\
\cT^{n-1}_1(Y_{\tau(\bx)})\ar[r,"\cT^{n-1}_1(Y_{\epsilon_\fa(\bx)})\times\cT^{n-1}_1(Y_{\epsilon_\fo(\bx)})"] &\cT^{n-1}_1(Y_{\fa(\bx)})\times\cT^{n-1}_1(Y_{\fo(\bx)}),
\end{tikzcd}
\]
i.e. sending $\bx$ to the source of the cartesian fibration over $\cT^{n-1}_1(Y_{\tau(\bx)})$ corresponding to the pullback of the presheaf $\lambda^{n-1}_1(Y_{\fa(\bx)})\boxtimes\lambda^{n-1}_1(Y_{\fo(\bx)})$ along the functor
$\cT^{n-1}_1(Y_{\epsilon_\fa(\bx)})\times\cT^{n-1}_1(Y_{\epsilon_\fo(\bx)})$.
\end{cor}
\begin{proof}
We apply Lemma \ref{lem:Kpullback} to $\iota_*(Y)$, and obtain after evaluation at $\bx$ the following pullback square in $\Catinfone$:
\[
\begin{tikzcd}[row sep=10pt, column sep=120pt]
\cK^{n-1}_1(\iota_*(Y)_{\bx})\ar[r]\ar[d] \ar[dr,phantom,"\lrcorner"very near start]&\cL^{n-1}_1(Y_{\fa(\bx)})\times \cL^{n-1}_1(Y_{\fo(\bx)})\ar[d]\\
\cT^{n-1}_1(Y_{\tau(\bx)})\ar[r,"\cT^{n-1}_1(Y_{\epsilon_\fa(\bx)})\times\cT^{n-1}_1(Y_{\epsilon_\fo(\bx)})"]&\cT^{n-1}_1(Y_{\fa(\bx)})\times\cT^{n-1}_1(Y_{\fo(\bx)}).
\end{tikzcd}
\]
Here we have used that for any $\by\in\bDelta_{01}^\op$ and any $X\in s^n_{01}\cS$ we have an equivalence $\iota^*(X_{\by})\simeq \iota^*(X)_{\tau(\by)}$; and that if moreover $\by\in\bDelta^\op_0\simeq\bDelta^\op$ we also have an equivalence $X_{\by}\simeq\tau^*(\iota^*(X)_{\by})$. Specifically, we have set $X=\iota_*(Y)$ and $\by=\bx,\fa(\bx),\fo(\bx)$, and we also have used the equivalence $\iota^*\iota_*\overset{\simeq}{\to}\Id_{s^n\cS}$.

We conclude by recalling that the straightening of $\cT^n_1(X)\simeq \cK^n_1(\iota_*(X))\to\bDelta_1$ evaluates $\cK^{n-1}_1(\iota_*(X)_\bx)$, i.e. the top left corner, at $\bx$.
\end{proof}

\begin{defn}
Let $X\in s^n_\N\cS$ and $Y\in s^n\cS$. A morphism in $\cK^n_1(X)$ is \emph{vertically inert} if it is inert and its image along the composite $\cK^n_1(X)\to\bDelta^n_1\xrightarrow{\tau}\bDelta^n\xrightarrow{\pi_1}\bDelta$ is an identity in $\bDelta$. We use the index ``$(-)_{\vinert}$'' to denote wide subcategories spanned by vertically inert morphisms.
We define $\ddot\cK^n(X):=L(\cK^n_1(X),\cK^n_1(X)_{\vinert})$. We define $\ddot\cT^n(Y):=\ddot\cK^n(\iota_*(Y))$, and by abuse of notation we also define $\ddot\cT^n(X)$ as $\ddot\cT^n(\iota^*(X))$. Finally, we set $\ddot\cL^n(Y):=\ddot\cK^n(\tau^*(Y))$.
\end{defn}
\begin{lem}
\label{lem:straighteningddotcK}
For $X\in s^n_\N\cS$, the localised functor $\ddot\cK^n(X)\to\bDelta_1$ is a cartesian fibration corresponding to the functor $\bDelta^\op_1\to\Catinfone$ sending $\bx\mapsto \dot\cK^{n-1}(X_\bx)$.
\end{lem}
\begin{proof}
We apply Lemma \ref{lem:vertlocfibration}
to the
relative category $(\cK^n_1(X),\cK^n_1(X)_{\vinert})$, which is obtained by unstraightening the functor $\bDelta_1^\op\to\RelCatinfone$ sending 
\[
\bx\mapsto(\cK^{n-1}_1(X_\bx),\cK^{n-1}_1(X_\bx)_{\inert}).
\]
\end{proof}

\begin{cor}
\label{cor:straighteningddotcT}
Let $\cC\in\Catinfn$; then the localised functor $\ddot\cT^n(\cC)\to\bDelta_1$ is again a cartesian fibration, and it corresponds to the functor $\bDelta^\op_1\to\Catinfone$ sending $\bx\in\bDelta^\op_1$ to the
pullback of the following cospan
\[
\begin{tikzcd}[row sep=10pt, column sep=140pt]
&\dot\cL^{n-1}(\cC_{\fa(\bx)})\times\dot\cL^{n-1}(\cC_{\fo(\bx)})\ar[d]\\
\dot\cT^{n-1}(\cC_{\tau(\bx)})\ar[r,"\cT^{n-1}(\cC_{\epsilon_\fa(\bx)})\times\dot\cT^{n-1}(\cC_{\epsilon_\fo(\bx)})"] &\dot\cT^{n-1}(\cC_{\fa(\bx)})\times\dot\cT^{n-1}(\cC_{\fo(\bx)}),
\end{tikzcd}
\]
i.e. sending $\bx$ to the
source of the cartesian fibration over $\dot\cT^{n-1}(\cC_{\tau(\bx)})$ corresponding to the pullback of the presheaf $\dot\lambda^{n-1}(\cC_{\fa(\bx)})\boxtimes\dot\lambda^{n-1}(\cC_{\fo(\bx)})$ along the functor $\dot\cT^{n-1}(\cC_{\epsilon_\fa(\bx)})\times \dot\cT^{n-1}(\cC_{\epsilon_\fo(\bx)})$.
\end{cor}
\begin{proof}
The pullback square from the proof of Corollary \ref{cor:straighteningLambda1} can be enhanced, naturally in $\bx\in\bDelta^\op_1$, to a pullback square in $\RelCatinfone$ by endowing each category with its wide subcategory of inert morphisms; the vertical arrows are then right fibrations fitting into the hypotheses of Lemma \ref{lem:horlocfibration}: for the right vertical arrow we invoke Lemma \ref{lem:kappaXproperties} (here we use that $\cC\in\Catinfn$), and for the left vertical arrow we use that the square is a pullback. Lemma \ref{lem:horlocfibration} allows us to localise the square pointwise at inert morphisms, obtaining again a pullback square whose vertical arrows are right fibrations. The localised top left corner is then by Lemma \ref{lem:straighteningddotcK} the functor $\bDelta^\op_1\to\Catinfone$ straightening the cartesian fibration $\ddot\cT^n(\cC)\to\bDelta_1$. 
\end{proof}

We can make the statement of Corollary \ref{cor:straighteningddotcT} more explicit as follows.
\begin{defn}
\label{defn:GammaC}
Let $\del(\bDelta_{01}^\op\times[1])\subset\bDelta^\op_{01}\times[1]$ denote the full subcategory 
\[
\del(\bDelta^\op_{01}\times[1]):=(\bDelta^\op_0\times[1])\cup(\bDelta^\op_{01}\times\set{1}).
\]
For $\cC\in\Catinfn$ we denote by $\hat\Gamma^n(\cC)\colon\bDelta^\op\times[1]\to\Catinfone$ the adjoint of the composite functor
\[
\begin{tikzcd}[column sep=30pt]
\bDelta^\op_{01}\ar[r,"\tau"]&\bDelta^\op\ar[r,"\cC_{(-)}"]&\Catinfnm\ar[rr,"(\dot\cL^{n-1}\to\dot\cT^{n-1})"]& &\PSh\subset\Fun([1],\Catinfone),
\end{tikzcd}
\]
and we denote by $\Gamma^n(\cC)\colon\bDelta^\op\times[1]\to\Catinfone$ the right Kan extension to $\bDelta^\op_{01}\times[1]$ of the restriction to $\del(\bDelta^\op_{01}\times[1])$ of $\hat\Gamma^n(\cC)$. For $j=0,1$ we let $\hat\Gamma^n(\cC)_j,\Gamma^n(\cC)_j\colon\bDelta^\op_{01}\to\Catinfone$ denote the restriction of $\hat\Gamma^n(\cC),\Gamma^n(\cC)$ to $\bDelta^\op_{01}\times\set{j}\simeq\bDelta^\op_{01}$.
\end{defn}
Corollary \ref{cor:straighteningddotcT} then identifies the straightening of $\ddot\cT^n(\cC)\to\bDelta_1$ with the restriction $\Gamma^n(\cC)_0|_{\bDelta^\op_1}\colon\bDelta^\op_1\to\Catinfone$. By Lemma \ref{lem:straighteningddotcK} we can similarly identify the straightening of $\ddot\cL^n(\cC)\to\bDelta_1$ with the restriction $\hat\Gamma^n(\cC)_0|_{\bDelta^\op_1}$; the left fibration $\ddot\cL^n(\cC)\to\ddot\cT^n(\cC)$ corresponds to the unit of the adjunction $\hat\Gamma(\cC)_0|_{\bDelta^\op_1}\to\Gamma(\cC)_0|_{\bDelta^\op_1}$.

\begin{lem}
\label{lem:GammacC01Segal}
Assume that Proposition \ref{prop:dotcTcLsm} holds for $n-1$, and let $\cC\in\Catinfn$. Then the functors $\hat\Gamma^n(\cC)_0,\Gamma^n(\cC)_0\colon\bDelta^\op_{01}\to\Catinfone$ satisfy the $01$-Segal condition, i.e. their restriction to $\bDelta^\op_{01,\inert}$ are right Kan extended from $\bDelta^\op_{01,\inert,\le1}$.
\end{lem}
\begin{proof}
Since $\cC\in\Catinfn$, we have that $\cC_{[0]}\simeq\cC^\simeq\in\Catinfnm$ is a space. The restrictions of both $\hat\Gamma^n(\cC)$ and $\Gamma^n(\cC)$ to $\set{([0],0)}\times[1]\simeq[1]$ give the arrow $\dot\cL^{n-1}(\cC_{[0]})\to\dot\cT^{n-1}(\cC_{[0]})$ in $\Fun([1],\Catinfone)$, which by Lemma \ref{lem:dotcTcLspace} can be identified with the diagonal arrow of spaces $\cC^{\simeq}\to(\cC^{\simeq})^{S^{n-2}}$.

We next observe that both restrictions of $\hat\Gamma^n(\cC)_0$ and $\Gamma^n(\cC)_0$ to $\bDelta^\op_0\simeq\bDelta^\op$ are equivalent to $\dot\cL^{n-1}(\cC_{(-)})$; and given $[k]\in\bDelta$, we may identify
\[
\dot\cL^{n-1}(\cC_{[k]})\xrightarrow{\simeq}\dot\cL^{n-1}\pa{\cC_{[1]}\underset{\cC^\simeq}\times\dots\underset{\cC^\simeq}\times\cC_{[1]}}\xrightarrow{\simeq}\dot\cL^{n-1}(\cC_{[1]})\underset{\cC^\simeq}\times\dots\underset{\cC^\simeq}\times\dot\cL^{n-1}(\cC_{[1]}),
\]
using the Segal condition for $\cC_{(-)}$ and the hypothesis that $\dot\cL^{n-1}$ preserves pullbacks over spaces. Recalling that $\hat\Gamma(\cC)_0$ is obtained from $\dot\cL^{n-1}(\cC_{(-)})$ by precomposing with $\tau$, we immediately get the $01$-Segal condition for $\hat\Gamma(\cC)_0$.

The formula for right Kan extension allows us to compute $\Gamma^n(\cC)_0([1],1)$ as the following pullback in $\Catinfone$:
\[
\begin{tikzcd}[row sep=10pt, column sep=8pt]
\Gamma^n(\cC)_0([1],1)\ar[d,"0\to1"]\ar[rrr,"{(d_1,d_0)}"]\ar[drrr,phantom,"\lrcorner"very near start] &&&\Gamma^n(\cC)_0([0],0)\times\Gamma^n(\cC)_0([0],0)\ar[d,"0\to1"]\ar[r,phantom,"\simeq"] &\cC^\simeq\times\cC^\simeq\\
\Gamma^n(\cC)_1([1],1)\ar[rrr,"{(d_1,d_0)}"]&&&\Gamma^n(\cC)_1([0],0)\times\Gamma^n(\cC)_1([0],0)\ar[r,phantom,"\simeq"]&  (\cC^\simeq)^{S^{n-2}}\times(\cC^\simeq)^{S^{n-2}}.
\end{tikzcd}
\]
For generic $\bx\in\bDelta^\op_1$, we can then compute $\Gamma^n(\cC)_0(\bx)$ as the following pullback square, in terms of $\Gamma^n(\cC)|_{\del(\bDelta^\op_{01}\times[1])}\simeq \hat\Gamma^n(\cC)|_{\del(\bDelta^\op_{01}\times[1])}$ and of $\Gamma^n(\cC)_0([1],1)$:
\[
\begin{tikzcd}[row sep=10pt, column sep=20pt]
\Gamma^n(\cC)_0(\bx)\ar[r]\ar[d,"0\to1"]\ar[dr,phantom,"\lrcorner"very near start]&\Gamma^n(\cC)_0(\fa(\bx))\underset{\cC^\simeq}{\times}\Gamma^n(\cC)_0([1],1)\underset{\cC^\simeq}{\times}\Gamma^n(\cC)_0(\fo(\bx))\ar[d,"0\to1"]\\
\Gamma^n(\cC)_1(\bx)\ar[r]&\Gamma^n(\cC)_1(\fa(\bx))\underset{(\cC^\simeq)^{S^{n-2}}}{\times}\Gamma^n(\cC)_1([1],1)\underset{(\cC^\simeq)^{S^{n-2}}}{\times}\Gamma^n(\cC)_1(\fo(\bx)).
\end{tikzcd}
\]
We now observe that the bottom row in the last square is an equivalence, since $\Gamma^n(\cC)_1$ agrees with $\hat\Gamma^n(\cC)_1$, and hence with $\dot\cT^{n-1}(\cC_{\tau(-)})$: here we use that $\dot\cT^{n-1}$ preserves pullbacks over spaces. Hence the top row in the last square is also an equivalence, and this readily implies the $01$-Segal condition for $\Gamma^n(\cC)_0$ at $\bx\in\bDelta^\op_1$.
\end{proof}

\subsection{Mapping spaces in \texorpdfstring{$\dot{\mathcal{T}}^n(\cC)$}{dotcTn(cC)} for \texorpdfstring{$n\ge2$}{n>=2}}
\begin{defn}
Let $\cD\to\bDelta_1$ be a cartesian fibration. We say that a morphism in $\cD$ is \emph{horizontally inert} if it is a cartesian lift of an inert morphism in $\bDelta_1$. We let $\cD_{\hinert}$ denote the wide subcategory of $\cD$ spanned by horizontally inert morphisms.
\end{defn}
We remark that in the cases $\cD=\ddot\cT^n(Y)$ and $\cD=\ddot\cL^n(Y)$ for some $Y\in s^n\cS$, the localisation $L(\cD,\cD_{\hinert})$ recovers $\dot\cT^n(Y)$ and $\dot\cL^n(Y)$, respectively. To save notation, we fix in the rest of the subsection a cartesian fibration $\cE\to\bDelta_{01}$, whose associated functor $\gamma\colon\bDelta_{01}^\op\to\Catinfone$ satisfies the $01$-Segal condition, and whose value $\gamma_0:=\gamma([0],0)$ is a space. The examples to keep in mind are of course $\gamma=\hat\Gamma(\cC)_0$ and $\gamma=\Gamma^n(\cC)_0$ for some $\cC\in\Catinfn$. We denote henceforth by $\cD:=\cE\times_{\bDelta_{01}}\bDelta_1\to\bDelta_1$ the restricted cartesian fibration. Our goal is to compute mapping spaces in $L(\cD,\cD_{\hinert})$, and we start with the observation that, by same argument used in the proof of Lemma \ref{lem:essentiallysurjective}, for every object $\bx\in\cD$ we can find an object $\by\in\cD$ lying over $([1],1)\in\bDelta_1$ such that $\bx$ and $\by$ become equivalent inside $L(\cD,\cD_{\hinert})$.

We fix therefore $\bx\in\cD$ of the form $([1],1,x)$, i.e. $x$ is an object in $\gamma([1],1)$, and we aim at computing mapping spaces into $\bx$ in $L(\cD,\cD_{\hinert})$.

\begin{defn}
We define $W(\bx)$ as the full subcategory of $\cD_{\bx/}$ spanned by horizontally inert morphisms, and let $\pi(\bx)\colon W(\bx)\hto\cD_{/\bx}$ denote the natural inclusion.
\end{defn}

We fix $\by=([k],\nu,y)\in\cD$, i.e., $y\in\gamma([k],\nu)$, and aim at computing the space $(d_0\pi(\bx))_!(*)(\by)\simeq|\cD_{\by/,\bx/\hinert}|$, where $\cD_{\by/,\bx/\hinert}$ denotes the category of cospans $\by\to\bz\ot\bx$ in $\cD$ with $\bx\to\bz$ horizontally inert.
\begin{nota}
We denote by $\by_\fa=(\fa([k],\nu),y_\fa)\to\by$, $\by_\fo=(\fo([k],\nu),y_\fo)\to\by$ and $\by'=([1],1,y')\to\by$ cartesian lifts in $\cE$ with target $\by$ of the morphisms $\epsilon_\fa([k],\nu)$, $\epsilon_\fo([k],\nu)$ and of the inert morphism $([1],1)\to([k],\nu))$ in $\bDelta_{01}$, respectively. We further denote by $\by'_\fa=([0],0,y'_\fa)\to\by'$ and $\by'_\fo=([0],0,y'_\fo)\to\by'$ cartesian lifts at $\by'$ of $d_1=\epsilon_\fa([1],1)$ and $d_0=\epsilon_\fo([1],1)$, respectively.
\end{nota}

\begin{nota}
For $([k'],\nu')\in\bDelta_1$ we abbreviate the object $([k'+2],(0,\nu',0))\in\bDelta_1$ by $([k'+2],0\nu'0)$, and we denote by
\[
\begin{tikzcd}
([k'],\nu')\ar[r,"\epsilon_{\fa\fo}(\nu')"]&([k+2],0\nu'0)&([1],1)\ar[l,"i(\nu')"']
\end{tikzcd}
\]
the unique cospan in $\bDelta_1$ with $i(\nu')$ inert, and whose projection to $\bDelta$ along $\tau$ exhibits $[k+2]$ as the disjoint union of the sets $[k]$ and $[1]$.
\end{nota}
Note that $\epsilon_{\fa}([k'],\nu')$ and $\epsilon_{\fo}([k'],\nu')$ exhibit $[k']$ as the concatenation of $\fa([k'],\nu')$ and $\fo([k'],\nu')$ under projection to $\bDelta$. Correspondingly, $\epsilon_{\fa\fo}(\nu')$ and $i(\nu')$ exhibit $[k'+2]$ as the concatenation of $\fa([k'],\nu')$, $[1]$ and $\fo([k'],\nu')$ under projection to $\bDelta$; this justifies the notation ``$\epsilon_{\fa\fo}$''.
\begin{defn}
A cospan $(\by\to\bz\ot\bx)\in\cD_{\by/,\bx/\hinert}$ is \emph{special} if:
\begin{enumerate} [label=(\roman*)]
\item the image of $\by\to\bz$ along $\cD\to\bDelta_1$ is $\epsilon_{\fa\fo}(\nu)$;
\item the composite morphisms $\by_\fa\to\by\to\bz$ and $\by_\fo\to\by\to\bz$ are cartesian lifts of their projections $\fa([k],\nu)\to[k+2]$ and $\fo([k],\nu)\to[k+2]$ to $\bDelta_{01}$.
\end{enumerate}
We denote by $\cD_{\by/,\bx/\hinert}^\special\subset\cD_{\by/,\bx/\hinert}$ the full subcategory of special objects, and by $\cD_{\by/,\bx/\hinert}^{\special'}\subset\cD_{\by/,\bx/\hinert}$ the full subcategory spanned by objects that satisfy (i).

\end{defn}

\begin{lem}
\label{lem:Dspecialinitial}
The following hold:
\begin{enumerate}
\item $\cD_{\by/,\bx/\hinert}^\special$ is initial in $\cD_{\by/,\bx/\hinert}$;
\item a horizontally inert morphism $\bt\to\by$ in $\cD$ induces an equivalence
\[
\cD_{\by/,\bx/\hinert}^\special\overset{\simeq}{\to}\cD_{\bt/,\bx/\hinert}^\special.
\]
\end{enumerate}
\end{lem}
\begin{proof}
For (1), we first prove that $\cD_{\by/,\bx/\hinert}^{\special'}$ is initial in $\cD_{\by/,\bx/\hinert}$: to see this, we observe that for an object $(\by\to\bz'\ot\bx)\in\cD_{\by/,\bx/\hinert}$ lying over a cospan in $\bDelta_1$ of the form $([k],\nu)\to([k'],\nu')\ot([1],1)$, the overcategory $(\cD_{\by/,\bx/\hinert}^{\special'})_{/(\by\to\bz'\ot\bx)}$ admits a terminal object $\by\to\bz\ot\bx$, obtained by letting $\bz\to\bz'$ be a cartesian lift of the unique morphism $([k+2],0\nu0)\to([k'],\nu')$ in $\bDelta_1$ making the following diagram in $\bDelta_1$ commute:
\[
\begin{tikzcd}[row sep=15pt]
([k],\nu)\ar[r,"\epsilon_{\fa\fo}(\nu)"]\ar[dr] &([k+2],0\nu0)\ar[d,"\exists!"]&([1],1)\ar[l,"i(\nu)"']\ar[dl,"\text{inert}"]\\
&([k'],\nu').
\end{tikzcd}
\]
We next argue that $\cD_{\by/,\bx/\hinert}^\special$ is initial in $\cD_{\by/,\bx/\hinert}^{\special'}$.
The $01$-Segal condition for $\gamma$ allows us to write a commutative diagram as follows:
\[
\begin{tikzcd}[row sep=15pt]
\gamma([1],1)\ar[r,equal]&\gamma([1],1)\\
\gamma([k+2],0\nu0)\ar[r,phantom,"\simeq"]\ar[d,"\gamma(\epsilon_{\fa\fo}(\nu))"]\ar[u,"\gamma(i(\nu))"']&\gamma(\fa([k],\nu))\times_{\gamma_0}\gamma([3],010)\times_{\gamma_0}\gamma(\fo([k],\nu))\ar[d,"\Id\times\gamma(\epsilon_{\fa\fo}(1))\times\Id"]\ar[u,"\gamma(i(1))"']
\\
\gamma([k],\nu)\ar[r,phantom,"\simeq"]&\gamma(\fa([k],\nu))\times_{\gamma_0}\gamma([1],1)\times_{\gamma_0}\gamma(\fo([k],\nu))\\
\gamma([k],\nu)_{y/}\ar[r,phantom,"\simeq"]\ar[u,"d_0"] & \gamma(\fa([k],\nu))_{y_\fa/}\times_{(\gamma_0)_{y'_\fa/}}\gamma([1],1)_{y'/}\times_{(\gamma_0)_{y'_\fo/}}\gamma(\fo([k],\nu))_{y_\fo/}\ar[u,"d_0\times d_0\times d_0"']
\end{tikzcd}
\]
Note that since $\gamma_0$ is a space, overcategories of $\gamma_0$ are contractible and hence the bottom right term in the previous diagram is a plain product of $\infone$-categories.

Using the left column in the previous diagram we can identify $\cD_{\by/,\bx/\hinert}^{\special'}$ with the limit of the following W-shaped diagram in $\Catinfone$
\[
\begin{tikzcd}[row sep=2pt]
*\ar[dr,"x"]& &\gamma([k+2],0\nu0)\ar[dl,"\gamma(i(\nu))"']\ar[dr,"\gamma(\epsilon_{\fa\fo}(\nu))"]& &\gamma([k],\nu)_{y/}\ar[dl,"d_0"']\\
&\gamma([1],1)& &\gamma([k],\nu),
\end{tikzcd}
\]
whereas $\cD_{\by/,\bx/\hinert}^{\special}$ is equivalent to the fibre at $(\Id_{y_\fa},\Id_{y_\fo})$ of the composite functor
\[
\begin{tikzcd}[column sep=20pt]
\cD_{\by/,\bx/\hinert}^{\special'}\ar[r]&\gamma([k],\nu)_{y/}\ar[rrr,"{\gamma(\epsilon_\fa([k],\nu))\times}","{\gamma(\epsilon_\fo([k],\nu))}"']& & &\gamma(\fa([k],\nu))_{y_{\fa}/}\times\gamma(\fo([k],\nu))_{y_{\fo}/}.
\end{tikzcd}
\]
Using the right column, we can then give fibre product decompositions:
\[
\begin{tikzcd}[row sep=10pt, column sep=-8pt]
\cD_{\by/,\bx/\hinert}^{\special}\simeq\ar[d] & *\quad\quad\times\!\pa{\!*\!\!\underset{\gamma([1],1)}\times\!\gamma([3],010)\underset{\gamma([1],1)}\times\!\gamma([1],1)_{y'/}\!}\!\times\quad\quad *\ar[d,"\Id_{y_\fa}\times(\Id)\times\Id_{y_\fo}"]\\
\cD_{\by/,\bx/\hinert}^{\special'}\simeq & \gamma(\fa([k],\nu))_{y_\fa/}\!\times\!\! \pa{\!*\!\!\underset{\gamma([1],1)}\times\!\gamma([3],010)\underset{\gamma([1],1)}\times\!\gamma([1],1)_{y'/}\!}\!\!\times\! \gamma(\fo([k],\nu))_{y_\fo/},
\end{tikzcd}
\]
where the fibre products in the parentheses are taken once along $\gamma(i(1))$ and once along $\gamma(\epsilon_{\fa\fo}(1))$; since a product of initial functors is initial, and $\Id_{y_\fa}$ and $\Id_{y_\fo}$ are initial objects in their respective overcategories,
$\cD_{\by/,\bx/\hinert}^{\special}$ is initial in $\cD_{\by/,\bx/\hinert}^{\special'}$.

For (2), the top equivalence in the previous diagram is nothing but the functor $\cD^\special_{\by/,\bx/\hinert}\xrightarrow{\simeq}\cD^\special_{\by'/,\bx/\hinert}$ induced by the horizontally inert morphism $\by'\to\by$. If $\bt\to\by$ is horizontally inert, then we have a composition of horizontally inert morphisms $\by'\to\bt\to\by$ and we conclude by ``two out of three''.
\end{proof}
\begin{cor}
\label{cor:LDDhintmappingspaces}
For objects $\bx=([1],1,x)$ and $\by=([1],1,y)$ in $\cD$, the morphism space $L(\cD,\cD_{\hinert})(\by,\bx)$ is equivalent to the classifying space $|\Xi(\gamma,\by,\bx)|$ of the limit $\Xi(\gamma,\by,\bx)$ in $\Catinfone$ of the following W-shaped diagram:
\[
\begin{tikzcd}[row sep=2pt]
*\ar[dr,"x"]& &\gamma([3],010)\ar[dl,"\gamma(i(1))"']\ar[dr,"\gamma(\epsilon_{\fa\fo}(1))"]& &\gamma([1],1)_{y/}\ar[dl,"d_0"']\\
&\gamma([1],1)& &\gamma([1],1).
\end{tikzcd}
\]
\end{cor}
\begin{proof}
By Lemma \ref{lem:Dspecialinitial}, Theorem \ref{thm:Cisinski} is applicable and we can identify the spaces 
\[
L(\cD,\cD_{\hinert})(\by,\bx)\simeq|\cD_{\by/,\bx/\hinert}|\simeq|\cD_{\by/,\bx/\hinert}^\special|,
\]
and the proof of \ref{lem:Dspecialinitial} identifies $\cD_{\by/,\bx/\hinert}^\special\simeq \Xi(\gamma,\by,\bx)$ in $\Catinfone$.
\end{proof}
\begin{proof}[Proof of Proposition \ref{prop:dotcTcLsm}]
We assume that the statement holds for $n-1$ and aim at proving it for a fixed $n\ge2$. We first prove that $\dot\cL^n$ preserves pullbacks over spaces. Let therefore $\cC\to Y\ot \cC'$ be a cospan in $\Catinfn$ with $Y\in\cS$. By Lemma \ref{lem:essentiallysurjective}, the rows in the commutative square
\[
\begin{tikzcd}[row sep=10pt]
\tau^*\Big(\cC\underset{Y}\times\cC'\Big)([1^n],1)\ar[r,]\ar[d,"\simeq"]&\dot\cL^n\Big(\cC\underset{Y}\times\cC'\Big)^\simeq\ar[d]\\
\tau^*(\cC)([1^n],1)\underset{Y}\times\tau^*(\cC')([1^n],1)\ar[r] &\pa{\dot\cL^n(\cC)\underset{Y}\times\dot\cL^n(\cC')}^\simeq
\end{tikzcd}
\]
are surjective on $\pi_0$, implying that the right vertical arrow is also surjective on $\pi_0$; in other words, the canonical functor $\dot\cL^n(\cC\times_Y\cC')\to\dot\cL^n(\cC)\times_{\dot\cL^n(Y)}\dot\cL^n(\cC')$ is essentially surjective. 
To prove full faithfulness, we first use the inductive hypothesis to obtain that the canonical map $\ddot\cL^n(\cC\times_Y\cC')\to\ddot\cL^n(\cC)\times_{\ddot\cL^n(Y)}\ddot\cL^n(\cC')$ is an equivalence of cartesian fibrations over $\bDelta_1$.
Now let $\bx=([1],1,x)$ and $\by=([1],1,y)$ be objects in $\ddot\cL^n(\cC\times_Y\cC')$, with $x,y\in\hat\Gamma^n(\cC\times_Y\cC')_0([1],1)$. Denote by $\bx_\cC,\bx_Y,\bx_{\cC'}$ the images of $\bx$ in $\hat\Gamma(\cC)_0([1],1)$, $\hat\Gamma(Y)_0([1],1)$ and $\hat\Gamma(\cC')_0([1],1)$, and define similarly $\by_\cC,\by_Y,\by_{\cC'}$.  We may then write a sequence of equivalences of spaces
\[
\begin{split}
\dot\cL^n(\cC\times_Y\cC')(\by,\bx)& \simeq|\Xi(\hat\Gamma(\cC\times_Y\cC')_0,\by,\bx)|\\
&\simeq |\Xi(\hat\Gamma(\cC)_0,\by_\cC,\bx_\cC)\times_{\Xi(\hat\Gamma(Y)_0,\by_Y,\bx_Y)}\Xi(\hat\Gamma(\cC')_0,\by_{\cC'},\bx_{\cC'}|\\
&\simeq |\Xi(\hat\Gamma(\cC)_0,\by_\cC,\bx_\cC)|\times_{|\Xi(\hat\Gamma(Y)_0,\by_Y,\bx_Y)|}|\Xi(\hat\Gamma(\cC')_0,\by_{\cC'},\bx_{\cC'})|\\
&\simeq\Big(\dot\cL^n(\cC)\times_{\dot\cL^n(Y)}\dot\cL^n(\cC')\Big)(\by,\bx);
\end{split}
\]
the first and fourth equivalences are given by Corollary \ref{cor:LDDhintmappingspaces};
the second equivalence comes from the formula for $\Xi$ as an iterated limit, and the fact that limits commute with limits; the third equivalence comes from the fact that $\Xi(\hat\Gamma(Y)_0,\by_Y,\bx_Y)$ is a space, and taking classifying spaces commutes with fibre products over a space. This proves that the canonical functor $\dot\cL^n(\cC\times_Y\cC')\to\dot\cL^n(\cC)\times_{\dot\cL^n(Y)}\dot\cL^n(\cC')$ is fully faithful, and therefore it is and equivalence of $\infone$-categories.

The proof that $\dot\cT^n$ preserves pullbacks over spaces is completely analogous: we just replace ``$\tau^*$'' by ``$\iota_*$'' and ``$Y$'' by ``$Y^{S^{n-1}}$'' in the first diagram above, and we replace ``$\hat\Gamma$'' by ``$\Gamma$'' in the rest of the argument.
\end{proof}

\subsection{A symmetric monoidal square}
The pullback square from Proposition \ref{prop:Catinfn01pullback} is not a square in $\CAlg(\Catinfone^{\fU_2})$, for instance because the left vertical functor $(\dot\cT^n,\dot\kappa^n)$ is not endowed with a symmetric monoidal structure. In this subsection we introduce an equivalent pullback square admitting a lift to $\CAlg(\Catinfone^{\fU_2})$.

In the following, let $M$ be an ungroup-like abelian monoid. Consider the functor $(\dot\cL^n\to\dot\cT^n)\colon\Catinfn\to\PSh$, which is symmetric monoidal (it preserves finite products) by Proposition \ref{prop:dotcTcLsm}. The terminal algebra $*\in\CAlg(\Catinfn)$ is sent to the terminal algebra $(*\to*)\in\CAlg(\PSh)$, and $\bB^n M\in\CAlg(\Catinfn)$ is sent along $(\dot\cL^n\to\dot\cT^n)$ to some commutative algebra $(\dot\cL^n\to\dot\cT^n)(\bB^n M)\in\CAlg(\PSh)$, about which we make the following conjecture.
\begin{conj}
For $X\in\Top$, denote by $\mathrm{SP}(X,M)$ the free abelian (strictly commutative) topological monoid generated by $X$ many copies of $M$, i.e. with the following universal property: for any abelian topological monoid $Y$ we have a bijection of sets, naturally in $Y$:
\[
\map_{\mathrm{TopAbMon}}(\mathrm{SP}(X,M),Y)\cong\map_\Top(X,\map_{\mathrm{AbMon}}(M,Y)).
\]
Then we have an equivalence $\dot\cT^n(\bB^nM)\simeq\bB(\mathrm{SP}(S^{n-1},M))$, and the right fibration $\dot\cL^n(\bB^nM)\to\dot\cT^n(\bB^nM)$ corresponds to the action of $\mathrm{SP}(S^{n-1},M)$ on $M$ induced by the topological monoid homomorphism $\mathrm{SP}(S^{n-1},M)\to\mathrm{SP}(*,M)\cong M$.
\end{conj}
We will content ourselves with a milder analysis of $\dot\cT^n(\bB^nM)$ and $\dot\cL^n(\bB^nM)$.
We observe that the fibre of the right fibration $(\cL^n_1\to\cT^n_1)(\bB^nM)$ at any object $\bx\in\cT^n_1(\bB^nM)$ is $\lambda^n_1(\bB^nM)(\bx)\simeq M$; the same holds for the localised right fibration $(\dot\cL^n\to\dot\cT^n)(\bB^nM)$; in particular, taking the fibre at the monoidal unit of $\dot\cT^n(\bB^nM)$, we obtain the following commutative diagram in $\CAlg(\Catinfone)$ with right fibrations as vertical arrows, and featuring a pullback square on right:
\[
\begin{tikzcd}[row sep=10pt, column sep=50pt]
*\ar[r,"\iota"]\ar[d,equal]&M\ar[r]\ar[d]\ar[dr,phantom,"\lrcorner"very near start]& \dot\cL^n(\bB^nM)\ar[d]\\
*\ar[r,equal]&*\ar[r,"{([1^n],1,*)}"] &\dot\cT^n(\bB^nM).
\end{tikzcd}
\]
Now let $(\bE M\to\bB M)\in\CAlg(\PSh)$ denote the right fibration corresponding to the action of $M$ on itself, and note that both $\bE M$ and $\bB M$ are plain 1-categories.

\begin{defn}
\label{defn:fdcT}
Let $M$ be an ungroup-like abelian monoid.
For an object $\bx=([\bk],\nu,x)\in\cT^n_1(\bB^nM)$, we consider the set of $(\prod_{i=1}^nk_i)-1$ inert morphisms $([1^n],0)\to([\bk],\nu)$ in $\bDelta^n_{01}$, take for each of these a cartesian lift $([1^n],0,y)$ with $y\in M$, and define the total grading of $\bx$ as the sum of the $(\prod_{i=1}^nk_i)-1$ elements $y$ thus obtained. We denote by
\[
\fd_\cT\colon\cT^n_1(\bB^nM)\to\bB M
\]
the functor sending each object to $*\in\bB M$, and sending a morphism $f\colon\bx\to\by$ in $\cT^n_1(\bB^nM)$ to the total grading of the essentially unique object $\bz\in\cT^n_1(\bB^nM)$ for which there is a commutative diagram in $\cT^n_1(\bB^nM)$ as follows
\[
\begin{tikzcd}[column sep=50pt, row sep=10pt]
\bx\ar[r,"f"] &\by\\
([1^n],1,*)\ar[u,"\text{inert}"']\ar[r,"\text{active}","f'"']&\bz.\ar[u,"\text{inert}"']
\end{tikzcd}
\]
We similarly define a functor $\fd_\cL\colon\cL^n_1(\bB^nM)\to\bE M$, sending an object $([\bk],\nu,x)$ to the element $y\in M$ such that $([1^n],1,y)\to([\bk],\nu,x)$ is an inert morphism, and whose behaviour on morphisms is dictated by the existence of the 
pullback square on left in the following:
\[
\begin{tikzcd}[row sep=10pt]
\cL^n_1(\bB^nM)\ar[r,"\fd_\cL"]\ar[d]\ar[dr,phantom,"\lrcorner"very near start] &\bE M\ar[d]\\
\cT^n_1(\bB^nM)\ar[r,"\fd_\cT"]&\bB M;
\end{tikzcd}
\hspace{2cm}
\begin{tikzcd}[row sep=10pt]
\dot\cL^n(\bB^nM)\ar[r,"\fd"]\ar[d] \ar[dr,phantom,"\lrcorner"very near start]&\bE M\ar[d]\\
\dot\cT^n(\bB^nM)\ar[r,"\fd"]&\bB M.
\end{tikzcd}
\]
We denote by $\fd_\cL$ and $\fd_\cT$ also the induced functors on localisations, as in the square on right, which by virtue of Lemma \ref{lem:horlocfibration} is again a pullback.
\end{defn}
A direct inspection shows that both localised functors $\fd_\cL,\fd_\cT$ are symmetric monoidal: since $\bE M$ and $\bB M$ are plain categories, this is a property. In fact the second pullback square in Definition \ref{defn:fdcT} is one in $\CAlg(\Catinfone)$.
\begin{rem}
\label{rem:twistedgrading}
Given an $M$-graded $\infn$-category $\cC$, the map $\fd\colon\cC\to\bB^nM$ in $\Catinfn$ gives rise to a map $\dot\cT^n(\fd)\colon\dot\cT^n(\cC)\to\dot\cT^n(\bB^nM)$, which composed with $\fd_\cT$ gives an $M$-grading on the $\infone$-category $\dot\cT^n(\cC)$.

Now suppose that we have a presheaf $A\in\PSh(\dot\cT^n(\cC))$ and a map $(\cC,A)\to(\dot\cT^n,\dot\lambda^n)(\bB^nM)=(\dot\cL^n\to\dot\cT^n)(\bB^nM)$ in $\PSh$. Then composing with $(\fd_\cL,\fd_\cT)$ we obtain a \emph{twisted $M$-grading} on $A$ in the following sense: for each $\bx\in\cC$ we have the datum of a map of spaces $A(\bx)\to M$, which we call an \emph{$M$-grading} on the space $A(\bx)$; these $M$-gradings satisfy that for each $f\colon\bx\to\by$ in $\cC$, $x\in A(\bx)$ and $y\in A(\by)$ with $A(f)(y)\simeq x$ we have that the $M$-grading of $x$ is the sum of the $M$-grading of $y$ and that of $f$. Note that this structure on $A$ is different than that of $M$-graded presheaf, i.e. a lift of $A$ to a functor $\cC^\op\to\cS_{/M}$: in the latter case we would just have that the $M$-gradings of $x$ and $y$ are equal.
\end{rem}

\begin{nota}
For $M=\N$, we let $\hat\iota$ and $\varphi$ denote the two morphisms fitting in the following composition in $\CAlg(\PSh)$, which is the result of the previous discussion:
\[
(*\to*)\xrightarrow{\hat\iota}(\N\to*)\xrightarrow{\varphi}(\dot\cL^n\to\dot\cT^n)(\bB^n\N)\xrightarrow{(\fd_\cL,\fd_\cT)}(\bE\N\to\bB\N).
\]
\end{nota}
We observe that $\varphi$ and $\fd$ are cartesian lifts of their projections along $d_0\colon\PSh\to\Catinfone$, and that $\varphi\hat\iota$ recovers the morphism $(*\to*)\to(\dot\cL^n\to\dot\cT^n)(\bB^n\N)$ in $\CAlg(\PSh)$ obtained by applying the symmetric monoidal functor 
$(\dot\cL^n\to\dot\cT^n)$ to the morphism $\bB^n\iota\colon \bB^n0\simeq*\to\bB^n\N$ in $\CAlg(\Catinfn)$.
Passing to overcategories, we obtain a sequence of lax symmetric monoidal functors:
\[
\begin{tikzcd}[column sep=20pt]
\Catinfn^\N\simeq(\Catinfn)_{/\bB^n\N}\ar[rr,"\dot\cL\dot\cT_\N"]& &\PSh_{/(\dot\cL^n\to\dot\cT^n)(\bB^n\N)}\ar[r,"\varphi^*"]&\PSh_{/(\N\to*)}\ar[r,"\hat\iota^*"]&\PSh.
\end{tikzcd}
\]
The first functor $\dot\cL\dot\cT_\N$ is induced on overcategories over commutative algebras by the symmetric monoidal functor $(\dot\cL^n\to\dot\cT^n)$, and it is also symmetric monoidal. The other two functors
$\varphi^*$ and $\hat\iota^*$, which take pullbacks in $\PSh$ along $\varphi$ and $\hat\iota$, respectively, are a priori only lax symmetric monoidal, as their left adjoints $\varphi_!$ and $\hat\iota_!$, given by postcomposition with $\varphi$ and $\hat\iota$, are symmetric monoidal.
\begin{lem}
\label{lem:iota*dotTnBnN}
Consider $\dot\cT^n(\bB^n\N)$ as an $\N$-graded $\infone$-category via the $\N$-gra\-ding $\fd$ constructed above. Then the functor $*\simeq\dot\cT^n(*)\to\dot\cT^n(\bB^n\N)$ restricts to an equivalence of $\infone$-categories $*\xrightarrow{\simeq}\iota^*(\dot\cT^n(\bB^n\N))$.
\end{lem}
\begin{proof}
Consider also $\fd\colon\cT^n_1(\bB^n\N)\to\bB\N$ as an $\N$-graded $\infone$-category.
Recall Definition \ref{defn:RelCatinfoneM}, and let $I=\set{0}\subset\N$. The counit of the adjunction $\iota_!\colon\Catinfone\leftrightarrows\Catinfone^\N\colon\iota^*$ gives rise to the following morphism in $\RelCatinfone^\N$:
\[
(\iota_!\iota^*(\cT^n_1(\bB^n\N)),\cT^n_1(\bB^n\N)_{\inert})\to(\cT^n_1(\bB^n\N),\cT^n_1(\bB^n\N)_{\inert}).
\]
This morphism satisfies the hypotheses of Lemma \ref{lem:Mgradedlocalisation}, so  we have an equivalence
\[
T_0 (L(\iota_!\iota^*(\cT^n_1(\bB^n\N)),\cT^n_1(\bB^n\N)_{\inert}))\xrightarrow{\simeq} T_0(L(\cT^n_1(\bB^n\N),\cT^n_1(\bB^n\N)_{\inert})).
\]

The counit of the adjunction $\iota^*\colon\Catinfone^\N\leftrightarrows\Catinfone\colon\iota_*$ induces an equivalence $\iota^*T_0\simeq\iota^*\iota_*\iota^*\xrightarrow{\simeq}\iota^*$, using that the unit of the same adjunction is an equivalence. Using this, and applying $\iota^*$ to the previous equivalence, we obtain a new equivalence involving $\iota^*(\dot\cT^n(\bB^n\N))$ on the right hand side; therefore we reduce to the computation of the new left hand side, namely
\[
\iota^*(L(\iota_!\iota^*(\cT^n_1(\bB^n\N)),\cT^n_1(\bB^n\N)_{\inert}))\simeq L(\iota^*(\cT^n_1(\bB^n\N)),\cT^n_1(\bB^n\N)_{\inert}).
\]
We next argue that every morphism $f$ in $\iota^*(\cT^n_1(\bB^n\N))$ is inverted if we localise at inert morphisms. The commutative square used to define the $\N$-grading $\fd$ on $\cT^n_1(\bB^n\N)$, together with the assumption that $f$ has vanishing $\N$-grading, allows us to only check that the corresponding active morphism of the form $f'\colon([1^n],1,*)\to\bz$ in $\iota^*(\cT^n_1(\bB^n\N))$ is inverted; note that the $\N$-grading of $f'$ is also zero. We observe that $\bz$ itself must have vanishing total $\N$-grading, i.e. $\bz\in\cT^n_1(\bB^n0)\subset\cT^n_1(\bB^n\N)$, so that there exists a retraction $g\colon\bz\to([1^n],1,*)$ of $f'$ in $\cT^n_1(\bB^n0)$, and hence also in $\cT^n_1(\bB^n\N)$. The composition of $g$ with the inert morphism $([1^n],1,*)\hto\bz$ is forced to be the identity of $([1^n],1,*)$, so that inverting inert morphisms forces us to invert $g$, which forces us to invert $f'$, which forces us to invert $f$.

We have therefore $L(\iota^*(\cT^n_1(\bB^n\N)),\cT^n_1(\bB^n\N)_{\inert})\simeq |\iota^*(\cT^n_1(\bB^n\N)))|\in\cS$; to prove that the latter space is contractible, we observe that the full subcategory of $\iota^*(\cT^n_1(\bB^n\N)))$ spanned by objects of vanishing total $\N$-grading is initial, and that
$([1^n],1,*)$ is terminal in this full subcategory.
\end{proof}

\begin{lem}
The functors $\varphi^*$ and $\hat\iota^*$ are symmetric monoidal.
\end{lem}
\begin{proof}
Let $(\cC\to\cC'),(\cD\to\cD')\in\PSh_{/(\dot\cL^n\to\dot\cT^n)(\bB^n\N)}$, and consider a commutative diagram as the following, in which $\cE\in\Catinfone$ and the right square is a pullback square by Lemma \ref{lem:iota*dotTnBnN} and by the fact that $\N$ is ungroup-like:
\[
\begin{tikzcd}[row sep=10pt]
\cE\ar[d]\ar[r]&*\times*\ar[r,"\simeq"']\ar[d]\ar[dr,phantom,"\lrcorner"very near start] & *\ar[d]\\
\cC'\times\cD'\ar[r]&\dot\cT^n(\bB^n\N)\times \dot\cT^n(\bB^n\N)\ar[r]&\dot\cT^n(\bB^n\N).
\end{tikzcd}
\]
Then $\cE$ is a pullback of the left square if and only if it is a pullback of the composite square; a priori, the first pullback computes $d_0(\varphi^*(\cC\to\cC')\otimes \varphi^*(\cD\to\cD'))$, whereas the second computes $d_0(\varphi^*((\cC\to\cC')\otimes(\cD\to\cD')))$. Similarly, we may consider a commutative diagram as the following, in which the right square is a pullback square by a standard argument combining the previous pullback square on the right with the pullback square defining $\varphi$:
\[
\begin{tikzcd}[row sep=10pt]
\cE\ar[d]\ar[r]&\N\times\N\ar[r]\ar[d]\ar[dr,phantom,"\lrcorner"very near start] & \N\ar[d]\\
\cC\times\cD\ar[r]&\dot\cL^n(\bB^n\N)\times \dot\cL^n(\bB^n\N)\ar[r]&\dot\cL^n(\bB^n\N).
\end{tikzcd}
\]
Again we observe that $\cE$ is a pullback of the left square if and only if it is a pullback of the composite square; by identifying the two pullbacks we obtain an equivalence $d_1(\varphi^*(\cC\to\cC')\otimes \varphi^*(\cD\to\cD'))\simeq d_1(\varphi^*((\cC\to\cC')\otimes(\cD\to\cD')))$. This proves that $\varphi^*$ is symmetric monoidal.
Similarly, $d_0(\hat\iota^*)$ is symmetric monoidal because it agrees with $d_0\colon\PSh_{/(\N\to*)}\to\Catinfone$, and the equivalence $*\simeq *\times_\N(\N\times\N)$ ensures that $d_1(\hat\iota^*)$ is symmetric monoidal.
\end{proof}
\begin{cor}
\label{cor:commsquareCatinfnN}
We have a commutative square in $\CAlg(\Catinfone^{\fU_2})$:
\[
\begin{tikzcd}[row sep=10pt]
\Catinfn^\N\ar[r,"\iota^*"]\ar[d,"\varphi^*\dot\cL\dot\cT_\N"']&\Catinfn\ar[d,"(\dot\cL^n\to\dot\cT^n)"]\\
\PSh_{/(\N\to*)}\ar[r,"\hat\iota^*"]&\PSh.
\end{tikzcd}
\]
\end{cor}
\begin{proof}
Both composite functors $\hat\iota^*\varphi^*\dot\cL\dot\cT_\N$ and $(\dot\cL^n\to\dot\cT^n)\iota^*$ are symmetric monoidal, and we have already identified them as lax symmetric monoidal functors.
\end{proof}
\begin{nota}
\label{nota:PShN}
We let $\cS^\N:=s^0_\N\cS\simeq\Fun(\N;\cS)\simeq\cS_{/\N}$, and refer to an object in $\cS^\N$ as an \emph{$\N$-graded space}. We let $\PSh^\N\to\Catinfone$ denote a cartesian fibration corresponding to the functor $\PSh(-;\cS^\N)\colon(\Catinfone)^\op\to\Catinfone^{\fU_2}$; the latter functor carries a lax symmetric monoidal structure given by exterior product of presheaves (see Subsection \ref{subsec:presheaves}), so that $\PSh^\N$ as well as the projection $\PSh^\N\to\Catinfone$ carry symmetric monoidal structures. We refer to an object in $\PSh(\cC;\cS^\N)$ as an \emph{$\N$-graded presheaf} over $\cC\in\Catinfone$. Equivalently, we may define $\PSh^\N$ as $\PSh_{/(\N\to*)}$.

Similarly, we let $\cS^{01}:=s^0_{01}\cS\simeq\Fun(\set{0,1},\cS)\simeq\cS^2$. The functor $(\iota^{01,\N})_*$ exhibits $\cS^{01}$ as the full subcategory of $\cS^\N$ spanned by $01$-truncated functors, i.e. functors $\N\to\cS$ that restrict to the terminal functor on $\N\setminus\set{0,1}$. We denote by $\PSh^{01}\to\Catinfone$ a cartesian fibration corresponding to $\PSh(-;\cS^{01})\colon(\Catinfone)^\op\to\Catinfone^{\fU_2}$. The adjunction $(\iota^{01,\N})^*\colon\cS^\N\rightleftarrows\cS^{01}\colon(\iota^{01,\N})_*$ gives rise to an analogous adjunction $(\iota^{01,\N})^*\colon\PSh^\N\rightleftarrows\PSh^{01}\colon(\iota^{01,\N})_*$, with $(\iota^{01,\N})_*$ exhibiting $\PSh^{01}\to\Catinfone$ as the full sub-cartesian fibration of $\PSh^\N\to\Catinfone$ spanned by objects $(\cC,A)$ with $A$ taking values in $01$-truncated $\N$-graded spaces.

Moreover the functor $(\iota^{01,\N})^*$ endows $\PSh^{01}$ with a symmetric monoidal structure for which it is a symmetric monoidal functor, and such that the projection $\PSh^{01}\to\Catinfone$ is symmetric monoidal as well.
\end{nota}
\begin{prop}
There is a pullback square in $\CAlg(\Catinfone^{\fU_2})$
\[
\begin{tikzcd}[row sep=10pt]
\Catinfn^{01}\ar[r,"\iota^*"]\ar[d,"\varphi^*\dot\cL\dot\cT_\N"']\ar[dr,phantom,"\lrcorner"very near start]&\Catinfn\ar[d,"(\dot\cL^n\to\dot\cT^n)"]\\
\PSh^{01}\ar[r,"\hat\iota^*"]&\PSh.
\end{tikzcd}
\]
\end{prop}
\begin{proof}
Recall the symmetric monoidal equivalence $\PSh^\N\simeq\PSh_{/(\N\to*)}$. Since the horizontal functors $\iota^*$ and $\hat\iota^*$ in the square from Corollary \ref{cor:commsquareCatinfnN} factor through the localisations $(\iota^{01,\N})^*$, we obtain a localised commutative square in $\CAlg(\Catinfone^{\fU_2})$ as in the statement.
Now observe that we have an equivalence $\PSh^{01}\simeq\PSh\times_{\Catinfone}\PSh$, by taking the parts of a presheaf of grading $0$ and $1$: under this equivalence, $\hat\iota^*$ corresponds to taking the part of a presheaf of degree 0. By Proposition \ref{prop:Catinfn01pullback} we then obtain that the square in the statement is a pullback square.
\end{proof}
Informally, the functor $\varphi^*\dot\cL\dot\cT_\N\colon\Catinfn^{01}\to\PSh^{01}$ sends $\cC\in\Catinfn^{01}$ to the triple $(\dot\cT^n(\iota^*(\cC)),\dot\lambda^n(\iota^*(\cC)),\dot\kappa^n(\cC))$, in which $\dot\lambda^n(\iota^*(\cC)),\dot\kappa^n(\cC)\in\PSh(\dot\cT^n(\iota^*(\cC)))$ account for gradings 0 and 1, respectively, of an $\N$-graded $01$-truncated presheaf over $\dot\cT^n(\iota^*(\cC))$.
\begin{cor}
\label{cor:CAlgCatinfn01pullback}
There is a pullback square in $\Catinfone^{\fU_2}$
\[
\begin{tikzcd}[row sep=10pt]
\CAlg(\Catinfn^{01})\ar[r,"\iota^*"]\ar[d,"\varphi^*\dot\cL\dot\cT_\N"']\ar[dr,phantom,"\lrcorner"very near start]&\CAlg(\Catinfn)\ar[d,"(\dot\cL^n\to\dot\cT^n)"]\\
\CAlg(\PSh^{01})\ar[r,"\hat\iota^*"]&\CAlg(\PSh).
\end{tikzcd}
\]
\end{cor}
\begin{rem}
\label{rem:dotkfreemodule}
Let $\cD=\cC[\ul X,\ul{\del X}]^\sm$ be an $\N$-admissible extension of symmetric monoidal $\infn$-categories as in Definition \ref{defn:Nadmissible}. Then
Corollary \ref{cor:CAlgCatinfn01pullback} allows us to compute $T_{01}(\cD)\in\CAlg(\Catinfn^{01})$ by computing its three coordinates in $\CAlg(\Catinfone)$, $\CAlg(\PSh)$ and $\CAlg(\PSh^{01})$.
The first is $\iota^*(\cD)\simeq\cC$ by Corollary \ref{cor:nsimeqinvariance}. The second is therefore $(\dot\cL^n(\cC)\to\dot\cT^n(\cC))\simeq(\dot\cT^n(\cC),\dot\lambda^n(\cC))\in\CAlg(\PSh)$.

Consider now $\PSh(\dot\cT^n(\cC),\cS^\N)$ as a symmetric monoidal category; by identifying $\dot\lambda^n(\cC)$ with $\iota_!(\dot\lambda^n(\cC))$, we may regard $\dot\lambda^n(\cC)\in\CAlg(\PSh(\dot\cT^n(\cC);\cS^\N))$. We can then define $\dot\lambda^n(\cC)[\ul X,\ul{\del X}]\in\CAlg(\PSh(\dot\cT^n(\cC),\cS^\N))$ as the free commutative $\dot\lambda^n(\cC)$-algebra generated by $\ul X$, which we put in grading 1, relative to $\ul{\del X}$. Then the third coordinate of $\cD$ is $(\dot\cT^n(\cC),(\iota^{01,\N})^*(\dot\lambda^n(\cC)[\ul X,\ul{\del X}]))\in\CAlg(\PSh^{01})$.

The part of $(\iota^{01,\N})^*(\dot\lambda^n(\cC)[\ul X,\ul{\del X}])$ of grading $0$ is $\dot\lambda^n(\cC)$; the part of grading 1 is $\dot\kappa^n(\cD)$, which agrees with the free $\dot\lambda^n(\cC)$-module in $\PSh(\dot\cT^n(\cC);\cS)$ generated by $\ul X$ relative to $\ul{\del X}$. In the basic case in which $\cD=\cC[X]^\sm$ we can explicitly identify:
\[
\dot\kappa^n(\cD)\simeq \Big(X\to\Fun(\mO_{n-1},\cC)^\simeq\to\dot\cT^n(\cC)\Big)_!(*)\otimes_{\Day}\dot\lambda^n(\cC)\in\PSh(\dot\cT^n(\cC)).
\]
In the general case, using Notation \ref{nota:multiextension}, we may iteratively identify $\cC_i$ as the pushout in $\CAlg(\Catinfn^\N)$ of the span $\cC_{i-1}\leftarrow\cC_{i-1}[\partial X_i]\to\cC_{i-1}[X_i]$; correspondingly, we may compute $\dot\kappa(\cD)$ as an iterated pushout in $\mathrm{Mod}_{\dot\lambda^n(\cC)}(\PSh(\dot\cT^n(\cC),\cS))$.
\end{rem}

\subsection{Mapping spaces in \texorpdfstring{$\dot{\mathcal{T}}^2(\cC)$}{dotcT2(cC)}}
We conclude the section with an explicit description of $\dot\cT^2(\cC)(\by,\bx)$ for $\cC\in\Catinftwo$ and for objects $\bx=([1^2],1,x),\by=([1^2],1,y)\in\dot\cT^2(\cC)$, starting from Corollary \ref{cor:LDDhintmappingspaces} for $\gamma=\Gamma^2(\cC)_0$. We expand $x,y\in\Fun(\mO_1,\cC)^\simeq$ as pairs of parallel 1-morphisms $x_u,x_d\colon x_0\to x_1$ and $y_u,y_d\colon y_0\to y_1$ in $\cC$.

The top row of the second pullback square in the proof of Lemma \ref{lem:GammacC01Segal}, which is an equivalence, together with Corollary \ref{cor:dotcTCCop}, allows us to compute $\Gamma^2(\cC)_0([3],010)$ as the iterated fibre product
$\Tw(\cC_{[1]})\times_{\cC^\simeq}\Gamma^2(\cC)_0([1],1)\times_{\cC^\simeq}\Tw(\cC_{[1]})$. It follows that the pullback of the left cospan in the W-shaped diagram from Corollary \ref{cor:LDDhintmappingspaces} is equivalent to the fibre at $(x_0,x_1)$ of the functor
\[
(\Tw(\cC_{d_0}),\Tw(\cC_{d_1}))\colon\Tw(\cC_{[1]})\times\Tw(\cC_{[1]})\to\Tw(\cC_{[0]})\times\Tw(\cC_{[0]})\simeq\cC^\simeq\times\cC^\simeq.
\]
We abbreviate this fibre by $\Tw(\cC_{[1]})_{d_0x_0}\times\Tw(\cC_{[1]})_{d_1x_1}$. We can now compute $\Xi(\Gamma^2(\cC)_0,\by,\bx)$ as the following pullback
\[
\begin{tikzcd}[row sep=10pt]
\Xi(\Gamma^2(\cC)_0,\by,\bx)\ar[r]\ar[d]\ar[dr,phantom,"\lrcorner"very near start]& \Tw(\cC_{[1]})_{d_0x_0}\times\Tw(\cC_{[1]})_{d_1x_1}\ar[d]\\
\big(\Gamma^2(\cC)_0([1],1)\big)_{y/}\ar[r]&\Gamma^2(\cC)_0([1],1).
\end{tikzcd}
\]
The first pullback square in the proof of Lemma \ref{lem:GammacC01Segal} computes $\Gamma^2(\cC)_0([1],1)$ as
\[
\begin{tikzcd}[column sep=50pt,row sep=10pt]
\Gamma^2(\cC)_0([1],1)\ar[r]\ar[d]\ar[dr,phantom,"\lrcorner"very near start] &\cC^\simeq\times\cC^\simeq\ar[d,"\text{diagonal}\times\text{diagonal}"]\\
\cC_{[1]}\times\cC_{[1]}^\op\ar[r,"{(d_1,d_0,d_1,d_0)}"]&(\cC^\simeq)^2\times(\cC^\simeq)^2,
\end{tikzcd}
\]
allowing us to identify $\Gamma^2(\cC)_0([1],1)_{y/}$ and $\cC(y_0,y_1)_{y_u/}\times(\cC(y_0,y_1)_{/y_d})^\op$.

The functor $\Tw(\cC_{[1]})_{d_0x_0}\times\Tw(\cC_{[1]})_{d_1x_1}\to\Gamma^2(\cC)_0([1],1)$ occurring in the pullback computing $\Xi(\Gamma^2(\cC)_0,\by,\bx)$ has top right coordinate given by the functor 
\[
\Tw(\cC_{d_1})\times \Tw(\cC_{d_0})\colon\Tw(\cC_{[1]})_{d_0x_0}\times\Tw(\cC_{[1]})_{d_1x_1}\to\cC^\simeq\times\cC^\simeq,
\]
and it has bottom left coordinate given by the following composite, where we abbreviate by $(\cC_{[1]})_{d_jx_j}$ the fibre at $x_j$ of $d_j\colon\cC_{[1]}\to\cC_{[0]}\simeq\cC^\simeq$:
\[
\begin{tikzcd}[row sep=15pt]
\Tw(\cC_{[1]})_{d_0x_0}\times\Tw(\cC_{[1]})_{d_1x_1}\ar[d]\\
\big((\cC_{[1]})_{d_0x_0}\times(\cC_{[1]}^\op)_{d_0x_0}\big)\times\big((\cC_{[1]})_{d_1x_1}\times(\cC_{[1]}^\op)_{d_1x_1}\big)\ar[d,"\simeq"]\\
\big((\cC_{[1]})_{d_0x_0}\times (\cC_{[1]})_{d_1x_1}\big)\times\big((\cC_{[1]}^\op)_{d_0x_0}\times(\cC_{[1]}^\op)_{d_1x_1}\big)\ar[d,"{(-\circhor x_u\circhor-)\times(-\circhor x_d\circhor-)^\op}"]\\
\cC_{[1]}\times\cC_{[1]}^\op.
\end{tikzcd}
\]
Combining all of the previous remarks, we obtain the following simplified formula for $\Xi(\Gamma^2(\cC)_0,\by,\bx)$, which we state as a proposition for future reference.
\begin{prop}
\label{prop:mappingspacesdotcT2}
Let $\cC\in\Catinftwo$ and let $\bx,\by\in\Fun(\mO_1,\cC)^\simeq$; using the notation above, we have a pullback square in $\Catinfone$ whose horizontal arrows are left fibrations:
\[
\begin{tikzcd}[column sep=0pt, row sep=13pt]
\Xi(\Gamma^2(\cC)_0,\by,\bx)\ar[r]\ar[dd]\ar[ddr,,phantom,"\lrcorner"very near start] &\Tw(\cC(y_0,x_0)\times\cC(x_1,y_1))\ar[d]\\
&(\cC(y_0,x_0)\times\cC(x_1,y_1))\times(\cC(y_0,x_0)\times\cC(x_1,y_1))^\op\ar[d,"{(-\circhor x_u\circhor-)\times(-\circhor x_d\circhor-)^\op}"]\\
\cC(y_0,y_1)_{y_u/}\times(\cC(y_0,y_1)_{/y_d})^\op
\ar[r,"d_0\times d_1"] & \cC(y_0,y_1)\times \cC(y_0,y_1)^\op.
\end{tikzcd}
\]
\end{prop}

\section{The universal property of \texorpdfstring{$\bGr^\twosimeq$}{bGr2simeq}}
\label{sec:proofthmAbis}
In this section we prove Theorem \ref{thm:Abis}; meanwhile, we start the proof of \ref{thm:A} under the assumption of the validity of Theorem \ref{thm:Abis}: some of the steps of the proofs of the two theorems are parallel, and we treat them simultaneously.
\subsection{Coloured graphs and coloured graph collapses}
\label{subsec:colouredreduction}
We start by introducing $S$-coloured versions of $\bGr^\twosimeq$ and of $\bGr$, for a finite set $S$.
\begin{defn}
Recall Definition \ref{defn:gaf} and Notation \ref{nota:gaf}.
We let $\VE\colon\bGr^\twosimeq\to\bB(\Fin^\simeq)$ denote the symmetric monoidal $\infone$-functor associating with a gaf $G$ the disjoint union $V\sqcup E$ of its sets of inner vertices and of edges. Recall also Definition \ref{defn:morgaf}. We let $\CE\colon\bGr\to\bB^2(\Fin^\simeq)$ denote the symmetric monoidal $\inftwo$-functor associating with a morphism of gaf's $f\colon G\to G'$ the subset of $E$ comprising edges that are collapsed along $f$.

For $S\in\Fin$ we identify $(\Fin_{/S})^\simeq$ as the $S$-fold coproduct of $\Fin^\simeq$ in $\CAlg(\cS)$, so that the forgetful map $(\Fin_{/S})^\simeq\to\Fin^\simeq$ agrees with the fold map. We define $\bGr^\twosimeq(S)\in\CAlg(\Catinfone)$ and $\bGr(S)\in\CAlg(\Catinftwo)$ by the following pullbacks:
\[
\begin{tikzcd}[row sep=10pt]
\bGr^\twosimeq(S)\ar[r,"\VE(S)"]\ar[d]\ar[dr,phantom,"\lrcorner"very near start] &\bB((\Fin_{/S})^\simeq)\ar[d,"\text{fold}"]\\
\bGr^\twosimeq\ar[r,"\VE"] &\bB(\Fin^\simeq);
\end{tikzcd}
\hspace{1cm}
\begin{tikzcd}[row sep=10pt]
\bGr(S)\ar[r,"\CE(S)"]\ar[d]\ar[dr,phantom,"\lrcorner"very near start] &\bB^2((\Fin_{/S})^\simeq)\ar[d,"\text{fold}"]\\
\bGr^\twosimeq\ar[r,"\CE"] &\bB^2(\Fin^\simeq).
\end{tikzcd}
\]
\end{defn}
\begin{rem}
The fact that the functors $\VE$ and $\CE$ are well defined on higher morphisms and carry a symmetric monoidal structure can in principle be checked ``manually'' by using that $\Fin^\simeq$ is an aspherical space, so that $\bB\Fin^\simeq$ and $\bB^2\Fin^\simeq$ have discrete spaces of 2-morphisms and 3-morphisms, respectively; hence all coheren\-ces of dimension at least 2 and 3, respectively, are properties and not structure. 
Alternatively, we could directly define $\bGr^\twosimeq(S)$ and $\bGr(S)$ by replacing Definition \ref{defn:gaf}, respectively Definition \ref{defn:morgaf}, with their counterparts in which we adjoin the datum of a suitable map of finite sets to $S$, from $V\sqcup E$ or from $E\setminus f^{-1}(E')$, respectively, and by then repeating verbatim the constructions from Section \ref{sec:graphcobsets}.
\end{rem}
\begin{rem}
We remark that $\bGr^\twosimeq(S)$ is a plain (2,1)-category, and $\bGr(S)$ is a plain 2-category: that is, if we consider both $\bGr^\twosimeq(S),\bGr(S)$ as objects in $\Catinftwo$, then all spaces of 2-morphisms are discrete. We also remark that $\bGr^\twosimeq(S)$ is in general \emph{not} equivalent to $\bGr(S)^\twosimeq$: the latter is instead equivalent to $\bGr^\twosimeq$.
\end{rem}
The assignments $S\mapsto\bGr^\twosimeq(S),\bGr(S)$ give rise for $n=1,2$ to functors $\Fin\to\CAlg(\Catinfn)$. Moreover, we have symmetric monoidal functors
\[
\bB^n(\#_{/S})\colon\bB^n((\Fin_{/S})^\simeq)\to\bB^n\N^S
\]
induced by taking cardinalities of finite sets, fibrewise over $S$; these give $\N^S$-gradings on $\bGr^\twosimeq(S)$ and $\bGr(S)$, respectively.
\begin{nota}
Recall Notation \ref{nota:cUcV}, Example \ref{ex:cUcVasextensions} and Definition \ref{defn:Sextension}. For a finite set $S$ we let $\cU(S):=\Fin[S|*\sqcup\bB C_2]^\sm$, i.e. the $S$-fold coproduct of $\cU$ in $\CAlg(\Catinfone)_{\Fin/}$.
We similarly let $\cV(S):=\cU[S| *\ ;\ \bB C_2,*]^\sm$, i.e. the $S$-fold coproduct of $\cV$ in $\CAlg(\Catinftwo)_{\cU/}$. We let $\fF(S)\colon\cU(S)\to\bGr^\twosimeq(S)$ and $\fG(S)\colon\cV(S)\to\bGr(S)$ denote the symmetric monoidal $\infn$-functors induced by the following commutative squares in $\CAlg(\Catinfone)$ and $\CAlg(\Catinftwo)$:
\[
\begin{tikzcd}[row sep=15pt, column sep=40pt]
\coprod^{\CAlg_{\Fin/}}_S\cU\ar[d,"{\fF\circ\text{fold}}"]\ar[r,"\coprod_S(\VE\circ\fF)"] &\coprod^{\CAlg}_S\bB\Fin^\simeq\ar[d,"\text{fold}"]\\
\bGr^\twosimeq\ar[r,"\VE"] &\bB(\Fin^\simeq);
\end{tikzcd}
\hspace{.5cm}
\begin{tikzcd}[row sep=15pt, column sep=40pt]
\coprod^{\CAlg_{\cU/}}_S\cV\ar[d,"{\fG\circ\text{fold}}"]\ar[r,"\coprod_S(\CE\circ\fG)"] &\coprod^{\CAlg}_S\bB^2\Fin^\simeq\ar[d,"\text{fold}"]\\
\bGr\ar[r,"\CE"] &\bB^2(\Fin^\simeq).
\end{tikzcd}
\]

\end{nota}
By Corollary \ref{cor:nsimeqinvariance}, for $S\in\Fin$ we have an equivalence $\fF(S)^\simeq\colon\cU(S)^\simeq\xrightarrow{\simeq}\bGr^\twosimeq(S)^\simeq\simeq\Fin^\simeq$ of spaces, and assumming Theorem \ref{thm:Abis}, an equivalence of $\infone$-catego\-ries $\fG(S)^\twosimeq\colon\cV(S)^\twosimeq\xrightarrow{\simeq}\bGr(S)^\twosimeq\simeq\bGr^\twosimeq$.

For $S\in\Fin$ we have the following commutative diagrams in $\CAlg(\Catinfone)$ and $\CAlg(\Catinftwo)$, in which the bottom and top composites recover the $\N$-admissible gradings on $\cU,\cV$ and the $\N^S$-gradings on $\cU(S),\cV(S)$ from Definition \ref{defn:fdS}:
\[
\begin{tikzcd}[column sep=50pt,row sep=12pt]
\cU(S)\ar[d,"\text{fold}"]\ar[r,"\fF(S)"]&\bGr^\twosimeq(S)\ar[d]\ar[r,"\VE(S)"]\ar[dr,phantom,"\lrcorner"very near start]&\bB((\Fin_{/S})^\simeq)\ar[d,"\text{fold}"]\ar[r,"\bB(\#_{/S})"]&\bB\N^S\ar[d,"\text{fold}"]\\
\cU\ar[r,"\fF"] &\bGr^\twosimeq\ar[r,"\VE"] &\bB\Fin^\simeq\ar[r,"\bB\#"]&\bB\N;
\end{tikzcd}
\]
\[
\begin{tikzcd}[column sep=50pt,row sep=12pt]
\cV(S)\ar[d,"\text{fold}"]\ar[r,"\fG(S)"]&\bGr(S)\ar[d]\ar[r,"\CE(S)"]\ar[dr,phantom,"\lrcorner"very near start]&\bB^2((\Fin_{/S})^\simeq)\ar[d,"\text{fold}"]\ar[r,"\bB^2(\#_{/S})"]&\bB^2\N^S\ar[d,"\text{fold}"]\\
\cV\ar[r,"\fG"] &\bGr\ar[r,"\CE"] &\bB^2\Fin^\simeq\ar[r,"\bB^2\#"]&\bB^2\N.
\end{tikzcd}
\]
\begin{nota}
Recall Notation \ref{nota:1S}. For finite sets $T\subseteq S$ we consider $\N^T$ as a submonoid of $\N^S$ and regard $1^T$ also as an element of $\N^S$.
\end{nota}

\begin{lem}
\label{lem:1Sreduction}
The following implications hold.
\begin{enumerate}
\item Assume that for all $S\in\Fin$ the $\infone$-functor $\fF(S)$ induces an equivalence on spaces of 1-morphisms of $\N^S$-grading $1^S$. Then  Theorem \ref{thm:Abis} holds.
\item Assume Theorem \ref{thm:Abis}, and assume that for all $S\in\Fin$ the $\inftwo$-functor $\fG(S)$ induces an equivalence on spaces of 2-morphisms of $\N^S$-grading $1^S$. Then  Theorem \ref{thm:A} holds.
\end{enumerate}
\end{lem}
\begin{proof}
For (1), by Example \ref{ex:cUcVcores}, it suffices to check that $\fF$ induces an equivalence on spaces of 1-morphisms. Since $\fF$ is an $\N$-graded functor, we can check the last statement in each grading $r\in\N$ separately, i.e. prove that $\fF$ induces an equivalence of spaces $
\Fun(\Theta_1,\cU)^\simeq_r\xrightarrow{\simeq}\Fun(\Theta_1,\bGr^\twosimeq)^\simeq_r$.
Let $S$ be a set of cardinality $r\ge0$; we may write the following commutative diagram of spaces, in which the top arrow is an equivalence by assumption, and the right arrow by direct inspection:
\[
\begin{tikzcd}[row sep=10pt]
\Fun(\Theta_1,\cU(S))^\simeq_{1^S}/\fS_S\ar[r,"\simeq"]\ar[d] &\Fun(\Theta_1,\bGr^\twosimeq(S))^\simeq_{1^S}/\fS_S\ar[d,"\simeq"]\\
\Fun(\Theta_1,\cU)^\simeq_r\ar[r]&\Fun(\Theta_1,\bGr^\twosimeq)^\simeq_r.
\end{tikzcd}
\]
It follows that the left arrow admits a retraction; by Proposition \ref{prop:retraction} it also admits a section, so that it has to be an equivalence of spaces. Hence also the bottom arrow is an equivalence.
The argument for (2) is analogous. By the assumption that Theorem \ref{thm:Abis} holds and by Example \ref{ex:cUcVcores}, it suffices to check, for each finite set $S$ of cardinality $r\ge0$, that the bottom arrow in the following diagram is an equivalence:
\[
\begin{tikzcd}[row sep=10pt]
\Fun(\Theta_2,\cV(S))^\simeq_{1^S}/\fS_S\ar[r,"\simeq"]\ar[d] &\Fun(\Theta_2,\bGr(S))^\simeq_{1^S}/\fS_S\ar[d,"\simeq"]\\
\Fun(\Theta_2,\cV)^\simeq_r\ar[r]&\Fun(\Theta_2,\bGr)^\simeq_r.
\end{tikzcd}
\]
We invoke Proposition \ref{prop:retraction} and conclude as before.
\end{proof}
\begin{nota}
For a finite set $S$ we denote by $S_+$ the finite set $S\sqcup\set{\bullet}$.
\end{nota}

\begin{defn}
For a finite set $S$ we denote by $\bGr^\twosimeq(S)^+$ the pushout of the span $\cU(S_+)\ot\cU(S)\xrightarrow{\fF(S)}\bGr^\twosimeq(S)$ in $\CAlg(\Catinfone)$, in which the left pointing arrow is induced by the inclusion $S\hto S_+$. Note that $\bGr^\twosimeq(S)^+$ is an extension of the form $\bGr^\twosimeq(S)[\bB C_2\sqcup *]^\sm$ as in Definition \ref{defn:extension}.
Similarly, we denote by $\bGr(S)^+$ the pushout of $\cV(S_+)\ot\cV(S)\xrightarrow{\fG(S)}\bGr(S)$ in $\CAlg(\Catinftwo)$; it is an extension of the form $\bGr(S)[S| *\ ;\ \bB C_2,*]^\sm$ as in Notation \ref{nota:multiextension}.

We denote by $\fj^\twosimeq(S)\colon\bGr^\twosimeq(S)^+\to\bGr^\twosimeq(S_+)$ and $\fj(S)\colon\bGr(S)^+\to\bGr(S_+)$ the symmetric monoidal functors induced on pushouts by the following commutative squares in $\CAlg(\Catinfone)$ and $\CAlg(\Catinftwo)$:
\[
\begin{tikzcd}[column sep=50pt,row sep=12pt]
\cU(S)\ar[r,"S\hto S_+"]\ar[d,"\fF(S)"] &\cU(S_+)\ar[d,"\fF(S_+)"]\\
\bGr^\twosimeq(S)\ar[r,"S\hto S_+"] &\bGr^\twosimeq(S_+);
\end{tikzcd}
\hspace{1cm}
\begin{tikzcd}[column sep=50pt,row sep=12pt]
\cV(S)\ar[r,"S\hto S_+"]\ar[d,"\fG(S)"] &\cV(S_+)\ar[d,"\fG(S_+)"]\\
\bGr(S)\ar[r,"S\hto S_+"] &\bGr(S_+).
\end{tikzcd}
\]
We restrict the $\N^{S_+}$-gradings on $\bGr^\twosimeq(S_+)$ and $\bGr(S_+)$ along $\fj^\twosimeq(S)$ and $\fj(S)$.
\end{defn}

\begin{lem}
\label{lem:plusreduction}
The following implications hold.
\begin{enumerate}
\item Let $r\ge0$ and assume that $\fj^\twosimeq(T)$ induces an equivalence
on spaces of 1-morphisms of $\N^{T_+}$-grading $1^{T_+}$
for all $T\in\Fin$ with $\#T<r$. Then for all $T\subseteq S\in\Fin$ with $\# S\le r$, $\fF(S)$ induces an equivalence on spaces of 1-morphisms of $\N^S$-grading $1^T$ for all $T\subseteq S$; in particular the hypothesis of Lemma \ref{lem:1Sreduction}(1) holds for $S$.
\item Assume Theorem \ref{thm:Abis}.  Let $r\ge0$ and assume that $\fj(T)$ induces an equivalence
on spaces of 2-morphisms of $\N^{T_+}$-grading $1^{T_+}$
for all $T\in\Fin$ with $\#T<r$. Then for all $T\subseteq S\in\Fin$ with $\# S\le r$, $\fG(S)$ induces an equivalence on spaces of 2-morphisms of $\N^S$-grading $1^T$; in particular the hypothesis of Lemma \ref{lem:1Sreduction}(2) holds for $S$.
\end{enumerate}
\end{lem}
\begin{proof}
For (1), we proceed inductively on $\# S\le r$. For $S=\emptyset$ we have to set $T=\emptyset$ and we may identify both spaces $\Fun(\Theta_1,\cU(S))^\simeq_{1^\emptyset}$ and $\Fun(\Theta_1,\bGr^\twosimeq(S))^\simeq_{1^\emptyset}$ with $\Fun(\Theta_1,\Fin^\simeq)$ by Corollary \ref{cor:nsimeqinvariance}. Suppose that the statement holds for a certain $S$ with $\# S<r$, let $T\subseteq S_+$, and define the coideal 
\[
I:=\set{0,1}^T\times\set{0}^{S_+\setminus T}\subset M:=\N^{S_+}.
\]

If $T\neq S_+$, we write a commutative diagram as follows, in which we embed $\CAlg(\Catinfone^{\N^T})\hto\CAlg(\Catinfone^M)$ via the functor $(\N^T\hto M)_!$:
\[
\begin{tikzcd}[row sep=12pt, column sep=11pt]
\Fun(\Theta_1,\bGr^\twosimeq(T))^\simeq_{1^T}\ar[d,"\simeq"]& &\Fun(\Theta_1,\cU(T))^\simeq_{1^T}\ar[ll,"\fF(T)"',"\simeq"]\ar[r,"\simeq"]\ar[d]&\Fun(\Theta_1,T_I(\cU(T)))^\simeq_{1^T}\ar[d,"\simeq"]\\
\Fun(\Theta_1,\bGr^\twosimeq(S_+))^\simeq_{1^T}& &\Fun(\Theta_1,\cU(S_+))^\simeq_{1^T}\ar[ll,"\fF(S_+)"']\ar[r,"\simeq"]&\Fun(\Theta_1,T_I(\cU(S_+)))^\simeq_{1^T}.
\end{tikzcd}
\]
The left vertical arrow is an equivalence by inspection, the top left horizontal arrow is an equivalence by inductive hypothesis, the two right horizontal arrows are equivalences by Corollary \ref{cor:Itruncateddescription}, and the right vertical arrow is an equivalence by Corollary \ref{cor:nsimeqinvariance}:
It follows that the bottom left horizontal arrow is also an equivalence.

Assume now $T=S_+$, and
factor the map $\Fun(\Theta_1,\fF(S_+))^\simeq_{1^{S_+}}$ as a composite
\[
\begin{tikzcd}[column sep=15pt]
\Fun(\Theta_1,\cU(S_+))^\simeq_{1^{S_+}}\ar[r]&\Fun(\Theta_1,\bGr^\twosimeq(S)^+)^\simeq_{1^{S_+}}\ar[rr,"\fj^\twosimeq(S)"]&&\Fun(\Theta_1,\bGr^\twosimeq(S_+))^\simeq_{1^{S_+}},
\end{tikzcd}
\]
where the second map is an equivalence by assumption. The first map is induced by the $M$-graded symmetric monoidal $\infone$-functor $\cU(S_+)\to\bGr^\twosimeq(S)^+$ induced on pushouts by the following map of spans in $\CAlg(\Catinfone^M)$, where again we use the functor $(\N^S\hto M)_!$ to put all objects in the same category:
\[
\begin{tikzcd}[row sep=12pt,column sep=50pt]
\cU(S_+)\ar[d,equal]&\cU(S)\ar[d,equal]\ar[r,equal]\ar[l,"S\hto S_+"'] &\cU(S)\ar[d,"\fF(S)"]\\
\cU(S_+)&\cU(S)\ar[r,"\fF(S)"]\ar[l,"S\hto S_+"']&\bGr^\twosimeq(S).
\end{tikzcd}
\]
By inductive hypothesis and by Corollary \ref{cor:Itruncateddescription}, the right vertical map is a $T_I$-equivalence: in particular, we observe that for all $\bullet\in T\subseteq S_+$ the spaces of 1-morphisms of $\N^{S_+}$-grading $1^T$ in both $\cU(S)$ and $\bGr^\twosimeq(S)$  are empty. By applying the functor $(\iota^{I,M})^*\colon\CAlg(\Catinfone^M)\to\CAlg(\Catinfone^I)$, which preserves pushouts, we obtain an equivalence $(\iota^{I,M})^*(\cU(S_+))\xrightarrow{\simeq}(\iota^{I,M})^*(\bGr^\twosimeq(S)^+)$, yielding the desired equivalence on spaces of 1-morphiss of $M$-grading $1^T$ for all $T\subseteq S_+$.

The analogous argument for (2) is omitted.
\end{proof}
For $S\in\Fin$, the projections $\N^{S_+}\to\N^{\set{\bullet}}\simeq\N$ and $\N^{S_+}\to\N^S$ allow us to split an $\N^{S_+}$-grading
as a pair of an $\N$-grading and an $\N^S$-grading. The first grading allows us to consider $\fj^\twosimeq(S)$ and $\fj(S)$ as morphisms in $\CAlg(\Catinfn^\N)$ for $n=1,2$. By Lemma \ref{lem:plusreduction}, it suffices for us to focus on the maps induced by $\fj^\twosimeq(S)$ and $\fj(S)$ on spaces of $n$-morphisms of $\N$-grading 1 (and $\N^S$-grading $1^S$), for $n=1,2$.
We can neglect the $\N^S$-grading and detect all $n$-morphisms of $\N$-grading 1 after having applied the functor $(\iota^{01,\N})^*\colon\CAlg(\Catinfn^\N)\to\CAlg(\Catinfn^{01})$,
i.e. we can focus on the following morphisms in $\CAlg(\Catinfn^{01})$:
\[
\begin{split}
(\iota^{01,\N})^*(\fj^\twosimeq(S))\colon & (\iota^{01,\N})^*(\bGr^\twosimeq(S)^+)\to (\iota^{01,\N})^*(\bGr^\twosimeq(S_+));\\
(\iota^{01,\N})^*(\fj(S))\colon & (\iota^{01,\N})^*(\bGr(S)^+)\to (\iota^{01,\N})^*(\bGr(S_+)).
\end{split}
\]
For $n=1$ we shall prove that $(\iota^{01,\N})^*(\fj^\twosimeq(S))$ is an equivalence by analysing its three coordinates with respect to the pullback square from Corollary \ref{cor:CAlgCatinfn01pullback}. The top right coordinate $\iota^*(\fj^\twosimeq(S))$ is an equivalence in the category $\CAlg(\Catinfone)$: by Corollary \ref{cor:nsimeqinvariance} and by a direct inspection, respectively, we have that both the source $\iota^*(\bGr^\twosimeq(S)^+)$ and the target $\iota^*(\bGr^\twosimeq(S_+))$ can be identified with $\bGr^\twosimeq(S)$. It follows that also the bottom right coordinate is an equivalence in $\CAlg(\PSh)$. The proof of Theorem \ref{thm:Abis} is complete once we prove that the bottom left coordinate $\varphi^*\dot\cL\dot\cT_\N((\iota^{01,\N})^*(\fj^\twosimeq(S)))$ is an equivalence in $\CAlg(\PSh^{01})$ as well; this boils down to proving that the morphism $\dot\kappa^1(\fj^\twosimeq(S))\colon\dot\kappa^1(\bGr^\twosimeq(S)^+)\to\dot\kappa^1(\bGr^\twosimeq(S_+))$ is an equivalence in $\PSh(\dot\cT^1(\bGr^\twosimeq(S)))$, and this is what we shall prove in Subsection \ref{subsec:proofAbis}.

For $n=2$ we will first make a reduction: we will introduce the property of being \emph{leaf-like} for 2-morphisms of $\N$-grading 1 in $\bGr(S)^+$ and $\bGr(S_+)$; it will be apparent that $\fj(S)$ preserves leaf-like morphisms. We will then reduce the proof of Theorem \ref{thm:A} to checking that $\fj(S)$ induces equivalences on spaces of leaf-like 2-morphisms of $\N^S$-grading $1^S$. After this we will again neglect the $\N^S$-grading and proceed in a similar manner as in the case $n=1$, showing that $\dot\kappa^2(\fj(S))\colon\dot\kappa^2(\bGr(S)^+)\to\dot\kappa^2(\bGr(S_+))$ induces objectwise an equivalence on spaces of leaf-like 2-morphisms.

\subsection{Proof of Theorem \ref{thm:Abis}}
\label{subsec:proofAbis}
We fix $S\in\Fin$ and aim at proving that the morphism $\dot\kappa^1(\fj^\twosimeq(S))\colon\dot\kappa^1(\bGr^\twosimeq(S)^+)\to\dot\kappa^1(\bGr^\twosimeq(S_+))$ in $\PSh(\dot\cT^1(\bGr^\twosimeq(S)))$ is an equivalence. 
We start by analysing the target presheaf. Let $(B,A)\in\dot\cT^1(\bGr^\twosimeq(S))\simeq \bGr^\twosimeq(S)\times \bGr^\twosimeq(S)^\op$, i.e. $A,B$ are finite sets. We have $\dot\kappa^1(\bGr^\twosimeq(S_+))(B,A)\simeq\bGr^\twosimeq(S_+)(B,A)_1$; the latter space can be identified with the moduli space of graph cobordisms $G$ from $B$ to $A$ endowed with a map of finite sets $\VE(G)\to S_+$ whose fibre over $\bullet$ is a singleton, i.e. the pullback of the cospan of spaces 
\[
\begin{tikzcd}
G^\twosimeq(B,A)\ar[r,"\VE"]&\Fin^\simeq&(\Fin_{/S})^\simeq\ar[l,"d_0(-)_+"'].
\end{tikzcd}
\]

\begin{nota}
We decompose $\bGr^\twosimeq(S_+)(B,A)_1$ as a disjoint union of spaces $X(B,A)_\ft\sqcup X(B,A)_\fe$, according to whether the unique element in the fibre over $\bullet$ of the map $\VE(G)\to S_+$ is an inner vertex or an edge.
\end{nota}

We next pass to the analysis of $\dot\kappa^1(\bGr^\twosimeq(S)^+)$.
\begin{nota}
Recall Remark \ref{rem:dotkfreemodule}. We denote by $\theta\in\PSh(\dot\cT^1(\bGr^\twosimeq(S)))$ the left Kan extension of the terminal presheaf $*\in\PSh(*\sqcup\bB C_2)$ along the composite
\[
\bB C_2\sqcup *\xrightarrow{F^\vee}\Fin^\simeq\times\Fin^\simeq\simeq(\bGr^\twosimeq(S)\times \bGr^\twosimeq(S)^\op)^\simeq\hto \bGr^\twosimeq(S)\times \bGr^\twosimeq(S)^\op,
\]
where the map $F^\vee$ is adjoint to the map $F$ from Example \ref{ex:cUcVasextensions}.
We similarly let $\theta_\ft,\theta_\fe\in\PSh(\dot\cT^1(\bGr^\twosimeq(S)))$ denote the left Kan extensions of $*\in\PSh(*)\simeq\cS$ and $*\in\PSh(\bB C_2)$ along the restrictions of the above composite to $*$ and $\bB C_2$, respectively.
\end{nota}
We abbreviate $\cC=\bGr^\twosimeq(S)$
and $\dot\lambda=\dot\lambda^1(\cC)\in\PSh(\dot\cT^1(\cC))\simeq\PSh(\cC\times\cC^\op)$ for the rest of the subsection.
As observed in Remark \ref{rem:dotkfreemodule}, we can identify $\dot\kappa^1(\bGr^\twosimeq(S)^+)$ with the tensor product $\theta\otimes_\Day\dot\lambda$ in $\PSh(\cC\times\cC^\op)$.
Moreover $\theta$ splits as a coproduct $\theta_\ft\sqcup\theta_\fe$ in $\PSh(\cC\times\cC^\op)$, so that also $\theta\otimes_\Day\dot\lambda$ also splits as $\theta_\ft\otimes_\Day\dot\lambda\sqcup\theta_\fe\otimes_\Day\dot\lambda$.

We start by analysing the second presheaf $\theta_\fe\otimes_\Day\dot\lambda$.
The restriction of $F^\vee$ to $\bB C_2$ is the product of the constant map at $\ul0\in\Fin^\simeq\simeq\cC^\simeq$ and the inclusion $j\colon\bB C_2\hto\Fin^\simeq\simeq(\cC^\op)\simeq$, so that $\theta_\fe\otimes_\Day\dot\lambda$ can be computed as the left Kan extension of $ *\boxtimes\dot\lambda\in\PSh(\bB C_2\times\cC\times\cC^\op;\cS^{\N^S})$ along the composite functor
\[
\begin{tikzcd}[column sep=30pt]
\bB C_2\times\cC\times\cC^\op\simeq*\times\cC\times\bB C_2\times\cC^\op\ar[r,"{\ul0,\Id,j,\Id}"]&\cC\times\cC\times\cC^\op\times\cC^\op\ar[r,"{\otimes,\otimes}"]&\cC\times\cC^\op.
\end{tikzcd}
\]
Since $\ul0$ is the monoidal unit in $\cC$, the previous composite functor agrees with
\[
\begin{tikzcd}[column sep=30pt]
\bB C_2\times\cC\times\cC^\op\simeq\cC\times\bB C_2\times\cC^\op\ar[r,"{\Id,j,\Id}"]&\cC\times\cC^\op\times\cC^\op\ar[r,"{\Id,\otimes}"]&\cC\times\cC^\op.
\end{tikzcd}
\]
Note that the last composition features the product of the identity of $\cC$ (appearing as middle factor in the source) and the functor $j\otimes\Id_{\cC^\op}\colon\bB C_2\times \cC^\op\to\cC^\op$, so that left Kan extension along the last composite fits as top row in the following commutative square in $\Catinfone^{\fU_2}$:
\[
\begin{tikzcd}[column sep=70pt,row sep=12pt]
\PSh(\bB C_2\times\cC\times\cC^\op)\ar[r]\ar[d,"\simeq"]&\PSh(\cC\times\cC^\op)\ar[d,"\simeq"]\\
\PSh(\cC;\PSh(\bB C_2\times\cC^\op))\ar[r,"\PSh(\cC;(j\otimes\Id_{\cC^\op})_!)"]&\PSh(\cC;\PSh(\cC^\op)).
\end{tikzcd}
\]
By Corollary \ref{cor:dotcTCCop}, the presheaf $*\boxtimes\dot\lambda$ in the top left category corresponds to the functor $*\boxtimes\yo_{\cC^\op}(-)\colon\cC^\op\to\PSh(\bB C_2\times\cC^\op)$ in the bottom left category.
We may moreover identify $*\simeq\colim_{\bB C_2}\yo_{\bB C_2}\in\PSh(\bB C_2)$, and write a chain of equivalences in the bottom right category $\PSh(\cC;\PSh(\cC^\op))\simeq\Fun(\cC^\op,\PSh(\cC^\op))$:
\[
\begin{split}
(j\otimes\Id_{\cC^\op})_!\circ\big(*\boxtimes\yo_{\cC^\op}(-)\big)&\simeq j_!(*)\otimes_\Day\yo_{\cC^\op}(-)\\
&\simeq j_!(\colim_{\bB C_2}\yo_{\bB C_2})\otimes_\Day\yo_{\cC^\op}(-)\\
&\simeq \colim_{\bB C_2}\yo_{\cC^\op}(j\otimes-).
\end{split}
\]

Hence for $B,A\in\Fin$ we may identify $(\theta_\fe\otimes_\Day\dot\lambda)(B,A)\simeq\bGr^\twosimeq(S)(B\sqcup\ul2,A)/C_2$, where $C_2$ acts trivially on $B,A$ and by swap on $\ul2$. 
An analogous, but slightly simpler computation allows us to identify
\[
(\theta_\ft\otimes_\Day\dot\lambda)(B,A)\simeq\bGr^\twosimeq(S)(B,A\sqcup\ul1).
\]
The map of presheaves $\dot\kappa^1(\fj^\twosimeq(S))\colon\dot\kappa^1(\bGr^\twosimeq(S)^+)\to\dot\kappa^1(\bGr^\twosimeq(S_+))$ gives, after evaluating at $(B,A)$, a map of spaces $(\theta\otimes_\Day\dot\lambda)(B,A)\to\bGr^\twosimeq(S_+)(B,A)_1$ which restricts to maps
$(\theta_\ft\otimes_\Day\dot\lambda)(B,A)\to X(B,A)_\ft$ and $(\theta_\fe\otimes_\Day\dot\lambda)(B,A)\to X(B,A)_\fe$.

Under the above identifications, the second map can be identified with the map induced on the $C_2$-quotient by the precomposition map
\[
-\circ(\Id_B\sqcup\fe_\bullet)\colon\bGr^\twosimeq(S)(B\sqcup\ul2,A)\to X(B,A)_\fe\subset\bGr^\twosimeq(S_+)(B,A)_1,
\]
where $\fe_\bullet\in\bGr^\twosimeq(S_+)(\ul0,\ul2)_1$ is the morphism $\fe$ with the unique edge labeled $\bullet$; and where we also identify $\bGr^\twosimeq(S)(B\sqcup\ul2,A)\simeq \bGr^\twosimeq(S_+)(B\sqcup\ul2,A)_0$.
A direct inspection shows that $\bGr^\twosimeq(S)(B\sqcup\ul2,A)/C_2\xrightarrow{\simeq} X(B,A)_\fe$ is indeed an equivalence of spaces. By a similar, slightly simpler argument, the first map $(\theta_\ft\otimes_\Day\dot\lambda)(B,A)\to X(B,A)_\ft$ can be identified with the postcomposition map
\[
(\Id_A\sqcup\ft_\bullet)\circ-\colon\bGr^\twosimeq(S)(B,A\sqcup\ul1)\xrightarrow{\simeq}X(B,A)_\ft\subset\bGr^\twosimeq(S_+)(B,A)_1,
\]
which is an equivalence. This concludes the proof of Theorem \ref{thm:Abis}.

\section{The universal property of \texorpdfstring{$\bGr$}{bGr}}
\label{sec:proofthmA}
In this section we conclude the proof of Theorem \ref{thm:A}.
We introduce some notations and conventions that we will use in the rest of the section.
\begin{itemize}
\item We fix $S\in\Fin$; by virtue of Lemmas \ref{lem:1Sreduction} and \ref{lem:plusreduction}, our goal 
is to prove that $\fj(S)$ induces an equivalence on spaces of 2-morphisms of $\N^{S_+}$ grading $1^{S_+}$.
\item We abbreviate $\cC=\bGr(S)$ and $\cC_+=\bGr(S_+)$.
\item Recall Definition \ref{defn:dotKappa}. We abbreviate
$\dot\cT:=\dot\cT^2(\cC)\simeq\dot\cT^2(\cC_+)$ (using the equivalence $\cC\simeq\iota^*\cC_+$), $\dot\cK:=\dot\cK^2(\cC_+)$, $\dot\cL:=\dot\cL^2(\cC)$, and we write $\dot\lambda:=\dot\lambda^2(\cC)$ and $\dot\kappa:=\dot\kappa^2(\cC_+)$, which are objects in $\PSh(\dot\cT)$. We further write $p\colon\dot\cL\to\dot\cT$ and $p_+\colon\dot\cK\to\dot\cT$ for the projections.
\item We suppress the first entry ``$\gamma$'' in a category of the form $\Xi(\gamma,-,-)$ as in Corollary \ref{cor:LDDhintmappingspaces}, leaving it understood to be equal to $\Gamma^2(\cC)_0=\Gamma^2(\bGr(S))_0$.
\end{itemize}

\begin{nota}
\label{nota:cOfm}
Given $B,A\in\Fin$ and $G_u,G_d\in\bGr(B,A)^\simeq$, we usually denote by $\cO=(B,A,G_u,G_d)\in\Fun(\mO_1,\cC)^\simeq\simeq\Fun(\mO_1,\bGr)^\simeq$ the corresponding $\mO_1$-shaped diagram. We consider $(B,A,G_u,G_d)$ also as an object of $\dot\cT$: by Lemma \ref{lem:essentiallysurjective}, up to equivalence we can represent any object in $\dot\cT$ in this way. We similarly expand $\cO'=(B',A',G'_u,G'_d)$, and so on.

A 2-morphism in $\bGr(S_+)$ from $G_u$ to $G_d$ is usually represented as a pair $(f,m)$ of a 2-morphism $f\colon G_u\to G_d$ in $\bGr$  and a map $m\colon\CE(f)\to S_+$. Given $\cD\in (\Catinftwo)_{/\cC_+}$ with $\cD^\twosimeq\xrightarrow{\simeq}\bGr^\twosimeq\simeq\cC_+^\twosimeq$, we denote by $\cD(f,m)\subseteq\cD(B,A)(G_u,G_d)$ the subspace of 2-morphisms mapping to $(f,m)$.

If $(f,m)$ has $\N$-grading 0, we consider $(f,m)$ also as an object in $p^{-1}(\cO)\subseteq\dot\cL$; and if $(f,m)$ has $\N$-grading 1, we consider it also as an object in $p_+^{-1}(\cO)\subseteq\dot\cK$.
\end{nota}
\subsection{The leaf-like reduction}
\label{subsec:leaflikereduction}
\begin{defn}
\label{defn:leaf}
A \emph{leaf} in a gaf $G$ is an edge $e\in E$ at least one of whose half-edges $h$ is attached, along $\sigma$, to an inner vertex $v\in V$ of valence 1, i.e. $\sigma^{-1}(v)=\set{h,v}$. An orientation on a leaf is a choice of a half-edge $h$ with the above properties (it is unique unless the component of $e$ in $G$ is isomorphic to the gaf $G_{\tbeta_1}$ from Subsection \ref{subsec:propertiesGr}, after disregarding the markings by $B$ and by $\ul2$, respectively).

We say that $(f,m)$ as in Notation \ref{nota:cOfm} is \emph{leaf-like} if the following hold:
\begin{itemize}
\item the $\N=\N^{\set{\bullet}}$-grading of $(f,m)$ is 1; that is, the fibre $m^{-1}(\bullet)$ is a singleton;
\item the unique element in $m^{-1}(\bullet)$ is a leaf in the based or unbased sub-trees of $G_u$ containing it, among those sub-trees that are collapsed along $f$.
\end{itemize}
In this case, we usually fix an orientation $h_\ell\in e_\ell$ of the special leaf $e_\ell=m^{-1}(\bullet)\in E_u$, and we denote $v_\ell=\sigma(h_\ell)\in V_u$ the corresponding vertex of valence 1.

Given $\cD\in (\Catinftwo)_{/\bGr(S_+)}$, we say that a 2-morphism in $\cD$ is \emph{leaf-like} if its image in $\bGr(S_+)$ is leaf-like. 
\end{defn}

\begin{lem}
\label{lem:leaflikereduction}
Assume the following:
\begin{itemize}
\item for all $T\in\Fin$ with $\#T<\#S$, the map $\fj(T)\colon\bGr(T)^+\to\bGr(T_+)$ induces an equivalence on spaces of 2-morphisms of $\N^{T_+}$-grading $1^{T_+}$.
\item for all $(f,m)$ as in Notation \ref{nota:cOfm} such that $(f,m)$ is leaf-like and has $\N^{S_+}$-grading $1^{S_+}$, the map of spaces $\fj(S)(f,m)\colon\bGr(S)^+(f,m)\to\bGr(S_+)(f,m)$ is an equivalence (equivalently, $\bGr(S)^+(f,m)$ is contractible).
\end{itemize}
Then $\fj(S)$ induces an equivalence on spaces of 2-morphisms of $\N^{S_+}$-grading $1^{S_+}$.
\end{lem}
\begin{proof}
We have to prove that $\fj(S)(f,m)$ is an equivalence also when $(f,m)$ has $\N^{S_+}$-grading $1^{S_+}$ but is not leaf-like. Since $m\colon\CE(f)\xrightarrow{\simeq} S_+$ is non-empty, we may find $s\in S$ that corresponds along $m$ to a leaf in some tree collapsed by $f$. Let $\bar m$ denote the composite of $m\colon\CE(f)\xrightarrow{\simeq} S_+$ with the transposition of $S_+$ swapping $s$ and $\bullet$. We may then write a commutative diagram of spaces as follows:
\[
\begin{tikzcd}[row sep=12pt]
\cU(S_+)(f,\bar m)\ar[r,phantom,"\simeq"]\ar[d,"\simeq"]&\cU(S_+)(f,m)\ar[r,"\simeq"]&\bGr(S)^+(f,m)\ar[d]\\
\bGr(S)^+(f,\bar m)\ar[r,"\simeq"]& \bGr(S_+)(f,\bar m)\ar[r,phantom,"\simeq"] &\bGr(S_+)(f,m).
\end{tikzcd}
\]
The first hypothesis implies by Lemma \ref{lem:plusreduction} that $\fG(S)\colon\cU(S)\to\bGr(S)$ is a $T_{01^S}$-equivalence, so that the map between extensions $\fG(S)[*\ ;\ *\sqcup\bB C_2]\colon\cU(S_+)\to\bGr(S)^+$ is a $T_{01^{S_+}}$-equivalence; in particular the top right horizontal map is an equivalence. An analogous argument applies to the left vertical map, which is therefore an equivalence. By the second hypothesis the bottom left horizontal map is an equivalence; we conclude that the right vertical map is also an equivalence.
\end{proof}
\begin{cor}
\label{cor:leaflikereduction}
Assume $\bGr(S)^+(f,m)\simeq*$ for all $S\in\Fin$ and all leaf-like 2-morphisms $(f,m)$ in $\bGr(S_+)$ as in Notation \ref{nota:cOfm}. Then Theorem \ref{thm:A} holds.
\end{cor}

\begin{rem}
Recall from Definition \ref{defn:leaf} that a leaf-like 2-morphism $(f,m)$ is in particular required to have $\N$-grading 1. In principle we could weaken the assumption in Corollary \ref{cor:leaflikereduction} by only requiring $\bGr(S)^+(f,m)$ to be contractible when $(f,m)$ has $\N^S$-grading $1^S$. In practice it will be easier to just neglect the $\N^S$-grading from now on and focus on the $\N=\N^{\set{\bullet}}$-grading. This will also spare us the difficulty to formalise twisted $\N^S$-gradings as introduced in Remark \ref{rem:twistedgrading}.
\end{rem}

\subsection{A core groupoid computation}
\begin{defn}
Given a full $\infone$-subcategory inclusion $\cD_1\subseteq\cD_2$ in $\Catinfone$, we say that $\cD_1$ is a \emph{split $\infone$-subcategory} of $\cD_2$ if there exists a full $\infone$-subcategory $\cD_3\subseteq\cD_2$ such that the two inclusions exhibit $\cD_2\simeq\cD_1\sqcup\cD_3$.
\end{defn}
\begin{defn}
For $\cO,\cO'\in\dot\cT$, $(f,m)\in\dot\kappa(\cO)$ and $(f',m')\in\dot\kappa(\cO')$ we denote by $\Xi((f,m),(f',m'))$ the split $\infone$-subcategory of $\Xi(\cO,\cO')=\Xi(\Gamma^2(\cC)_0,\cO,\cO')$ corresponding to the subspace $\dot\cK((f,m),(f',m'))\subseteq\dot\cT(\cO,\cO')\simeq|\Xi(\cO,\cO')|$. We define similarly $\Xi((f,m),(f',m'))$ for $(f,m)\in\dot\lambda(\cO)$ and $(f',m')\in\dot\lambda(\cO')$.
\end{defn}

\begin{nota}
\label{nota:morphismdotcT}
We usually represent a morphism $\zeta\colon\cO\to\cO'$ in $\dot\cT$ by an object in $\Xi(\cO,\cO')$, i.e. by an annular diagram as follows, using Notation \ref{nota:cOfm} and relying on Proposition \ref{prop:mappingspacesdotcT2}:
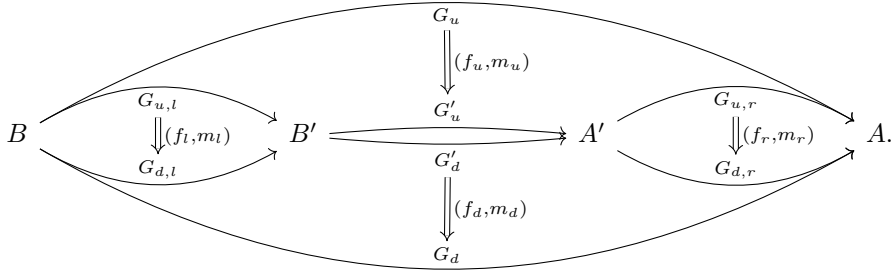
\[
\begin{tikzcd}[column sep=90pt, row sep=5pt]
B  
\ar[r,bend left=30,"G_{u,l}"'{name=UL}]\ar[r,bend right=30,"G_{d,l}"{name=DL}]
\ar[rrr,bend left=30,"G_u"'{name=U}]\ar[rrr,bend right=30,"G_d"{name=D}]
&B'\ar[r,bend left=5,"G'_u"{name=Up}]\ar[r,bend right=5,"G'_d"'{name=Dp}] &A'\ar[r,bend left=30,"G_{u,r}"'{name=UR}]\ar[r,bend right=30,"G_{d,r}"{name=DR}]
&A.
\ar[from=U, to=Up, Rightarrow,"{(f_u,m_u)}"]
\ar[from=Dp, to=D, Rightarrow,"{(f_d,m_d)}"]
\ar[from=UL, to=DL, Rightarrow,"{(f_l,m_l)}"]
\ar[from=UR, to=DR, Rightarrow,"{(f_r,m_r)}"]
\end{tikzcd}
\]
Here for $\star=l,r,u,d$, the pair $(f_\star,m_\star)$ is a 2-morphism in $\cC$, in particular it has $\N$-grading 0.
We say that a diagram as the one above is ``special'' if $f_l$ and $f_r$ are equivalences; up to equivalence, in a special diagram we have $G_{u,l}=G_{d,l}$ and $G_{u,r}=G_{d,r}$: we then denote the first gaf by $G_l$ and the second by $G_r$.

We represent a morphism $\zeta\colon(f,m)\to(f',m')$ in $\dot\cK$ or in $\dot\cL$ using the same notation, i.e. by an object in $\Xi((f,m),(f',m'))$.
\end{nota}
\begin{rem}
\label{rem:specialsuffices}
We observe each object in $\Xi(\cO,\cO')$ is connected by some morphism to an object corresponding to a special diagram; hence any point in the space $|\Xi(\cO,\cO')|$ is equivalent to one represented by a special diagram as in Notation \ref{nota:morphismdotcT}. In other words, every morphism in $\dot\cT$ can be represented up to (non-canonical) equivalence by a special diagram. A similar remark holds for $\dot\cL$ and $\dot\cK$.
\end{rem}

\begin{lem}
\label{lem:dotcTGrcore}
The map $\Fun(\mO_1,\cC)^\simeq\to\dot\cT^\simeq$ from Lemma \ref{lem:essentiallysurjective} is an equivalence.
\end{lem}
\begin{proof}
Let $\cO,\cO'$ be as in Notation \ref{nota:cOfm}. We have a functor $\fd\colon\Xi(\cO,\cO')\to\N$, sending an object as in Notation \ref{nota:morphismdotcT} to the sum of the cardinalities of the finite sets $\CE(f_l),\CE(f_r),\CE(f_u),\CE(f_d)$; this assignment on objects is constant along morphisms, so it descends indeed to a functor $\fd\colon\Xi(\cO,\cO')\to\N$ and hence to a map of spaces $\fd\colon\dot\cT(\cO,\cO')\to\N$, using Proposition \ref{prop:mappingspacesdotcT2}. 
The maps of spaces $\fd\colon\dot\cT(\cO,\cO')\to\N$ assemble into an $\N$-grading on $\dot\cT$. If we let $\Xi(\cO,\cO')_{\fd=0}$ denote the split $\infone$-subcategory $\fd^{-1}(0)\subseteq\Xi(\cO,\cO')$,
it follows that the subspace of $\dot\cT(\cO,\cO')$ comprising equivalences is contained in the subspace $|\Xi(\cO,\cO')_{\fd=0}|\subset|\Xi(\cO,\cO')|$. Note also that $\Xi(\cO,\cO')_{\fd=0}$ is already an $\infty$-groupoid.

We can now define another functor $\fd'\colon\Xi(\cO,\cO')_{\fd=0}\to\N$: this time we send an object, which is in particular represented by a special diagram as in Notation \ref{nota:morphismdotcT}, to the sum $\#\VE(G_l)+\#\VE(G_r)$; again this assignment is constant on (iso)morphisms in $\Xi(\cO,\cO')_{\fd=0}$, so it descends indeed to a map of spaces $\fd'\colon\Xi(\cO,\cO')_{\fd=0}\to\N$; the subspaces of morphism spaces $\Xi(\cO,\cO')_{\fd=0}\subset|\Xi(\cO,\cO')|\simeq\dot\cT(\cO,\cO')$ assemble into a wide subcategory $\dot\cT'\subseteq\dot\cT$, and the maps $\fd'$ give rise to an $\N$-grading on $\dot\cT'$. It follows that the subspace of $\dot\cT(\cO,\cO')$ comprising equivalences is contained in the subspace $\Xi(\cO,\cO')_{\fd=0,\fd'=0}=(\fd')^{-1}(0)\subseteq\Xi(\cO,\cO')_{\fd=0}$.

Finally, we have a functor $\omega\colon\dot\cT'\to\Fin^\op\times\Fin$, sending $\cO\mapsto(B,A)$ and sending a morphism $\zeta\colon\cO\to\cO'$ represented by a special diagram as in Notation \ref{nota:morphismdotcT} to the pair of \emph{maps of finite sets} $G_l\colon B'\to B$ and $G_r\colon A\to A'$. It follows that the subspace of $\dot\cT(\cO,\cO')$ comprising equivalences is contained in the subspace $\Xi(\cO,\cO')_{\fd=0,\fd'=0,\omega\simeq}\subseteq\Xi(\cO,\cO')_{\fd=0,\fd'=0}$ of morphisms of $\dot\cT'$ that are sent to equivalences along $\omega$.

We now observe that $\Xi(\cO,\cO')_{\fd=0,\fd'=0,\omega\simeq}$ is equivalent to the space of natural isomorphisms between $\cO,\cO'\in\Fun(\mO_1,\cC)^\simeq$. This proves that the map of spaces $\Fun(\mO_1,\cC)^\simeq\to\dot\cT^\simeq$ from Lemma \ref{lem:essentiallysurjective} is fully faithful; and by the same lemma, it is essentially surjective.
\end{proof}
\begin{cor}
\label{cor:dotcKGrcore}
We have $\Fun(\Theta_2,\cC)^\simeq\simeq\dot\cL^\simeq$ and $\Fun(\Theta_2,\cC_+)_1^\simeq\simeq\dot\cK^\simeq$.
\end{cor}
\begin{proof}
We use that $p\colon\dot\cL\to\dot\cT$ and $p_+\colon\dot\cK\to\dot\cT$ are right fibrations corresponding to the presheaves $\dot\lambda,$ and $\dot\kappa$, whose restrictions to the space $\Fun(\mO_1,\cC)^\simeq$ are the presheaves sending $\cO$ to the space of 2-morphisms in $\cC_+$ from $G_u$ to $G_d$ having $\N$-grading 0 and 1, respectively.
\end{proof}
We conclude the subsection with a construction of functors between categories of the form $\Xi(\cO,\cO')$, and by restriction of the form $\Xi((f,m),(f',m'))$.
\begin{defn}
\label{defn:zetacircblank}
Let $\zeta\colon\cO\to\cO'$ be a morphism in $\dot\cT$ represented by a special diagram as in Notation \ref{nota:morphismdotcT}, and let $\cO''\in\dot\cT$. We define a functor $\zeta\circ-\colon\Xi(\cO'',\cO)\to\Xi(\cO'',\cO')$. 
Consider the following lax commutative square, whose vertical functors are given by the right vertical functors in the pullback decompositions for $\Xi(\cO'',\cO)$ and $\Xi(\cO'',\cO')$ as in Proposition \ref{prop:mappingspacesdotcT2}, respectively:
\[
\begin{tikzcd}[column sep=50pt, row sep=40pt]
\Tw(\cC(B'',B))\times\Tw(\cC(A,A''))\ar[r,"\Tw(G_l\circhor-)","\Tw(-\circhor G_r)"']\ar[d,"d_1(-)\circhor G_u\circhor d_1(-)"',"d_0(-)\circhor G_d\circhor d_0(-)"]&
\Tw(\cC(B'',B'))\times\Tw(\cC(A',A''))\ar[d,"d_1(-)\circhor G'_u\circhor d_1(-)"',"d_0(-)\circhor G'_d\circhor d_0(-)"]\\
\cC(B'',A'')\times\cC(B'',A'')^\op\ar[r,equal]\ar[ur,Rightarrow,"f_u","f_d"']&\cC(B'',A'')\times\cC(B'',A'')^\op.
\end{tikzcd}
\]
We consider the pullback decompositions of $\Xi(\cO'',\cO)$ and $\Xi(\cO'',\cO')$ from \ref{prop:mappingspacesdotcT2}.
We define the top right coordinate of $\zeta\circ-$ as the composite of the top right projection $\Xi(\cO'',\cO)\to\Tw(\cC(B'',B))\times\Tw(\cC(A,A''))$ with the top horizontal functor above; we define the bottom right coordinate of $\zeta\circ-$ as the further composite of the previous with the right vertical functor above; and we define the bottom left coordinate of $\zeta\circ-$ as the concatenation of the bottom left projection $\Xi(\cO'',\cO)\to\cC(B'',A'')_{G''_u/}\times\cC(B'',A'')^\op_{G''_d/}$ with the composite functor 
\[
\Xi(\cO'',\cO)\to \Tw(\cC(B'',B))\times\Tw(\cC(A,A''))\to\Fun([1],\cC(B'',A'')\times\cC(B'',A'')^\op),
\]
where the second arrow is given by the above lax square.

We define similarly a functor $-\circ\zeta\colon\Xi(\cO',\cO'')\to\Xi(\cO,\cO'')$. Whenever we have objects $(f,m),(f',m'),(f'',m'')$ in $\dot\cL$ or $\dot\cK$ as in Notation \ref{nota:cOfm} and we have that $\dot\lambda(\zeta)$ or $\dot\kappa(\zeta)$ sends $(f',m')\mapsto(f,m)$, we also write $\zeta\circ-\colon\Xi((f'',m''),(f,m))\to\Xi(\Gamma(f'',m''),(f,m))$ and $-\circ\zeta\colon\Xi((f',m'),(f'',m''))\to\Xi(\Gamma(f,m),(f',m'))$ for the restricted functors.
\end{defn}
In the setting of Definition \ref{defn:zetacircblank}, the induced maps between classifying spaces $\zeta\circ-\colon|\Xi(\cO'',\cO)|\to|\Xi(\cO'',\cO')|$ and $-\circ\zeta\colon|\Xi(\cO',\cO'')|\to|\Xi(\cO,\cO'')|$ are equivalent to the maps between morphism spaces $\dot\cT(\cO'',\cO)\to\dot\cT(\cO'',\cO')$ and $\dot\cT(\cO',\cO'')\to\dot\cT(\cO,\cO'')$ induced by precomposition and postcomposition with $\zeta\in\dot\cT(\cO,\cO')$, respectively. A similar remarks holds for the restricted functors, compared to morphism spaces in $\dot\cL$ and $\dot\cK$.

\subsection{A pushout of Day convolutions}
\begin{nota}
\label{nota:Afm}
For a presheaf $A\in\PSh(\dot\cT)_{/\dot\kappa}$ and for $(f,m)$ as in Notation \ref{nota:cOfm} with $(f,m)$ of $\N$-grading 1, we denote by $A(f,m)\subseteq A(\cO)$ the preimage of $(f,m)\in\dot\kappa(\cO)$, using that the latter space is homotopy discrete.
\end{nota}
Using Notations \ref{nota:cOfm} and \ref{nota:Afm}, we have an identification of spaces $\bGr(S)^+(f,m)\simeq \dot\kappa^2(\bGr(S)^+)(f,m)$.
By virtue of Corollary \ref{cor:leaflikereduction}, our next goal becomes to compute the presheaf $\dot\kappa^2(\bGr(S)^+)\in\PSh(\dot\cT)$.
Recall from Example \ref{ex:cUcVasextensions} that $\bGr(S)^+\simeq\cC[*\ ;\ \bB C_2,*]^\sm$ is a two-step symmetric monoidal extension of $\inftwo$-categories. This can be rephrased as follows in terms of elementary and non-relative extensions, i.e. as in Definition \ref{defn:extension} with $\del X=\emptyset$.
\begin{nota}
\label{nota:phipsi}
Recall Lemma \ref{lem:dotcTGrcore} and Corollary \ref{cor:dotcKGrcore}.
We let $\phi_1\colon *\to\dot\cT$ and $\phi_2\colon \bB C_2\to\dot\cT$ denote the functors classifying the following $\mO_1$-shaped diagrams in $\cC$, with the second diagram carrying a $C_2$-equivariant structure; see Subsection \ref{subsec:propertiesGr} for the notation and for more details:
\[
\begin{tikzcd}[column sep=70pt]
\ul1\ar[r,bend left=10,"G_{\beta,1}"]\ar[r,bend right=10,"G_{\beta,2}"'] &\ul1;
\end{tikzcd}
\hspace{1cm}
\begin{tikzcd}[column sep=70pt]
\ul2\ar[r,bend left=10,"G_{\tbeta,1}"]\ar[r,bend right=10,"G_{\tbeta,2}"'] &\ul0.
\end{tikzcd}
\]
We let $\phi_3$ denote $\phi_2|_*$ for the inclusion $*\hto\bB C_2$. For uniform notation, we let $X_1=X_3=*\in\cS$ and $X_2=\bB C_2\in\cS$.

We define similarly the functors $\psi_i\colon X_i\to\dot\cK\subseteq\dot\cL^2(\cC_+)$ by using the 2-morphism $\beta_\bGr$ (for $i=1$) and the $C_2$-equivariant 2-morphism $\tbeta_\bGr$ (for $i=2,3$), together with the maps $\CE(\beta_\bGr)\simeq\CE(\tbeta_\bGr)\simeq\ul1\to S_+$ with value $\bullet$.
\end{nota}
We may identify $\bGr(S)^+$ as the pushout in $\CAlg(\Catinftwo)_{\cC/}$, or equivalently in $\CAlg(\Catinftwo)$, of the following span diagram of extensions of $\cC$, where we make explicit the structure maps left implicit in Definition \ref{defn:extension}:
\[
\begin{tikzcd}[column sep=70pt]
\cC[X_1\ \phi_1]^\sm &\cC[X_3\ \phi_3]^\sm\ar[l,"\beta_3\mapsto\ft\mu(\beta_1\otimes\Id_{\ul1})"']\ar[r,"*\hto\bB C_2"]&\cC[X_2\ \phi_2]^\sm.
\end{tikzcd}
\]
Here $\beta_i$ denotes the 2-morphism which we adjoin in the extensions $\cC[X_i\ \phi_i]$ for $i=1,2,3$, respectively, with $\beta_2$ being a $C_2$-equivariant 2-morphism. 
The left pointing map sends $\beta_3\mapsto\ft\mu(\beta_1\otimes\Id_{\ul 1})$: this assignment is compatible with the two diagrams classified by $\phi_1$ and $\phi_3$. Since all extensions in the previous span diagram are $\N$-admissible,
we may also regard the previous span diagram as a span diagram in  $\CAlg(\Catinftwo^\N)$. 
We may then apply the colimit preserving functor $(\iota^{01,\N})^*$ and obtain an analogous span diagram in $\CAlg(\Catinftwo^{01})$, 
whose pushout is $(\iota^{01,\N})^*(\bGr(S)^+)$:
\[
\begin{tikzcd}[row sep=10pt,column sep=10pt]
(\iota^{01,\N})^*(\cC[X_3\ \phi_3]^\sm)\ar[d]\ar[r]\ar[dr,phantom,"\ulcorner"very near end] &
(\iota^{01,\N})^*(\cC[X_1\ \phi_1]^\sm)\ar[d]\\
(\iota^{01,\N})^*(\cC[X_2\ \phi_2]^\sm)\ar[r]&(\iota^{01,\N})^*(\bGr(S)^+).
\end{tikzcd}
\]

Recall the pullback square from Corollary \ref{cor:CAlgCatinfn01pullback}; by Corollary \ref{cor:nsimeqinvariance}, the image of the last pushout diagram along the top functor $\iota^*$ in the said square is the constant square diagram in 
$\CAlg(\Catinftwo)$ with value $\cC$, 
so that the further image along the right functor $(\dot\cL^2\to\dot\cT^2)$ is the constant square diagram with value $(\dot\cT,\dot\lambda)$ in $\CAlg(\PSh)$. 
The constancy of the last two diagrams allows us to conclude that the left vertical functor $\varphi^*\dot\cL\dot\cT_\N$ in the said square sends the above pushout diagram in  $\CAlg(\Catinftwo^{01})$ 
to a pushout square in $\CAlg(\PSh^{01})$, and in fact even in the fibre
$(\hat\iota^*)^{-1}(\dot\cT,\dot\lambda)\subseteq\CAlg(\PSh^{01})$.
The fibre $(\hat\iota^*)^{-1}(\dot\cT,\dot\lambda)$ may be identified with the module category 
$\Mod_{\dot\lambda}(\PSh(\dot\cT))$.
We thus obtain the following pushout square in $\Mod_{\dot\lambda}(\PSh(\dot\cT))$, computing the presheaf $\dot\kappa^2(\bGr(S)^+)$:
\[
\begin{tikzcd}[row sep=10pt,column sep=10pt]
\dot\kappa^2(\cC[X_3\ \phi_3]^\sm)\ar[r,phantom,"\simeq"]&\theta_3\otimes_\Day\dot\lambda\ar[r]\ar[d]\ar[dr,phantom,"\ulcorner"very near end] &\theta_1\otimes_\Day\dot\lambda\ar[d]\ar[r,phantom,"\simeq"]&\dot\kappa^2(\cC[X_1\ \phi_1]^\sm)\\
\dot\kappa^2(\cC[X_2\ \phi_2]^\sm)\ar[r,phantom,"\simeq"]&\theta_2\otimes_\Day\dot\lambda\ar[r]&\dot\kappa^2(\bGr(S)^+).
\end{tikzcd}
\]
Here $\theta_{i}:=(\phi_i)_!(*)$ is the left Kan extension of the terminal presheaf; the left vertical arrow is induced by the inclusion $*\hto\bB C_2$; and the top horizontal arrow is induced by the point $\ft\mu(\beta\otimes\Id_{\ul1})\in \phi_3^*(\theta_1\otimes_\Day\dot\lambda)(*)$. The identifications in the above diagram are obtained as in Remark \ref{rem:dotkfreemodule}. We obtain the following, which we state as a lemma for future reference.
\begin{lem}
\label{lem:dotkappafmpushout}
For $(f,m)$ as in Notation \ref{nota:cOfm} of $\N$-grading 1, and in particular for $(f,m)$ leaf-like, the space $\dot\kappa^2(\bGr(S)^+)(f,m)$ is equivalent to the pushout of the span of spaces $(\theta_1\otimes_\Day\dot\lambda)(f,m)\ot(\theta_3\otimes_\Day\dot\lambda)(f,m)\to(\theta_2\otimes_\Day\dot\lambda)(f,m)$.
\end{lem}
\subsection{A categorical interpretation}

We next introduce a categorical interpretation of the spaces $(\theta_i\otimes_\Day\dot\lambda)(f,m)$. 
\begin{nota}
\label{nota:cKmoduleovercL}
We will use in the following that $\dot\cK$ has a structure of module over $\dot\cL\in\CAlg(\Catinfone)$; the module structure is obtained by restricting the symmetric monoidal structure on $\dot\cL^2(\cC_+)$ to its full $\infone$-subcategories $\dot\cL$ and $\dot\cK$. We shall use the notation ``$-\sqcup-$'' for this module structure, i.e. the same notation used for the symmetric monoidal structure on $\dot\cL$.
\end{nota}

\begin{lem}
\label{lem:defQi}
Let $\cO\in\dot\cT$
and $(f,m)$ 
be as in Notation \ref{nota:cOfm}.
with $(f,m)$ of $\N$-grading 1. 
Then for $i=1,2,3$ the spaces $(\theta_i\otimes_\Day\dot\lambda)(\cO)$ and $(\theta_i\otimes_\Day\dot\lambda)(f,m)$ are equivalent to the classifying spaces of the categories $Q_i(\cO)$ and $Q_i(f,m)$ defined by the following pullback squares in $\Catinfone$, respectively:
\[
\begin{tikzcd}[row sep=10pt, column sep=60pt]
Q_i(\cO)\ar[r,"\pi_{\cO}^i"]\ar[d,"\pi_{X_i}\times\pi^i_{\dot\cL}"']\ar[dr,phantom,"\lrcorner"very near start]&\dot\cT_{\cO/}\ar[d,"d_0"]\\
X_i\times\dot\cL\ar[r,"\phi_i(-)\sqcup p(-)"]&\dot\cT.
\end{tikzcd}
\hspace{1cm}
\begin{tikzcd}[row sep=10pt,column sep=60pt]
Q_i(f,m)\ar[r,"\pi_{(f,m)}^i"]\ar[d,"\pi_{X_i}\times\pi^i_{\dot\cL}"']\ar[dr,phantom,"\lrcorner"very near start]&\dot\cK_{(f,m)/}\ar[d,"d_0"]\\
X_i\times\dot\cL\ar[r,"\psi_i(-)\sqcup -"]&\dot\cK.
\end{tikzcd}
\]
\end{lem}
\begin{proof}
Using the equivalences $\theta_i\simeq(\phi_i)_!(*)$ and $\dot\lambda\simeq p_!(*)$, the Day convolution $\theta_i\otimes_\Day\dot\lambda$ can be computed as the left Kan extension of $*$ along the composite functor $X_i\times\dot\cL\xrightarrow{\phi_i\times p}\dot\cT\times\dot\cT\xrightarrow{-\sqcup-}\dot\cT$. The formula for left Kan extension of presheaves then readily implies the statement for $Q_i(\cO)$. The second statement follows from the splitting $\dot\cK\times_{\dot\cT}\dot\cT_{\cO/}\simeq\coprod_{(f,m)\in\dot\kappa(\cO)}\dot\cK_{(f,m)/}$, implying a splitting of $\infone$-categories $Q_i(\cO)\simeq\coprod_{(f,m)\in\dot\kappa(\cO)}Q_i(f,m)$ which is compatible with the splitting of spaces 
$(\theta_i\otimes_\Day\dot\lambda)(\cO)\simeq\coprod_{(f,m)\in\dot\kappa(\cO)}(\theta_i\otimes_\Day\dot\lambda)(f,m)$.
\end{proof}

\begin{nota}
\label{nota:Qiobject}
We usually represent an object in $Q_1(\cO)$ or in any of its split $\infone$-subcategories $Q_1(f,m)$ as a diagram $\bar D$ as follows, in which $\bar\cO\in\dot\cT$ is as in Notation \ref{nota:cOfm}, $(\bar f,\bar m)\in\dot\lambda(\bar\cO)$ (in particular $(\bar f,\bar m)$ has $\N$-grading 0), and the data $\bar G_l$, $\bar G_r$, $(\bar f_u,\bar m_u)$, $(\bar f_d,\bar m_d)$ give a special diagram as in Notation \ref{nota:morphismdotcT}, corresponding to a morphism $\bar\zeta\colon\cO\to\phi_i\sqcup\bar\cO$ in $\dot\cT$ as well as a morphism $\bar\zeta\colon(f,m)\to\psi_i\sqcup(\bar f,\bar m)$ in $\dot\cK$; we rely here on Remark \ref{rem:specialsuffices}:
\[
\begin{tikzcd}[column sep=80pt]
B \ar[r,"\bar G_l"']\ar[rrr,bend left=27,"G_u"'{name=U}]\ar[rrr,bend right=27,"G_d"{name=D}]&\ul1\sqcup \bar B\ar[r,bend left,"G_{\beta,1}\sqcup \bar G_u"'{name=Up}]\ar[r,bend right,"G_{\beta,2}\sqcup \bar G_d"{name=Dp}] &\ul1\sqcup \bar A\ar[r,"\bar G_r"'] &A.
\ar[from=Up, to=Dp, Rightarrow,"{\psi_1\sqcup (\bar f,\bar m)}"]
\ar[from=U, to=Up, Rightarrow,"{(\bar f_u,\bar m_u)}"]
\ar[from=Dp, to=D, Rightarrow,"{(\bar f_d,\bar m_d)}"]
\end{tikzcd}
\]
We similarly represent by a diagram $\bar D$ as follows an object in $Q_3(\cO)$ or $Q_3(f,m)$:
\[
\begin{tikzcd}[column sep=80pt]
B \ar[r,"\bar G_l"']\ar[rrr,bend left=27,"G_u"'{name=U}]\ar[rrr,bend right=27,"G_d"{name=D}]&\ul2\sqcup \bar B\ar[r,bend left,"G_{\tbeta,1}\sqcup \bar G_u"'{name=Up}]\ar[r,bend right,"G_{\tbeta,2}\sqcup \bar G_d"{name=Dp}] &\ul0\sqcup \bar A\ar[r,"\bar G_r"'] &A.
\ar[from=Up, to=Dp, Rightarrow,"{\psi_3\sqcup (\bar f,\bar m)}"]
\ar[from=U, to=Up, Rightarrow,"{(\bar f_u,\bar m_u)}"]
\ar[from=Dp, to=D, Rightarrow,"{(\bar f_d,\bar m_d)}"]
\end{tikzcd}
\]
We use different decorations such as $\check D,\tilde D,\dots$ for other objects in $Q_i(\cO)$ or $Q_i(f,m)$, and repeat the same decorations on the consituent parts of such objects.
\end{nota}
\begin{nota}
Recall Definition \ref{defn:leaf}.
We consider the unique edges $e_\beta$ and $e_{\tbeta}$ of the graphs $G_{\beta,1}$ and $G_{\tbeta,1}$ as leaves; we equip the first leaf with the unique possible orientation, which we denote $h_\beta\in H_{\beta,1}$, and the second with the orientation selecting the first half-edge $1\in\ul2\cong H_{\tbeta,1}$, which we denote $h_{\tbeta}\in H_{\tbeta,1}$. The second orientation is of course not invariant under the $C_2$-action on $G_{\tbeta,1}$.
\end{nota}

\begin{nota}
\label{nota:orientationpreserving}
Let $(f,m)$ be leaf-like as Definition \ref{defn:leaf}. 
We say that an object $\bar D\in Q_1(f,m)$ (respectively, $\bar D\in Q_3(f,m)$) as in Notation \ref{nota:Qiobject} is \emph{orientation-preserving} if $\bar f_u$ sends $h_\ell\in H_u$ to $h_\beta$  (respectively, to $h_{\tbeta}$); otherwise, if $\bar f_u(h_\ell)$ is the other half-edge, we say that $\bar D$ is \emph{orientation-reversing}. We observe that a morphism in $Q_1(f,m)$ or $Q_3(f,m)$ is always between a pair of orientation-preserving objects, or between a pair of orientation-reversing objects. For $i=1,3$ we denote by $Q_i(f,m)^\to$ and $Q_i(f,m)^\ot$ the split $\infone$-subcategories of $Q_i(f,m)$ spanned by orientation-preserving and orientation-reversing objects, respectively; we thus have a splitting $Q_i(f,m)\simeq Q_i(f,m)^\to\sqcup Q_i(f,m)^\ot$.

For $(\bar f,\bar m)\in\dot\cL$ we similarly consider the decomposition
\[
\Xi((f,m),\psi_i\sqcup(\bar f,\bar m))\simeq \Xi((f,m),\psi_i\sqcup(\bar f,\bar m))^\to\sqcup\Xi((f,m),\psi_i\sqcup(\bar f,\bar m))^\ot.
\]
\end{nota}

\begin{lem}
\label{lem:trivialcovering}
For $(f,m)$ leaf-like as in Notation \ref{nota:cOfm}, the map $(\theta_3\otimes_\Day\dot\lambda)(f,m)\to(\theta_2\otimes_\Day\dot\lambda)(f,m)$ is a trivial 2-fold covering of spaces.
\end{lem}
\begin{proof}
The definition of $Q_i(f,m)$ as a pullback implies that we have an equivalence of categories $Q_3(f,m)\simeq Q_2(f,m)\times_{\bB C_2}*$; since geometric realisations preserve pullbacks over spaces, the map $(\theta_3\otimes_\Day\dot\lambda)(f,m)\to(\theta_2\otimes_\Day\dot\lambda)(f,m)$ is a 2-fold covering of spaces. 
To see that it is trivial, we consider the $C_2$-action on $Q_3(f,m)$ whose quotient is $Q_2(f,m)\simeq Q_3(f,m)/C_2$: this $C_2$-action is induced by the $C_2$-action on $G_{\tbeta,1}$, which inverts the orientation of the unique leaf of $G_{\beta,1}$; consequently, this $C_2$-action swaps the subcategories $Q_3(f,m)^\to$ and $Q_3(f,m)^\ot$. Hence the space $(\theta_3\otimes_\Day\dot\lambda)(f,m)\simeq |Q_3(f,m)^\to|\sqcup|Q_3(f,m)^\ot|$ is the disjoint union of two spaces, each of which is equivalent to $(\theta_2\otimes_\Day\dot\lambda)(f,m)$ along the covering map.
\end{proof}

Combining Lemmas \ref{lem:dotkappafmpushout} and \ref{lem:trivialcovering}, for $(f,m)$ leaf-like as in Notation \ref{nota:cOfm} we may express $\dot\kappa^2(\bGr(S)^+)(f,m)$ as the colimit of the following W-shaped diagram in $\cS$, whose middle arrows are equivalences:
\[
\begin{tikzcd}[column sep=0pt,row sep=7pt]
{|Q_1(f,m)^\to|}& &{|Q_2(f,m)|}& &{|Q_1(f,m)^\ot|}\\
&{|Q_3(f,m)^\to|}\ar[ul]\ar[ur,"\simeq"]& &{|Q_3(f,m)^\ot|}.\ar[ul,"\simeq"']\ar[ur]
\end{tikzcd}
\]
In order to prove that $\dot\kappa^2(\bGr(S)^+)(f,m)$ is contractible, and thus complete the proof of Theorem \ref{thm:A}, we shall
prove the following proposition.
\begin{prop}
\label{prop:finalthmA}
Let $(f,m)$ be leaf-like as in Notation \ref{nota:cOfm}.
With reference to the above diagram, the following hold:
\begin{enumerate}
\item the space $|Q_1(f,m)^\to|$ is contractible;
\item the map $|Q_3(f,m)^\ot|\to|Q_1(f,m)^\ot|$ is an equivalence of spaces.
\end{enumerate}
\end{prop}
\subsection{Morphism spaces in \texorpdfstring{$Q_i(f,m)$}{Qi(f,m)}}
In the rest of the section we let the index $i$ be equal to 1 or 3, in particular $X_i\simeq *$ and we may think of $\phi_i$ and $\psi_i$ as objects in $\dot\cT$ and $\dot\cK$, respectively.
Let $\cO$ and $(f,m)$ be as in Notation \ref{nota:cOfm} with $(f,m)$ of $\N$-grading 1, and let $\bar D,\check D$ be objects in $Q_i(f,m)$ as in Notation \ref{nota:Qiobject}. By the definition of $Q_i(f,m)$ as a pullback from Lemma \ref{lem:defQi}, the morphism space $Q_i(f,m)(\bar D,\check D)$ is the pullback of the cospan of spaces
\[
\begin{tikzcd}[column sep=-10pt, row sep=5pt]
\dot\cK_{(f,m)/}(\psi_i\sqcup(\bar f,\bar m),\psi_i\sqcup(\check f,\check m))\ar[dr,"d_0"near end]& &\dot\cL((\bar f,\bar m),(\check f,\check m))\ar[dl,"\psi_i\sqcup-"']\\
&\dot\cK(\psi_i\sqcup(\bar f,\bar m),\psi_i\sqcup(\check f,\check m))
\end{tikzcd}
\]
\begin{lem}
\label{lem:QitodotcKfullyfaithful}
The above right diagonal map $\psi_i\sqcup-$ is an inclusion of spaces. 
\end{lem}
\begin{proof}
We may regard $\dot\cL((\bar f,\bar m),(\check f,\check m))$ and $\dot\cK(\psi_i\sqcup(\bar f,\bar m),\psi_i\sqcup(\check f,\check m))$ as subspaces of $\dot\cT(\bar\cO,\check\cO)$ and $\dot\cT(\phi_i\sqcup\bar\cO,\phi_i\sqcup\check\cO)$, respectively. We may then identify the map $\psi_i\sqcup-$ as the restriction of the map $\phi_i\sqcup-$ between the last two morphism spaces in $\dot\cT$. The last map agrees with the map induced on geometric realisations by the functor $\phi_i\sqcup-\colon\Xi(\bar\cO,\check\cO)\to\Xi(\phi_i\sqcup\bar\cO,\phi_i\sqcup\check\cO)$. This last functor exhibits the source category as a split $\infone$-subcategory of the target, as a consequence of the following observations:
\begin{itemize}
\item for all $\cO_1,\cO_2,\cO'\in\dot\cT$, the category $\Xi(\cO_1\sqcup\cO_2,\cO')$ splits as the disjoint union of the categories $\Xi(\cO_1,\cO'_1)\times\Xi(\cO_2,\cO'_2)$, indexed by the set of all splittings of $\cO'$ as a disjoint union $\cO'_1\sqcup\cO'_2$;
\item in the case $\cO_1=\phi_i$, we have $\Xi(\cO_1,\cO_2)\simeq\emptyset$ unless $\cO_2\simeq\phi_i$, in which case we have that $\Xi(\cO_1,\cO_2)$ is equivalent to $*$ ($i=1$) or to $*\sqcup *$ ($i=3$).
\end{itemize}
The functor $\phi_i\sqcup-$ exhibits $\Xi(\bar\cO,\check\cO)$ as (one of) the split $\infone$-subcategories of $\Xi(\phi_i\sqcup\bar\cO,\phi_i\sqcup\check\cO)$ corresponding to the choice $\cO'_1=\phi_i$ and $\cO'_2=\check\cO$.
\end{proof}
It follows that we may also regard $Q_i(\bar D,\check D)$ as a subspace of the morphism space $\dot\cK_{(f,m)/}(\psi_i\sqcup(\bar f,\bar m),\psi_i\sqcup(\check f,\check m))$. The latter space can be identified with the fibre at $\check\zeta\in\dot\cK((f,m),\psi_i\sqcup(\check f,\check m))$ of the map of spaces
\[
-\circ\bar\zeta\colon\dot\cK(\psi_i\sqcup(\bar f,\bar m),(\check f,\check m))\to \dot\cK((f,m),(\check f,\check m)).
\]
\begin{lem}
\label{lem:QitodotcKessentiallysurjective}
The map of spaces $Q_i(\bar D,\check D)\to \dot\cK_{(f,m)/}(\psi_i\sqcup(\bar f,\bar m),\psi_i\sqcup(\check f,\check m))$ induced by the functor $\pi^i_{(f,m)}$ is essentially surjective.
\end{lem}
\begin{proof}
Let $\zeta\colon\psi_i\sqcup(\bar f,\bar m)\to\psi_i\sqcup(\check f,\check m)$ be a morphism in $\dot\cK_{(f,m)/}$, i.e. we are equipped with an equivalence $\zeta\circ\bar\zeta\simeq\check\zeta$ of morphisms in $\dot\cK$, and use Notation \ref{nota:morphismdotcT} for $\zeta$. Consider first the case $i=1$; 
then $f_u\colon G_{\beta,1}\sqcup\bar G_u\to G_r\circhor(G_{\beta,1}\sqcup\check G_u)\circhor G_l$ must send $h_\beta\mapsto h_\beta$, since both half-edges denoted $h_\beta$ are the image of $h_\ell\in H_u$ along $\bar f_u$ and $\check f_u$, respectively. It follows that $\zeta$ has the form $\Id_{\psi_1}\sqcup\zeta'$ for some $\zeta'\colon(\bar f,\bar m)\to(\check f,\check m)$ in $\dot\cL$. The case $i=3$ is analogous, replacing $G_{\beta,1}$ by $G_{\tbeta,1}$.
\end{proof}
Lemmas \ref{lem:QitodotcKfullyfaithful} and \ref{lem:QitodotcKessentiallysurjective} prove that the functor $\pi^i_{(f,m)}$ is fully faithful, allowing us to identify morphism spaces in $Q_i(f,m)$ with corresponding morphism spaces in $\dot\cK_{(f,m)/}$. We conclude the subsection with three lemmas comparing morphism spaces of categories of the form $Q_i(f,m)$.
\begin{lem}
\label{lem:Qispacesequivalence1}
Let $\tilde D$ be an object in  $Q_i(f,m)$ as in Notation \ref{nota:Qiobject}, let $\zeta\colon(\bar f,\bar m)\to(\check f,\check m)$ be a morphism in $\dot\cL$ represented by a special diagram as in Notation \ref{nota:morphismdotcT}, and fix an equivalence $(\Id_{\psi_i}\sqcup\zeta)\circ\bar\zeta\simeq\check\zeta$, upgrading $(\Id_{\psi_i}\sqcup\zeta)$ to a morphism in $Q_i(\bar D,\check D)$. Consider the following commutative diagram of categories:
\[
\begin{tikzcd}[row sep=10pt]
{\Xi((\check f,\check m),(\tilde f,\tilde m))}\ar[r,"-\circ\zeta"]\ar[d,"\psi_i\sqcup-"]&{\Xi((\bar f,\bar m),(\tilde f,\tilde m))}\ar[d,"\psi_i\sqcup-"]\\
{\Xi(\psi_i\sqcup(\check f,\check m),\psi_i\sqcup(\tilde f,\tilde m))}\ar[d,"-\circ\check\zeta"] & {\Xi(\psi_i\sqcup(\bar f,\bar m),\psi_i\sqcup(\tilde f,\tilde m))}\ar[d,"-\circ\bar\zeta"]\\
{\Xi((f,m),\psi_i\sqcup(\tilde f,\tilde m))}\ar[r,equal]&{\Xi((f,m),\psi_i\sqcup(\tilde f,\tilde m))}.
\end{tikzcd}
\]
Suppose that there are split subcategories
\[
\cD\subseteq\Xi((f,m),\psi_i\sqcup(\tilde f,\tilde m)),\quad \cE_1\subseteq\Xi((\check f,\check m),(\tilde f,\tilde m)),\quad\ \cE_2\subseteq\Xi((\bar f,\bar m),(\tilde f,\tilde m))
\]
such that
$\tilde\zeta\in\cD$, the preimage of $\cD$ along the left vertical composite $(\psi_i\sqcup-)\circ\check\zeta$ is contained in $\cE_1$, and similarly $((\psi_i\sqcup-)\circ\bar\zeta)^{-1}(\cD)\subseteq\cE_2$. Assume also that the top horizontal functor $-\circ\zeta$ restricts to a functor $\cE_1\to\cE_2$ inducing an equivalence $|\cE_1|\xrightarrow{\simeq}|\cE_2|$ on classifying spaces.
Then we have an equivalence of spaces
\[
-\circ\zeta\colon Q_i(f,m)(\check D,\tilde D)\xrightarrow{\simeq} Q_i(f,m)(\bar D,\tilde D).
\]
\end{lem}
\begin{proof}
We can take classifying spaces in the above diagram of categories, and identify the map between vertical fibres at $\tilde\zeta\in|\Xi((f,m),\psi_i\sqcup(\tilde f,\tilde m))|$ with the map $-\circ\zeta\colon Q_i(f,m)(\check D,\tilde D)\xrightarrow{\simeq} Q_i(f,m)(\bar D,\tilde D)$. We can compute these vertical fibres after passing to the following commutative square of subspaces:
\[
\begin{tikzcd}[row sep=10pt]
{|\cE_1|}\ar[r,"-\circ\zeta","\simeq"']\ar[d,"(\psi_i\sqcup-)\circ\check\zeta"']&{|\cE_2|}\ar[d,"(\psi_i\sqcup-)\circ\bar\zeta"]\\
\tilde\zeta\in{|\cD|}\quad\quad\ar[r,equal]&\quad\quad{|\cD|}\ni\tilde\zeta.
\end{tikzcd}
\]
As the horizontal arrows are equivalences, so is the induced map between fibres.
\end{proof}
\begin{lem}
\label{lem:Qispacesequivalence2}
In the setting of Lemma \ref{lem:Qispacesequivalence1}, consider the following commutative diagram of categories:
\[
\begin{tikzcd}[row sep=10pt,column sep=40pt]
{\Xi((\tilde f,\tilde m),(\bar f,\bar m))}\ar[r,"\zeta\circ-"]\ar[d,"\psi_i\sqcup-"]&{\Xi((\tilde f,\tilde m),(\check f,\check m))}\ar[d,"\psi_i\sqcup-"]\\
{\Xi(\psi_i\sqcup(\tilde f,\tilde m),\psi_i\sqcup(\bar f,\bar m))}\ar[d,"-\circ\tilde\zeta"] & {\Xi(\psi_i\sqcup(\tilde f,\tilde m)\psi_i\sqcup(\check f,\check m))}\ar[d,"-\circ\tilde\zeta"]\\
{\Xi((f,m),\psi_i\sqcup(\bar f,\bar m))}\ar[r,"(\Id_{\psi_i}\sqcup\zeta)\circ-"]&{\Xi((f,m),\psi_i\sqcup(\check f,\check m))}.
\end{tikzcd}
\]
Assume that there are split subcategories 
\[
\bar\zeta\in\cD_1\subseteq\Xi((f,m),\psi_i\sqcup(\bar f,\bar m)),\quad\check\zeta\in\cD_2\subseteq\Xi((f,m),\psi_i\sqcup(\check f,\check m))
\]
such that the bottom functor $(\Id_{\psi_i}\sqcup\zeta)\circ-$ restricts to a functor $\cD_1\to\cD_2$ inducing an equivalence $|\cD_1|\xrightarrow{\simeq}|\cD_2|$. Assume also that there are split subcategories 
\[
\begin{split}
((\psi_i\sqcup-)\circ\tilde\zeta)^{-1}(\cD_1)&\subseteq\cE_1\subseteq\Xi((\tilde f,\tilde m),(\bar f,\bar m)),\\
((\psi_i\sqcup-)\circ\tilde\zeta)^{-1}(\cD_2)&\subseteq\cE_2\subseteq\Xi((\tilde f,\tilde m),(\check f,\check m))
\end{split}
\]
such that the top functor $\zeta\circ-$ restricts to a functor $\cE_1\to\cE_2$ inducing an equivalence $|\cE_1|\xrightarrow{\simeq}|\cE_2|$.
Then we have an equivalence of spaces 
\[
\zeta\circ-\colon Q_i(f,m)(\tilde D,\bar D)\xrightarrow{\simeq}Q_i(f,m)(\tilde D,\check D).
\]
\end{lem}
\begin{proof}
The argument is analogous to the one for Lemma \ref{lem:Qispacesequivalence1}: we take classifying spaces of the categories in the above diagram and we observe that the hypotheses imply that the induced map between vertical fibres at $\bar\zeta$ and $\check\zeta$ is an equivalence.
\end{proof}

\begin{lem}
\label{lem:Qispacesequivalence3}
Let $\zeta\colon(f',m')\to(f,m)$ be a morphism in $\dot\cK$ represented by a special diagram as in Notation \ref{nota:morphismdotcT}, and denote by $-\circ\zeta\colon Q_i(f,m)\to Q_i(f',m')$ the functor induced by $\zeta$ in the light of the pullback definition of $Q_i(f,m)$ and $Q_i(f',m')$ from Lemma \ref{lem:defQi}. Let $\bar D,\check D\in Q_i(f,m)$ as in Notation \ref{nota:Qiobject}; assume that there are split subcategories $\check\zeta\in\cD_1\subseteq\Xi((f,m),\psi_i\sqcup(\check f,\check m))$ and $\check\zeta\circ\zeta\in\cD_2\subseteq\Xi((f',m'),\psi_i\sqcup(\check f,\check m))$ such that $-\circ\zeta$ restricts to a functor $\cD_1\to\cD_2$ inducing an equivalence $|\cD_1|\xrightarrow{\simeq}|\cD_2|$.
Then the map $Q_i(f,m)(\bar D,\check D)\to Q_i(f',m')(\bar D\circ\zeta,\check D\circ \zeta)$ induced by $-\circ\zeta$ on morphism spaces is an equivalence.
\end{lem}
\begin{proof}
The argument is analogous to the one for Lemma \ref{lem:Qispacesequivalence1}, and is based on taking vertical fibres at $\check\zeta$ and $\check\zeta\circ\zeta$ in the following square of spaces:
\[
\begin{tikzcd}[row sep=10pt]
{|\Xi((\bar f, \bar m),(\check f,\check m))|}\ar[r,equal]\ar[d,"(\psi_i\sqcup-)\circ\bar\zeta"]&{|\Xi((\bar f, \bar m),(\check f,\check m))|}\ar[d,"(\psi_i\sqcup-)\circ\bar\zeta\circ\zeta"]\\
{|\Xi((f,m),\psi_i\sqcup(\check f,\check m))|}\ar[r,"-\circ\zeta",hook]&{|\Xi((f',m'),\psi_i\sqcup(\check f,\check m))|}
\end{tikzcd}
\]

\end{proof}

\subsection{The one-vertex reduction}
Our next step is to reduce the proof of Proposition \ref{prop:finalthmA} to the case in which the target $G_d$ of $f$ is a gaf consisting of a single vertex and having no edges.
\begin{defn}
\label{defn:Qipfm}
Let $\cO$ and $(f,m)$ be as in Notation \ref{nota:cOfm}, with $(f,m)$ leaf-like.
For $i=1,3$ we denote by $Q'_i(f,m)\subseteq Q_i(f,m)$ the full $\infone$-subcategory spanned by special diagrams $\bar D$ as in Notation \ref{nota:Qiobject} satisfying the following property: \begin{itemize}
\item[(i)] the set $\bar f_d^{-1}(E_d)$ is contained in $\bar E_l\sqcup\bar E_d$. 
\end{itemize}
We further denote by $Q''_i(f,m)\subseteq Q'_i(f,m)$ the full $\infone$-subcategory spanned by objects $\bar D$ satisfying the following property:
\begin{itemize}
\item[(ii)] the set $\bar f_d^{-1}(E_d)$ is contained in $\bar E_l$. 
\end{itemize}
\end{defn}
In Definition \ref{defn:Qipfm}, note that a priori $\bar f_d^{-1}(E_d)$ is a subset of the set $\bar E_l\sqcup \bar E_d\sqcup\bar E_r$.
\begin{defn}
\label{defn:fXmX}
Let $\cO$ and $(f,m)$ be as in Notation \ref{nota:cOfm},
with $(f,m)$ of $\N$-grading $\le1$, 
and let $X\subseteq E_d$ be a subset. For $x\in X$, regard $x$ as a subset of $H_d$ with two elements and fix an identification $\ul2\simeq x$ (an edge orientation). We let $\cO^X$ and $(f^X,m^X)$ be the essentially unique data fitting in a diagram in $\cC_+$ as follows:
\[
\begin{tikzcd}[column sep=60pt]
B\ar[r,"\Id_B\sqcup\coprod_XG_\fe"]&B\sqcup X\times\ul2=B^X\ar[r,bend left,"G^X_u"'{name=U}]\ar[r,bend right,"G^X_d"{name=D}]&A^X=A,
\ar[from=U, to=D, Rightarrow,"{(f^X,m^X)}"]
\end{tikzcd}
\]
with horizontal composition identified with $(f,m)\colon G_u\Rightarrow G_d$ in such a way that for all $x\in X$, the inclusion of finite sets $\ul2\simeq x\hto H_d$ coincides with the inclusion of $\ul2=H_\fe$ into the set of half-edges of $G^X_d\circhor \Id_B\sqcup\coprod_XG_\fe$ coming from the $x$\textsuperscript{th} copy of $G_\fe$.

We let $\zeta^X\colon(f,m)\to(f^X,m^X)$ be the morphism in $\dot\cL$ or $\dot\cK$ given by the above factorisation: in the light of Notation \ref{nota:morphismdotcT}, we have 
$G^X_l=\Id_B\sqcup\coprod_XG_\fe$ and all other data $G^X_r,f^X_u,f^X_d$ are invertible 1-morphisms and 2-morphisms.
\end{defn}
In the setting of Definition \ref{defn:fXmX} we have a canonical identification $\CE(f)\simeq\CE(f^X)$; in particular $(f,m)$ and $(f^X,m^X)$ have the same $\N$-grading; it is also immediate that if $(f,m)$ is leaf-like, then so is $(f^X,m^X)$. We also observe that in the case $X=E_d$ we have that $G_d^X$ is a gaf without edges.

\begin{lem}
\label{lem:zetaXfullyfaithful}
Let $(f,m)$ be leaf-like as in Notation \ref{nota:cOfm}, and let $i=1,3$.
Consider the functor $-\circ\zeta^{E_d}\colon Q_i(f^{E_d},m^{E_d})\to Q_i(f,m)$ induced by the morphism $\zeta^{E_d}$ in the light of the pullback definition of $Q_i(f,m)$ and $Q_i(f^{E_d},m^{E_d})$ from Lemma \ref{lem:defQi}. Then $-\circ\zeta^{E_d}$ is fully faithful with essential image $Q''_i(f,m)$.
\end{lem}
\begin{proof}
The fact that an object is in the essential image of $-\circ\zeta^{E_d}$ if and only if it lies in $Q_i''(f,m)$ is immediate. To prove that $-\circ\zeta^{E_d}$ is fully faithful, fix $\bar D,\check D\in Q_i(f^{E_d},m^{E_d})$; we want to prove that $-\circ\zeta^{E_d}$ induces an equivalence of spaces $Q_i(f^{E_d},m^{E_d})(\bar D,\check D)\xrightarrow{\simeq} Q_i(f,m)(\bar D\circ\zeta^{E_d},\check D\circ\zeta^{E_d})$. By Lemma \ref{lem:Qispacesequivalence3}, it suffices to prove that the functor
\[
-\circ\zeta^{E_d}\colon\Xi((f^{E_d},m^{E_d}),\psi_i\sqcup(\check f,\check m))\to \Xi((f,m),\psi_i\sqcup(\check f,\check m))
\]
is a split inclusion of categories. This is evidently the case: in fact we may characterise the essential image of the last functor as the split subcategory of $\Xi((f,m),\psi_i\sqcup(\check f,\check m))$ spanned by objects $\zeta$ as in Notation \ref{nota:morphismdotcT} such that $f_d^{-1}(E_d)$ is contained in $E_{d,l}$ (a priori, it is contained in $E_{d,l}\sqcup\check E_d\sqcup E_{d,r}$).
\end{proof}

\begin{lem}
\label{lem:edgelessreduction}
In the setting of Lemma \ref{lem:zetaXfullyfaithful}, the inclusions of $\infone$-categories $Q''_i(f,m)\subseteq Q'_i(f,m)\subseteq Q_i(f,m)$
induce equivalences on classifying spaces.
\end{lem}
\begin{proof}
We will prove that the inclusion $Q''_i(f,m)\hto Q'_i(f,m)$ admits a left adjoint, and that the inclusion $Q'_i(f,m)\hto Q_i(f,m)$ admits a right adjoint.

For the first claim, given an object $\bar D\in Q'_i(f,m)$ as in Notation \ref{nota:Qiobject}, we denote $X=\bar f_d^{-1}(E_d)\cap\bar E_d$ and we define a candidate left adjoint object $\check D\in Q''_i(f,m)$ together with a morphism $\eta\colon\bar D\to\check D$ in $Q'_i(f,m)$ as follows: 
\begin{itemize}
\item  we let 
$(\check f,\check m):=(\bar f^X,\bar m^X)\in\dot\cL$, we denote by $\bar\zeta^X\colon(\bar f,\bar m)\to(\bar f^X,\bar m^X)$ the morphism from Definition \ref{defn:fXmX}, and we let $\check\zeta$ be the composite in $\dot\cK$
\[
(f,m)\xrightarrow{\bar\zeta}\psi_i\sqcup(\bar f,\bar m)\xrightarrow{\Id_{\psi_i}\sqcup\bar\zeta^X}\psi_i\sqcup(\bar f^X,\bar m^X);
\]

\item the definition of $\check\zeta$ allows us to upgrade $\eta:=\Id_{\psi_i}\sqcup\bar\zeta^X$ to a morphism $\bar D\to\check D$ in $Q_i(f,m)$.
\end{itemize}
It is left to prove that for all objects $\tilde D\in Q''_i(f,m)$, precomposition with $\eta$ induces an equivalence of spaces
$Q_i(f,m)(\check D,\tilde D)\xrightarrow{\simeq}Q_i(f,m)(\bar D,\tilde D)$. By Lemma \ref{lem:Qispacesequivalence1}, it suffices to find a split subcategory $\tilde\zeta\in\cD\subseteq\Xi((f,m),\psi_i\sqcup(\tilde f,\tilde m))$ such that $-\circ\eta$ restricts to an equivalence $\cE_1:=((\psi_i\sqcup-)\circ\check\zeta)^{-1}(\cD)\xrightarrow{\simeq} \cE_2:=((\psi_i\sqcup-)\circ\bar\zeta)^{-1}(\cD)$.
This is achieved by letting $\tilde\zeta\in\cD$ be the split subcategory of $\Xi((f,m),\psi_i\sqcup(\tilde f,\tilde m))$ spanned by objects $\zeta$ as in Notation \ref{nota:morphismdotcT} such that $f_d^{-1}(E_d)$ is contained in $E_{d,l}$: note that a priori we have $f_d^{-1}(E_d)\subseteq E_{d,l}\sqcup\tilde E_d\sqcup E_{d,r}$. We may then identify $\cE_2$ with the split subcategory of $\Xi((\bar f,\bar m),(\tilde f,\tilde m))$ spanned by objects $\zeta$ such that $f_d^{-1}(X)\subseteq E_{d,l}$, whereas $\cE_1=\Xi((\check f,\check m),(\tilde f,\tilde m))$; the functor $-\circ\eta$ restricts to an equivalences $\cE_1\xrightarrow{\simeq}\cE_2$, which a fortiori induces an equivalence of spaces $|\cE_1|\xrightarrow{\simeq}|\cE_2|$.
This completes the proof that $Q''_i(f,m)\hto Q'_i(f,m)$ admits a left adjoint.

We similarly show that the inclusion $Q'_i(f,m)\hto Q_i(f,m)$ admits a right adjoint. Given an object $\check D\in Q_i(f,m)$ as in Notation \ref{nota:Qiobject} we define a candidate right adjoint object $\bar D\in Q'_i(f,m)$ and a morphism $\varepsilon\colon\bar D\to\check D$ in $Q_i(f,m)$ as follows:
\begin{itemize}
\item we let $X:=\check E_r\cap \check f_d^{-1}(E_d)\subseteq \check E_r$; we regard each $x\in X$ as a subset of two elements of $\check H_r$ and fix an identification $x\simeq\ul2$; by abuse of notation we regard $X$ also as a subset of $E_d$;
\item we let $k=1$ if $i=1$, and $k=0$ if $i=3$;
\item we factor $\check G_r=\bar G_r\circhor(\Id_{\ul k\sqcup\check A}\sqcup\coprod_XG_\fe)$ in the essentially unique way such that for all $x\in X$, the inclusion $\ul2\simeq x\hto\check H_r$ agrees with the inclusion of $\ul2=H_\fe$ in the set of half-edges of $\check G_r\circhor(\Id_{\ul k\sqcup\check A}\sqcup\coprod_XG_\fe)$ coming from the
$x$\textsuperscript{th} copy of $G_\fe$.
\item we let $(\bar f,\bar m):=(\Id_{\check A}\sqcup\coprod_XG_\fe)\circhor(\check f,\check m)$, and we let $\bar\zeta\colon(f,m)\to\psi_i\sqcup(\bar f,\bar m)$  be the morphism in $\dot\cK$ given by a special diagram as in Notation \ref{nota:morphismdotcT} by setting $\bar G_l=\check G_l$ and $(\bar f_\star,\bar m_\star)=(\check f_\star,\check m_\star)$ for $\star=u,d$, as well as by considering the gaf $\bar G_r$ constructed above;
\item we let $\zeta_X\colon(\bar f,\bar m)\to(\check f,\check m)$ be the morphism in $\dot\cL$ given by a special diagram as in Notation \ref{nota:morphismdotcT} by setting $G_{X,r}=(\Id_{\check A}\sqcup\coprod_XG_\fe)$ and letting all other data $G_{X,l}$ and $(f_{X,\star},m_{X,\star})$ for $\star=u,d$ be the obvious equivalences;
\item the definition of $\bar\zeta$ allows us to upgrade $\varepsilon:=\Id_{\psi_i}\sqcup\zeta_X$ to a morphism $\bar D\to\check D$ in $Q_i(f,m)$, since we have an evident equivalence $\varepsilon\circ\bar\zeta\simeq\check\zeta$.
\end{itemize}
It is left to prove that for all $\tilde D\in Q'_i(f,m)$, postcomposition with $\varepsilon$ induces an equivalence $Q_i(f,m)(\tilde D,\bar D)\xrightarrow{\simeq} Q_i(f,m)(\tilde D,\check D)$; for this we shall apply Lemma \ref{lem:Qispacesequivalence2}. By further abuse of notation, let $X\subseteq \tilde E_l\sqcup\tilde E_d$ denote also the subset $\tilde f_d^{-1}(X)$. 
We let $\cD_1\subseteq\Xi((f,m),\psi_i\sqcup(\bar f,\bar m))$ denote the split subcategory spanned by objects $\zeta$ as in Notation \ref{nota:morphismdotcT} such that $f_d$ restricts to the canonical, edgewise oriented identification $X\cong X$ between the two subsets of $\bar E_d\subseteq E_{d,l}\sqcup\bar E_d\sqcup E_{d,r}$ and of $E_d$ that we have denoted $X$; and we let $\cD_2\subseteq\Xi((f,m),\psi_i\sqcup(\bar f,\bar m))$ denote the split $\infone$-subcategory spanned by objects $\zeta$ such that $f_d^{-1}(X)\subseteq E_{d,r}$. Then
\[
\varepsilon\circ-\colon\Xi((f,m),\psi_i\sqcup(\bar f,\bar m))\to\Xi((f,m),\psi_i\sqcup(\check f,\check m))
\]
restricts to an equivalence $\cD_1\xrightarrow{\simeq}\cD_2$, inducing an equivalence on classifying spaces. 

Moreover, the preimage $\cE_1$ of $\cD_1$ in $\Xi((\tilde f,\tilde m),(\bar f,\bar m))$ is the split subcategory spanned by objects $\zeta$ as in Notation \ref{nota:morphismdotcT} such that $f_d$ restricts to the canonical, edgewise oriented identification $X\cong X$ between the two subsets of $\bar E_d\subseteq E_{d,l}\sqcup\bar E_d\sqcup E_{d,r}$ and of $\tilde E_d$ denoted $X$; and the preimage $\cE_2$ of $\cD_2$ in $\Xi((\tilde f,\tilde m),(\check f,\check m))$ is the split subcategory spanned by objects $\zeta$ such that $f_d^{-1}(X)\subseteq E_{d,r}$. Again, $\zeta_X\circ-\colon\Xi((\tilde f,\tilde m),(\bar f,\bar m))\to\Xi((\tilde f,\tilde m),(\check f,\check m))$ restricts to an equivalence $\cE_1\xrightarrow{\simeq}\cE_2$, inducing an equivalence $|\cE_1|\xrightarrow{\simeq}|\cE_2|$. This completes the proof that $Q'_i(f,m)\hto Q_i(f,m)$ admits a right adjoint.
\end{proof}
The combination of Lemmas \ref{lem:zetaXfullyfaithful} and \ref{lem:edgelessreduction} allows to reduce the proof of Proposition \ref{prop:finalthmA} to the special case of a leaf-like $(f,m)\in\dot\cK$ as in Notation \ref{nota:cOfm} whose target $G_d$ is a gaf without edges.
We conclude the subsection by further reducing to the case in which $G_d$ consists of a single (attaching or inner) vertex.

Let $(f,m)\in\dot\cK$ be leaf-like as in Notation \ref{nota:cOfm} with $E_d=\emptyset$, and recall Notation \ref{nota:cKmoduleovercL}. Then there is a canonical splitting $(f,m)\simeq (f_1,m_1)\sqcup(f_2,m_2)$ such that $(f_2,m_2)\in\dot\cL$ and such that $G_{1,d}$ consists of a single vertex, which is the image of the special leaf in $G_d$ along $f$. For $i=1,3$, we have a commutative diagram in $\Catinfone$ as follows, which we regard as a morphism between cospans:
\[
\begin{tikzcd}[column sep=60pt,row sep=10pt]
\dot\cL\times\dot\cL\ar[r,"(\psi_i\sqcup-)\times-"]\ar[d,"-\sqcup-"]&\dot\cK\times\dot\cL\ar[d,"-\sqcup-"]&\dot\cK_{(f_1,m_1)/}\times\dot\cL_{(f_2,m_2)/}\ar[d,"-\sqcup-"]\ar[l,"d_0\times d_0"']\\
\dot\cL\ar[r,"\psi_i\sqcup-"]&\dot\cK&\dot\cK_{(f,m)/}\ar[l,"d_0"'].
\end{tikzcd}
\]
The induced functor between pullbacks $Q_i(f_1,m_1)\times\dot\cL_{(f_2,m_2)/}\to Q_i(f,m)$ is an equivalence; taking classifying spaces, we obtain an equivalence $|Q_i(f_1,m_1)|\xrightarrow{\simeq}|Q_i(f,m)|$, restricting to equivalences between subsapces corresponding to split $\infone$-subcategories with decorations ``$\to$'' and ``$\ot$''. The following corollary summarizes the reductions achieved in this subsection.
\begin{cor}
\label{cor:singlevertexGd}
Assume that Proposition \ref{prop:finalthmA} holds when $(f,m)$ is such that $G_d$ consists of a single vertex. Then Theorem \ref{thm:A} holds.
\end{cor}

\subsection{The equivalence between \texorpdfstring{$|Q_3(f,m)^\ot|$}{|Q3(f,m)ot|} and \texorpdfstring{$|Q_3(f,m)^\ot|$}{|Q1(f,m)ot|}}
The goal of this subsection is to prove Proposition \ref{prop:finalthmA}(2). We fix $\cO$ and $(f,m)$ as in Notation \ref{nota:cOfm} throughout the subsection, with $(f,m)$ leaf-like and with $G_d$ consisting of a single vertex. For the arguments in this subsection, it suffices in fact to assume $E_d=\emptyset$.

\begin{nota}
\label{nota:zetaplus}
For $T\in\Fin$ we abuse notation and denote by $T$ all of the following:
\begin{itemize}
\item the identity gaf of $T$;
\item the object in $\dot\cT$ represented by the $\mO_1$-shaped diagram comprising twice the identity gaf of $T$;
\item the object in $\dot\cL$ represented by the identity of the identity gaf of $T$.
\end{itemize}
We denote by $\zeta^+\colon\psi_3\to(\psi_1\sqcup\ul1)$ the morphism in $\dot\cK$ given, in the light of Notation \ref{nota:morphismdotcT}, by letting $G^+_l:=\Id_{\ul2}$ and $G^+_r:=G_\ft\circhor\mu$, and by letting $f^+_u$ and $f^+_d$ be the evident identifications.
\end{nota}
The map of spaces $|Q_3(f,m)^\ot|\to|Q_1(f,m)^\ot|$ is induced on geometric realisations by the functor $a^\ot\colon Q_3(f,m)^\ot\to Q_1(f,m)^\ot$ which we describe in the following.
The morphism $\zeta^+$ from Notation \ref{nota:zetaplus} induces a natural transformation 
$z\colon(\psi_3\sqcup-)\Rightarrow(\psi_1\sqcup\ul1\sqcup-)\simeq(\psi_1\sqcup-)\circ(\ul1\sqcup-)$
of functors $\dot\cL\to\dot\cK$.
We may now define a functor $a\colon Q_3(f,m)\to Q_1(f,m)$ using the pullback definition of the latter category from Lemma \ref{lem:defQi}: we let the $\dot\cL$-coordinate of $a$ be $\ul1\sqcup\pi^3_{\dot\cL}$; we are then forced to let the $\dot\cK$-coordinate of $a$ be $(\psi_1\sqcup\ul1\sqcup p\pi^3_{\dot\cL})$; and we let the $\dot\cK_{(f,m)/}$-coordinate of $a$ be the concatenation of $\pi_{(f,m)}^3$ and the functor $Q_3(f,m)\to\Fun([1],\dot\cK)$ corresponding to the natural transformation $zp\pi_{\dot\cL}^3$.
The functor $a$ restricts to a functor $a^\ot\colon Q_3(f,m)^\ot\to Q_1(f,m)^\ot$ between split $\infone$-subcategories, and our goal is to show that $a^\ot$ induces an equivalence on classifying spaces.

To make sense of the next definition, we observe that if $\bar D\in Q_1(f,m)^\ot$ is an object represented by a diagram as in Notation \ref{nota:Qiobject}, then the vertex $\bar\rho_r(1)\in\bar V_r\sqcup A$ is in fact an inner vertex in $\bar V_r$, and it actually constitutes an isolated component of the gaf $\bar G_r$: this follows from the fact that $\bar f_u^{-1}(\bar v)$ must be a subtree of $G_u$ to which the special leaf $e_\ell$ is attached along $h_\ell$; the maximal such subtree in $G_u$ consists of the sole vertex $v_\ell$: here we are using in particular that $E_d=\emptyset$, so all edges of $G_u$ are collapsed along $f$ and hence $e_\ell$ is the unique edge attached to $v_\ell$ in $G_u$ (and not only the unique edge attached to $v_\ell$ that is collapsed along $f$). Hence $\bar f_u^{-1}(\bar v)=\set{v_\ell}$, and it follows that $\bar v$ has valence $0$ in $\bar G_r$. For the same reason, no edge in either $\bar G_u$ or $\bar G_d$ can be attached to any point in $\bar\rho_r^{-1}(\bar\rho_r(1))\cap\bar A\subseteq\bar A$. 
\begin{defn}
\label{defn:Q1vee}
We denote by $Q_1^{\vee}(f,m)^\ot\subseteq Q_1(f,m)^\ot$ the full $\infone$-subcate\-gory spanned by objects $\bar D$ as in Notation \ref{nota:Qiobject} satisfying the following property:
\begin{itemize}
\item[(iii)] let $1\in\ul1\subseteq\ul1\sqcup\bar A$, and let $\bar v:=\bar\rho_r(1)$; then the subset $\bar\rho_r^{-1}(\bar v)\subseteq\ul1\sqcup\bar A$ comprises precisely two elements. The first element is $1\in\ul1$; we denote in this case by $\bar 1\in\bar A$ the second element, and we split $\bar A=\bar{\ul1}\sqcup\bar A^\vee$.
\end{itemize}
We further denote by $Q_1^{\vee\vee}(f,m)^\ot\subseteq Q_1^\vee(f,m)^\ot$ the full $\infone$-subcategory span\-ned by objects $\bar D$ satisfying the following additional property:
\begin{itemize}
\item[(iv)] there is a (necessarily unique) splitting $(\bar f,\bar m)\simeq\bar{\ul1}\sqcup(\bar f^\vee,\bar m^\vee)$ restricting on target objects to the above splitting $\bar A=\bar{\ul1}\sqcup\bar A^\vee$.
\end{itemize}
\end{defn}
\begin{lem}
\label{lem:aotfullyfaithful}
The functor $a^\ot$ is fully faithful with essential image $Q_1^{\vee\vee}(f,m)^\ot$.
\end{lem}
\begin{proof}
It is immediate to identify the essential image of $a^\ot$ with $Q_1^{\vee\vee}(f,m)^\ot$. To prove full faithfulness, let $\bar D,\check D\in Q_3(f,m)^\ot$; we want to prove that $a^\ot$ induces an equivalence of spaces $Q_3(f,m)(\bar D,\check D)\xrightarrow{\simeq} Q_1(f,m)(a^\ot(\bar D),a^\ot(\check D))$. Similarly as in Lemma \ref{lem:Qispacesequivalence3}, we identify this map of morphism spaces with the map between vertical fibres at $\check\zeta$ and $(\zeta^+\sqcup\Id_{(\check f,\check m)})\circ\check\zeta$ in the following commutative square: 
\[
\begin{tikzcd}[column sep=70pt]
{|\Xi((\bar f,\bar m),(\check f,\check m))|}\ar[r,equal]\ar[d,"(\psi_3\sqcup-)\circ\bar\zeta"] & {|\Xi((\bar f,\bar m),(\check f,\check m))|}\ar[d,"(\zeta^+\sqcup\Id_{(\check f,\check m)})\circ(\psi_1\sqcup\ul1\sqcup-)\circ\bar\zeta"]\\
{|\Xi((f,m),\psi_3\sqcup(\check f,\check m))|}\ar[r,"(\zeta^+\sqcup\Id_{(\check f,\check m)})\circ-"] & {|\Xi((f,m),\psi_1\sqcup\ul1\sqcup(\check f,\check m))|}.
\end{tikzcd}
\]
We conclude by observing that the functor
\[
(\zeta^+\sqcup\Id)\circ-\colon\Xi((f,m),\psi_3\sqcup(\check f,\check m))\to \Xi((f,m),\psi_1\sqcup\ul1\sqcup(\check f,\check m))
\]
restricts to an equivalence $\cD_1\xrightarrow{\simeq}\cD_2$, where $\check\zeta\in\cD_1:=\Xi((f,m),\psi_3\sqcup(\check f,\check m))^\ot$,
and where $(\zeta^+\sqcup\Id_{(\check f,\check m)})\circ\check\zeta\in\cD_2\subseteq\Xi((f,m),\psi_1\sqcup\ul1\sqcup(\check f,\check m))^\ot$ is the split subcategory spanned by objects $\zeta$ as in Notation \ref{nota:morphismdotcT} such that 
$\rho_{d,r}^{-1}(\rho_{d,r}(1))=\set{1,\bar1}$.
\end{proof}
\begin{lem}
\label{lem:Q1veeequivalences}
The inclusions of $\infone$-categories $Q_1^{\vee\vee}(f,m)^\ot\subseteq Q_1^\vee(f,m)^\ot\subseteq Q_i(f,m)^\ot$ induce equivalences on classifying spaces.
\end{lem}
\begin{proof}
We prove that the inclusion $Q_1^{\vee\vee}(f,m)^\ot\subseteq Q_1^\vee(f,m)^\ot$ admits a left adjoint, and the inclusion $Q_1^\vee(f,m)^\ot\subseteq Q_1(f,m)^\ot$ admits a right adjoint. For the first claim, given $\bar D\in Q_1^\vee(f,m)^\ot$, we define a candidate left adjoint object $\check D\in Q_1^{\vee\vee}(f,m)^\ot$ with a morphism $\eta\colon\bar D\to\check D$ in $Q_1^\vee(f,m)^\ot$ as follows:
\begin{itemize}
\item we let $\bar1$ be the unique element in $\bar\rho_r^{-1}(\bar\rho_r(1))\cap\bar A$; we observe that the leaf-like assumption on $(f,m)$ implies that no edge in either $\bar G_u,\bar G_d$ can be attached to $\bar 1$; we may therefore split $(\bar f,\bar m)=\Id_{C\to\bar{\ul1}}\sqcup(\bar f^\vee,\bar m^\vee)$, where $C=\bar\rho_d^{-1}(\bar 1)\subseteq\bar B$; we have in fact $\bar B=C\sqcup\bar B^\vee$;
\item we let $(\check f,\check m)=\bar{\ul1}\sqcup(\bar f^\vee,\bar m^\vee)$ and we let $\zeta^C\colon(\bar f,\bar m)\to(\check f,\check m)$ be the morphism  in $\dot\cL$ given by a special diagram as in Notation \ref{nota:morphismdotcT} by letting $G^C_l$ be the map of finite sets $(C\to\bar{\ul1})\sqcup\Id_{\bar B^\vee}$, and by letting $G^C_r$, $f^C_u$ and $f^C_d$ be the evident identifications;
\item we let $\check\zeta\colon(f,m)\to(\check f,\check m)$ be the following composite in $\dot\cK$:
\[
(f,m)\xrightarrow{\bar\zeta}\psi_1\sqcup(\bar f,\bar m)\xrightarrow{\Id_{\psi_1}\sqcup\zeta^C}\psi_1\sqcup(\check f,\check m);
\]
\item by definition of $\check\zeta$ we have that $\eta:=\Id_{\psi_1}\sqcup\zeta^C$ upgrades to a morphism $\bar D\to\check D$ in $Q_1(f,m)$.
\end{itemize}
We need to check that for all $\tilde D\in Q_1^\vee(f,m)^\ot$, precomposition with $\eta$ induces an equivalence of spaces $Q_1(f,m)(\check D,\tilde D)\xrightarrow{\simeq} Q_1(f,m)(\bar D,\tilde D)$.
In order to apply Lemma \ref{lem:Qispacesequivalence1}, we let $\tilde\zeta\in\cD:=\Xi((f,m),\psi_1\sqcup(\tilde f,\tilde m))^\ot$.
The preimages of $\cD$ in
$\Xi((\check f,\check m),(\tilde f,\tilde m))$ and $\Xi((\bar f,\bar m),(\tilde f,\tilde m))$ are the split subcategories $\cE_1$ and $\cE_2$ spanned by objects $\zeta$ such that $\rho_{d,r}^{-1}(\rho_{d,r}(1))$ is the set $\set{1,\check 1}\subseteq\ul1\sqcup\check A$, respectively the set $\set{1,\bar 1}\subseteq\ul1\sqcup\bar A$; the functor $-\circ\zeta^C\colon\Xi((\check f,\check m),(\tilde f,\tilde m))\to\Xi((\bar f,\bar m),(\tilde f,\tilde m))$ restricts to an equivalence $\cE_1\xrightarrow{\simeq}\cE_2$, inducing an equivalence $|\cE_1|\xrightarrow{\simeq}|\cE_2|$.

For the second claim, given $\check D\in Q_1(f,m)^\ot$, we define a candidate right adjoint objects $\bar D\in Q_1^\vee(f,m)^\ot$ with a morphism $\varepsilon\colon\bar D\to\check D$ in $Q_1(f,m)^\ot$ as follows:
\begin{itemize}
\item we let $C=\check\rho_r^{-1}(\check\rho_r(1))\cap\check A\subseteq\check A$, and we split $\check A=C\sqcup\check A^\vee$; 
\item we factor $\check G_r=\bar G_r\circhor(\Id_{\ul1}\sqcup (C\to\check{\ul1})\sqcup \Id_{\check A^\vee})$ in the essentially unique way, where the second factor is a map of finite sets;
\item we let $(\bar f,\bar m)=((C\to\check{\ul1})\sqcup\Id_{\check A^\vee})\circhor(\check f,\check m)$, and we define $\bar\zeta\colon(f,m)\to\psi_1\sqcup(\bar f,\bar m)$ to be the morphism in $\dot\cK$ given by a special diagram as in Notation \ref{nota:morphismdotcT} by setting $\bar G_l=\check G_l$ and $(\bar f_\star,\bar m_\star)=(\check f_\star,\check m_\star)$ for $\star=u,d$, as well as by considering the gaf $\bar G_r$ constructed above;
\item we let $\zeta_C\colon(\bar f,\bar m)\to(\check f,\check m)$ be the morphism in $\dot\cL$ given by a special diagram as in Notation \ref{nota:morphismdotcT} by letting $G_{C,r}$ be the map of finite sets $(C\to\check{\ul1})\sqcup\Id_{\check A^\vee}$ and by letting $G_{C,l},f_{C,u},f_{C,d}$ be the evident identifications;
\item the definition of $\bar\zeta$ allows us to upgrade $\varepsilon:=\Id_{\psi_i}\sqcup\zeta_C$ to a morphism $\bar D\to\check D$ in $Q_i(f,m)$, since we have an evident equivalence $\varepsilon\circ\bar\zeta\simeq\check\zeta$.
\end{itemize}
We need to check that for all $\tilde D\in Q_1^\vee(f,m)^\ot$, postcomposition with $\varepsilon$ induces an equivalence $Q_1(f,m)(\tilde D,\bar D)\xrightarrow{\simeq} Q_1(f,m)(\tilde D,\check D)$. 
In order to apply Lemma \ref{lem:Qispacesequivalence2}
we let $\check\zeta\in\cD_2\subseteq\Xi((f,m),\psi_1\sqcup(\check f,\check m))^\ot$ denote the split subcategory spanned by objects $\zeta$ as in Notation \ref{nota:morphismdotcT} such that 
$\rho_{d,r}^{-1}(\rho_{d,r}(1))=\ul1\sqcup C\subseteq\ul1\sqcup\check A$;
and we let $\bar\zeta\in\cD_1\subseteq \Xi((f,m),\psi_1\sqcup(\bar f,\bar m))^\ot$ denote the split subcategory spanned by objects $\zeta$ such that 
$\rho_{d,r}^{-1}(\rho_{d,r}(1))=\set{1,\bar 1}\subseteq\ul1\sqcup\bar A$. Then $\varepsilon\circ-$ restricts to an equivalence $\cD_1\xrightarrow{\simeq}\cD_2$, inducing and equivalence $|\cD_1|\xrightarrow{\simeq}|\cD_2|$.

Moreover, the preimage $\cE_1$ of $\cD_1$ in $\Xi((\tilde f,\tilde m),(\bar f,\bar m))$ is the split subcategory spanned by objects $\zeta$ such that $G_r$ splits as a disjoint union of the map of finite sets $(\bar{\ul1}\xrightarrow{\simeq}\tilde{\ul1})$ and another gaf; and the preimage $\cE_2$ of $\cD_2$ in $\Xi((\tilde f,\tilde m),(\check f,\check m))$ is the split subcategory spanned by objects $\zeta$ such that $G_r$ splits as a disjoint union of the map of finite sets $(C\to\tilde{\ul1})$ and another gaf; we have that $\zeta_C\circ-$ restricts to an equivalence $\cE_1\xrightarrow{\simeq}\cE_2$, inducing and equivalence $|\cE_1|\xrightarrow{\simeq}|\cE_2|$.
\end{proof}
\begin{proof}[Proof of  Proposition \ref{prop:finalthmA}(2)]
We reduce to the setting in which $G_d$ consists of a single vertex using Corollary \ref{cor:singlevertexGd}; we then apply
Lemmas \ref{lem:aotfullyfaithful} and \ref{lem:Q1veeequivalences}.
\end{proof}

\subsection{The spine reduction}
The goal of this subsection and the next is to prove Proposition \ref{prop:finalthmA}(1). We fix $\cO$ and $(f,m)$ as in Notation \ref{nota:cOfm} throughout the next two subsections, with $(f,m)$ leaf-like and with $G_d$ consisting of a single vertex $x$. We distinguish two cases: the ``attaching case'' is the one in which $A=\set{x}$ and $V_d=\emptyset$; the ``inner case'' is the one in which $V_d=\set{x}$ and $A=\emptyset$. We let $v_\ell\in V_u$, $h_\ell\in H_u$ and $e_\ell\in E_u$ be as in Definition \ref{defn:leaf}. The key result of this subsection is that we can reduce to the case in which $G_u$ is a linear graph, i.e. a connected graph having precisely two vertices of valence 1 (the ``endpoints'') and all other vertices of valence 2.
More precisely, in the inner case we will reduce to the case in which $G_u$ is linear with unique edge $e_\ell$; in the attaching case we will reduce to the case in which $G_u$ is linear and its endpoints are $x\in A$ and $v_\ell$. Along the way we will replace $Q_1(f,m)^\to$ by a sequence of $\infone$-subcategories having equivalent classifying spaces.

The first step is 
to replace the categories $\Xi((\bar f,\bar m),(\check f,\check m))$, for varying $(\bar f,\bar m)$ and $(\check f,\check m)$ in $\dot\cK$, with smaller subcategories having equivalent classifying space.

\begin{rem}
\label{rem:Gronlytrees}
Let $\bar D\in Q_1(f,m)^\to$ as in Notation \ref{nota:Qiobject}; then each edge in the gaf $\bar G_l\circhor\bar G_d\circhor\bar G_r$ is contracted along $\bar f_d$, since $E_d=\emptyset$. Hence $\bar G_l\circhor\bar G_d\circhor\bar G_r$, and a fortiori its subgaf $\bar G_r$, are disjoint unions of (based and unbased) trees.
\end{rem}

\begin{defn}
\label{defn:Xi**}
Let $\bar\cO,\check\cO\in\dot\cT$ be as in Notation \ref{nota:cOfm}. We denote by $\Xi^*(\bar\cO,\check\cO)$ the full subcategory of $\Xi(\bar\cO,\check\cO)$ spanned by objects $\zeta$ as in Notation \ref{nota:morphismdotcT} such that every edge of $G_{d,r}\circhor\check G_d\circhor G_{d,l}$ coming from $G_{d,r}$, as well as every leaf coming from $G_{d,l}$, is \emph{not} contracted along $f_d$.

We further denote by $\Xi^{**}(\bar\cO,\check\cO)\subseteq\Xi^*(\bar\cO,\check\cO)$ the full subcategory spanned by objects $\zeta$ such that, in addition, $(f_r,m_r)$ is an equivalence, and no leaf in $G_{u,l}$ is contracted along $f_{d,l}$. We will in this case simplify the notation of the data associated with $\zeta$ and abbreviate $G_r:=G_{u,r}\simeq G_{d,r}$.

For $(\bar f,\bar m),(\check f,\check m)$ in $\dot\cL$ or in $\dot\cK$ we define similarly the full subcategories
\[
\Xi^{**}((\bar f,\bar m),(\check f,\check m))\subseteq\Xi^*((\bar f,\bar m),(\check f,\check m))\subseteq\Xi((\bar f,\bar m),(\check f,\check m)).
\]
Whenever we have a split subcategory inclusion $\cD\subseteq\Xi((\bar f,\bar m),(\check f,\check m))$, we denote by $\cD^{**}$ the intersection $\cD\cap\Xi^{**}((\bar f,\bar m),(\check f,\check m))$.
\end{defn}
\begin{lem}
\label{lem:Xi*reduction}
Let $(\bar f,\bar m),(\check f,\check m)$ be objects in $\dot\cL$ or in $\dot\cK$. Then the inclusion $\Xi^{**}((\bar f,\bar m),(\check f,\check m))\subseteq\Xi^*((\bar f,\bar m),(\check f,\check m))$ admits a left adjoint, and the inclusion $\Xi^*((\bar f,\bar m),(\check f,\check m))\subseteq\Xi((\bar f,\bar m),(\check f,\check m))$ admits a right adjoint.
\end{lem}
\begin{proof}
We omit ``$((\bar f,\bar m),(\check f,\check m))$'' from the notation throughout the proof.
For the first claim, given an object $\zeta\in\Xi^*$, we define a candidate left adjoint object $\zeta'\in\Xi^{**}$ together with a morphism $\eta\colon\zeta\to\zeta'$ as follows. We define both $G'_{u,r}$ and $G'_{d,r}$ as $G_{d,r}\in\cC(\check A,\bar A)$, we let 
$G'_{d,l}:= G_{d,l}\in\cC(\bar B,\check B)$, and we let
$G'_{u,l}\in\cC(\bar B,\check B)$
be the gaf obtained from $G_{u,l}$ by iteratively collapsing all leaves that are collapsed along $f_l$. We have essentially unique factorisations
\[
(f_\star,m_\star)\colon G_{u,\star}\xrightarrow{(f^*_\star,m^*_\star)}G'_{u,\star}\xrightarrow{(f'_\star,m'_\star)} G'_{d,\star}=G_{d,\star}
\]
for $\star=l,r$, such that $f_r^*$ is the identity of $G_{d,r}$, and $f_l^*$ is the aforementioned leaf-collapsing map of gaf's. We further set $(f'_d,m'_d):=(f_d, m_d)$ and $(f'_u,m'_u):=((f^*_r,m^*_r)\circhor \check G_u\circhor(f^*_l,m^*_l))\circ(f_u,m_u)$. We let $\eta$ be the morphism induced by the morphisms $(f^*_\star,m^*_\star)$ for $\star=l,r$, in the light of the above factorisations of $(f_\star,m_\star)$ and $(f'_u,m'_u)$. It is immediate that for any $\zeta''\in\Xi^{**}$ we have a bijection of sets
$-\circ\eta\colon \Xi(\zeta',\zeta'')\xrightarrow{\simeq}\Xi(\zeta,\zeta'')$.

For the second claim, given $\zeta'\in\Xi$, we define $\zeta\in\Xi^*$ and $\varepsilon\colon\zeta\to\zeta'$ as follows: we let $G_{d,r}$ be the gaf obtained from $G'_{d,r}$ by collapsing all edges that are collapsed along $f'_d$, and we let $G_{d,l}$ be the gaf obtained from $G'_{d,l}$ by iteratively collapsing all leaves that are collapsed along $f'_d$; there are then essentially unique 2-morphisms $(f^*_\star,m^*_\star)\colon G'_{d,\star}\to G_{d,\star}$ and $(f_d,m_d)\colon G_{d,r}\circhor\check G_d\circhor G_{d,l}\to\bar G_d$ in $\cC$ fitting into a factorisation of 2-morphisms $(f'_d,m'_d)=(f_d,m_d)\circ((f^*_r,m^*_r)\circhor G'_d\circhor(f^*_l,m^*_l))$, such that $f^*_\star$ is the edge-collapsing or iterated leaf-collapsing map described above.
We further set $G_{u,\star}:= G'_{u,\star}$, $(f_\star,m_\star):=(f^*_\star,m^*_\star)\circ(f'_\star,m'_\star)$ for $\star=l,r$, and $(f_u,m_u)=(f'_u,m'_u)$. We let $\varepsilon$ be the morphism induced by $(f^*_\star,m^*_\star)$ in the light of the above factorisations. It is immediate that for any $\zeta''\in\Xi^*$ we have a bijection of sets $\varepsilon\circ-\colon \Xi(\zeta'',\zeta)\xrightarrow{\simeq}\Xi(\zeta'',\zeta')$.
\end{proof}
\begin{nota}
\label{nota:Q1object*}
Thanks to Lemma \ref{lem:Xi*reduction}, we will be able henceforth to represent any object in $Q_1(f,m)^\to$, up to equivalence, by a special diagram $\bar D$ as in Notation \ref{nota:Qiobject}, such that $\bar\zeta\in\Xi^{**}((f,m),(\bar f,\bar m))$. We denote by $\bar x:=\bar\rho_r(1)\in A\sqcup\bar V_r$ the image of the unique element $1\in\ul1\subseteq\ul1\sqcup\bar A$.
\end{nota}
\begin{rem}
\label{rem:uniquebijection}
Let $\bar D\in Q_1(f,m)^\to$ as in Notation \ref{nota:Q1object*}; then Remark \ref{rem:Gronlytrees} implies that $\bar G_r$ is a gaf without edges. We further observe that the image of $\bar\rho_r\colon\ul1\sqcup\bar A\to A\sqcup\bar V_r$ contains all of $\bar V_r$: if we had an element $v\in \bar V_r$ which is not in the image of $\bar\rho_r$, then $v\neq\bar x$ and $v$ would constitute on its own a component in $\bar G_r\circhor\bar G_u\circhor\bar G_l$; however the latter gaf must be connected, since it is a quotient of the tree $G_u$. 
\end{rem}
\begin{defn}
We denote by $Q_1^\dagger(f,m)^\to\subseteq Q_1(f,m)^\to$ the full $\infone$-subcatego\-ry spanned by objects $\bar D$ as in Notation \ref{nota:Q1object*} satisfying the following properties:
\begin{itemize}
\item[(v)] $\bar G_d$ has no leaves; moreover $A\sqcup \bar V_r=A\cup\set{\bar x}$, i.e., each vertex in $\bar G_r$ is an attaching vertex or is equal to $\bar x$ (or both).
\end{itemize}
We denote by $Q_1^{\dagger\dagger}(f,m)^\to\subseteq Q_1^\dagger(f,m)^\to$ the full $\infone$-subcategory spanned by objects $\bar D$ satisfying the following additional requirement:
\begin{itemize}
\item[(vi)] $\bar G_u$ has no leaves; moreover $\bar G_u$ and $\bar G_d$ have no component consisting of a single inner vertex.
\end{itemize}
\end{defn}

\begin{lem}
\label{lem:Q1daggerequivalence}
The inclusions $Q_1^{\dagger\dagger}(f,m)^\to\subseteq Q_1^\dagger(f,m)^\to\subseteq Q_1(f,m)^\to$ induce equivalences on classifying spaces.
\end{lem}
\begin{proof}
We will prove that the first inclusion admits a left adjoint, and the second a right adjoint. 
For the first claim, given $\bar D\in Q^\dagger_1(f,m)^\to$ as in Notation \ref{nota:Q1object*}, we define $\check D\in Q_1^{\dagger\dagger}(f,m)^\to$ together with a morphism $\eta\colon\bar D\to\check D$ as follows:
\begin{itemize}
\item we let $G^{\dagger\dagger}\in\cC(\bar B,\bar A)$ denote the gaf obtained from $\bar G_u$ by iteratively collapsing leaves; since $\bar G_d$ has no leaves, we have a factorisation in $\cC(\bar B,\bar A)$
\[
(\bar f,\bar m)\colon \bar G_u\xrightarrow{(f^\dagger_u,m^\dagger_u)}G^{\dagger\dagger}\xrightarrow{(f^{\dagger\dagger},m^{\dagger\dagger})}\bar G_d;
\]
\item we have an essentially unique horizontal factorisation of $(f^{\dagger\dagger},m^{\dagger\dagger})$ as
\[
\begin{tikzcd}[column sep=90pt]
\bar B\ar[r,"G^\dagger_l"]&\check B\ar[r,bend left=25,"\check G_u"'{name=U}]\ar[r,bend right=25,"\check G_d"{name=D}]&\check A:=\bar A,
\ar[from=U, to=D, Rightarrow,"{(\check f,\check m)}"near start]
\end{tikzcd}
\]
where $\check G_u,\check G_d$ have no components consisting of a single inner vertex, $G^\dagger_l$ has no edge, and the restricted map $\rho_l^\dagger\colon(\rho_l^\dagger)^{-1}(\check B)\xrightarrow{\simeq}\check B$ is bijective; this restricts to horizontal factorisations $G^{\dagger\dagger}= \check G_u\circhor G^\dagger_l$ and $\bar G_d=\check G_d\circhor G^\dagger_l$;
\item we let $G^\dagger_r$ be the identity gaf of $\check A:=\bar A$, and we let
$(f^\dagger_d,m^\dagger_d)\colon \bar G_d\xrightarrow{\simeq}G^\dagger_r\circhor\check G_d\circhor G^\dagger_l $ be the above identification;
\item the above data define a morphism $\zeta^\dagger\colon(\bar f,\bar m)\to(\check f,\check m)$ in $\dot\cL$; we let $\check\zeta:=(\Id_{\psi_1}\sqcup\zeta^\dagger)\circ\bar\zeta$;
\item the above data define $\check D$, represented by a diagram as in Notation \ref{nota:Q1object*}; by definition of $\check\zeta$, the morphism $\eta:=\Id_{\psi_1}\sqcup\zeta^\dagger$ upgrades to a morphism $\bar D\to\check D$ in $Q_1(f,m)$.
\end{itemize}
Given $\tilde D\in Q^{\dagger\dagger}_1(f,m)^\to$, we argue that $-\circ\eta\colon Q_1(f,m)(\check D,\tilde D)\xrightarrow{\simeq}Q_1(f,m)(\bar D,\tilde D)$ is an equivalence by checking that the hypotheses of 
Lemma \ref{lem:Qispacesequivalence1} are fulfilled.
Let $\tilde\zeta\in\cD\subseteq\Xi((f,m),\psi_1\sqcup(\tilde f,\tilde m))$ denote the split subcategory comprising objects $\zeta$ as in Notation \ref{nota:morphismdotcT} such that the following holds: at most one component of $G_{d,r}$ is an unbased tree, and if such component exists, then it contains the element $\rho_{d,r}(1)$, where $1\in\ul1\subseteq\ul1\sqcup\tilde A$.
Let moreover $\cE_1\subseteq\Xi((\check f,\check m),(\tilde f,\tilde m))$ and $\cE_2\subseteq\Xi((\bar f,\bar m),(\tilde f,\tilde m))$ denote the split subcategories spanned by objects $\zeta$ such that, in both cases, the following holds: no component of $G_{d,r}$ is an unbased tree. Then we have $((\psi_1\sqcup-)\circ\check\zeta)^{-1}(\cD)\subseteq\cE_1$: indeed, if $\zeta\in\Xi((\check f,\check m),(\tilde f,\tilde m))$ and if we know that $\check G_r\circhor(\ul1\sqcup G_{d,r})$ has at most one unbased tree marked by $\ul1$, then $G_{d,r}$ cannot contain any unbased trees. Similarly, $((\psi_1\sqcup-)\circ\bar\zeta)^{-1}(\cD)\subseteq\cE_2$.

We observe that the functor 
$-\circ\zeta^\dagger \colon\Xi((\check f,\check m),(\tilde f,\tilde m))\to\Xi((\bar f,\bar m),(\tilde f,\tilde m))$
restricts to a functor $\Xi^{**}((\check f,\check m),(\tilde f,\tilde m))\to\Xi^{**}((\bar f,\bar m),(\tilde f,\tilde m))$, as a consequence of the fact that $(f_d^\dagger,m_d^\dagger)$ is an equivalence.
Moreover $\cE^{**}_2\subseteq\Xi^{**}((\bar f,\bar m),(\tilde f,\tilde m))$ may be characterised as the split subcategory spanned by objects $\zeta$ as in Notation \ref{nota:morphismdotcT} such that $G_r$ contains no isolated inner vertices. We observe the following for $\zeta\in\cE^{**}_2$.
\begin{itemize}
\item Since $\tilde G_u$ contains no isolated inner vertex and no leaf, we have that any leaf in the gaf $G_r\circhor\tilde G_u\circhor G_{u,l}$ must come from a leaf in $G_{u,l}$; however, since $\zeta\in \Xi^{**}((\check f,\check m),(\tilde f,\tilde m))$, we have that also $G_{u,l}$ has no leaves and we conclude that $G_r\circhor\tilde G_u\circhor G_{u,l}$ has no leaves as well. It follows that for $\zeta\in \cE^{**}_2$ we have that $f_u\colon \bar G_u\to G_r\circhor\tilde G_u\circhor G_{u,l}$ must iteratively collapse all leaves of $\bar G_u$, i.e. $f_u$ factors through $f^\dagger_u$.
\item Recall that the set of isolated vertices of $\bar G_d$ is in bijection with $V_l^\dagger$. For an isolated inner vertex $v$ of $\bar G_d$, the unbased tree $f_d^{-1}(v)$, which is a subtree of $G_r\circhor\tilde G_d\circhor G_{d,l}$, cannot contain any inner vertices coming from $G_r$, because the latter has no leaves and no isolated inner vertices. It follows that $f_d^{-1}(v)$ can be regarded as an unbased subtree of $\tilde G_d\circhor G_{d,l}$; again, however, no inner vertices can come from $\tilde G_d$. We conclude that $f_d^{-1}(v)$ is an unbased subtree of $G_{d,l}$, and since the latter has no leaves, we have that $f_d^{-1}(v)$ consists of a single inner vertex coming from $G_{d,l}$. Moreover, every isolated inner vertex in $G_{d,l}$ has preimage along $f_l$ consisting again of an isolated inner vertex, as $G_{u,l}$ has no leaves. It follows that the 2-morphism $(f_l,m_l)$ canonically factors as the horizontal composition of $G_l^\dagger$ and another 2-morphism. 
\end{itemize}
The previous remarks show that the restricted functor $-\circ\zeta^\dagger\colon\cE^{**}_1\to\cE^{**}_2$ is an equivalence of categories; we have therefore $-\circ\zeta^\dagger\colon|\cE^{**}_1|\xrightarrow{\simeq}|\cE^{**}_2|$ and hence, by Lemma \ref{lem:Xi*reduction}, we also have $-\circ\zeta^\dagger\colon|\cE_1|\xrightarrow{\simeq}|\cE_2|$.

For the second claim, given $\check D\in Q_1(f,m)^\to$ as in Notation \ref{nota:Q1object*}, we define $\bar D\in Q_1^{\dagger}(f,m)^\to$ together with a morphism $\varepsilon\colon\bar D\to\check D$ as follows:
\begin{itemize}
\item we factor $\check G_r$ in the essentially unique way as a horizontal composition of gaf's without edges $\ul1\sqcup\check A\xrightarrow{\ul1\sqcup G_r^\dagger}\ul1\sqcup\bar A\xrightarrow{\bar G_r}A$, where $\bar V_r=\bar V_r\cap\set{\bar\rho_r(1)}$, and the restriction $\rho^\dagger_r\colon(\rho^\dagger_r)^{-1}(\bar A)\xrightarrow{\simeq}\bar A$ is a bijection;
\item we let $G^\dagger_l$ be the identity gaf of $\bar B:=\check B$, and we let $\bar G_l:=\check G_l$;
\item we let $\bar G_d\in\cC(\bar B,\bar A)$ denote the gaf obtained from $G_r^\dagger\circhor\check G_d$ by iteratively collapsing all leaves, and we let $\bar G_u= G_r^\dagger\circhor \check G_u$;
\item we factor $(\check f_d,\check m_d)=(\bar f_d,\bar m_d)\circ(\bar G_r\circhor(\ul1\sqcup (f^\dagger_d,m^\dagger_d))\circhor\bar G_l)$ in the essentially unique way for which $f^\dagger_d\colon G_r^\dagger\circhor \check G_d\circhor G^\dagger_l\simeq  G_r^\dagger\circhor \check G_d\to\bar G_d$ is the above leaf-collapse map;
\item we let $(f^\dagger_u,m^\dagger_u)\colon \bar G_u\xrightarrow{\simeq}G_r^\dagger\circhor \check G_u\circhor G^\dagger_l$ be the obvious identification, and we let $(\bar f_u,\bar m_u):=(\check f_u,\check m_u)$;
\item the above data define $\bar D$ and a morphism $\zeta^\dagger\colon(\bar f,\bar m)\to(\check f,\check m)$ in $\dot\cL$; the morphism $\varepsilon:=\Id_{\psi_1}\sqcup\zeta^\dagger$ in $\dot\cK$ upgrades to a morphism $\bar D\to\check D$ in $Q_1(f,m)$.
\end{itemize}
Given $\tilde D\in Q_1^\dagger(f,m)^\to$, we argue that $\varepsilon\circ-\colon Q_1(f,m)(\tilde D,\bar D)\xrightarrow{\simeq} Q_1(f,m)(\tilde D,\check D)$ is an equivalence by checking that the hypotheses of Lemma \ref{lem:Qispacesequivalence2} are fulfilled.

We let $\cD_1\subseteq\Xi((f,m),\psi_1\sqcup(\bar f,\bar m))$ denote the split subcategory spanned by objects $\zeta$ as in Notation \ref{nota:morphismdotcT} such that $G_{d,r}$ contains at most one unbased tree, and if it does, this unbased tree contains the vertex $\rho_{d,r}(1)$; we similarly let $\cD_2\subseteq\Xi((f,m),\psi_1\sqcup(\check f,\check m))$ denote the split subcategory spanned by objects $\zeta$ such that the following holds: there is a bijection between the set of components $\pi_0(G_{d,r})$ of $G_{d,r}$ and the set $A\sqcup \check V_r$, such that the inclusion $A\hto A\sqcup\check V_r$ agrees with the composite $A\hto\pi_0(G_{d,r})\simeq A\sqcup\check V_r$,
and such that the composite map of finite sets $\ul1\sqcup\check A\xrightarrow{\rho_{d,r}} A\sqcup V_{d,r}\twoheadrightarrow \pi_0(G_{d,r})\simeq A\sqcup\check V_r$ coincides with $\check\rho_r$. Note that by Remark \ref{rem:uniquebijection}, such a bijection $\pi_0(G_{d,r})\simeq A\sqcup\check V_r$, if it exists, is unique.

We claim that $\varepsilon\circ-$ restricts to an equivalence $|\cD_1|\xrightarrow{\simeq}|\cD_2|$. For this, observe that $\varepsilon\circ-$ restricts to a functor $\Xi^{**}((f,m),\psi_1\sqcup(\bar f,\bar m))\to\Xi^{**}((f,m),\psi_1\sqcup(\check f,\check m))$, as a consequence of the fact that $G^\dagger_r$ has no edges and $G^\dagger_l$ is an equivalence.
By Lemma \ref{lem:Xi*reduction}, we reduce to proving that the restricted functor $\varepsilon\circ-\colon\cD_1^{**}\to\cD_2^{**}$ induces an equivalence on classifying spaces.
In fact $\cD_1^{**}$ is characterised as the split subcategory of $\Xi^{**}((f,m),\psi_1\sqcup(\bar f,\bar m))$ spanned by objects $\zeta$ such that $V_r=V_r\cap\set{\rho_r(1)}$; similarly, $\cD^{**}_2\subseteq\Xi^{**}((f,m),\psi_1\sqcup(\check f,\check m))$ is the split subcategory spanned by objects $\zeta$ such that there is a (unique) bijection $V_r\simeq \check V_r$ such that $\check\rho_r$ factors as $\ul1\sqcup\check A\xrightarrow{\rho_r} A\sqcup V_r\xrightarrow{\simeq}A\sqcup\check V_r$. The restricted functor
$\varepsilon\circ-\colon\cD_1^{**}\to\cD_2^{**}$ is already an equivalence of categories.

We further let $\cE_1\subseteq\Xi((\tilde f,\tilde m),(\bar f,\bar m))$ denote the split subcategory spanned by objects $\zeta$ as in Notation \ref{nota:morphismdotcT} such that no component of $G_{d,r}$ is an unbased tree; and we let $\cE_2\subseteq\Xi((\tilde f,\tilde m),(\check f,\check m))$ denote the split subcategory spanned by objects $\zeta$ such that the following holds: let $\pi_0(G_{d,r}^{\mathrm{ut}})$ denote the set of components of $G_{d,r}$ that are unbased trees, and consider the subset $A^\dagger:=(\rho_r^\dagger)^{-1}(V^\dagger_r)\subseteq\check A$; then there exists a (necessarily unique) bijection $V^\dagger_r\simeq \pi_0(G_{d,r}^{\mathrm{ut}})$, such that the following composite maps of finite sets agree:
\[
A^\dagger\xrightarrow{\rho_r^\dagger} V^\dagger_r\simeq \pi_0(G_{d,r}^{\mathrm{ut}})\hto \pi_0(G_{d,r});\quad A^\dagger\xrightarrow{\rho_{d,r}}\tilde A\sqcup V_{d,r}\twoheadrightarrow\pi_0(G_{d,r}). 
\]
We have that $\cE_1$ and $\cE_2$ contain the preimages of $\cD_1$ and $\cD_2$ along the two functors denoted $(\psi_1\sqcup-)\circ\tilde\zeta$, respectively.
Moreover $\zeta^\dagger\circ-$ restricts to a functor $\cE_1\to\cE_2$ and to a functor $\Xi^{**}((\tilde f,\tilde m),(\bar f,\bar m))\to\Xi^{**}((\tilde f,\tilde m),(\check f,\check m))$. Finally, $\cE_1^{**}$ can be characterised as the full subcategory of $\Xi^{**}((\tilde f,\tilde m),(\bar f,\bar m))$ spanned by objects $\zeta$ as in Notation \ref{nota:morphismdotcT} such that $G_r$ has no isolated inner vertices; 
and $\cE_2^{**}$ as the full subcategory of $\Xi^{**}((\tilde f,\tilde m),(\check f,\check m))$ spanned by objects $\zeta$ for which the following holds: let $V_r^{\mathrm{iiv}}\subseteq V_r$ denote the set of isolated inner vertices of $G_r$; then there is a (necessarily unique) bijection $V_r^\dagger\simeq V_r^{\mathrm{iiv}}$ such that the composite
\[
A^\dagger\xrightarrow{\rho_r^\dagger}V_r^\dagger\simeq V_r^{\mathrm{iv}}\hto \tilde A\sqcup V_r
\]
agrees with $\rho_r|_{A^\dagger}$. The restricted functor $\zeta^\dagger\circ-\colon\cE_1^{**}\to\cE_2^{**}$ is an equivalence, inducing an equivalence $\zeta^\dagger\circ-\colon|\cE_1^{**}|\xrightarrow{\simeq}|\cE_2^{**}|$; it follows from Lemma \ref{lem:Xi*reduction} that $\zeta^\dagger\circ-\colon|\cE_1|\xrightarrow{\simeq}|\cE_2|$ is also an equivalence.
\end{proof}
\begin{nota}
\label{nota:spine}
We denote by $G^s_u$ the linear subgraph of $G_u$ defined as follows:
\begin{itemize}
\item in the inner case, $G^s_u$ consists of the unique edge $e_\ell$, together with the two inner vertices giving the endpoints of $e_\ell$;
\item in the attaching case, $G^s_u$ is the linear subgraph 
joining $x\in A$ with $v_\ell$.
\end{itemize}
We refer to $G^s_u$ as the ``spine'' of $G_u$, and denote by $E^s_u\subseteq E_u$ the set of edges of $G^s_u$. We have an essentially unique factorisation in $\cC$ of the form
\[
(f,m)\colon G_u\xrightarrow{(f^b_u,m^b_u)}G^s_u\xrightarrow{(f^s,m^s)}G_d,
\]
in which $f^s$ collapses all edges of $G_u^s$, whereas $f^b_u$ collapses all edges in $E_u\setminus E^s_u$.
Letting $G^b_l=\Id_B$, $G^b_r=\Id_A$ and $(f^b_d,m^b_d)=\Id_{G_d}$, we obtain a special object $\zeta^b\in\Xi((f,m),(f^s,m^s))$, giving a morphism $\zeta^b\colon(f,m)\to(f^s,m^s)$ in $\dot\cK$.
\end{nota}
\begin{defn}
\label{defn:bridge}
Let $\bar A,\bar B\in\Fin$ with $\#\bar A\ge2$. We say that a gaf $\bar G\in\cC(\bar B,\bar A)$ is a \emph{bridge} if there is a 2-element subset $\del\bar G\subseteq \bar A$ such that, when the marking by $\bar B$ is disregarded, $\bar G$ is the disjoint union of the attaching subset $\bar A\setminus\del\bar G$ and of a linear graph intersecting $\bar A$ at its set of endpoints $\del\bar G$.
For uniform notation, whenever a gaf $\bar G\in\cC(\bar B,\bar A)$ is \emph{not} a bridge, we declare $\del\bar G:=\emptyset\subseteq\bar A$.
\end{defn}

\begin{defn}
\label{defn:classificationQ1dagger}
We give a classification of objects in $Q_1^{\dagger\dagger}(f,m)^\to$, represented by diagrams as in Notation \ref{nota:Q1object*}, according to types.
\begin{itemize}
\item[(I)] An object $\bar D$ is of type I if $A\sqcup\bar V_r=\set{\bar x}$ is a singleton, where we use Notation \ref{nota:Q1object*}; in this case the following hold:
\begin{itemize}
\item $(\bar f,\bar m)$ is an equivalence, and the gaf's $\bar G_u\simeq\bar G_d$ are given by the same map of finite sets $\bar\rho_u=\bar\rho_d\colon\bar B\to\bar A$;
\item $\bar G_l$ is also given by a map of finite sets $\bar\rho_l\colon B\to \ul1\sqcup\bar B$, sending $\rho_u^{-1}(v_\ell)$ to $1\in\ul1$ and restricting to a map $B\setminus \rho_u^{-1}(v_\ell)\to\bar B$;
\item $\bar f_d\colon \bar G_r\circhor\bar G_d\circhor\bar G_l\to G_d$ is the unique equivalence, with $\bar f_d\colon\bar x\mapsto x$;
\item $\bar f_u$ collapses all edges of $G_u$ except $e_\ell$, and $\bar m_u$ is the restriction of $m$.
\end{itemize}
\end{itemize}

Note that if $\bar D$ is not of type I, then $A\sqcup\bar V_r$ has two elements $x\in A$ and $\bar x\in\bar V_r$. 
\begin{itemize}
\item[(II)] An object $\bar D$ not of type I is of type II if $(\bar f,\bar m)$ is an equivalence, and the gaf's $\bar G_u\simeq\bar G_d$ are given by the same map of finite sets $\bar \rho_u=\bar\rho_d\colon\bar B\to\bar A$; in this case the following hold:
\begin{itemize}
\item $\bar G_l$ is a bridge, and the set $\bar E_l$ corresponds along $\bar f_u$ to a nonempty subset $\bar X_l\subseteq E^s_u\setminus\set{e_\ell}$;
\item $\bar\rho_r\circ\bar\rho_u$ restricts to a bijection $\del\bar G_l\xrightarrow{\simeq}\set{x,\bar x}$; in particular $\bar\rho_u$ is injective on $\del\bar G_l$;
\item $\bar f_u$ collapses all edges in $E_u\setminus(\bar X_l\sqcup\set{e_\ell})$, and $\bar m_u$ 
is restricted from $m$;
\item $\bar f_d$ collapses all edges in $\bar X_l$, and $\bar m_d$ is the restriction of $m$;
\item $\bar\rho_l\colon B\to \ul1\sqcup\bar B\sqcup\bar V_l$ is a map of finite sets, sending $\rho_u^{-1}(v_\ell)$ to $1$ and the rest of $B$ to $\bar B\sqcup\bar V_l$.
\end{itemize}
\item[(III)] An object $\bar D$ is of type III if it is neither of type I nor of type II; in this case the following hold:
\begin{itemize}
\item $\bar G_u$ and $\bar G_d$ are bridges with $\del\bar G_u=\del\bar G_d\subseteq\bar A$, and the sets of edges $\bar E_u$ and $\bar E_d$ correspond, along $\bar f_u$ and $(\bar G_r\circhor\bar f\circhor\bar G_l)\circ\bar f_u$, to two nonempty, nested subsets $\bar X_d\subseteq\bar X_u\subseteq E^s_u\setminus\set{e_\ell}$, respectively;
\item $\bar\rho_r$ restricts to a bijection $\del\bar G_u\xrightarrow{\simeq}\set{x,\bar x}$;
\item $\bar f_u$, $\bar f$ and $\bar f_d$ collapse the sets of edges $E_u\setminus(\bar X_u\sqcup\set{e_\ell})$, $\bar X_u\setminus\bar X_d$ and $\bar X_d$, respectively, and $\bar m_u$, $\bar m$ and $\bar m_d$ are restricted from $m$;
\item $\bar G_l$ is given by a map of finite sets sending $\rho_u^{-1}(v_\ell)$ to $1\in\ul1$ and the rest of $B$ to $\bar B$.
\end{itemize}
\end{itemize}
\end{defn}

\begin{rem}
We observe that only objects of type I may occure in the inner case $A=\emptyset$; the same holds in the attaching case if the edge $e_\ell$ of $G_u$ is attached directly to $x\in A$, i.e., if $E^s_u=\set{e_\ell}$.
\end{rem}
We are ready for the main result of the subsection.
\begin{lem}
\label{lem:linearreduction}
Recall Notation \ref{nota:spine}, and
consider the functor $-\circ\zeta^b\colon Q_1(f^s,m^s)\to Q_i(f,m)$ induced by the morphism $\zeta^b$ in the light of the pullback definition of $Q_1(f,m)$ and $Q_1(f^s,m^s)$ from Lemma \ref{lem:defQi}. Then $-\circ\zeta^b$ restricts to an equivalence $Q_1^{\dagger\dagger}(f^s,m^s)^\to\xrightarrow{\simeq}Q_1^{\dagger\dagger}(f,m)^\to$. 
\end{lem}
\begin{proof}
First, we observe that $-\circ\zeta^b$ indeed restricts to a functor $Q_1^{\dagger\dagger}(f^s,m^s)^\to\to Q_1^{\dagger\dagger}(f,m)^\to$, since $f^b_u\colon h_\ell\mapsto h_\ell$ and since $G^b_\star$ is an identity gaf for $\star=l,r$. The restricted functor is moreover essentially surjective: given $\bar D\in Q_1^{\dagger\dagger}(f,m)^\to$ as in Definition \ref{defn:classificationQ1dagger}, for all types I, II and III we have that $\bar f_u$ must collapse all edges in $E_u\setminus E^s_u$, so $(\bar f_u,\bar m_u)$ factors through $(f^b_u,m^b_u)$.

Finally, we apply Lemma \ref{lem:Qispacesequivalence3} to show that the restricted functor is fully faithful: given $\bar D,\check D\in Q_1^{\dagger\dagger}(f^s,m^s)^\to$,
we let $\check\zeta\in\cD_1\subseteq\Xi((f^s,m^s),\psi_1\sqcup(\check f,\check m))^\to$ and $\check\zeta\circ\zeta^b\in\cD_2\subseteq\Xi((f,m),\psi_1\sqcup(\check f,\check m))^\to$ denote the split subcategories spanned in both cases by objects $\zeta$ as in Notation \ref{nota:morphismdotcT} such that
at most one component of $G_{d,r}$ is an unbased tree, and if it exists it contains $\rho_{d,r}(1)$.
We have that $-\circ\zeta^b$ restricts to an equivalence $\cD_1^{**}\xrightarrow{\simeq}\cD_2^{**}$, inducing an equivalence $|\cD_1^{**}|\xrightarrow{\simeq}|\cD_2^{**}|$; by Lemma \ref{lem:Xi*reduction} we conclude that $-\circ\zeta^b\colon|\cD_1|\xrightarrow{\simeq}|\cD_2|$ is also an equivalence.
\end{proof}
Lemma \ref{lem:linearreduction} allows us to reduce the proof of Proposition \ref{prop:finalthmA}(1) to the case in which $G_u$ is linear and coincides with its own spine. We conclude the subsection with a further mild reduction, which we operate mostly to save notation later. The proof of the following lemma is left to the reader.
\begin{lem}
\label{lem:Balmostempty}
Let $B':=B\setminus\rho_u^{-1}(v_\ell)\subseteq B$, let $G_l$ be the gaf corresponding to the inclusion of finite sets $B'\hto B$, let $(f',m'):=(f,m)\circhor G_r$, and let $\zeta\in\Xi((f',m'),(f,m))$ denote the special object defined by the previous factorisation, by letting $G_l$, $f_u$ and $f_d$ be equivalences. Then the functor $-\circ\zeta\colon Q_1(f,m)\to Q_1(f',m')$ induced by $\zeta$ in the light of the pullback definition of $Q_1(f,m)$ and $Q_1(f',m')$ restricts to an equivalence $-\circ\zeta\colon Q_1(f,m)^\to\xrightarrow{\simeq} Q_1(f',m')^\to$.
\end{lem}
Lemma \ref{lem:Balmostempty} allows us to further reduce to the setting in which $\rho_u^{-1}(v_\ell)=\emptyset$.

\subsection{Contractibility of \texorpdfstring{$|Q_1(f,m)^\to|$}{|Q1(f,m)to|}
}
In this subsection we conclude the proof of Proposition \ref{prop:finalthmA}(1), under the assumption that $G_d$ consists of a single vertex, the tree $G_u$ is linear and coincides with its own spine, and $\rho_u^{-1}(v_\ell)=\emptyset$.

\begin{defn}
\label{defn:Q1flat}
Recall Definition \ref{defn:bridge}.
We denote by $Q_1^\flat(f,m)^\to\subseteq Q_1^{\dagger\dagger}(f,m)^\to$ the full $\infone$-subcategory spanned by objects $\bar D$ as in Definition \ref{defn:classificationQ1dagger} satisfying the following property:
\begin{itemize}
\item[(vii)] $\bar\rho_l$ restricts to a bijection $\bar\rho_l^{-1}(\bar B\setminus\del\bar G_l)\xrightarrow{\simeq}\bar B\setminus\del\bar G_l$.
\end{itemize}
We denote by $Q_1^{\flat\flat}(f,m)^\to\subseteq Q_1^\flat(f,m)^\to$ the full $\infone$-subcategory spanned by objects $\bar D$ satisfying
the following additional property:
\begin{itemize}
\item[(viii)] $\bar\rho_u$ restricts to a bijection $\bar\rho_u^{-1}(\bar A\setminus(\del\bar G_u\cup\bar\rho_u(\del\bar G_l)))\xrightarrow{\simeq}\bar A\setminus(\del\bar G_u\cup\bar\rho_u(\del\bar G_l))$.
\end{itemize}
\end{defn}
\begin{lem}
\label{lem:Q1flatequivalence}
The inclusions $Q_1^{\flat\flat}(f,m)^\to\subseteq Q_1^\flat(f,m)^\to\subseteq Q_1^{\dagger\dagger}(f,m)^\to$ induce equivalences on classifying spaces.
\end{lem}
\begin{proof}
We prove that the first inclusion admits a left adjoint, and the second a right adjoint. For the first claim, given $\bar D\in Q_1^\flat(f,m)^\to$, we define a candidate left adjoint object $\check D\in Q_1^{\flat\flat}(f,m)^\to$ together with a morphism $\eta\colon\bar D\to\check D$ as follows:
\begin{itemize}
\item we let $\check B$ be the following disjoint union of two finite sets:
\[
\check B:=(\del\bar G_u\cup\bar\rho_u(\del\bar G_l))\sqcup \bar\rho_u^{-1}(\bar A\setminus(\del\bar G_u\cup\bar\rho_u(\del\bar G_l)));
\]
we let $G_r^\flat$ be the gaf corresponding to the map of finite sets $\check B\to\bar B$ restricting to the inclusion, respectively to $\bar\rho_u$, on the two parts;
\item we let $(\bar f,\bar m)= G_r^\flat\circhor(\check f,\check m)$ be the essentially unique horizontal factorisation; in particular we have $\bar G_\star=G_r^\flat\circhor\check G_\star$ for $\star=u,d$;
\item we let the $G_l^\flat$ be the identity gaf of $\check A:=\bar A$;
we let $(f_u^\flat,m_u^\flat)$ be the identification $\bar G_u\xrightarrow{\simeq} G_r^\flat\circhor\check G_u\circhor G_l^\flat$ and $(f_d^\flat,m_u^\flat)$ be the identification $G_r^\flat\circhor\check G_d\circhor G_l^\flat\xrightarrow{\simeq}\bar G_d$;
\item we let $\check G_r:=\bar G_r\circhor (\ul1\sqcup G_r^\flat)$; and we let $\check G_l:=\bar G_l$ and $(\check f_\star,\check m_\star):=(\bar f_\star,\bar m_\star)$ for $\star=u,d$;
\item we upgrade $\eta:=\Id_{\psi_1}\sqcup\zeta^\flat$ to a morphism in $Q_1(f,m)$.
\end{itemize}
For $\tilde D\in Q_1^{\flat\flat}(f,m)^\to$, we show that $-\circ\eta\colon Q_1(f,m)(\check D,\tilde D)\xrightarrow{\simeq}Q_1(f,m)(\bar D,\tilde D)$ is an equivalence by checking that the hypotheses of Lemma \ref{lem:Qispacesequivalence1} are fulfilled. We let $\tilde\zeta\in\cD\subseteq\Xi((f,m),\psi_1\sqcup(\tilde f,\tilde m))$ denote the split subcategory spanned by objects $\zeta$ as in Notation \ref{nota:morphismdotcT} such that the following holds: for each $\tilde b\in(\tilde B\setminus\del\tilde G_l)$, the component of $G_{d,l}$ containing $\tilde b$ is a tree intersecting $\tilde B$ only in $\tilde b$, and whose preimage along $\rho_{d,l}$ consists only of the element $\tilde\rho_l^{-1}(\tilde b)\in B$. We let $\cE_1$ and $\cE_2$ be the preimages of $\cD$ along the functors $(\psi_1\sqcup-)\circ\check\zeta$ and $(\psi_1\sqcup-)\circ\bar\zeta$.

Since the data giving $\zeta^\flat$ are all equivalences or maps of finite sets, we have that $-\circ\zeta^\flat$
restricts to a functor $\Xi^{**}((\check f,\check m),(\tilde f,\tilde m))\to\Xi^{**}((\bar f,\bar m),(\tilde f,\tilde m))$. 
We claim that the restricted functor
$-\circ\zeta^\flat\colon\cE_1^{**}\to\cE_2^{**}$ is an equivalence. For this, let $\zeta\in\cE_2^{**}$; denote by $\bar B':=\bar\rho_u^{-1}(\bar A\setminus(\del \bar G_u\cup\rho_u(\del\bar G_l))\subseteq\bar B$, and let $B':=\bar \rho_l^{-1}(\bar B')\subseteq B$; then $\bar\rho_l\colon B'\xrightarrow{\simeq}\bar B'$ is a bijection. Let moreover $\tilde A':=\rho_{u,r}^{-1}(\bar A\setminus(\del \bar G_u\cup\rho_u(\del\bar G_l))\subseteq\tilde A$. Then the composition of gaf's
\[
B\xrightarrow{\bar G_l} \ul1\sqcup\bar B\xrightarrow{G_{u,l}}\ul1\sqcup\tilde B\xrightarrow{\tilde G_u}\ul1\sqcup\tilde A\xrightarrow{G_{u,r}}\ul1\sqcup\bar A
\]
induces a composite map of finite sets as follows, using also that all based trees in $G_{u,\star}$ are isolated attaching vertices for $\star=l,r$:
\[
B'\xrightarrow{\bar\rho_l}\bar B'\xrightarrow{\rho_{u,l}}\tilde \rho_l(B')\xrightarrow{\tilde\rho_u}\tilde\rho_u\tilde\rho_l(B')\xrightarrow{\rho_{u,r}}\bar A\setminus(\del \bar G_u\cup\rho_u(\del\bar G_l)).
\]
All arrows but the last in the previous composite are bijections. This shows that $G_{u,r}$ factors through $G^\flat_r$ in a canonical way. The same argument shows that $G_{d,r}$ factors canonically through $G^\flat_r$; extending the argument to analyse morphisms in $\cE_1^{**},\cE_2^{**}$, we obtain that $-\circ\zeta^\beta\colon\cE_1^{**}\to\cE_2^{**}$ is an equivalence. Hence we have equivalences $-\circ\zeta^\beta\colon|\cE_1^{**}|\to|\cE_2^{**}|$ and, by Lemma \ref{lem:Xi*reduction}, $-\circ\zeta^\beta\colon|\cE_1|\to|\cE_2|$.

For the second claim, given $\check D\in Q_1^{\dagger\dagger}(f,m)^\to$, we define $\bar D\in Q_1^\flat(f,m)^\to$ together with a morphism $\varepsilon\colon\bar D\to\check D$ as follows:
\begin{itemize}
\item we let $\bar B$ denote the set $\del\check G_l\sqcup\check\rho_l^{-1}(\check B\setminus\del\check G_l)$, and let $G^\flat_l$ be the gaf corresponding to the map of finite sets $\bar B\to\check B$ restricting to the inclusion and to $\check\rho_l$, respectively, on the two parts;
\item we factor $\check G_l$  in the essentially unique way as a horizontal composition
\[
B\xrightarrow{\bar G_l}\ul1\sqcup\bar B\xrightarrow{\ul1\sqcup G_l^\flat}\ul1\sqcup\check B;
\]
\item 
we let $(\bar f,\bar m):=(\check f,\check m)\circhor G_l^\flat$, in particular $\bar G_\star:=\check G_\star\circhor G_l^\flat$ for $\star=u,d$;
\item we let $G_r^\flat$ be the identity gaf of $\bar A:=\check A$; we let $(f_u^\flat,m_u^\flat)$ be the identification $\bar G_u\xrightarrow{\simeq}G_r^\flat\circhor\check G_u\circhor G_l^\flat$, and $(f_d^\flat,m_d^\flat)$ be the identification $G_r^\flat\circhor\check G_d\circhor G_l^\flat\xrightarrow{\simeq} \bar G_d$;
\item we let $\bar G_r:=\check G_r$, $(\bar f_u,\bar m_u):=(\check f_u,\check m_u)$, and $(\bar f_d,\bar m_d):=(\check f_d,\check m_d)$; 
\item we upgrade $\varepsilon:=\Id_{\psi_1}\sqcup\zeta^\flat$ to a morphism in $Q_1(f,m)$.
\end{itemize}
For $\tilde D\in Q_1^\flat(f,m)^\to$, we show that $\varepsilon\circ-\colon Q_1(f,m)(\tilde D,\bar D)\xrightarrow{\simeq} Q_1(f,m)(\tilde D,\check D)$ by checking tha the hypotheses of Lemma \ref{lem:Qispacesequivalence2} are fulfilled. Let $\bar\zeta\in\cD_1\subseteq\Xi((f,m),\psi_1\sqcup(\bar f,\bar m))$ denote the split subcategory spanned by objects $\zeta$ as in Notation \ref{nota:morphismdotcT} such that the following holds: for each $\bar b\in \bar B\setminus\del\bar G_l$, the component of $G_{d,l}$ containing $\bar b$ is a tree intersecting $\bar B$ only in $\bar b$, and whose preimage along $\rho_{d,l}$ consists only of the element $\bar\rho_l^{-1}(\bar b)\in B$. Define similarly $\check\zeta\in\cD_2\subseteq\Xi((f,m),\psi_1\sqcup(\check f,\check m))$ as the split subcategory spanned by objects $\zeta$ such that the following holds: for each $\check b\in \check B\setminus\del\check G_l$, the component of $G_{d,l}$ containing $\check b$ is a tree intersecting $\check B$ only in $\check b$, and whose preimage along $\rho_{d,l}$ consists precisely of the set $\check\rho_l^{-1}(\check b)\subseteq B$.
Since the data giving $\zeta^\flat$ are all equivalences or maps of finite sets, we have that $\varepsilon\circ-$ and $\zeta^\flat\circ-$ restrict to functors $\Xi^{**}((f,m),\psi_1\sqcup(\bar f,\bar m))\to \Xi^{**}((f,m),\psi_1\sqcup(\check f,\check m))$ and $\Xi^{**}((\tilde f,\tilde m),(\bar f,\bar m))\to\Xi^{**}((\tilde f,\tilde m),(\check f,\check m))$, respectively.
The restricted functor $\varepsilon\circ-\colon\cD_1^{**}\xrightarrow{\simeq}\cD_2^{**}$ is an equivalence. By Lemma \ref{lem:Xi*reduction} we obtain an equivalence $\varepsilon\circ-\colon|\cD_1|\xrightarrow{\simeq}\cD_2$.
Similarly as above, we let $\cE_1$ and $\cE_2$ denote the preimages of $\cD_1$ and $\cD_2$ along the two functors denoted $((\psi_1\sqcup-)\circ\tilde\zeta)$, respectively.
An analogous but simpler argument as the one above shows that $\zeta^\flat\circ-$ restricts to an equivalence $\cE_1^{**}\xrightarrow{\simeq}\cE_2^{**}$; by Lemma \ref{lem:Xi*reduction} we obtain an equivalence $\zeta^\flat\circ-\colon|\cE_1|\xrightarrow{\simeq}|\cE_2|$.
\end{proof}

The $\infone$-category $Q_1^{\flat\flat}(f,m)^\to$ is small enough to allow a direct analysis. In fact we will show that it is a (plain) poset, with contractible classifying space.

\begin{defn}
Given a linear gaf $\bar G\in\cC(\bar B,\bar A)$ and a subset $T$ of its set of all edges and vertices, we say that $T$ is \emph{convex} if whenever $y,y'\in T$, then all edges and vertices in $\bar G$ lying between $y$ and $y'$ also belong to $T$. For a subset $T'\subset \bar A\sqcup\bar V\sqcup\bar E$ we abbreviate by $\bar G\setminus T'$ the set $(\bar A\sqcup\bar V\sqcup\bar E)\setminus T'$. A \emph{segment} in $\bar G\setminus T'$ is a maximal convex subset for $\bar G$ that is contained in $\bar G\setminus T'$ and that contains at least one edge.
\end{defn}

\begin{defn}
We associate with an object $\bar D\in Q_1^{\flat\flat}(f,m)^\to$ a \emph{profile} $\pi(\bar D)$, i.e. a list of invariants of the object; this is defined according to the type of $\bar D$.
\begin{itemize}
\item[(I)] We associate the empty profile, denoted $\pi(\bar D)=()$.
\item[(II)] We associate the profile $\pi(\bar D)=(\bar B^\circ,\bar A^\circ,\bar X)$ consisting of two nested subsets $\emptyset\subseteq\bar A^\circ\subseteq\bar B^\circ\subseteq B$, and a \emph{nonempty} subset $\bar X\subseteq E_u\setminus\set{e_\ell}$ consisting of all edges of a segment of $G_u\setminus(\set{e_\ell}\cup\rho_u(\bar B^\circ))$. We define $\bar B^\circ=\bar B\setminus\del\bar G_l$, considered as a subset of $B$ by taking the preimage along $\bar\rho_l$; we similarly define $\bar A^\circ:=\bar A\setminus\bar\rho_u(\del\bar G_l)$; and we define $\bar X$ as the set of edges of the segment of $G_u\setminus(\set{e_\ell}\cup\rho_u(\bar B^\circ))$ containing $\bar X_l=\bar f_u^{-1}(\bar E_l)$.
\item[(III)] We associate the profile $\pi(\bar D)=(\bar A^\circ,\bar X_u,\bar X_d)$ consisting of a subset $\emptyset\subseteq\bar A^\circ\subseteq B$ and two \emph{nonempty}, nested subsets $\bar X_d\subseteq X_u\subseteq E_u\setminus\set{e_\ell}$, such that $\bar X_d$ and $\bar X_u$ are contained in a segment of $G_u\setminus(\set{e_\ell}\cup\rho_u(\bar A^\circ))$. We define $\bar A^\circ:=\bar A\setminus\del\bar G_u$; and we define $\bar X_d\subseteq \bar X_u$ as the preimages of $\bar E_d$ and $\bar E_u$ along $(\bar G_r\circhor\bar f\circhor\bar G_l)\circ\bar f_u$ and $\bar f_u$, respectively.
\end{itemize}
We denote by $\cD_{\pi(\bar D)}\subseteq\Xi^{**}((f,m),\psi_1\sqcup(\bar f,\bar m))$ the split subcategory  spanned by those objects $\zeta$ such that the following holds: if $\zeta'$ is a special object connected to $\zeta$ by some morphism in $\Xi^{**}((f,m),\psi_1\sqcup(\bar f,\bar m))$, we have that $\zeta'$ and $(\bar f,\bar m)$ combine in a new object $\bar D'\in Q_1^{\flat\flat}(f,m)^\to$ with $\pi(\bar D')=\pi(\bar D)$.
\end{defn}

\begin{lem}
\label{lem:discretecore}
The core groupoid of $Q_1^{\flat\flat}(f,m)^\to$ is homotopy discrete, and its equivalence classes of objects are in bijection with all possible profiles.
\end{lem}
\begin{proof}
It is clear that every profile can be achieved by some object. Given an object $\bar D$, the moduli space of objects exhibiting the profile $\pi(\bar D)$ can be computed as the classifying space of $\cD_{\pi(\bar D)}$.
We have that $\cD_{\pi(\bar D)}$
is the terminal category whenever $\bar D$ is of type I or III. If instead $\bar D$ is of type II, then $\cD_{\pi(\bar D)}$ is equivalent to $\Tw(\cP_{\neq\emptyset}(\bar X))$, where $\cP_{\neq\emptyset}$ denotes the poset of nonempty subsets. Concretely, we associate with $\zeta\in\cD$ the pair of nonempty subsets $X_{d,l}\subseteq X_{u,l}\subseteq \bar X$ obtained as preimages of $E_{d,l}$ and $E_{u,l}$ along $(G_{u,r}\circhor \bar G_u\circhor f_l)\circ f_u$ and $f_u$, respectively. We conclude by observing that $|\Tw(\cP_{\neq\emptyset}(\bar X))|\simeq*$.
\end{proof}

\begin{rem}
\label{rem:morphismQ1flat}
Let $\bar D,\check D\in Q_1^{\flat\flat}(f,m)^\to$; using that $|\cD_{\pi(\check D)}|\simeq*$, we obtain that the morphism space $Q_1(f,m)(\bar D,\check D)$ is equivalent to the classifying space of the split subcategory $\cE_{\bar D,\check D}:=((\psi_1\sqcup-)\circ\bar\zeta)^{-1}(\cD_{\pi(\check D)})\subseteq\Xi((\bar f,\bar m),(\check f,\check m))$, or by Lemma \ref{lem:Xi*reduction}, to the classifying space of $\cE^{**}_{\bar D,\check D}$. A routine analysis shows the following:
\begin{itemize}
\item $\cE^{**}_{\bar D,\check D}$ is the terminal category whenever one of the following occurs:  
\begin{itemize}
\item $\bar D$ and $\check D$ are of type I;
\item $\bar D$ and $\check D$ are of type II, and we have inclusions $\bar A^\circ\subseteq\check A^\circ\subseteq\check B^\circ\subseteq\bar B^\circ$, and correspondingly $\bar X\subseteq\check X$;
\item $\bar D$ and $\check D$ are of type III, and we have inclusions $\bar X_d\subseteq\check X_d\subseteq\check X_u\subseteq\bar X_u$ and $\bar A^\circ\subseteq\check A^\circ$;
\end{itemize}
\item $\cE^{**}_{\bar D,\check D}\simeq\Tw(\cP(T))$ for some finite set $T$, where $\cP$ denotes the poset of all subsets, whenever $\bar D$ is of type III and one of the following occurs:
\begin{itemize}
\item $\check D$ is of type I; in this case we let $T=\bar X_u\setminus\bar X_d$;
\item $\check D$ is of type II, and we have inclusions $\bar A^\circ\subseteq\check A^\circ$ and $\bar X_d\subseteq\check X$; in this case we let $T=(\check X\cap\bar X_u)\setminus\bar X_d$;
\end{itemize}
\item in all other cases $\cE^{**}_{\bar D,\check D}$ is empty.
\end{itemize}
The previous analysis show in particular that $Q_1^{\flat\flat}(f,m)^\to$ is equivalent to a (finite) poset, whose underlying set of objects is the set of all possible profiles.
\end{rem}
\begin{proof}[Proof of Proposition \ref{prop:finalthmA}(1)]
By Corollary \ref{cor:singlevertexGd} and
Lemmas \ref{lem:linearreduction} and
\ref{lem:Balmostempty}, it suffices to consider the case of an object $(f,m)\in\dot\cK$ as in Notation \ref{nota:cOfm} such that $G_d$ consists of a single vertex, $G_u$ is a linear graph and coincides with its own spine, and $\rho_u^{-1}(v_\ell)=\emptyset\subseteq B$. By Lemmas 
\ref{lem:Q1daggerequivalence} and 
\ref{lem:Q1flatequivalence} it suffices to prove that $|Q_1^{\flat\flat}(f,m)^\to|$ is contractible. By Lemma \ref{lem:discretecore} and Remark \ref{rem:morphismQ1flat}, we have that $Q_1^{\flat\flat}(f,m)^\to$ is a finite poset.
We denote by $Q\subseteq Q_1^{\flat\flat}(f,m)^\to$ the full subposet spanned by objects of type I, and by objects $\bar D$ of type III whose profile satisfies the following property: $\bar X_u$ is maximal, i.e. it consists of all edges of the segment of $G_u\setminus(\set{e_\ell}\cup\rho_u(\bar A^\circ))$ containing it. The inclusion $Q\subseteq Q_1^{\flat\flat}(f,m)^\to$ admits a right adjoint: it fixes all objects of type I (indeed, it has to fix all objects of $Q$); it sends the object $\bar D$ of type III with profile $(\bar A^\circ,\bar X_u,\bar X_d)$ to the object of type III with profile $(\bar A^\circ,\bar X,\bar X_d)$, where $\bar X$ is the set of all edges of the segment of $G_u\setminus(\set{e_\ell}\cup\rho_u(\bar A^\circ))$ containing $\bar X_u$; and it sends $\bar D$ of type II with profile $(\bar B^\circ,\bar A^\circ,\bar X)$ to the object of type III with profile $(\bar A^\circ,\bar X,\bar X)$. It follows that the inclusion $Q\subseteq Q_1^{\flat\flat}(f,m)^\to$ induces an equivalence on classifying spaces. Finally, we have $|Q|\simeq*$ because $Q$ has a terminal object, namely the essentially unique object of type I.
\end{proof}

\section{The universal property of \texorpdfstring{$\Gr$}{Gr}}
\label{sec:proofA1A2}
In this final section we prove Corollaries \ref{cor:A1} and \ref{cor:A2}, and we give an application exhibiting $\Gr$ as an $\infone$-subcategory of the (large) $\infone$-category $\Cospan(\cS)$ of cospans between spaces.

\subsection{Proof of the corollaries of Theorem \ref{thm:A}}

\begin{proof}[Proof of Corollary \ref{cor:A1}]
The object $\Gr=|\bGr|_2\in\CAlg(\Catinfone)$ represents the restriction of the functor $\GL\colon\CAlg(\Catinftwo)\to\cS$ to the full $\infone$-subcategory $\CAlg(\Catinfone)$ of $\CAlg(\Catinftwo)$, so our goal is to provide an equivalence 
\[
\GL|_{\CAlg(\Catinfone)}\simeq\CFrob.
\]
We start by providing a natural transformation between the two functors: we send a graph-like structure $(\fI,\ft,\fe,\beta)$ on $\cC\in\CAlg(\Catinfone)$ to the pair $(\fI,\ft)$. We need to check that this pair is indeed a commutative Frobenius algebra in $\cC$. Let $a:= \fI(\ul1)$ and let $c:=\ft\circ\fI(\mu)$. To prove that $(a,c)$ is a duality datum for $a$, we need to prove that there exists $u$ and homotopies as in Definition \ref{defn:duality}. We let $u:=\fe\colon\one\simeq\fI(\ul0)\to a\otimes a \simeq\fI(\ul2)$; then the 2-morphism $\beta$, which is invertible as $\cC\in\Catinfone$, provides an equivalence $(c\otimes\Id_a)\circ(\Id_a\otimes u)\simeq\Id_a$. Now let $\tau\colon a^{\otimes2}\to a^{\otimes2}$ denote the map swapping the factors, and let $\tau':=(\tau\otimes\Id_a)(\Id_a\otimes\tau)\colon a^{\otimes 3}\to a^{\otimes3}$ denote the map permuting cyclically the factors. Since both $c$ and $u$ are $C_2$-equivariant morphisms, we have equivalences $c\simeq c\tau$ and $u\simeq\tau u$; combining these we get a chain of equivalences 
\[
\begin{split}
(\Id_a\otimes c)(u\otimes\Id_a)&\simeq
(\Id_a\otimes c\tau)(\tau u\otimes\Id_a)\simeq
(\Id_a\otimes c)(\Id_a\otimes\tau)(\tau\otimes\Id_a)(u\otimes\Id_a)\\
&\simeq(\Id_a\otimes c)\tau'\tau'(u\otimes\Id_a)\simeq(c\otimes\Id_a)\circ(\Id_a\otimes u)\simeq\Id_a.
\end{split}
\]

We next prove that the constructed natural transformation $\GL|_{\CAlg(\Catinfone)}\Rightarrow\CFrob$ is an equivalence. For this, we will show that for $\cC\in\CAlg(\Catinfone)$ and for $(\fI,\ft)\in\CFrob(\cC)$, the fibre at $(\fI,\ft)$ of the map $\GL(\cC)\to\CFrob(\cC)$ is contractible. The mentioned fibre is the space of ways to upgrade $(\fI,\ft)$ to a graph-like structure. For simplicity we set again $a:=\fI(\ul1)$ and denote by $\mu$ also the morphism $\fI(\mu)\colon a^{\otimes2}\to a$. We start by showing that the space of choices of the following two data is contractible:
\begin{itemize}
\item a \emph{non-$C_2$-equivariant} morphism $\fe\colon\one\to a^{\otimes2}$;
\item a homotopy $\beta$ connecting $(\ft\mu\otimes\Id_a)(\Id_a\otimes\fe)$ to $\Id_a$.
\end{itemize}
Setting $x=z=a$ and $y=\one$ in Lemma \ref{lem:dualityequivalence} we obtain an equivalence of spaces $\cC(\one, a^{\otimes 2})\overset{\simeq}{\to}\cC(a,a)$, which makes the choice of a point $\fe$ in the first space equivalent to the choice of its image $(\ft\mu\otimes\Id_a)(\Id_a\otimes\fe)$ in the second space; the joint choice of the latter point together with a path $\beta$ joining it to the given point $\Id_x$ is evidently contractible.

We next prove that, given a completely non-$C_2$-equivariant datum $(\fI,\fe,\ft,\beta)$, the space of choices of the following two data is also contractible:
\begin{itemize}
\item a $C_2$-equivariant structure on $\fe$;
\item a $C_2$-equivariant structure on the homotopy $\tbeta:=\ft\mu(\beta\otimes\Id_a)$ connecting the $C_2$-equivariant morphisms $\ft\mu((\ft\mu\otimes\Id_a)(\Id_a\otimes\fe)\otimes\Id_a)$ and $\ft\mu$. Here the first morphism carries a $C_2$-equivariant structure by being equivalent to $(\ft\mu)^{\otimes\ul2}(\Id_a\otimes\fe\otimes\Id_a)$, where the $C_2$-equivariant structure on $(\ft\mu)^{\otimes\ul2}$ swaps the two tensor factors, and the $C_2$-equivariant structure on $(\Id_a\otimes\fe\otimes\Id_a)$ swaps the first and last tensor factors, and uses the $C_2$-equivariant structure on $\fe$ chosen above.
\end{itemize}
Setting $x=a^{\otimes\ul2}$ and $y=z=\one$ in Lemma \ref{lem:dualityequivalence}, we obtain a $C_2$-equivariant equivalence $\cC(\one,a^{\otimes\ul2})\overset{\simeq}{\to}\cC(a^{\otimes\ul2},\one)$; here we put the trivial $C_2$-action on $\one$ and the one swapping the two tensor factors on $a^{\otimes\ul2}$; both maps in the composite of Lemma \ref{lem:dualityequivalence} have an evident $C_2$-equivariant structure. Therefore a choice of a $C_2$-equivariant structure on $\fe\in\cC(\one,a^{\otimes\ul2})$ is equivalent to a choice of a $C_2$-equivariant structure on its image $(\ft\mu)^{\otimes\ul2}(\Id_a\otimes\fe\otimes\Id_a)$; the joint choice of the latter $C_2$-equivariant structure and of a $C_2$-equivariant structure on the path $\tbeta$ joining the point $(\ft\mu)^{\otimes\ul2}(\Id_a\otimes\fe\otimes\Id_a)$ to the given $C_2$-equivariant point $\ft\mu$ is evidently contractible.
\end{proof}
The argument of proof of Corollary \ref{cor:A2} as a consequence of Corollary \ref{cor:A1} is due to Barkan--Steinebrunner; it is included here for the sole purpose of completeness.
\begin{proof}[Proof of Corollary \ref{cor:A2}]
By Corollary \ref{cor:A1}, it suffices to construct a natural equivalence $\CFrob|_{\CAlg(\cS)}\overset{\simeq}{\to}\Omega$ of functors $\CAlg(\cS)\to\cS$.
As a natural transformation, this is given by sending $(\fI,\ft)\in\CFrob(M)$ to the composite path
\[
\ft\circ\fI(\ul0\to\ul1)\colon\fI(\ul0)\to\fI(\ul0).
\]

We next prove that the constructed natural transformation is an equivalence. We first observe that, since $|\Fin|\simeq*$, the space $\Fun^\otimes(\Fin,M)\simeq\Fun^\otimes(|\Fin|,M)$ is contractible. Consequently, $\CFrob(M)$ is equivalent to the space of commutative Frobenius algebra structures on the constant symmetric monoidal functor $\Fin\to M$ at $\one\in M$, i.e. of morphism $c\colon \one\simeq \one\otimes \one\to \one$ such that $(\one,c)$ is a duality datum for $c$. The fact that $M$ is a space ensures that every endomorphism of $\one\in M$ makes $(\one,c)$ into a duality datum. This concludes the proof that $\CFrob(M)\simeq\Omega M$.
\end{proof}

\subsection{Geometric realisation of graphs}
\label{subsec:geomrealgraphs}
\begin{defn}
We denote by $\Cospan(\cS)\in\Catinfone^{\fU_2}$ the large $\infone$-category of cospans between spaces. Its objects are spaces; for $X,Y\in\cS$, the morphism space from $X$ to $Y$ is the (large) core groupoid $(\cS_{X\sqcup Y/})$, i.e. the moduli space of cospans of spaces of the form $X\to W\ot Y$. Composition in $\Cospan(\cS)$ is given by taking pushouts. We endow $\Cospan(\cS)$ with the symmetric monoidal structure given by disjoint unions of spaces and (pointwise) of cospans. 
\end{defn}
We have a symmetric monoidal functor $\cS\to\Cospan(\cS)$, sending $X\mapsto X$ and sending a map of spaces $X\to Y$ to the cospand $X\to Y=Y$.

\begin{defn}
\label{defn:Re}

By virtue of Corollary \ref{cor:A1}, we define a symmetric monoidal $\infone$-functor
$\Re\colon\Gr\to\Cospan(\cS)$ by providing a commutative Frobenius algebra $(\fI_\Re,\ft_\Re)$
in $\Cospan(\cS)$:
\begin{itemize}
 \item we let $\fI_\Re\colon\Fin\to\cD$ be the composition of the symmetric monoidal functors $\Fin\hto\cS\to\Cospan(\cS)$.
 \item we let $\ft_\Re\colon \fI_\Re(\ul 1)\to \fI_\Re(\ul 0)$ be the cospan $\ul 1\to\ul 1\ot\ul 0$.
 \end{itemize}
\end{defn}
 The fact that $(\fI_\Re,\ft_\Re)$ is a commutative Frobenius algebra is witnessed by the following data
 \begin{itemize}
 \item we let $u\colon \fI_\Re(\ul 0)\to \fI_\Re(\ul 2)$ be the cospan $\ul 0\to\ul 1\ot\ul 2$;
 \item letting $c=\ft_\Re\circ\fI_\Re(\mu)$, we have (essentially unique) equivalences of morphisms $ (c\sqcup\Id_{\ul 1})\circ(\Id_{\ul 1}\sqcup u)\simeq\Id_{\ul 1}$ and $ (\Id_{\ul 1}\sqcup c)\circ(u\sqcup\Id_{\ul 1})\simeq\Id_{\ul 1}$ in $\Cospan(\cS)$: we use here that in both cases the source and the target of the equivalence are equivalent to the cospan $\ul 1\to\ul 1\ot\ul 1$, which is a terminal object in $\cS_{\ul 1\sqcup\ul1/}$.
\end{itemize}
\begin{rem}
The notation ``$\Re$'' is similar to the one used in Definition \ref{defn:ReG}
to denote the cell complex $\Re(G)$ constructed from a gaf $G$: in fact, up to replacing ``cell complexes'' with the $\infty$-category $\cS$, we have that $\Re$ sends the 1-morphism $G\colon B\to A$ in $\Gr$ to the cospan of spaces $A\to\Re(G)\ot B$.

To see this formally, let $\bCospan(\Top)$ denote the (large) symmetric monoidal 2-category of cospans of topological spaces: its objects are topological spaces; for $X,Y\in\Top$, the (large) 1-category $\bCospan(\Top)(X,Y)$ is the 1-category $\Top_{X\sqcup Y/}$; horizontal composition is given by taking pushouts in $\Top$; the monoidal product is given by taking disjoint unions. We may similarly introduce a symmetric monoidal $\inftwo$-category $\bCospan(\cS)$; the (product preserving, hence symmetric monoidal) localisation at weak equivalences $\Top\to\cS$ induces a symmetric monoidal $\inftwo$-functor $\bCospan(\Top)\to\bCospan(\cS)$.

We may now define a symmetric monoidal 2-functor $\tilde\Re\colon\bGr\to\Cospan(\Top)$ as follows: at the level of objects, we send each finite set to itself, considered as a topological space; for $B,A\in\Fin$ the functor $\tilde\Re\colon\bGr(B,A)\to\Cospan(\Top)(B,A)$ sends $G$ to the cospan of topological spaces $B\to\Re(G)\ot A$, and sends $f\colon G\to G'$ to the evident weak equivalence of topological spaces $\Re(G)\to\Re(G')$ collapsing 1-cells corresponding to edges collapsed along $f$, and restricting to a linear identification on 1-cells that are not collapsed. The composite symmetric monoidal $\inftwo$-functor
\[
\bGr\xrightarrow{\tilde\Re}\bCospan(\Top)\to\bCospan(\cS)
\]
factors through $\Cospan(\cS)\simeq(\bCospan(\cS))^\twosimeq$, and it agrees with $\Re$, as can be checked by comparing the associated graph-like structures, or even more directly, the associated commutative Frobenius algebras.
\end{rem}

We conclude the article with the following theorem.
\begin{thm}
\label{thm:Reinclusion}
The functor $\Re\colon\Gr\to\Cospan(\cS)$ is faithful, in the sense that it induces an inclusion of spaces at the level of core groupoids and at the level of morphism spaces. More precisely, $\Re$ exhibits $\Gr$ as the (small) $\infone$-subcategory of $\Cospan(\cS)$ whose objects are spaces homotopy equivalent to finite sets, and whose morphisms from $B$ to $A$ are those cospans of spaces $B\ot W\to A$ such that $W$ is homotopy equivalent to a finite cell complex of dimension at most 1. 
\end{thm}
\begin{proof}
Given finite sets $B,A$, representing objects in $\Gr$, our goal is to compute the morphism space $\Gr(B,A)$; by definition of $\Gr$, this is the classifying space $|\bGr(B,A)|$ of the morphism category $\bGr(B,A)$ from Definition \ref{defn:GrAB}. More precisely, our goal is to prove that the map $|\bGr(B,A)|\to(\cS_{A\sqcup B/})^\simeq$ induced by the functor $\Re$ is an inclusion of spaces.

Let $\bGr(B,A)'$ denote the full subcategory of $\bGr(B,A)$ spanned by gaf's $G$ satisfying the following properties:
\begin{itemize}
\item each element of $A$ has valence 1 in $G$ (it is an ``attaching leaf'');
\item the map $\rho\colon B\to A\sqcup V$ is injective, has image inside $V$, and all elements of its image are leaves.
\end{itemize}
Now recall the gaf $G_{\beta,1}$ from Subsection \ref{subsec:propertiesGr}, and consider the endofunctor
\[
\xi:=\pa{\coprod_BG_{\beta,1}}\circhor-\circhor\pa{\coprod_AG_{\beta,1}}\colon\bGr(B,A)\to\bGr(B,A).
\]
Then $\xi$ has image inside $\bGr(B,A)'$; moreover we have a natural transformation $\pa{\coprod_B\beta_\bGr}\circhor-\circhor\pa{\coprod_A\beta_\bGr}\colon\xi\to\Id_{\bGr(B,A)}$. It follows that $|\bGr(B,A)|\simeq|\bGr(B,A)'|$.

We next consider the full subcategory $\bGr(B,A)''$ of $\bGr(B,A)'$ spanned by gaf's $G$ whose only vertices of valence 1 are those contained in $A$ or in the image of $\rho$. The inclusion $\bGr(B,A)''\hto\bGr(B,A)'$ admits a left adjoint $\bGr(B,A)'\to\bGr(B,A)''$, sending a gaf $G$ to the gaf $G'$ obtained by repeatedly collapsing all non-marked leaves of $G$ and of the gaf's that one obtains througout the procedure. It follows that $|\bGr(B,A)'|\simeq|\bGr(B,A)''|$.

Let now $(B\to W\ot A)\in(\cS_{B\sqcup A/})^\simeq$ be a cospan of spaces with $W$ equivalent to a finite cell complex of dimension $\le1$, and let $\bGr(B,A)''_W$ denote the fibre of the $\infone$-functor $\Re\colon\bGr(B,A)''\to(\cS_{B\sqcup A/})^\simeq$ at $W$. Our goal is to prove that $|\bGr(B,A)''_W|$ is contractible.
If $W$ splits as a disjoint union of connected spaces $\coprod_i W_i$, letting $A_i:=A\cap W_i$ and $B_i:=B\cap W_i$ we have a product decomposition $\bGr(B,A)''_W\simeq\prod_i\bGr(B_i,A_i)''_{W_i}$, inducing an analogous product decomposition on classifying space. We may therefore assume that $W$ is connected.

We first consider the special case $W\simeq *$. Let $*\hto(\cS^\simeq_{B\sqcup A/})^\simeq$ denote the component of $B\to*\ot A$. Then the fibre product $\bGr(B,A)''\times_{\cS_{B\sqcup A/}}*$ is the subcategory of $\bGr(B,A)''$ spanned by the single gaf $G$ having precisely one inner vertex not marked by $B$; the classifying space of this category is contractible and hence also $|\bGr(B,A)''_*|\simeq*$.

We next consider the special case $W\simeq S^1$ and $A=B=\emptyset$. Let $\bB\hAut(S^1)\subset\cS^\simeq$ denote the component of $S^1$, and let $\bB\hAut^+(S^1)\simeq\mathbb{C}P^\infty$ denote the double (and universal) cover of the latter.
Then the fibre product $\bGr(\emptyset,\emptyset)''\times_{\cS^{\simeq}}\bB\hAut(S^1)$ is equivalent to the unstraightening of the cocartesian fibration over $\bB C_2$ corresponding to the $C_2$-action on the cyclic category $\Lambda$ by reflections. Since $|\Lambda|\simeq\bB\hAut^+(S^1)$, we obtain that $|\bGr(\emptyset,\emptyset)''_{S^1}|$ is indeed contractible.

From now on we assume $\chi(W,B\sqcup A)<0$. This implies in particular that the component $\bB\hAut_{B\sqcup A/}(W)\subset(\cS_{B\sqcup A/})^\simeq$ containing $B\to W\ot A$ is aspherical, and hence that $\subset\bGr(B,A)''_W$ is a plain category. It is instructive to think of objects in $\bGr(B,A)''_W$ as gaf's with a marking by $W$, compatible with the marking by $B$ and the inclusion of $A$.

Let $\bGr(B,A)'''\subset\bGr(B,A)''_W$ denote the full subcategory spanned by objects whose image in $\bGr(B,A)''$ are gaf's $G$ satisfying the following property: there is no ``bridge'' in $G$, where a bridge is an edge of $G$ that is not a leaf (i.e., neither of its endpoints is in $A$ or in the image of $\rho$), and it separates $G$ into two non-empty gaf's. The inclusion $\bGr(B,A)'''\subset\bGr(B,A)''_W$ admits a left adjoint $\bGr(B,A)''_W\to\bGr(B,A)'''$, sending a gaf $G$ to the gaf obtained by collapsing all bridges of $G$. It follows that $|\bGr(B,A)'''|\simeq|\bGr(B,A)''_W|$.

Finally, we invoke the work of
Culler--Vogtmann \cite{CullerVogtmann}, in the case $A=B=\emptyset$ and $\chi(W)<0$, and the generalisation of
Hatcher--Vogtmann
\cite[Section 3]{HatcherVogtmann} in the remaining cases. If $n$ denotes the rank of $W$ and if $s=\#A+\#B$, then the (topological) geometric realisation of the simplicial set given by the nerve of $\bGr(B,A)'''$ is homeomorphic to the spine of outer space, denoted $K$ in \cite{CullerVogtmann}, or to the spine of the space $A_{n,s}$ from \cite{HatcherVogtmann}; in both cases it is a contractible space.
\end{proof}

\bibliography{Bibliographygraphs.bib,Bibliographyother.bib}{}
\bibliographystyle{plain}

\end{document}